%% file: mdiag.tex
\documentclass[11pt, oneside, article]{memoir}

\input{preamble-mdiag}

\title{Manifold diagrams and tame tangles}

\author{Christoph Dorn \& Christopher L. Douglas}
\date{}

\begin{document}

\maketitle


\begin{abstract}
Diagrammatic notation has become a ubiquitous computational tool; early examples include Penrose's graphical notation for tensor calculus, Feynman's diagrams for perturbative quantum field theory, and Cvitanovi\'c's birdtracks for Lie algebras.  Category theory provides a robust framework in which to understand the nature of such diagrams, and Joyal and Street formalized this framework by introducing string diagrams, governed by the syntax of monoidal 1-categories.  The notion of `manifold diagrams' generalizes string diagrams to higher dimensions, and can be interpreted in higher-categorical terms by a process of geometric dualization.  The closely related notion of `tame tangles' describes a well-behaved class of embedded manifolds that can likewise be interpreted categorically.  In this paper we formally introduce the notions of manifold diagrams and of tame tangles, and show that they admit a combinatorial classification, by using results from the toolbox of framed combinatorial topology.  We then study the stability of tame tangles under perturbation; the local forms of perturbation stable tame tangles provide combinatorial models of differential singularities.  As an illustration we describe various such combinatorial singularities in low dimensions.  We conclude by observing that all smooth 4-manifolds can be presented as tame tangles, and conjecture that the same is true for smooth manifolds of any dimension.
\end{abstract}

\vspace{2\baselineskip}
{\scriptsize
{\nid\scshape Mathematical Institute, University of Oxford, Oxford OX2 6GG, United Kingdom} \\
{\nid\textit{Email addresses:} \texttt{dorn@maths.ox.ac.uk \& cdouglas@maths.ox.ac.uk}} \\

}



\newpage

\renewcommand{\baselinestretch}{1.1}\small
\setcounter{tocdepth}{2}
\tableofcontents*
\renewcommand{\baselinestretch}{1}\normalsize

\newpage

\renewcommand{\thesection}{I.{\arabic{section}}}
\renewcommand{\thesubsection}{\thesection.{\arabic{subsection}}}
\renewcommand\thefigure{I.\arabic{figure}}

\chapter*{Introduction}
\addcontentsline{toc}{chapter}{Introduction}

%
%

Examples of diagrammatic notation include the Penrose graphical notation for tensor calculations \cite{penrose1971applications}, and its adaptions to spin and tensor networks \cite{rovelli1995spin} \cite{vidal2009entanglement},  Feynman diagrams for perturbative quantum field theory \cite{t1974diagrammar},
birdtracks for the classification of Lie algebras \cite{cvitanovi2008}, Petri nets \cite{rozenberg1996elementary} \cite{baez2018quantum}, Bayesian networks \cite{fong2013causal}, and many more. Higher category theory provides a unifying context for these examples, by providing a framework for compositional structures based on the mechanism of `composing morphisms in diagrams'.  Most of the preceding examples can be translated into special instances of diagrams in 2-categories and variations thereof---the required formal diagrammatic notation for 2-categories is known as `string diagram calculus' \cite{joyal1991geometry}.  The concept of manifold diagrams arises when generalizing string diagrams to higher dimensions; while string diagrams represent compositions of morphisms in 2-categories, manifold diagrams describe compositions in $n$-categories.  Intuitively, the notion of manifold diagrams is geometrically dual to the more familiar representation of such compositions by diagrams of cells.

%
%

While on a heuristic level, both the manifold diagram and cell diagram approaches to representing compositions in higher category theory have been known for many decades, essentially all current approaches to the foundations of higher categories are based on diagrams of cells.  In contrast, manifold diagrams have only been investigated in low dimensions.  One reasons for this discrepancy may have been the difficult of formalizing a key geometric feature of manifold diagrams, namely that manifold strata in manifold diagrams can be deformed by isotopies and such isotopies should themselves be higher manifold diagrams.  In cellular approaches, such deformations have received little attention because geometric dualization converts isotopies into certain mysterious, but still essential, `degenerate' diagrams of cells whose role is less evident.  Despite much work on cellular models of higher categories, the problem of formalizing the notion of manifold diagrams has previously remained unaddressed.

%
%

Even in the absence of a formal definition, manifold diagrams are actively being used as a rough conceptual tool in areas such as quantum algebra, knot theory, topological quantum field theory, and related areas \cite{yetter2001functorial} \cite{kock2004frobenius} \cite{baez2011prehistory} \cite{pstrkagowski2014dualizable} \cite{ara2017a} \cite{douglas2020dualizable}.  The advantage that manifold diagrams bring to these areas is that their isotopies make visible certain higher categorical coherences and thereby allow the expression of non-trivial homotopical-algebraic calculations.  In fact, the idea of encoding higher categorical coherences in isotopies is one origin of the tangle hypothesis \cite{baez1995higher}, a central result in the study of topological quantum field theories \cite{lurie2009classification}.  Indeed, tangles can be thought of as manifold diagrams in which manifold strata are amalgamated into a single embedded manifold.  A formal theory of manifold diagrams should therefore also provide a toolset for studying tangles and their local forms.

%
%

In this paper we will formalize both the notions of manifold diagrams and of `tame' tangles in a unifying combinatorial-topological language.  We will explain how this language naturally describes the geometric dualization between cell and manifold diagrams.  The language also provides notions of bundles of manifold diagrams and tame tangles, and thus allows us to study the stability of tame tangles under perturbations.  We will discuss how the local forms of perturbation stable tame tangles provide combinatorial models of differential singularities.  

\section{Overview} \label{sec:overview}

Consider the six `diagrams' in dimensions  $n = 1, 2,3$ shown in \cref{fig:manifold-diagrams-inductive-idea}. As we will see later, all six are in fact examples of manifold $n$-diagrams (in particular, the 2-dimensional diagrams are examples of ordinary string diagrams). Observe that the diagrams are related across dimensions by projecting onto one another (depicted in the same figure by sideways projections). This is indicative of an `inductive structure' in the data of manifold diagrams---and our definitions will precisely exploit this structure.

\begin{figure}[ht]
    \centering
    \def\svgwidth{1\columnwidth}
    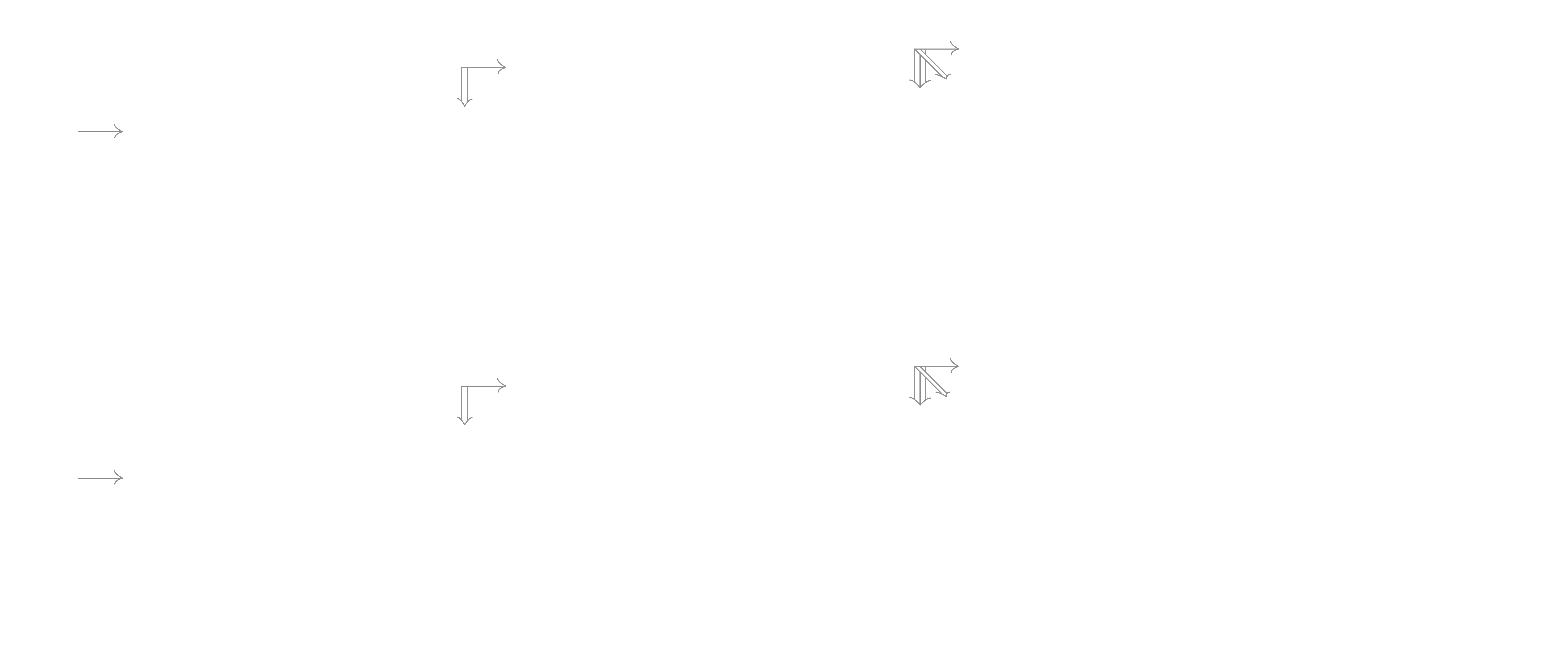

    \caption[Manifold diagrams]{Manifold diagrams and their projections}
    \label{fig:manifold-diagrams-inductive-idea}
\end{figure}

First, we have to formalize the notion of `projections': we do so by introducing a notion of `constructible fiber bundles with framed stratified 1-dimensional fibers'---we refer this notion as \emph{1-mesh bundles}. Inductively stacking 1-mesh bundles on top of one another, defines the notion of \emph{$n$-mesh bundles}. If the base is trivial we simply speak of \emph{$n$-meshes}. Note, however, what we end up with here is \emph{not} yet a manifold diagram: instead, an $n$-mesh will look like a complex of (framed regular) cells as shown in \cref{fig:meshes-cellulate-manifold-diagrams}, on the right. The role of such $n$-meshes is to `cellulate' (think `triangulate', but with framed regular cells) manifold diagrams; this is illustrated in the same figure for two of our earlier examples of manifold diagrams. Let us call a stratification \emph{tame} if it can be cellulated by a mesh.

\begin{figure}[ht]
    \centering
    \def\svgwidth{1\columnwidth}
    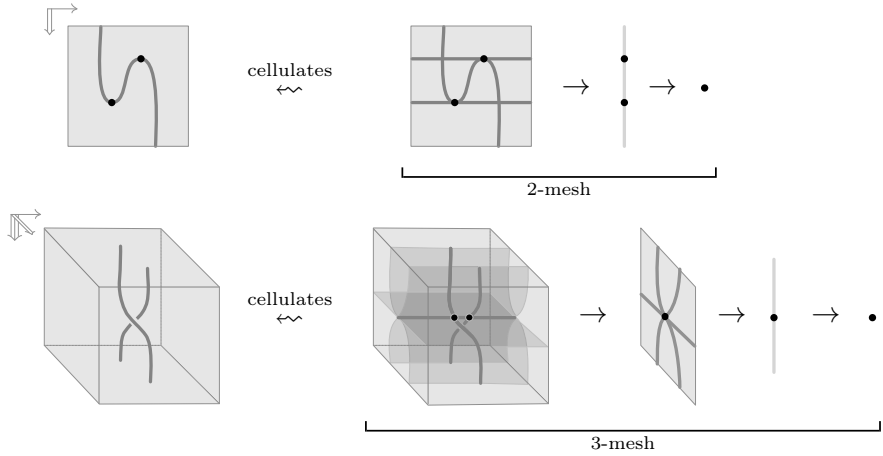

    \caption[Meshes cellulate manifold diagrams]{Meshes cellulate manifold diagrams}
    \label{fig:meshes-cellulate-manifold-diagrams}
\end{figure}

A manifold diagram can now be defined as a tame stratification of the open $n$-cube $\II^n \equiv (-1,1)^n$ satisfying an additional condition which we call framed conicality: roughly speaking, this means every point has a stratified neighborhood that is the product of a `tangential' $\lR^k$ and a `normal' stratified cone appropriately compatible with the ambient framing. (Framed conicality is therefore a `framed' variation of the usual notion of conicality for stratifications.) The framed conicality condition in this definition is illustrated in \cref{fig:framed-conicality-condition}.
\begin{figure}[ht]
    \centering
    \def\svgwidth{1\columnwidth}
    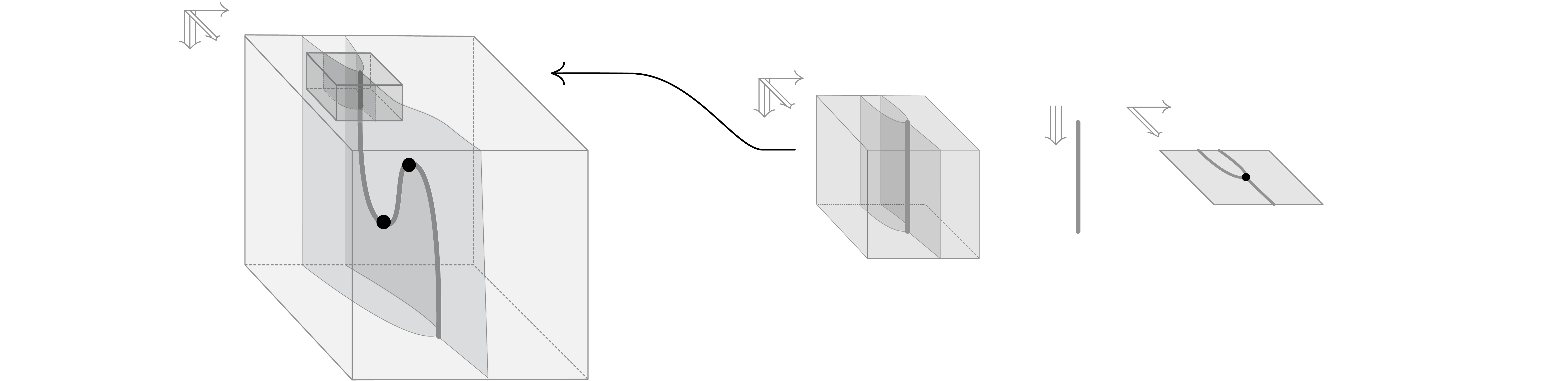

    \caption[The framed conicality condition]{The framed conicality condition}
    \label{fig:framed-conicality-condition}
\end{figure}

\begin{introdef} A \emph{manifold $n$-diagram} is a framed conical tame stratification of the open $n$-cube.
\end{introdef}

Closely related to our definition of manifold diagrams, the second central notion of this paper is the notion of tame tangles. For us, the role of tame tangles is two-fold: on one hand, we will show that tame tangles can be canonically refined by manifold diagrams, and the manifold diagrams obtained from tame tangles in this way can be thought of as describing composites of (appropriately `structured') dualizable morphisms---this line of thinking, of course, follows the tangle hypothesis \cite{baez1995higher}. On the other hand, however, the notion of tame tangles by itself will provide a new bridge between classical manifold theory and our framework of framed combinatorial topology \cite{fct}, and many of our later observations and conjectures will aim to better understand this bridge, by relating tame tangles to classical piecewise linear (PL) and smooth manifolds.

Formally, tame $m$-tangles in dimension $n$ will be defined as topological manifolds embedded in the open $n$-cube, with the requirement that this embedding be tame (i.e.\ define a tame stratification) and `framed transversal': roughly speaking, the latter condition requires that each point on the tangle manifold has a neighborhood that is the product of a `transversal' part which looks like $\lR^k$, and a `non-transversal' part, which looks like the cone of an $(m-k-1)$-sphere embedded in the boundary of the $(n-k)$-cube. (Note the conceptual similarity to our earlier description of framed conicality.) We arrive at the following definition.

\begin{introdef} A \emph{tame $m$-tangle in dimension $n$} is an $m$-manifold tamely embedded in the open $n$-cube which is framed transversal.
\end{introdef}

So far all of the above is phrased purely in the language of stratified topology. However, as discussed in detail in \cite[\S5]{fct}, `framed tame topology and framed combinatorial topology are locally the same thing'. We will use results from \emph{loc.cit.}\ to write out a fully combinatorial story paralleling the above discussion of manifold diagrams and tame tangles.

We briefly outline this parallel story. The combinatorial counterparts of meshes are called \emph{$n$-trusses}: these will be introduced as the structured (namely, combinatorially framed) entrance path posets of meshes. As we recall, there is a weak equivalence that takes $n$-meshes to their `fundamental' $n$-trusses. This can be generalized to a construction that takes manifold diagrams (resp.\ tame tangles) to their `fundamental combinatorial manifold diagrams' (resp.\ their `fundamental combinatorial tangles'). In this way we obtain a combinatorialization (and thereby a classification) of manifold diagrams and tame tangles as follows.

\begin{introthm}[Combinatorialization of manifold diagrams resp.\ tame tangles] Manifold diagrams up to framed stratified homeomorphism are classified by combinatorial manifold diagrams. Tame tangles up to framed stratified homeomorphism are classified by combinatorial tangles.
\end{introthm}

\nid The theorem is proven (in the respective cases of manifold diagrams and tame tangles) in {\cref{thm:comb-mdiag}} resp.\ {\cref{obs:combinatorializing-tame-tangles}.

    As a first immediate application of these combinatorialization results we will find that manifold diagrams and tame tangles are canonically endowed with PL structures (see \cref{ssec:mdiag-PL} resp.\ \cref{ssec:PL-struct}). 

    As a second application, we will construct geometric duals for manifold diagrams. This relies on the combinatorial theory of trusses being naturally self-dual. Via the previous theorem, this duality carries over to the case of manifold diagrams, yielding a dual notion of `cell diagrams'. The process of dualizing a manifold diagram is illustrated in \cref{fig:a-manifold-diagram-and-its-dual-pasting-diagram}. (In \cref{ssec:cell-diag-interpretation} we discuss an interpretation of cell diagrams in familiar higher categorical terms as `pasting diagrams of higher morphisms'.)

\begin{introobs}[Manifold diagrams dualize to cell diagrams] Up to framed stratified homeomorphism, manifold and cell diagrams are in correspondence by geometric dualization.
\end{introobs}

\begin{figure}[ht]
    \centering
    \def\svgwidth{1\columnwidth}
    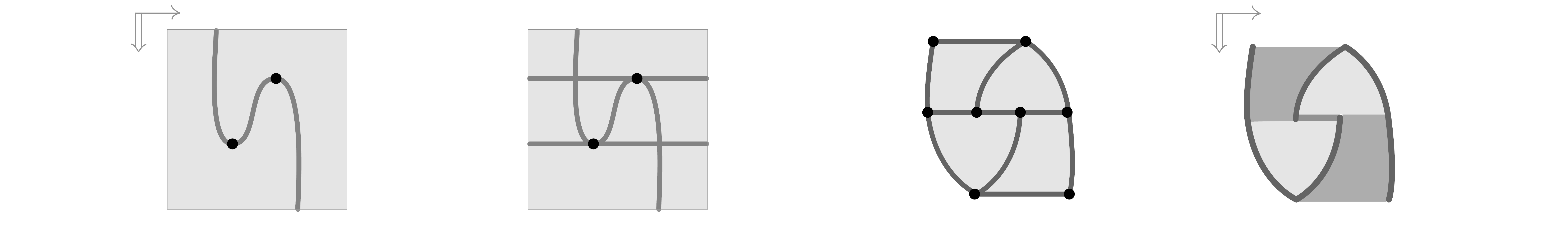

    \caption[Dualizing a manifold diagrams to cell diagrams]{Dualizing a manifold diagram to its dual cell diagram}
    \label{fig:a-manifold-diagram-and-its-dual-pasting-diagram}
\end{figure}

The combinatorialization of tangles will also play a crucial part in our study of stability of tangles. For this, we will first use the fact that theory of meshes and trusses naturally accommodates notions of bundles, which allows us to generalize the above definition of tame tangles to a notion of tame tangle \emph{bundles} (the generalization similarly applies to manifold diagrams and lets us study `manifold diagram bundles', even though we will not do so here). Our main interest will lie with \emph{perturbations} of tame tangle which are bundles of tame tangles over the base stratification $\{0\} \subset [0,1]$ (fiber over $0$ is usually called `special', whereas the fiber over its complement $(0,1]$ is called `generic'). Perturbations allow us to describe the situation in which one tangle can be perturbed into a `more generic' tangle; importantly, to decide if one tangle is strictly more generic than another tangle, we will introduce a measure of complexity based on the combinatorialization of tangles. If a tangle cannot be perturbed into a strictly less complex tangle, then we will say it is \emph{perturbation stable}. We are particularly interested case of tangle \emph{singularities}, which are the local neighborhoods around critical points in a tangle (think for instance of a saddle point). In low dimensions, we will show that our definitions lead to the following classes of stable singularities.

\begin{introthm}[Stable tangle singularities in dimension {$\leq 4$}] Up to reflection and appropriate equivalence, the stable tangle singularities in dimensions less than 4 can be determined as follows.
    \begin{enumerate}
        \item There is one stable 1-tangle singularity in dim 2 and dim 3.
        \item There are three stable 2-tangle singularities in dim 3: these are shown in \cref{fig:stable-2-tangle-singularities-in-dimension-3}.
        \item There are four stable 2-tangle singularities in dim 4.
    \end{enumerate}
\end{introthm}

\begin{figure}[ht]
    \centering
    \def\svgwidth{1\columnwidth}
    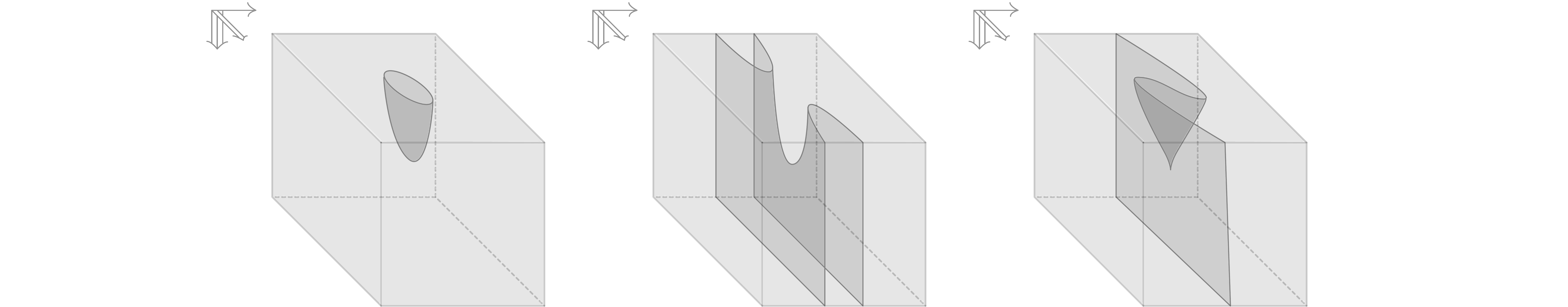

    \caption[Stable 2-tangle singularities in dimension 3]{The stable 2-tangle singularities in dimension 3 (up to reflection)}
    \label{fig:stable-2-tangle-singularities-in-dimension-3}
\end{figure}

\nid The three cases of the theorem will be addressed in {\cref{thm:stable-1-sing-dim-3}}, {\cref{thm:stable-2-sing}}, and {\cref{thm:stable-2-sing-dim-4}} respectively. Pictures of the cases 1 and 3 of the theorem are given in the main text: see \cref{fig:basic-tangles-singularities}, \cref{fig:stable-1-tangle-singularities-in-dimension-3} for case 1, and \cref{fig:three-stable-2-tangle-singularities-in-dimension-4} and \cref{fig:the-braid-eating-singularity} for case 2. There we also discuss the case of stable 3-tangle singularities in dim 4, see \cref{rmk:3-tang-in-dim-4}, in which case we claim that there are five classes singularities, see \cref{fig:stable-3-tangle-singularities}, but do not give a rigorous proof of this claim. Finally, for the case of tangle singularities in codimension 1, we will sketch how the classification could propagate into higher dimensions.

Our discussion will exhibit a surprising relationship between the combinatorics of tame tangles and the classical differential machinery of singularity theory. This will motivate us to make several conjectures about the relationship between framed combinatorial topology and smooth topology---the gist of these conjectures is summarized in the following hypothesis, and further discussion can be found in \cref{sec:classical-sings}.

\begin{introhyp}[Combinatorialization of smooth structures] Smooth structures on manifolds can be captured by tame tangles up to framed stratified homeomorphisms.
\end{introhyp}

\nid The hypothesis later reappears (in strengthened form) in three separate statements, respectively given in \cref{conj:tame-dense}, \cref{conj:smooth-faithfull} and \cref{conj:smooth-full}. Together with the earlier observation that tame tangles and manifold diagrams have canonical PL structures, the hypothesis lends substance to the idea that different topological worlds may to some extent be `unified' in the combinatorial language of framed combinatorial topology; we briefly return to the idea in \cref{sec:promises}.

\begin{intrormk}[Combinatorialization of smooth 4-structures] The above hypothesis can be made precise and holds for manifolds of dimension $\leq 4$ (see \cref{obs:comb-of-dim4-struct}). A particular consequence is that all smooth structures of 4-manifolds can be combinatorially represented as tame tangles; in the case of the 4-sphere, this allows smooth structures to be visualized as tame 4-tangles in dimension 5, leading to a new `diagrammatic' approach to such structures (see \cref{rmk:SPC4-pictorially}).
\end{intrormk}

\subsection*{Disclaimer} What this paper does not contain is any attempt at using manifold diagrams (or their dual diagrams of cells) to build geometric models of higher categories (see \cref{sec:promises} for further discussion). Indeed, due to the presence of isotopies, this becomes a somewhat more delicate issue than simply imposing conditions on presheafs of cells, and goes beyond the scope of the present work.

\section{Related work} \label{sec:related-work}

\subsection*{Manifold diagrams} Manifold diagrams have previously been described in low dimensions. Joyal and Street \cite{joyal1991geometry} gave the first definition of manifold 2-diagrams (also called `string diagrams' in this dimension). McIntyre and Trimble \cite{trimble1999} worked towards a definition of manifold 3-diagrams (also called `surface diagram' in this dimension) but the project was never completed. Later, manifold 3-diagrams for Gray categories (with duals) were defined in both \cite{hummon2012surface} and \cite{barrett2012gray}. In dimension $4$, there has been substantial work on the `diagrammatic algebra' of immersed and knotted surfaces \cite{carter1996diagrammatics} \cite{carter2011excursion}, however, this was not aimed directly at defining manifold 4-diagrams.
A `combinatorial presentation' of 4-diagrams in terms of generators and relations (more specifically, in terms of elementary 4-coherences and their 5-diagram identities) was discussed by Bar and Vicary in \cite{bar2017data}, and was in part based on earlier work by \cite{crans2000braiding}. The thesis \cite{thesis} described a combinatorial datastructure called `singular cubes' for the purpose of capturing manifold diagrams in all dimensions. This provided a precursor to the notion of `trusses' introduced in \cite{fct}, which then also introduced the stratified topological notion of `meshes', and formalized the theory of both trusses and meshes in a geometrically self-dual way. Sketch definitions of manifold diagrams and tame tangles were given in \emph{loc.cit.}, and we expand upon these sketches here.

\subsection*{Tangles and singularities} The notion of `ordinary' 1-dimensional tangles in dimension 3 has been an important object of study in knot theory for several decades \cite{adams2004knot} (the word `tangle' itself goes back to \cite{conway1970enumeration}). It was soon realized that 1-tangles in this dimension bore close relation to category theoretic ideas, namely, to the notion of braided monoidal categories \cite{yetter1986markov} \cite{freyd1989braided} \cite{turaev1990operator}.
This relation was generalized to $m$-tangles in any dimension in the statement of the so-called `tangle hypothesis' by Baez and Dolan in \cite{baez1995higher}; while the statement was given only at a heuristic level (leaving much ``to be made precise'' as the authors noted), it sparked a flurry of new research in the surrounding areas of topological quantum field theories and quantum algebra \cite{kapustin2010topological} \cite{bartlett2015modular} \cite{teleman2016five} \cite{ayala2017factorization}, and a proof of the hypothesis was sketched in \cite{lurie2009classification}.
Full definitional accounts of higher dimensional tangles in the literature remained nonetheless sparse (and \emph{loc.cit.}\ only gives a sketch definition)---one encounters difficulties comparable to those in the case of manifold diagrams. Several descriptions were given in low dimensions; work on diagrams of knotted surfaces by Carter and others \cite{carter1997combinatorial} \cite{carter1998knotted} falls into this category. Likewise in dimension 2, a thorough treatment of cobordisms (tangles in infinite codimension) was given by Schommer-Pries \cite{schommer2009classification}, who describes a presentation of their category by generators and relation using Morse-Cerf theoretic methods.
The idea of `generators and relations' for categories of tangles (which already featured in the original work by Baez and Dolan) is in turn directly related to our discussion of singularities: any tame tangle can be perturbed into a tangle built only from perturbation stable singularities, and in this sense perturbation stable singularities precisely play the role of `generating tangles' (indeed, our study of stable 2-tangle singularities will recover results from \emph{loc.cit.}\ in the setting of tame tangles). In higher dimensions, the classical theory of smooth singularities roots in the works of Morse \cite{morse1925relations} and Cerf \cite{cerf1970stratification}, with works by Thom, Mather, Arnold, Wall and many other ensuing  \cite{mather1971stability} \cite{arnol1972normal} \cite{thom1976structural} \cite{arnold1985singularities} \cite{du1995geometry}. A good recent summary can be found in \cite{mond2020singularities}.

\section{Outlook} \label{sec:promises}

\subsection*{Towards geometric higher category theory} A core promise of the theory of manifold diagrams and their isotopies is to provide geometric semantics for the `coherences' of homotopy theory and higher category theory. For instance, the Eckmann-Hilton argument, which categorically can be expressed as sequence of coherences induced by a zig-zag of evaluations (as shown in \cref{fig:eckmann-hilton-vs-braid} on the left) can dually be expressed by an isotopy of two points in the plane (illustrated by its graph in $\lR^3$ in the same figure on the right)---this isotopy is also known as the `braid'.

\begin{figure}[ht]
    \centering
    \def\svgwidth{1\columnwidth}
    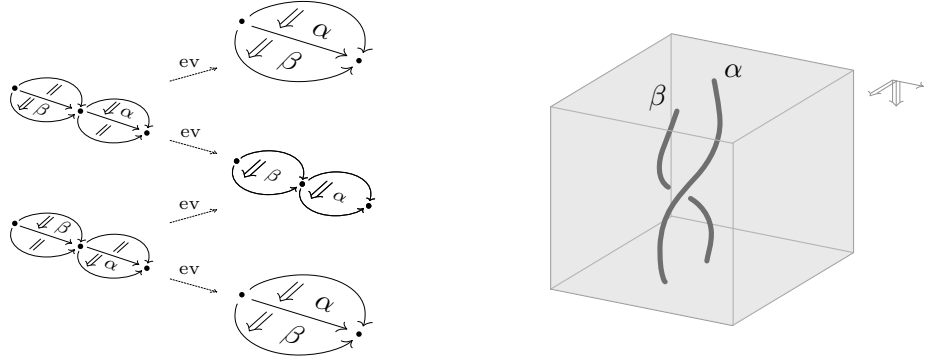

    \caption[Eckmann-Hilton argument as an isotopy]{Eckmann-Hilton argument for 2-cells $\alpha, \beta : \id \to \id$ witnessed by a zig-zag of evaluations (on the left) and by the braid isotopy (on the right)}
    \label{fig:eckmann-hilton-vs-braid}
\end{figure}

The heuristic idea of relating isotopies and coherences in this way has been long-known and it motivated, for instance, the early project \cite{trimble1999} to develop manifold 3-diagrams as a geometric model for Gray 3-categories \cite{gordon1995coherence}. (Gray 3-categories are a `maximally strict' variation of the notion of weak 3-categories, which is still though equivalent to the weak one; in other words, Gray 3-categories retain only homotopically `essential' coherences). One may, of course, also argue in the converse direction: understanding manifold diagrams and their isotopies could provide an understanding of essential coherences, allowing us to build geometric models of higher categories. We refer to this approach to constructing coherences as the paradigm of \emph{isotopy}. In contrast, most cellular models of higher categories in use are based on the paradigm of \emph{contraction} for the construction of coherences. The following summary specifies the difference between the two paradigms:
\begin{enumerate}
    \item \emph{Contraction}: A coherence between two $n$-diagrams can be constructed as a sequence of `contractions'; a contraction is an $(n+1)$-cell between two $n$-diagrams that exists whenever both diagrams can be obtained as evaluations of another $n$-diagram.
    \item \emph{Isotopy}: A coherence between two $n$-diagrams can be constructed as an $(n+1)$-diagram which (as a manifold $(n+1)$-diagram) is an isotopy between the two diagrams, i.e.\ it does not introduce new manifold strata.
\end{enumerate}
Comparing the two we note the following. First, implementing the contraction paradigm requires us to enforce the existence of structure (namely, of contractions) whereas, in the paradigm of isotopy, coherences are constructed only from already existing structure via the `primitive' mechanism of composing morphisms in diagrams. Arguably, this makes the latter the simpler paradigm to work with from a constructive foundational perspective. Moreover, individual contractions are `homotopically trivial' (the space of contractions of a given diagram is  `contractible'). In contrast, coherences in the isotopy sense may describe non-trivial equivalences, such as the braid isotopy. This makes manifold diagrams a framework suitable for studying the `non-trivial pieces' of general homotopical behavior.

\subsection*{A combinatorial language for smooth structures} Even if we weren't interested in models of higher category theory per se, there is a last, and (to us) central, promise of the theory of manifold diagrams and tame tangles. Namely, we will hypothesize that manifold diagrams and tame tangles enable the description of essential phenomena from the worlds of `tame topology', `PL topology' and `smooth topology' in a single unifying combinatorial-topological language, without the need to separately define concepts in the respective three worlds. This contrasts that in classical terms the topological, PL and smooth categories are built in fundamentally different languages. As an example of our claim, we will see that manifold strata in manifold diagrams, despite being defined in tame topological terms, carry unique PL and smooth structures. We will also see that tangle manifolds of tame tangle can faithfully represent all (compact) PL manifold structures. Finally, a central hypothesis of this paper is that tame tangles can also be used to represent (compact) smooth structures and diffeomorphism. The hypothesis, if true, means that the language developed here can be used to work with smooth structures in purely combinatorial terms.

\section{Reader's Guide}

The combinatorial-topological language for manifold diagrams that we introduce here will build on the theory of framed combinatorial topology developed in our book \cite{fct}. In \cref{ch:recollection}, we will recall central definitions and results from \cite{fct} (as well as basic notions from the theory of stratifications), which ensures that the present paper is as self-contained as possible. Note that even readers already familiar with this material might benefit from skimming the section, as our presentation and terminology differs slightly from that in the book. All mentions of $\infty$-categories can be ignored; ultimately, our results do not rely on them. 

In \cref{ch:mdiag}, as a first contribution of this paper, we give a formal definition of manifold diagrams in all dimensions, elaborating on a sketch given in \cite[\S5]{fct}. We will show that, while our definition is phrased in stratified topological terms, manifold diagrams can in fact be classified by purely combinatorial means. This combinatorial classification will also enable the translation of manifold diagrams into diagrams of (directed) cells by a process of geometric dualization.

In \cref{ch:tangles-and-sings}, we then introduce the related notion of tame tangles. We will show that tame tangle can be canonically refined to manifold diagrams, which makes the relation between the two notions precise. Analogous to manifold diagrams, tame tangles have a combinatorial counterpart which we exploit in order to study the stability of tame tangles under perturbations. We will classify perturbation stable tangle singularities in low dimensions. The fact that this bears resemblance to deep results from classical differential singularity theory will lead us to make several conjectures about the relation between the framed topology of tame tangles and smooth manifolds---these and other ideas discussed in the final section \cref{sec:classical-sings} are speculative, and are yet to be made rigorous.

\section*{Acknowledgements}

This work builds directly on our book of framed combinatorial topology \cite{fct}. We thank Jan Steinebrunner, Filippos Sytilidis, and Christoph Weis for questions and feedback during a preliminary reading group on the material in Oxford in spring 2021. CD would like to thank Lukas Heidemann for many insightful discussions. Both authors were partially supported by the EPSRC grant EP/S018883/1, ``Higher algebra and quantum protocols'', which was conceived and prepared in collaboration with Jamie Vicary and David Reutter. CLD was partially supported by a research professorship in the MSRI programs ``Quantum symmetries'' and ``Higher categories and categorification'' in the spring of 2020 which, after a Covid related break, was continued in June 2022 at the Unidad Cuernavaca del Instituto de Matem\'aticas UNAM.

\renewcommand{\thesection}{\thechapter.{\arabic{section}}}
\renewcommand{\thesubsection}{\thesection.{\arabic{subsection}}}
\renewcommand\thefigure{\thechapter.\arabic{figure}}

\chapter{Meshes and trusses} \label{ch:recollection}

We recall definitions of iterated `1-framed, stratified' bundles, which give rise to the notion of \emph{meshes} (\cite[\S4]{fct}). As outlined in the introduction, meshes provide cellulations of manifold diagrams (that is, they refine the manifold diagrams by cell complexes). The construction has a faithful combinatorial counterpart in the notion of \emph{trusses} (\cite[\S2]{fct}), which is obtained by describing iterated stratified bundles on the level of their fundamental categories.

\section{Recollections from stratified topology}

We briefly recall the basic notions of the theory of stratifications (and refer to \cite[App.\ B]{fct} for further details).

\begin{conv}[Posets as categories] Given a poset (or preorder) $(P,\leq)$ we regard it as a category with the same objects as $P$, and an arrow $x \to y$ between two objects if and only if $x \leq y$. 
\end{conv}

\subsection*{Stratifications and stratified maps} \label{ssec:strat-recoll}

\begin{recoll}[Stratifications and entrance path posets] A \emph{prestratification} $(X,f)$ is a decomposition $f$ of a topological space $X$ into a disjoint union of connected subspaces called \emph{strata}. Its \emph{entrance path preorder} $\Entr(f)$ is the preorder whose objects are strata $s$, and whose arrows are generated by arrows $s \to r$ for all non-empty intersections $\overline s \cap r$. (The opposite poset $\Entr(f)\op$ is often called the \emph{exit path preorder} and denoted by $\Exit(f)$.) If $\Entr(f)$ is a poset then we call $(X,f)$ a \emph{stratification}. The canonical map $X \to \Entr(f)$, taking $x$ to $s$ if $x \in s$, is called the \emph{characteristic map} of the stratification $f$ and often denoted by $f$.
\end{recoll}

\begin{conv}[Finiteness] We assume all stratifications to be finite, i.e.\ to contain only finitely many strata. Equally, we assume all posets to be finite.
\end{conv}

\begin{conv}[Manifold strata] Strata in any stratification are assumed to be manifolds.
\end{conv}

\begin{conv}[Abbreviations] We often abbreviate stratifications $(X,f)$ by their characteristic map $f$, or even by their underlying space $X$.
\end{conv}


\begin{recoll}[Trivial and product stratifications] By default, spaces $X$ are given the \emph{trivial} stratification (with characteristic map being the terminal map $X \to \ast$). Given two stratification $(X,f)$ and $(Y,g)$, their \emph{product} stratifications $(X,f) \times (Y,g)$ is the stratification of $X \times Y$ with characteristic map $f \times g$.
\end{recoll}

\begin{recoll}[Stratified maps] \label{recoll:strat-maps} A \emph{stratified map} $F : (X,f) \to (Y,g)$ is a map $F : X \to Y$ that factors through the characteristic maps $f : X \to \Entr(f)$ and $g : Y \to \Entr(g)$ by a (necessarily unique) poset map $\Entr(F) : \Entr(f) \to \Entr(g)$.
\begin{itemize}
    \item[$-$] If $F : X \to Y$ is a subspace and $\Entr(F)$ is conservative (i.e.\ reflects identities), then we say $F : (X,f) \to (Y,g)$ is a \emph{substratification}.
    \item[$-$] If $F : X \to Y$ is a subspace, $\Entr(F) : \Entr g \to \Entr f$ is a subposet, and $X = f\inv(\Entr g)$, then we say $F$ is a \emph{constructible substratification}.
    \item[$-$] If $F : X \to Y$ is a homeomorphism, we say $F : (X,f) \to (Y,g)$ is a \emph{coarsening} of $(X,f)$ to $(Y,f)$---equivalently, we say $F$ is a \emph{refinement} of $(Y,f)$ by $(X,f)$. We say a coarsening (or refinement) is \emph{strict} if $F = \id_X : X = X$.
\item[$-$] If $F : X \to Y$ is a homeomorphism and $\Entr(F)$ is an isomorphism, we say $F : (X,f) \to (Y,g)$ is a \emph{stratified homeomorphism}.
\end{itemize}
We also write `$\into$', `$\epi$' and `$\iso$' in place of `$\to$' to indicate that a stratified map is a substratification, coarsening or stratified homeomorphism respectively.
\end{recoll}

\begin{recoll}[Stratified bundles] A \emph{stratified bundle} is a stratified map $p : (X,f) \to (Y,g)$ such that for each stratum $s$ of the `base' stratification $(Y,g)$, each point $x$ in $s$ has a neighborhood $V$ (in $s$) over which the map $p$ trivializes to a stratified `projection' map $V \times (Z,h) \to V$.
\end{recoll}

The passage from stratifications to their entrance path posets has a converse.

\begin{recoll}[Classifying stratifications] \label{recoll:classifying-strat} Every poset $P$ has a \emph{classifying stratification} $\CStr{P}$, whose underlying space is the classifying space $\abs{P}$ of $P$ (i.e.\ the realization of the nerve of $P$), and whose characteristic map is the map $\abs{P} \to P$ that maps $\abs{P^{\geq x}} \setminus \abs{P^{>x}}$ to $x$ (here, $P^{\geq x}$ and $P^{>x}$ are the upper resp.\ strict upper closures of an element $x$ in $P$). Moreover, given a poset map $F : P \to Q$, the realization of its nerve yields a stratified map $\CStr F : \CStr P \to \CStr Q$.
\end{recoll}


\begin{recoll}[Cellular posets, regular cell complexes] \label{recoll:cellular-poset} A \emph{cellular poset} $P$ is a poset in which each strict upper closure is a sphere, i.e.\ for all $x \in P$, $\abs{P^{>x}}$ is homeomorphic to a topological $k$-sphere $S^k$ (for $k \in \lN$). A \emph{regular cell complex} is a stratification that is stratified homeomorphic to the classifying stratification of a cellular poset. (Standard results in combinatorial topology show that this coincides with the usual definition of regular cell complexes \cite{lundell1969topology} \cite{bjorner1984posets}.)
\end{recoll}

\begin{rmk}[Dual stratifications] Given a poset $P$, we say $\CStr {P\op}$ is the `dual stratification' of $\CStr {P}$. This definition recovers the following familiar situation: given a PL manifold $M$ with a triangulating simplicial complex $K$, then $K$ is in particular a regular cell complex and thus $K \iso \CStr P$ where $P = \Entr K$. The usual condition for $(M,K)$ to define a PL manifold requires that links of simplices in $K$ are PL standard spheres. Equivalently, this can now be phrased as saying that strict lower closures $P^{<x}$ are standard PL spheres. Thus $P\op$ is a cellular poset: the regular cell complex $\CStr {P\op}$ is the `dual cell complex' of $K$, which is often used in the proof of Poincar\'e duality (cf.\ \cite[\S65]{munkres2018elements} \cite{hatcher2002AT}).
\end{rmk}

\subsection*{Conical, regular, and cellulable stratifications} We recall two classes of `nice' stratifications: \emph{conical} and \emph{regular} stratifications (with the latter further specializing to \emph{cellulable} stratifications). One upshot of these conditions is that they enable simple definitions of `higher' fundamental categories of stratifications, which categorify the notion of entrance path posets. We assume all posets $P$ to be `locally finite', meaning all lower closures $P^{\leq x}$ of $P$ are finite.

\begin{recoll}[Cone stratifications] \label{rmk:cub-cone-strat} Given a stratification $(X,f)$, the \emph{stratified open cone} $(\cone(X), \cone(f))$ stratifies the topological open cone $\cone(X) = X \times [0,1) / X \times \{0\}$ by the product $(X,f) \times (0,1)$ away from the cone point $0$, and by setting $\{0\}$ to be its own stratum. To define the \emph{closed cone} $(\overline\cone(X),\overline\cone(f))$ replace `$1)$' by `$1]$'.
\end{recoll}

\begin{recoll}[Conical stratifications] \label{recoll:conical-strat} A \emph{conical} stratification is a stratification in which each point $x$ has a neighborhood that is a stratified product $U \times (\cone(Z),\cone(l))$ with $x \in U \times \{0\}$.
\end{recoll}

\begin{rmk}[Entrance path $\infty$-categories via homs] \label{rmk:TEntr} In \cite[App.\ A]{lurie2012higher} Lurie constructs \emph{entrance path $\infty$-categories} $\TEntr(f)$ of conical stratifications $(X,f)$: the $k$-simplices of the quasicategory $\TEntr(f)$ are precisely stratified maps $\CStr [k] \to (X,f)$, where $[k]$ is the `$k$-simplex' poset $(0 \to 1 \to ... \to k)$.
\end{rmk}

\begin{conv}[Entrance paths] Unnamed entrance paths $s \to t$ refer to arrows in $\Entr(f)$, while named entrance paths $\alpha :  s \to t$ refer to 1-morphisms in $\TEntr(f)$.
\end{conv}

\begin{rmk}[Posets are conical] Given a poset $P$, its classifying stratification $\CStr P$ is conical \cite[App.\ B]{fct}.
\end{rmk}

\begin{recoll}[$0$-Truncated stratifications] A conical stratification $(X,f)$ is called \emph{$0$-truncated} if $\TEntr(f)$ is $0$-truncated (i.e.\ a poset). Roughly speaking, in a $0$-truncated stratification it doesn't matter which entrance path we take between any two given strata, as we can identify $\TEntr(f) \eqv \Entr(f)$ canonically.
\end{recoll}

\nid In \cite[App.\ B]{fct} we showed that classifying stratifications of posets are $0$-truncated. In particular, regular cell complexes are $0$-truncated. The property of being $0$-truncated is furthermore inherited by all constructible substratifications.

\begin{recoll}[Regular stratifications] An \emph{(open) regular stratification} $(X,f)$ is a stratification that admits a `posetal subrefinement': this is a span $(X,f) \leftepi (Y,g) \into \CStr P$ where $(Y,g) \epi (X,f)$ is refinement of $(X,f)$ and $(Y,g) \into \CStr P$ an (open) constructible substratification of the classifying stratification of a poset $P$. (Here, an `open' substratification is an open subspace on underlying spaces.)
\end{recoll}

\nid Note, every stratification that can be triangulated or cellulated (i.e.\ refined by a simplicial or regular cell complex) is open regular.

\begin{conv}[Base stratifications] Base stratifications $B$ in any stratified bundle $E \to B$ are assumed to be `sufficiently nice' (which can be taken to be mean e.g.\ `conical and open regular'). For our purposes, we moreover assume base stratifications are always $0$-truncated unless otherwise noted.
\end{conv}

\begin{rmk}[Cellulable stratifications] \label{rmk:cellulable-strat} Replacing the `poset $P$' by `cellular poset $P$' in the previous recollection leads to the notion of \emph{(open) cellulable} stratification $(X,f)$ (i.e.\ a stratification which admits a `cellular subrefinement' $(X,f) \leftepi (Y,g) \into \CStr P$ for $P$ a cellular poset). Many proofs in \cite{fct} were given on the case of cellulable stratifications---but most arguments in fact generalize to the case of regular stratifications.
\end{rmk}

\begin{rmk}[Entrance path $\infty$-categories via localization] \label{rmk:posets-wwe} Given a regular stratification $(X,f)$ its entrance path $\infty$-category has a presentation as follows. Pick a posetal refinement $F : g \epi f$ of $f$. Then $\TEntr(g)$ is $0$-truncated and thus $\Entr(g) \eqv \TEntr(g)$.
    We define $\TEntr(f)$ to be the $\infty$-localization of $\Entr(g)$ at the set of arrows that $\Entr(F)$ maps to identities---thus $\TEntr(f)$ is presented by a poset with weak equivalences.\footnote{While this makes it easy to construct a candidate $\infty$-category $\TEntr(f)$, checking well-definedness of the construction requires us to verify independence of the choice $(Y,g)$; a reasonable argument can be given for the case of \emph{conical regular} stratifications, and checks equivalence of the given presentation with the definition of $\TEntr(f)$ in the conical case (see \cref{rmk:TEntr}). We omit the argument as we'll not need it here.}
\end{rmk}

\subsection*{Stratifications as \texorpdfstring{$\infty$}{infinity}-posets} For the categorically-minded reader it may be worth pointing out where stratifications are situated in the categorical landscape---see \cref{table:strat-in-cat-land}. Here, intuitively, an `$\infty$-X' is an $(\infty,\infty)$-category which admits a conservative functor to an X, where X can e.g.\ stand for `set', `poset', or `category'. Yet more generally, X can be an $(n,k)$-category for $n,k < \infty$ (note that, if $n < k$, then for $n < m \leq k$ it is convention to require that there is at most one $m$-arrow between any two $(m-1)$-arrows; in particular, posets and preorders are $(0,1)$-categories by this convention). A `set with weak equivalences' means a poset with weak equivalences in which each arrow is a weak equivalence (see \cref{rmk:posets-wwe}). Note that both `$\infty$-fication' can be `$\infty$-localization' become functors when appropriately defined. This illustrates that the role of posets in the theory of stratifications is analogous to the role of sets in the theory of spaces.

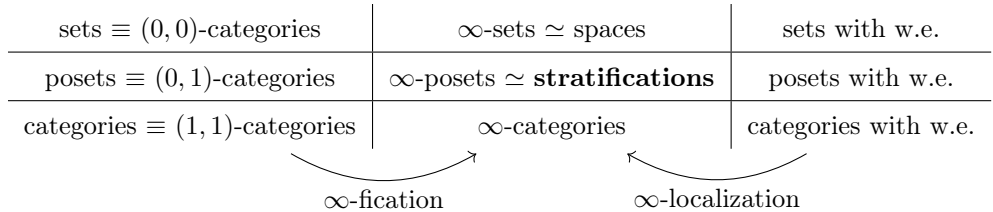
\begin{figure}[h!]
\centering\small
\def\arraystretch{1.5}
\begin{tabular}{ c|c|c }
 sets $\equiv$ $(0,0)$-categories & $\infty$-sets $\eqv$ spaces & sets with w.e. \\ \hline
 posets $\equiv$ $(0,1)$-categories & $\infty$-posets $\eqv$ \textbf{stratifications} & posets with w.e. \\ \hline
 categories $\equiv$ $(1,1)$-cate\tikzmark{a}gories & \tikzmark{b}$\infty$-categories\tikzmark{c} & categ\tikzmark{d}ories with w.e. \\
\end{tabular} \vspace{40pt}
\begin{tikzpicture}[overlay, remember picture, shorten >=.5pt, shorten <=.5pt, transform canvas={yshift=.25\baselineskip}]
    \draw [->] ([yshift=-10pt]{pic cs:a}) to [bend right] node[midway,below] {{$\infty$-fication}} ([yshift=-10pt]{pic cs:b}) ;
    \draw [->] ([yshift=-10pt]{pic cs:d}) to [bend left] node[midway,below] {{$\infty$-localization}} ([yshift=-10pt]{pic cs:c});
  \end{tikzpicture}
\caption[Stratifications in the categorical landscape]{Stratifications in the categorical landscape}
\label{table:strat-in-cat-land}
\end{figure}

\section{Meshes and trusses in dimension 1}

\subsection{1-Meshes and their bundles} A 1-dimensional mesh, or simply a 1-mesh, is a  contractible $k$-manifold, $k \leq 1$, which is stratified `by cells', together with a choice of framing on that manifold.\footnote{We are somewhat agnostic to what we mean by `framing', since in dimension $k \leq 1$ all reasonable notions are equivalent (whether we work in the topological, PL or smooth category). The reader may also take framing to mean `orientation' if $k = 1$, and to mean the trivial structure if $k = 0$.}

\begin{defn}[1-Meshes] A \textbf{1-mesh} $M$ is a stratified framed contractible 0- or 1-manifold, whose strata are open 0- or 1-disks.
\end{defn}

\nid Given a 1-mesh $M$, assigning dimension to strata yields a functor $\Entr(M) \to [1]\op$ from the entrance path poset to the (opposite) interval poset $[1]\op = (0 \ot 1)$.

\begin{term}[Closed and open 1-meshes] If the underlying manifold of a 1-mesh $M$ is compact (resp.\ an open interval) then we say $M$ is `closed' (resp.\ `open'). If neither applies, we say $M$ is `mixed'.
\end{term}

\begin{eg}[1-Meshes] Examples of 1-meshes are shown in \cref{fig:1-meshes-of-open-closed-and-mixed-boundary-type}: we use arrows to indicate the framing of underlying manifolds.
\begin{figure}[ht]
    \centering
    \def\svgwidth{1\columnwidth}
    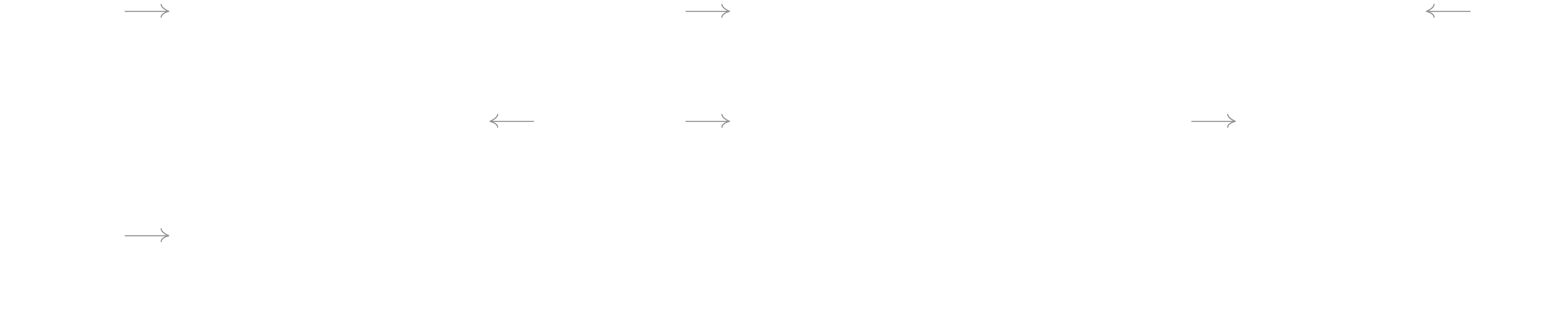

    \caption{1-Meshes of open, closed. and mixed type}
    \label{fig:1-meshes-of-open-closed-and-mixed-boundary-type}
\end{figure}
\end{eg}

\begin{term}[Coordinates and bounds] A `coordinatization' $\gamma$ of a 1-mesh $M$ is a framed bounded embedding $\gamma : M \into \lR$. (Here, the embedding being `framed' means that the standard framing of $\lR$ pulls back along $\gamma$ to a framing equivalent to that of $M$.) Since $\gamma$ is bounded, its image closure $\overline{\im(\gamma)}$ is a subspace $[\gamma_-,\gamma_+]$ of $\lR$: we call $\gamma_-$ and $\gamma_+$ the `lower' resp.\ `upper bound' of $\gamma$.
\end{term}

\begin{defn}[1-Mesh maps] A \textbf{map of 1-meshes} $F : M \to N$ is a stratified map that preservers the framing of the 1-meshes.
\end{defn}

\nid The space of maps of 1-meshes inherits a topology from the mapping space of the underlying spaces.

\begin{notn}[The $\infty$-category of 1-meshes] The topologically enriched category of 1-meshes and their maps is denoted by $\tmesh 1$.
\end{notn}


We next recall the definition of bundles of 1-meshes.

\begin{notn}[Fiber dimension] Given a stratified bundle $p : E \to B$, whose domain and codomain both have strata that are manifolds, we write $\fibdim(s)$ for the `fiber dimension' $\dim(s) - \dim(p(s))$ of a stratum $s$ in $E$.
\end{notn}

\begin{defn}[1-Mesh bundles] \label{defn:1-mesh-bun} A \textbf{1-mesh bundle} $p : M \to B$ is a stratified bundle together with a choice of 1-mesh structure $M^b$ for each stratified fiber $p\inv(b)$, $b \in B$, and with the following compatibility condition between fibers.
\begin{enumerate}
    \item[$-$] \textit{Coordinatizability}: there exists a `coordinatizing' bundle embedding $\gamma : M \into B \times \lR$ into the trivial bundle $B \times \lR \to B$, which on each 1-mesh fiber $M^b$ restricts to a coordinatization $\gamma^b : M^b \into \lR$ of $M^b$, such that the `bounding sections' $b \mapsto \gamma_\pm(b) := (b,\gamma^b_\pm)$ are continuous maps $B \to B \times \lR$.
    \item[$-$] \textit{0-Constructibility}: For a stratum $s$ in $M$ with $\fibdim(s)=0$, any entrance path $p(s) \to u$ in $\Entr(B)$ has a unique lift $s \to t$ in $\Entr(M)$, and that lift is such that $\fibdim(t) = 0$. \qedhere
\end{enumerate}
\end{defn}

\nid Observe that the constructibility condition ensures that there is a functor $\Entr(M) \to [1]\op$ which assigns to each stratum $s$ its fiber dimension $\fibdim(s)$.

\begin{term}[Closed and open 1-mesh bundles] If all fibers of a 1-mesh bundle $p$ are closed (resp. open) then we call $p$ itself `closed' (resp. `open').
\end{term}

\begin{eg}[1-Mesh bundles] Two 1-mesh bundles are shown in \cref{fig:1-mesh-bundles}.
\begin{figure}[ht]
    \centering
    \def\svgwidth{1\columnwidth}
    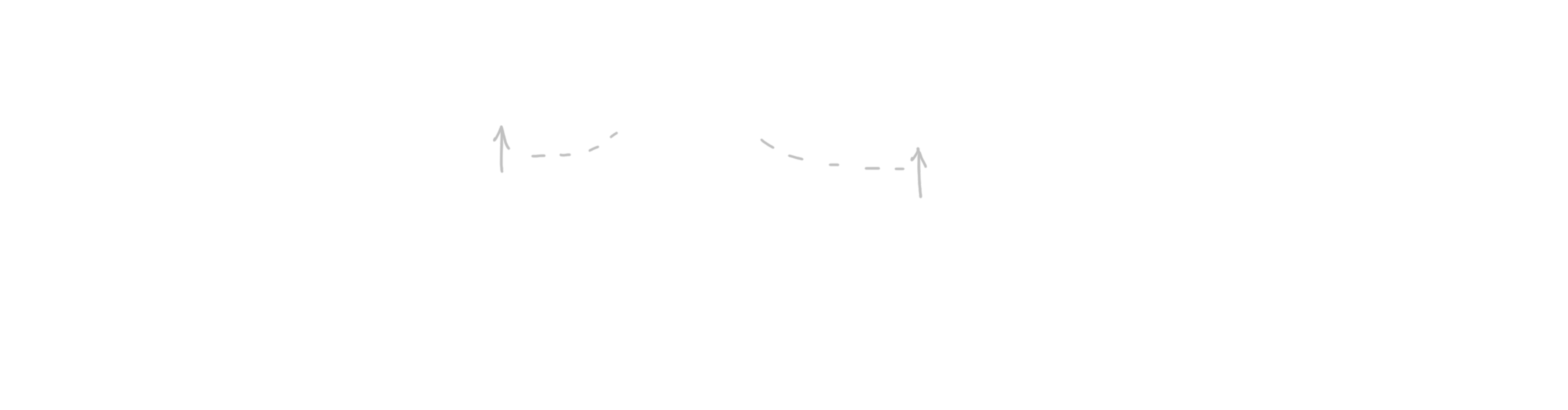

    \caption{1-Mesh bundles}
    \label{fig:1-mesh-bundles}
\end{figure}
\end{eg}

\nid The constructibility condition required for 1-mesh bundles acts at the `$0$th categorical level': namely, the condition will ensure that 1-mesh bundles $p : M \to B$ can be classified by functors on the entrance path poset $\Entr(B)$. To describe 1-mesh bundles classified by functors on the entrance path $\infty$-category $\TEntr(B)$, we can replace the condition as follows.

\begin{rmk}[1-Constructibility, {\cite[Rmk.\ 4.1.33]{fct}}] \label{rmk:1-constructible} The `0-constructibility' condition in \cref{defn:1-mesh-bun} is categorified by the following condition:
\begin{enumerate}
\item[$-$] \emph{1-Constructibility}: For a stratum $s$ in $M$ with $\fibdim(s)=0$, any entrance path $\alpha : p(s) \to u$ in $B$ has a unique lift $\beta : s \to t$ in $M$ with $p \circ \beta = \alpha$, and that lift satisfies $\fibdim(t) = 0$.
\end{enumerate}
If the base $B$ is a $0$-truncated stratification then 0-constructibility and 1-constructibility are equivalent conditions. We refer to 1-mesh bundles satisfying 1-constructibility (in place of $0$-constructibility) as `$1$-constructible 1-mesh bundles'.
\end{rmk}

\begin{rmk}[Coordinatizations as structure] Our definitions of 1-meshes and their bundles differ slightly from those given in \cite[\S4]{fct}, where we required coordinatizations $\gamma$ as part of the structure of both 1-meshes and 1-mesh bundles. In contrast, here we only ask for the existence of such coordinatizations. Categorically both approaches are equivalent (see also \cref{ssec:coordinates}).
\end{rmk}

\subsection{1-Trusses and their bundles} The entrance path posets of 1-meshes naturally carry additional structure given by strata dimensions and framing. This leads a notion of structured entrance path posets of 1-meshes, which we refer to as the `fundamental 1-trusses' of the 1-meshes.

\begin{notn}[Frame order $\fleq$ in meshes] Given a 1-mesh $M$ and strata $s, t$ in $M$, we write $s \fles t$ if for some (or equivalently, any) coordinatization $\gamma  : M \into \lR$ we have $\gamma(s) < \gamma(t)$ (meaning $\gamma(s)$ lies to the left of $\gamma(t)$ in the standard order of $\lR$).
\end{notn}

\begin{defn}[1-Trusses from 1-meshes] \label{defn:1-truss} A \textbf{1-truss} $T = (T,\leq,\dim,\fleq)$ is a poset $(T,\leq)$ together with a `dimension' map $\dim : (T,\leq) \to [1]\op$ and a `frame order' $(T,\fleq)$, for which we can find a 1-mesh $M$ such that
\begin{enumerate}
\item[$-$] there exists poset isomorphism $\phi : (T,\leq) \iso  \Entr(M)$,
\item[$-$] which is compatible with dimensions, i.e.\ $\dim(s) = \dim(\phi(s))$,
\item[$-$] and compatible with framings, i.e.\ for $s,t \in T$, we have $s \fleq t$ iff $\phi(s) \fleq \phi(t)$.
\end{enumerate}
(Note that for fixed $M$ the choice of $\phi$ is necessarily unique.) We call $M$ a `classifying 1-mesh' of $T$, and say $T$ is a `fundamental 1-truss' of $M$.
\end{defn}

\begin{notn}[Fundamental 1-trusses] Fundamental 1-trusses of 1-meshes $M$ are `essentially unique' (this means, any two fundamental trusses of a given 1-mesh are related by unique structure preserving isomorphisms), and we denote them by $\FTrs M$.
\end{notn}

\begin{notn}[Classifying 1-meshes] Classifying 1-meshes of 1-trusses $T$ are `unique up to contractible choice' (this means, any two classifying 1-meshes of a given 1-truss are related by a contractible space of framed stratified homeomorphisms), and we denote them by $\CMsh T$.
\end{notn}


\nid The structure $(T,\leq,\dim,\fleq)$ of a 1-truss $T$ can alternatively be characterized in purely combinatorial terms, without requiring the existence of classifying 1-meshes. Recall a functor is called `full' if it is surjective on hom sets, and `conservative' if it reflects isomorphisms.

\begin{altdefn}[1-Trusses, combinatorially] \label{rmk:comb-1-trusses} A 1-truss is a poset $(T,\leq)$ together with a full and conservative functor $\dim : (T,\leq) \to [1]\op$, and with a second order $(T,\fleq)$ which is total, and whose generating arrows $t \fles s$ satisfy either $t < s$ or $s < t$.
\end{altdefn}

\begin{eg}[1-Trusses] The fundamental trusses of the 1-meshes from \cref{fig:1-meshes-of-open-closed-and-mixed-boundary-type} are shown in \cref{fig:fundamental-truss-of-our-earlier-1-mesh-examples} below: the order $\leq$ is indicated by arrows between objects, dimension assignments are indicated by numbers, and the total order $\fleq$ is indicated by a single arrow along all (linearly arranged) objects.
\begin{figure}[ht]
    \centering
    \def\svgwidth{1\columnwidth}
    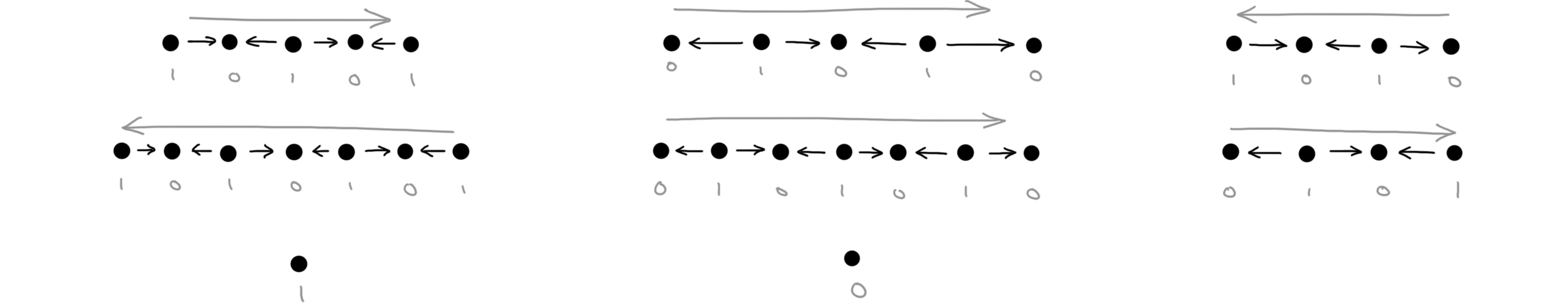

    \caption[Fundamental trusses]{The fundamental trusses of the earlier 1-mesh examples}
    \label{fig:fundamental-truss-of-our-earlier-1-mesh-examples}
\end{figure}
\end{eg}

\begin{notn}[Dimension-$i$ objects] \label{notn:dim-sets} Given a 1-truss $(T,\leq,\dim,\fleq)$ we denote by $T_{(i)}$ the subset $\dim\inv(i)$ of $T$ (where $i \in [1]$).
\end{notn}

\begin{defn}[Maps of 1-trusses] A \textbf{map of 1-trusses} $T \to S$ is a map of posets $(T,\leq) \to (S,\leq)$ that preserves the frame orders.
\end{defn}

\begin{notn}[Category of 1-trusses] The category of 1-trusses and their maps is denoted by $\truss 1$.
\end{notn}

\begin{rmk}[Fundamental 1-truss functor] Any 1-mesh map $F : M \to N$ determines, by its mapping $\Entr(F)$ on entrance path posets, a unique 1-truss map $\FTrs M \to \FTrs N$, which we call its `fundamental 1-truss map' and denote by $\FTrs(F)$. This yields a functor $\FTrs : \tmesh 1 \to \truss 1$.
\end{rmk}

\begin{rmk}[Classifying 1-mesh functor] Any 1-truss map $G : T \to S$ has a `classifying 1-mesh map' $ : M \to N$ such that $\FTrs (\CMsh G) = G$. $\CMsh G$ is determined up to contractible choice (which yields an `weak inverse' $\CMsh : \truss 1 \to \tmesh 1$ to the functor $\FTrs$; we return to this in \cref{rmk:cmsh-functor-choice}).
\end{rmk}


We next consider bundles of 1-trusses. We first define 1-truss bundles as the combinatorial structures underlying 1-mesh bundles, and then also give an independent, purely combinatorial definition.

\begin{defn}[1-Truss bundles from 1-mesh bundles] \label{defn:1-truss-bun} For a stratification $B$, a \textbf{1-truss bundle} over the poset $\Entr(B)$ is a poset map $q : T \to \Entr(B)$ together with the structure of a 1-truss $T^s$ on each fiber $q\inv(s)$, $s \in \Entr(B)$, for which there exists a 1-mesh bundle $p : M \to B$ with an isomorphism $\phi : T \iso \Entr(M)$ such that $q = \Entr(p) \circ \phi$, and $\phi$ exhibits fibers $T^s$ as the fundamental 1-trusses of fibers $M^b$, $b \in s$ (in the sense of \cref{defn:1-truss}). We call $p$ a `classifying 1-mesh bundle' for $q$, and conversely say $q$ is a `fundamental 1-truss bundle' of $p$.
\end{defn}

\begin{notn}[Fundamental 1-truss bundles] Fundamental 1-truss bundles of 1-mesh bundles $p$ are essentially unique, and we denote them by $\FTrs p$.
\end{notn}

\begin{notn}[Classifying 1-mesh bundles] Classifying 1-mesh bundles of 1-truss bundles $q$ are unique up to contractible choice, and we denote them by $\CMsh q$.
\end{notn}


\begin{eg}[1-Truss bundles]  The fundamental trusses of the 1-meshes from \cref{fig:1-mesh-bundles} are shown in \cref{fig:fundamental-1-truss-bundles-of-earlier-1-mesh-bundle} below.
\begin{figure}[ht]
    \centering
    \def\svgwidth{1\columnwidth}
    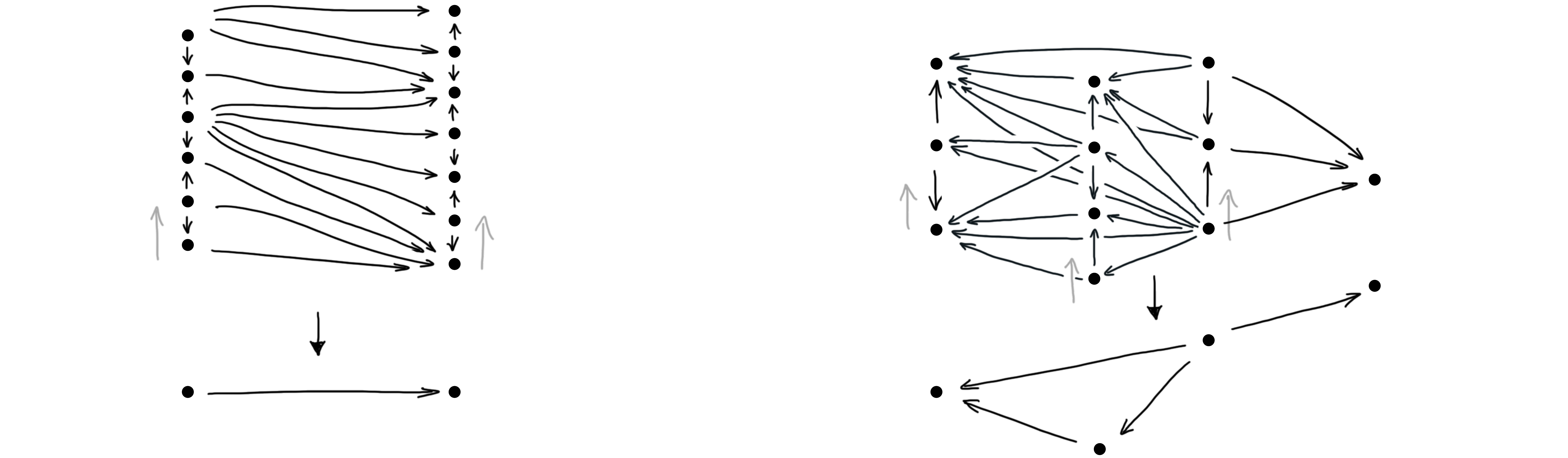

    \caption[Fundamental 1-truss bundles]{Fundamental 1-truss bundles of earlier 1-mesh bundles}
    \label{fig:fundamental-1-truss-bundles-of-earlier-1-mesh-bundle}
\end{figure}
\end{eg}

\nid Analogous to our combinatorial definition of 1-trusses in \cref{rmk:comb-1-trusses}, we next give an alternative, purely combinatorial definition of 1-truss bundles. This will in particular imply that the definition of 1-truss bundles is independent of the choice of stratification $B$ in \cref{defn:1-truss-bun}.

\begin{rmk}[Boolean profunctors] \label{rmk:bool-prof} A `Boolean profunctor' is an ordinary profunctor $H : \iC \proto{} \iD$ whose values are either the initial set $\emptyset \equiv \bot$ or the terminal set $\ast \equiv \top$. If $\iC$ and $\iD$ are discrete, then such a profunctor $H$ is simply a relation of sets. In this case, we call the profunctor $H$ a `function' if it is a functional relation or a `cofunction' if the dual profunctor $H\op$ is a function.
\end{rmk}

\begin{rmk}[Boolean profunctors from posets maps] \label{rmk:bool-prof-from-pos-map} For any map of posets $F : P \to Q$, the fiber $F\inv(x \to y)$ over an arrow $x \to y$ of $Q$ defines a Boolean profunctor $F\inv(x) \proto{} F\inv(y)$ by mapping $(a,b)$ to $\top$ iff $a \to b$ is an arrow in $P$.
\end{rmk}

\nid Recall our notation for dim-$i$ objects $T_{(i)}$ of a 1-truss $T$ (see \cref{notn:dim-sets}).

\begin{defn}[Category of 1-truss bordisms] \label{defn:1-truss-bordisms} Given 1-trusses $T$ and $S$, a \textbf{1-truss bordism} $R : T \proto{} S$ is a Boolean profunctor $T \proto{} S$ satisfying the following:
\begin{itemize}
    \item[\textsf{(A)}] $R$ restricts to a function $R_{(0)} : T_{(0)} \proto{} S_{(0)}$ and a cofunction $R_{(1)} : T_{(1)} \proto{} S_{(1)}$.
    \item[\textsf{(B)}] Whenever $R(t,s) = \top = R(t',s')$, then either $t \fles t'$ or $s' \fles s$ but not both.
\end{itemize}
Given 1-truss bordisms $R : T \proto{} S$ and $Q : S \proto{} U$, their composite profunctor $R \circ Q$  (composed as ordinary profunctors) is again a 1-truss bordism.\footnote{In contrast, composites of Boolean profunctors (composed as ordinary profunctors) in general need not themselves be Boolean.} This gives rise to the \textbf{category $\kT^1$ of 1-trusses and their bordisms}.
\end{defn}

\begin{altdefn}[1-Truss bundles, combinatorially] \label{rmk:1-truss-bun-comb-defn} A 1-truss bundle over a poset $P$ is a poset map $q : T \to P$ in which each fiber $q\inv(x)$, $x \in P$, is equipped with the structures of a 1-truss $T^x$, and for each arrow $x \to y$ in $P$, the fiber $q\inv(x \to y)$ is a 1-truss bordisms $T^x \proto{} T^y$ (note $q\inv(x \to y)$ is a Boolean profunctor by \cref{rmk:bool-prof-from-pos-map}).
\end{altdefn}

\nid When $P = \Entr(B)$ (for some stratification $B$), then this alternative definition is equivalent to \cref{defn:1-truss-bun} (one verifies that the conditions \textsf{(A)} and \textsf{(B)} in \cref{defn:1-truss-bordisms} precisely guarantee the existence of a classifying mesh bundle over $B$, as required in \cref{defn:1-truss-bun}). Using this alternative characterization, we now recall that the category $\kT^1$ of 1-truss bordisms acts as a classifying category for 1-truss bundles.

\begin{obs}[Essential gauntness of $\kT^1$] \label{rmk:gaunt} The category $\kT^1$ is `essentially gaunt', in the sense that it does not have non-identity automorphisms. Thus, passing to any skeleton, we can replace `natural isomorphism' of functors $P \to \kT^1$ by `equality' without harm.
\end{obs}

\begin{constr}[Classification of 1-truss bundles] \label{constr:1trussbord} Given a functor $\chi : P \to \kT^1$, the usual Grothendieck construction for profunctors yields a functor $q_{\chi} : T \to P$ over $P$ (more specifically, a functor constructed in this way is called an `exponentiable fibration', see \cite{street2001powerful}). Since 1-truss bordisms are Boolean profunctors, $q_{\chi}$ is in fact a map of posets. Moreover, each fiber $q_{\chi}\inv(x)$, $x \in P$, has the structure of a 1-truss (given by the image $\chi(x)$), and this makes $q_\chi$ a 1-truss bundle. The mapping $\chi \mapsto q_\chi$ yields a 1-to-1 correspondence between functors $P \to \kT^1$ and 1-truss bundles over $P$ (up to structure preserving bundle isomorphism). The inverse mapping will be denoted by $q \mapsto \chi_q$, and we refer to $\chi_q$ as the `classifying functor' of the 1-truss bundle $q$.
\end{constr}

\begin{rmk}[Alternative definition of $\kT^1$] \label{rmk:ttr1alt} The fact that the category $\kT^1$ classifies 1-truss bundles means that it can equivalently be defined of as follows: objects of $\kT^1$ are 1-truss bundles over the $[0]$; morphisms are 1-truss bundles over $[1]$ (with domain and codomain given by restricting to fibers over $\Set{0,1} \subset [1]$); two morphisms compose to a third morphism iff there is a 1-truss bundle over $[2]$ that restricts over $(0 \to 1)$, $(1 \to 2)$, and $(0 \to 2)$ to the first, second, resp.\ third morphism.
\end{rmk}

The relation of 1-mesh bundles and 1-truss bundles, together with the classification of 1-truss bundles by the category $\kT^1$, combines to give the following observation.

\begin{lem}[Classification of 1-mesh bundles] \label{lem:class1meshbun} Bundle isomorphism classes of 1-mesh bundles $p : M \to B$ bijectively correspond to functors $\Entr(B) \to \kT^1$. The correspondence takes $p$ to the classifying functor $\chi_{\FTrs(p)}$ of the 1-truss bundle $\FTrs(p)$.
\end{lem}

\begin{proof} Follows from \cite[Prop.\ 4.2.26]{fct} and \cite[Obs.\  4.2.68]{fct}.
\end{proof}

\begin{note}[Classification of 1-constructible mesh bundles] \label{rmk:base-cat} The preceding lemma still applies, with essentially the same proof, in the 1-constructible case (see \cref{rmk:1-constructible}): bundle isomorphism classes of 1-constructible 1-mesh bundles $p : M \to B$ (with $B$ not necessarily $0$-truncated) correspond to functors $\TEntr(B) \to \kT^1$.
\end{note}



\subsection{Geometric duality} The category of trusses admits natural duality operation. Combinatorially, this mirrors the dualization functor $C \mapsto C\op$ of categories. Geometrically, the operation passes to dual cell structures in the spirit of classical Poincar\'e duality.

\begin{defn}[Dualization of trusses] Given a 1-truss $T \equiv (T,\leq,\dim,\fleq)$, its \textbf{dual 1-truss} $T^\dagger$ is the truss $(T,\leq\op,\dim\op,\fleq)$ (where the definition of $\dim\op: (T,\leq\op) \to [1]\op$ uses that $[1] \iso [1]\op$). Given a 1-truss map $F : T \to S$, its \textbf{dual map} $F^\dagger : T^\dagger \to S^\dagger$ is the truss map of dual 1-trusses, that maps objects by $F$.
\end{defn}

\nid As a result, we obtain an involutive dualization functor $\dagger : \truss 1 \iso \truss 1$. This dualization functor maps closed trusses to open trusses. To refine our description of how dualization acts on maps, we now distinguish certain types of maps by their action on dimensions.

\begin{term}[Cellular and cocellular 1-truss maps] \label{term:1-truss-maps} Let $F : T \to S$ be a 1-truss map.
\begin{itemize}
    \item[-] We say $F$ is `cellular' if it restricts to a map $T_{(0)} \to S_{(0)}$, i.e.\ it maps dimension 0 objects to dimension 0 objects. If $T$ and $S$ are closed trusses, then equivalently $F$ preserves upper closures, that is, $F(T^{\geq x}) = S^{\geq F(x)}$.
    \item[-] We say $F$ is `cocellular' if it restricts to a map $T_{(1)} \to S_{(1)}$, i.e.\ it maps dimension 1 objects to dimension 1 objects. If $T$ and $S$ are open trusses, then equivalently $F$ preserves lower closures, that is, $F(T^{\leq x}) = S^{\leq F(x)}$.
\end{itemize}
If $F$ is both cellular and cocellular we say $F$ is `balanced'. A balanced injection is a called `subtruss', and a balanced bijection a `truss isomorphism'. The terminology similarly applies to 1-mesh maps.\footnote{Note our terminology here differs from that in \cite[\S2]{fct}, where `cellular' resp.\ `cocellular' maps of trusses were called `singular' resp.\ `regular' maps, since they map singular objects (i.e.\ the objects of $T_{(0)}$) to singular objects resp.\ regular objects (i.e.\ the objects of $T_{(1)}$) to regular objects. Our terminology choice here instead mirrors the terminology for `cellular maps' of cellular posets, cf.\ \cite[Def.\ 1.3.18]{fct}.}
\end{term}

\nid The dualization functor $\dagger : \truss 1 \iso \truss 1$ maps cellular maps to cocellular maps, and vice versa.

\begin{term}[Duality for 1-meshes and maps] We say meshes $M$ and $N$ are dual iff $\FTrs (M) = \FTrs (N)^\dagger$. Duality of mesh maps works similarly, and is illustrated in the next example.
\end{term}

\begin{eg}[Dualization of mesh maps] In \cref{fig:dualization-of-mesh-maps} we depict pairs of dual maps: the left column shows cellular maps which dualize to the cocellular maps on the right. (These maps may be further terminologically distinguished as `face', `degeneracy', `embedding', and `coarsening' maps, as indicated in the figure; we refer the reader to \cite[Eg.\ 2.3.96]{fct} for a discussion.)
\begin{figure}[ht]
    \centering
    \def\svgwidth{1\columnwidth}
    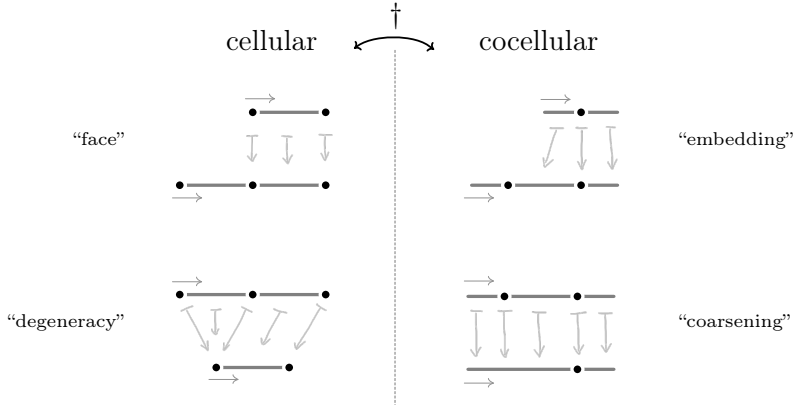

    \caption{Dualization of 1-mesh maps}
    \label{fig:dualization-of-mesh-maps}
\end{figure}
\end{eg}

\begin{defn}[Duality on truss bundles] \label{defn:dual-bundle} Given a 1-truss bundle $q : T \to B$, its \textbf{dual 1-truss bundle} $q^\dagger : T\op \to B\op$ is the opposite map of posets together with 1-truss fiber structures dual to those of $q$.
\end{defn}

\nid The fact that $q^\dagger$ is indeed a 1-truss bundle is easiest to verify using the classification of 1-truss bundles (see \cref{constr:1trussbord}): recall, 1-truss bundles $q$ are classified by functors $\chi_q : B \to \kT^1$. Now, given a truss bordism $R : T \proto{} S$, the dual Boolean profunctor $R\op$ defines a truss bordism $R^\dagger : S^\dagger \to T^\dagger$ (that is, $R\op$ satisfies conditions \textsf{(A)} and \textsf{(B)}). This yields an involution
\[
	\dagger : \kT^1 \iso (\kT^1)\op .
\]
\nid Post-composing $\chi_p$ with this involution yields a functor $(\dagger \circ \chi_q)\op : B\op \to \kT^1$. This functor classifies a truss bundle, which equals the dual bundle $q^\dagger : T\op \to B\op$ defined above (in particular,  $q^\dagger$ is a 1-truss bundle).

\begin{term}[Duality for 1-mesh bundles] Given 1-mesh bundles $p$ and $p'$ we say they are dual if their fundamental 1-truss bundles are, i.e.\ if $\FTrs(p) = \FTrs(p')^\dagger$.
\end{term}

\begin{eg}[Dualization of mesh bundles] In \cref{fig:dualization-of-mesh-bundles} we depict two dual 1-mesh bundles.
\begin{figure}[ht]
    \centering
    \def\svgwidth{1\columnwidth}
    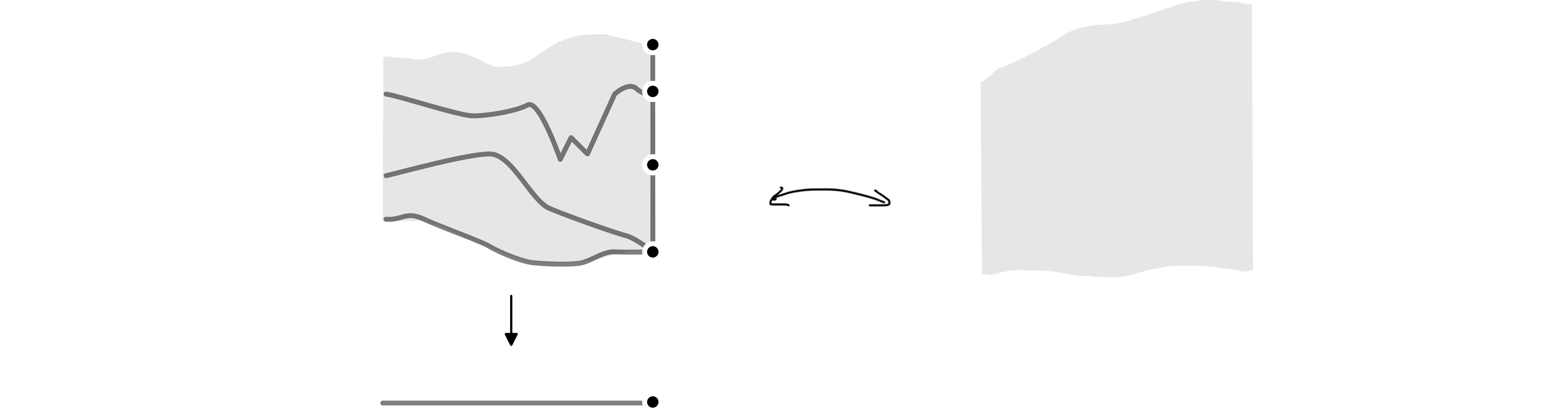

    \caption{Dualization of 1-mesh bundles}
    \label{fig:dualization-of-mesh-bundles}
\end{figure}
\end{eg}

\begin{note}[The dual worlds of cells and cocells] \label{note:cell-vs-string}  Given a \emph{cellular} 1-mesh map $F$ we may think of it as a 1-mesh bundle, by projecting its mapping cylinder $M_F$ to the stratified 1-simplex $\CStr [1]$. Passing to the dual of that bundle now represents the \emph{co}cellular map $F^\dagger$ via some (less familiar) notion of `\emph{co}mapping cylinder'. This illustrates the following heuristic observation: as topologists we are used to think about the world using cell structures and cell complexes, but less used to thinking in the dual world of cocells. Using the self-duality of meshes and trusses one can cleanly connect the two worlds.
\end{note}

\subsection{Labelings and stratifications}

We now endow 1-trusses and their bundles with an additional structure, called a `labeling'. This structure will serve two purposes: firstly, it will enable the later inductive classification of $n$-truss bundles, and secondly, it specializes to a structure of `stratifications' on 1-truss bundles which, analogous to stratifications on spaces, will decompose the bundle's total poset into subposet `strata'.

\subsubsection*{Labelings and stratifications of 1-trusses} \label{ssec:strattruss}

A `$\cC$-labeling' of a stratified space $(X,f)$ is a functor $\TEntr(f) \to \cC$ from its entrance path $\infty$-category to another ($\infty$-)category $\cC$. When labeling $0$-truncated stratifications we may work instead with entrance path posets $\Entr(f)$; in particular, for 1-meshes we work with their fundamental 1-trusses. This leads to the following notion of trusses with `labelings'. Throughout this section, let $\iC$ be an ordinary category.

\begin{defn}[Labeled 1-trusses and 1-truss bundles] Let $\iC$ be a category. A \textbf{$\iC$-labeled 1-truss} $(T,f)$ is an `underlying' 1-truss $T$ together with a `labeling' functor $f : T \to \iC$. A  \textbf{$\iC$-labeled 1-truss bundle} $(q,f)$ is an `underlying' 1-truss bundle $q : T \to P$ together with a `labeling' functor $f : T \to \iC$.
\end{defn}

\nid Analogous to 1-truss bundles being classified by functors into the category 1-truss bordisms $\kT^1$, we now construct a classifying category for the case of labeled 1-truss bundles. We define this category by considering labeled bundles over simplices, with objects being bundles over the 0-simplex, morphisms being bundles over the 1-simplex, and their composites being described by bundles over the 2-simplex---this mirrors our earlier construction of $\kT^1$ via unlabeled bundles over simplices in \cref{rmk:ttr1alt}. (However, an alternative construction, which mirrors the construction of $\kT^1$ via 1-truss bordisms in \cref{defn:1-truss-bordisms}, can also be given, see \cite[Obs.\ 2.3.21]{fct}.)

\begin{constr}[Labeled 1-truss bordisms] \label{constr:lab-1-truss-bord} Given a category $\iC$, the \textbf{category $\kT^1(\iC)$ of $\iC$-labeled 1-trusses and their bordisms} is defined as follows: objects of $\kT^1(\iC)$ are $\iC$-labeled 1-truss bundles over $[0]$; morphisms are $\iC$-labeled 1-truss bundles over $[1]$ (with domain and codomain given by restricting to fibers over  $0$ resp.\ $1$); two morphisms compose to a third iff there is a $\iC$-labeled bundle over $[2]$ that restricts over $(0 \to 1)$, $(1 \to 2)$, and $(0 \to 2)$ to the first, second, resp.\ third morphism.
The fact that this defines a category is shown in \cite[\S2.3.1]{fct}.
\end{constr}

\begin{rmk}[Label forgetting functor] Comparing the preceding construction of $\kT^1(\iC)$ to our earlier construction of $\kT^1$ in \cref{rmk:ttr1alt}, note that there is a `label forgetting' functor $\mathfrak{un} :\kT^1(\iC) \to \kT^1$ which discards labelings.
\end{rmk}

\begin{obs}[Classifying labeled 1-truss bundles] \label{constr:lab1bord} Given a functor $\chi : P \to \kT^1(\iC)$, the composite $\mathfrak{un} \circ \chi$ classifies a 1-truss bundle $q_{\chi}$ over $P$ by \cref{constr:1trussbord}. Note, the restriction of $q_{\chi}$ over each arrow $x \to y$ in $P$ equals the bundle classified by $\mathfrak{un} \circ \chi (x \to y)$ (up to a unique bundle isomorphism).
Using this identification, the labelings of the labeled 1-truss bundle $\chi (x \to y)$ assemble into a labeling $f_\chi$ of $q_{\chi}$. The mapping $\chi \mapsto (q_{\chi},f_\chi)$ yields a bijective correspondence between functors $P \to \kT^1(\iC)$ and $\iC$-labeled 1-truss bundles over $P$ (up to bundle isomorphisms that commute with labelings). The inverse will be denoted by $(q,f) \mapsto \chi_{(q,f)}$.
\end{obs}

\begin{obs}[Labeled bordisms functor] Observe that the construction $\iC \mapsto \kT^1(\iC)$ is functorial in $\iC$. Indeed, given a functor $\iF : \iC \to \iD$, then $\kT^1(\iF) : \kT^1(\iC) \to \kT^1(\iD)$ acts on objects and morphisms of $\kT^1(\iC)$ (see \cref{constr:lab-1-truss-bord}) by post-composing their labelings with $\iF$.
\end{obs}

A `stratification' is a particular type of labeling structure. Let us first recall a few useful facts about such structures. We call a map $X \to P$ from a space to a poset `characteristic' if it is the characteristic map of a stratification of $X$.

\begin{rmk}[Stratified posets] \label{rmk:strat-posets} Any poset $P$ can be considered as a topological space by defining the basic opens to be the lower closures $P^{\leq x}$ of objects $x \in P$ (we speak of the `poset space $P$' when endowing $P$ with this topology). A `stratified poset' $(P,f)$ is a poset $P$ together with a stratification $f$ of the poset space $P$. Usually, we think of $f$ in terms of its characteristic map $f : P \to \Entr(f)$. Importantly, the characteristic map $f$ of a stratified poset $(P,f)$ is always a poset map $P \to \Entr(f)$.
\end{rmk}

\nid In particular, given a stratified space $(X,g)$, a stratification of the poset $\Entr(g)$ given by a characteristic map $f : \Entr(g) \to \Entr(f)$ is a special case of a labeling of $(X,g)$.

\begin{rmk}[Poset labelings of stratifications] \label{obs:lab-strat-top} Let $P$ be a poset.
\begin{enumerate}
\item Given a stratification $(X,g)$, with characteristic map $g : X \to \Entr(g)$, then strict coarsenings of $X$ are in bijective correspondence with stratifications of the poset $\Entr(g)$: the correspondence takes a strict coarsening $F$ to its entrance path poset map $\Entr(F)$. For details, see \cite[Lem.\ B.2.12]{fct}.
\item Given any continuous map $f : X \to P$ from a space $X$ to a poset space $P$, there is an essentially unique decomposition $f = \dcs f \circ \ccs f$ (called the `connected component splitting' of $f$) such that $\ccs f$ is characteristic, and $\dcs f$ is conservative. For details, see \cite[Constr.\ B.1.34]{fct}.
\item The last observation applies to all poset maps $Q \to P$ (since all poset maps are continuous as maps of poset spaces). Thus, any poset map splits essentially uniquely into a characteristic map and a conservative map. \qedhere
\end{enumerate}
\end{rmk}

\begin{defn}[Stratified 1-trusses] A \textbf{stratified 1-truss} $(T,f)$ is a labeled 1-truss such that $f$ stratifies the poset $(T,\leq)$ (in the sense of \cref{rmk:strat-posets}).
\end{defn}

\begin{rmk}[Stratified trusses from poset labeled trusses] \label{rmk:strat-truss-from-pos-lab} Using connected component splittings from \cref{obs:lab-strat-top} above, note that every poset labeled truss $(T, f : T \to P)$ canonically gives rise to a stratified truss $(T,\ccs f)$.
\end{rmk}

Let us also introduce an analogous notion for 1-meshes.

\begin{defn}[Stratified 1-meshes] A \textbf{stratified 1-mesh} is a tuple $(M,f)$ consisting of a 1-mesh $M$ and a strict coarsening $M \to f$ to a stratification $f$ (that is, $f$ has the same underlying space as $M$, and $M$ refines $f$).
\end{defn}

\begin{rmk}[Relating stratified 1-meshes and 1-trusses] Given a stratified 1-mesh $(M,f)$ and stratified 1-truss $(T,g)$, we say $(T,g)$ is a `stratified fundamental 1-truss' of $(M,f)$ (and conversely, that $(M,f)$ a `stratified classifying 1-mesh' of $(T,g)$), if, firstly, $T = \FTrs M$ is the fundamental 1-truss of $M$, and secondly, there is a (necessarily unique) poset isomorphism $\Entr(g) \iso \Entr(f)$ such that the following commutes
\[\begin{tikzcd}[baseline=(W.base)]
	{\Entr(M)} & T \\
{\Entr(f)} & |[alias=W]| {\Entr(g)}
	\arrow["\iso", from=1-1, to=1-2]
	\arrow["g", from=1-2, to=2-2]
	\arrow["{\Entr(M \to f)}"', from=1-1, to=2-1]
	\arrow["\iso"', from=2-1, to=2-2]
\end{tikzcd}. \qedhere
\]
\end{rmk}

\begin{notn}[Stratified fundamental 1-trusses] Stratified fundamental 1-trusses of stratified 1-meshes $(M,f)$ are essentially unique, and will be denoted by $\FTrs(M,f)$.
\end{notn}

\begin{notn}[Stratified classifying 1-meshes] Stratified classifying 1-meshes of stratified 1-trusses $(T,g)$ are unique up to contractible choice (when considered in an appropriate category, see \cref{cor:strat-mesh-truss-bundle}), and will be denoted by $\CMsh(T,g)$.
\end{notn}

\begin{eg}[Stratified 1-meshes and stratified 1-trusses] In \cref{fig:stratified-1-meshes-and-1-trusses}, on the left, we depict a stratified 1-mesh $(M,f)$ consisting of a 1-mesh $M$ and a coarsening $M \to f$; and on the right, the corresponding stratified 1-truss consisting of a 1-truss $T$ and a characteristic poset map $T \to \Entr(f)$.
\begin{figure}[ht]
    \centering
    \def\svgwidth{1\columnwidth}
    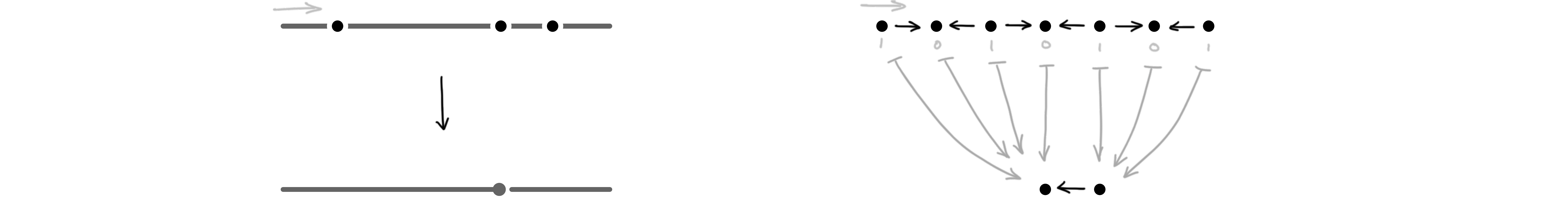

    \caption{Stratified 1-mesh and its stratified fundamental 1-truss}
    \label{fig:stratified-1-meshes-and-1-trusses}
\end{figure}
\end{eg}

\nid Just as we defined labeled 1-truss bundles earlier, we could now also define stratifications for 1-truss (and 1-mesh) bundles. We won't do so here, but will return to the notion in the more general case of $n$-trusses (and $n$-meshes) in the next section.

\section{Meshes and trusses in dimension \texorpdfstring{$n$}{n}}

Recall the bigger picture from \cref{sec:overview}: when studying manifold diagrams one observes an apparent inductive pattern; manifold $n$-diagrams project to manifold $(n-1)$-diagrams, which in turn project to manifold $(n-2)$-diagrams, and so on. With 1-mesh bundles playing a role analogous to that of such projections, we are lead to the following definition.

\begin{defn}[$n$-Meshes] An \textbf{$n$-mesh} $M$ is a tower of 1-mesh bundles
\[
	M_n \xto {p_n} M_{n-1} \xto {p_{n-1}}  ... \xto{p_2} M_1 \xto {p_1} M_0 = \ast. \qedhere
\]
\end{defn}

\nid At the level of structured entrance path posets, this has the following combinatorial counterpart.

\begin{defn}[$n$-Trusses] An \textbf{$n$-truss} $T$ is a tower of 1-truss bundles
\[
	T_n \xto {q_n} T_{n-1} \xto {q_{n-1}}  ... \xto{q_2} T_1 \xto {q_1} T_0 = \ast. \qedhere
\]
\end{defn}

\begin{notn}[Fundamental $n$-trusses] Given an $n$-mesh $M$ consisting of $1$-mesh bundles $p_i : M_i \to M_{i-1}$, $1 \leq i \leq n$, its `fundamental $n$-truss' $\FTrs M$ is the $n$-truss given by the 1-truss bundles $\FTrs p_i$.
\end{notn}

\begin{eg}[$2$-Meshes and 2-trusses] In \cref{fig:open-and-closed-2-meshes-and-their-respective-2-trusses} we depict an open and a closed 2-mesh, together with their corresponding fundamental 2-trusses.
\begin{figure}[ht]
    \centering
    \def\svgwidth{1\columnwidth}
    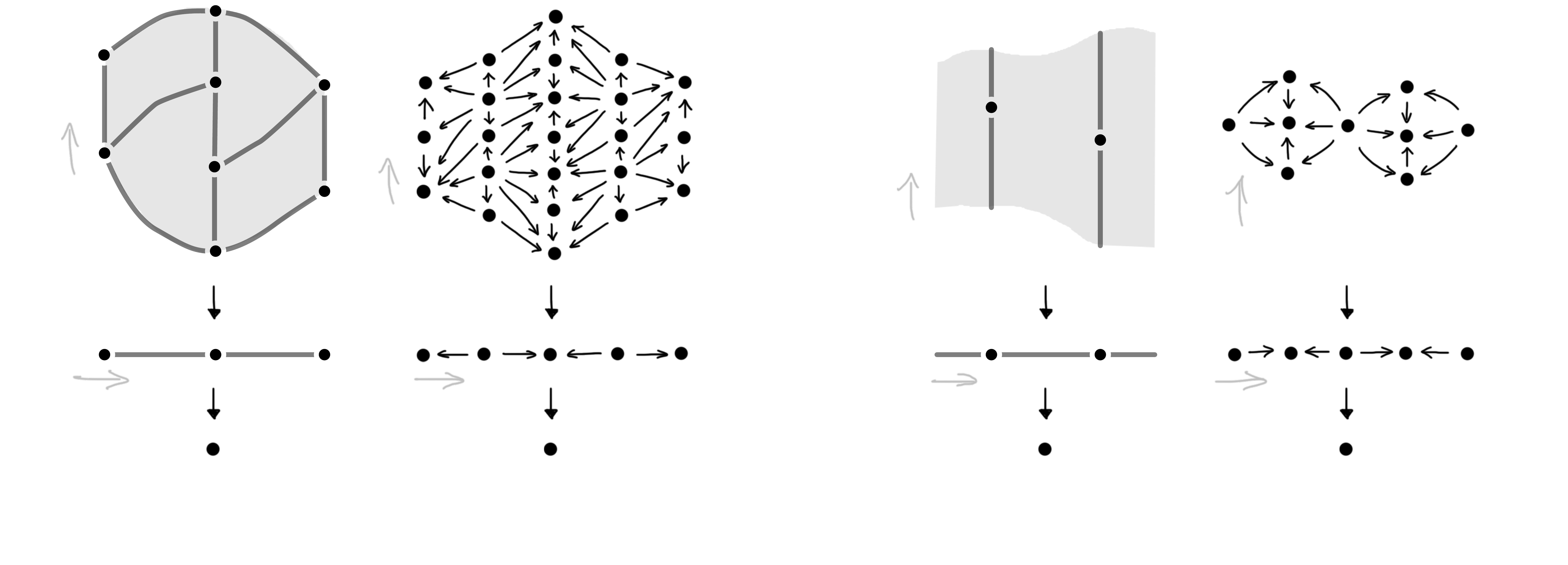

    \caption[Open and closed meshes and trusses]{Open and closed 2-meshes and their respective fundamental 2-trusses}
    \label{fig:open-and-closed-2-meshes-and-their-respective-2-trusses}
\end{figure}
\end{eg}

\nid The preceding example illustrates that all strata (in the total stratification $M_n$) of an $n$-mesh $M$ are $k$-cells, $k \leq n$---indeed, as outlined in the introduction, the role of $n$-meshes is precisely to `cellulate' other stratifications (such as manifold diagrams, to be defined in \cref{ch:mdiag}). For trusses, cell dimensions can be expressed in combinatorial terms as follows.

\begin{rmk}[Cell dimensions in $n$-trusses] \label{rmk:n-truss-dim} For an $n$-truss $T$ given by 1-truss bundles $q_i : T_i \to T_{i-1}$, and for $x \in T_n$, we write $x_i = q_{> i}(x)$ where $q_{> i} := q_{i+1} \circ ... \circ q_{n-1} \circ q_n$ maps $T_n \to T_i$. Write $\dim_i(x) = j$ if $\dim(x_i) = j$ in the 1-truss fiber over $x_{i-1}$. We define $\dim(x) := \sum_i \dim_i(x)$ and speak of that sum as the `cell dimension' of $x$. The cell dimension is a poset map $\dim : T_n \to [n]\op$. Note if $T = \FTrs M$, then $\dim(x)$ is the dimension of the cell in $M_n$ corresponding to $x$.
\end{rmk}

\subsection{Bundles, maps, and classification} \label{ssec:truss-mesh-bundles}

The definitions of $n$-meshes and $n$-trusses immediately generalize to bundles thereof.

\begin{defn}[$n$-Mesh bundles] An \textbf{$n$-mesh bundle $p$} over a `base' stratification $B$ is a tower of 1-mesh bundles $p_i : M_i \to M_{i-1}$, $1 \leq i \leq n$, ending in $M_0 = B$.
\end{defn}

\begin{defn}[$n$-Truss bundles] An \textbf{$n$-truss bundle $q$} over a `base' poset $P$ is a tower of 1-truss bundles $q_i : T_i \to T_{i-1}$, $1 \leq i \leq n$, ending in $T_0 = P$.
\end{defn}

\nid Note that both definitions also apply to the case $n = 0$: a $0$-mesh bundle is simply a stratification $B$, while a $0$-truss bundle is simply a poset $P$. As before, the two definitions are related as follows.

\begin{notn}[Fundamental $n$-truss bundles] \label{rmk:fund-truss-bundle} Given an $n$-mesh bundle $p = \{p_i : M_i \to M_{i-1}\}_{1 \leq i \leq n}$ over $B$, its `fundamental $n$-truss bundle' $\FTrs(p)$ is the essentially unique $n$-truss bundle over $\Entr(B)$ given by the $1$-truss bundles $\{\FTrs p_i : \Entr M_i \to \Entr M_{i-1}\}_{1 \leq i \leq n}$.
\end{notn}

\begin{notn}[Classifying $n$-mesh bundles] \label{term:classifying-mesh-bundles} An $n$-mesh bundle over $B$ is said to be the `classifying $n$-mesh bundle' of an $n$-truss bundle $q$ over $\Entr(B)$ if $q$ is the fundamental $n$-truss bundle of $p$. Classifying $n$-mesh bundles of $q$ are unique up to contractible choice (which will follow from \cref{thm:mesh-truss-bundle}, at least in the case of closed and open bundles), and we denote any such $n$-mesh bundle by $\CMsh q$.
\end{notn}

\begin{term}[Closed and open bundles] An $n$-mesh bundle $p$ is called `closed' (resp. `open') if all of its 1-mesh bundles $p_i$ are closed (resp.\ open). The terminology similarly applies to $n$-truss bundles.
\end{term}

\begin{eg}[$n$-Mesh bundles] We depict two examples of 2-mesh bundles over the stratified 1-simplex $\CStr [1]$ in \cref{fig:2-mesh-bundles-over-the-stratified-1-simplex}. (In each case, the reader may construct a 2-truss bundle by passing to fundamental truss bundles).
\begin{figure}[ht]
    \centering
    \def\svgwidth{1\columnwidth}
    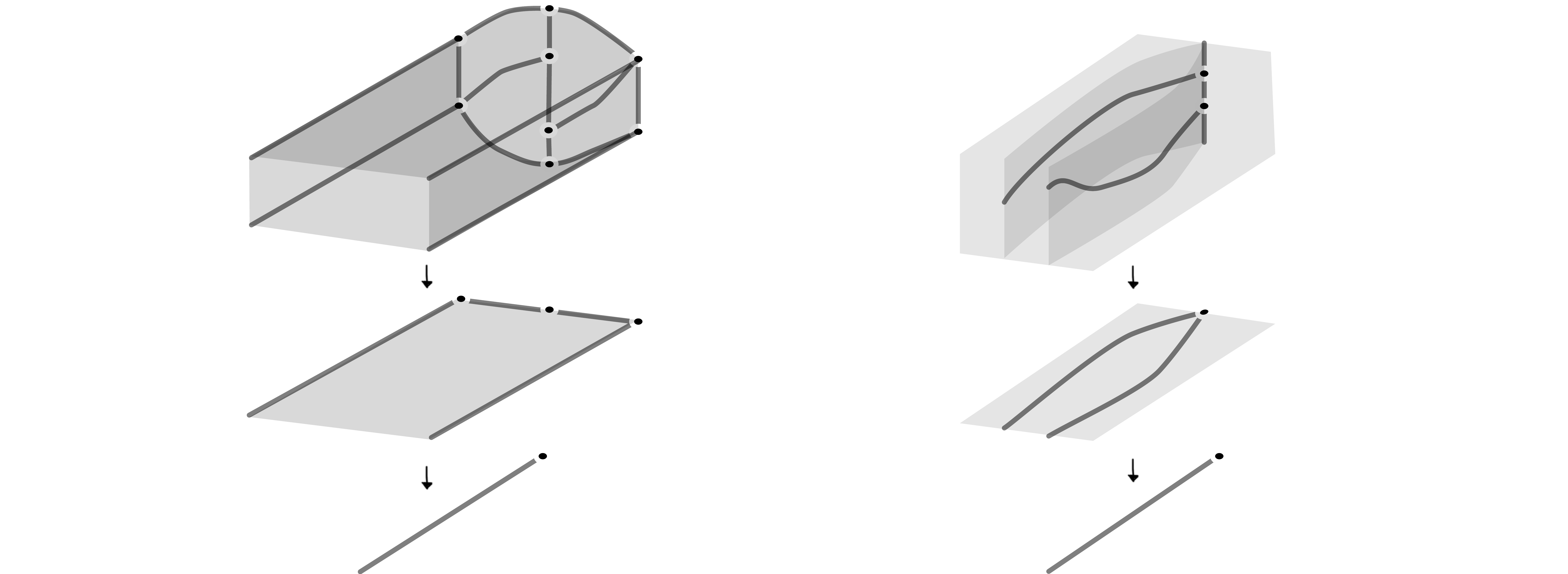

    \caption[2-Mesh bundles]{A closed resp.\ open 2-mesh bundle over the stratified 1-simplex}
    \label{fig:2-mesh-bundles-over-the-stratified-1-simplex}
\end{figure}
\end{eg}

We now discuss the theory mesh and truss bundles in more detail. Note that many of our definitions specialize to the case of $n$-meshes and $n$-trusses as well, but we will not explicitly repeat them.

\begin{notn}[Total stratifications and posets] \label{notn:total-strat-pos} Given an $n$-mesh bundle $p = \{p_i : M_i \to M_{i-1}\}_{1 \leq i \leq n}$, we call $M_n$ the `total stratification' of $p$ and denote it by $\Totz p$. Similarly, in the case of an $n$-truss bundles $q = \{q_i : T_i \to T_{i-1}\}_{1 \leq i \leq n}$ we call $T_n$ the `total poset' and denote it by $\Totz q$.
\end{notn}


\begin{defn}[Mesh bundle maps] Given $n$-mesh bundles $p = \{p_i : M_i \to M_{i-1}\}_{1 \leq i \leq n}$ and $p' =  \{p'_i : M'_i \to M'_{i-1}\}_{1 \leq i \leq n}$, an \textbf{$n$-mesh bundle map} $F : p \to p'$  is a tower of stratified maps $F_i : M_i \to M'_i$ that commute with all $p_i$, $p'_i$, and such that the $F_i$ restrict to 1-mesh maps on fibers.
\end{defn}

\begin{defn}[Truss bundle maps] Given $n$-truss bundles $q = \{q_i : T_i \to T_{i-1}\}_{1 \leq i \leq n}$ and $q' = \{q'_i : T'_i \to T'_{i-1}\}_{1 \leq i \leq n}$, an \textbf{$n$-truss bundle map}  $F : q \to q'$ is a tower of poset maps $F_i : T_i \to T'_i$ that commute with all $q_i$, $q'_i$, and such that the $F_i$ restrict to 1-truss maps on fibers.
\end{defn}

\begin{conv}[Fixing the base] \label{conv:fixing-the-base} When working with maps $F$ of bundles over the same base $B$, we assume that $F_0 = \id_B$ (for both mesh and truss bundle maps).
\end{conv}

\begin{notn}[Mesh and truss categories] $\tmeshbun n$ denotes the topologically enriched category of $n$-mesh bundles and their maps, $\tmesh n(B)$ its subcategory of bundles over $B$, and $\tmesh n$ its subcategory of $n$-meshes. Similarly, $\trussbun n$ denotes 1-category of $n$-trusses bundles and their maps, $\truss n(P)$ its subcategory of bundles over $P$, and $\truss n$ its subcategory of $n$-trusses.
\end{notn}

\begin{rmk}[Fundamental truss bundle functor] \label{notn:ftrs-bundle-map} Note that any $n$-mesh bundle map $F : p \to p'$ determines an $n$-truss bundle map $\FTrs p \to \FTrs p'$ by passing to entrance path posets componentwise; we denote this map by $\FTrs F$. This yields a topologically enriched functor $\FTrs : \tmesh n(B) \to \truss n (\Entr B)$ (see \cite[\S4.2.1]{fct}).
\end{rmk}

\nid For the next observation, recall that we call a category `essentially gaunt' if any two isomorphic objects are isomorphic by a unique isomorphism.

\begin{obs}[Essential gauntness of truss bundles] Truss bundle isomorphisms $F : q \to q'$ between $n$-truss bundles $q$, $q'$ over the same base (assumed to be fixed by $F$), are unique if they exist. By passing to any skeleton, we will replace all isomorphisms $q \iso q'$ of $n$-truss bundles by equalities $q = q'$.
\end{obs}

\begin{term}[Cellular and co-cellular $n$-truss bundle maps] \label{term:n-truss-bun-maps} An $n$-truss bundle map $F$ is called `cellular',  `cocellular', or `balanced' if each 1-truss bundle map $F_{i}$, $i > 0$, is fiberwise so in the sense of \cref{term:1-truss-maps}. It is a `subtruss bundle' (resp.\ a `truss bundle isomorphism') if firstly, each $F_{i}$, $i > 0$, is fiberwise so, and secondly, $F_0$ is a subposet (resp.\ a poset isomorphism). We similarly use the same terms `cellular', `cocellular', `balanced', `submesh bundle', and `mesh bundle isomorphism' for mesh bundle maps $F$, whenever it applies to the fundamental $n$-truss bundle map $\FTrs F$.
\end{term}

Next, we recall the classification of $n$-truss and $n$-mesh bundles.

\begin{defn}[$n$-Truss bordisms] The \textbf{category of $n$-trusses and their bordisms} $\kT^n$ is the category obtained by $n$ times applying the functor $\kT^1(-)$ to the terminal category $\ast$.
\end{defn}

\nid As in the case $n = 1$ (see \cref{rmk:gaunt}), the category of $n$-truss bordisms $\kT^n$ is essentially gaunt. Consequently, passing to a skeleton, we may again consider functors into $\kT^n$ up to equality.

\begin{obs}[Classification of $n$-truss bundles] \label{obs:ntrussbord} $n$-Truss bundles over a poset are classified by functors from the poset into $\kT^n$. To see this, consider the following diagram, whose top row is an $n$-truss bundle $q$ over $P$.
\[\begin{tikzcd}
	{T_n} & {T_{n-1}} & {T_{n-2}} & \cdots & {T_1} & {T_0 = P} \\
	\ast & {\kT^1(\ast)} & {\kT^2(\ast)} && {\kT^{n-1}(\ast)} & {\kT^n(\ast)}
	\arrow["{q_n}", from=1-1, to=1-2]
	\arrow["{q_{n-1}}", from=1-2, to=1-3]
	\arrow["{q_{n-2}}", from=1-3, to=1-4]
	\arrow["{q_2}", from=1-4, to=1-5]
	\arrow["{q_1}", from=1-5, to=1-6]
	\arrow[from=1-1, to=2-1]
	\arrow["{\chi_q^1}",from=1-2, to=2-2]
	\arrow["{\chi_q^2}",from=1-3, to=2-3]
	\arrow["{\chi_q^{n-1}}",from=1-5, to=2-5]
	\arrow["{\chi_q^n}",from=1-6, to=2-6]
\end{tikzcd}\]
Just considering the top 1-truss bundle $q_n$ together with the `trivial labeling' $T_n \to \ast$, we can apply \cref{constr:lab1bord} to find a classifying functor $\chi_q^1 : T_{n-1} \to \kT^1(\ast)$. But this classifying functor now provides a labeling for the 1-truss bundle $q_{n-1} : T_{n-1} \to T_{n-2}$; thus, applying \cref{constr:lab1bord} again, this $\kT^1(\ast)$-labeled bundle is now classified by a functor $\chi_q^2 : T_{n-2} \to \kT(\kT^1(\ast)) = \kT^2(\ast)$. Continuing inductively and defining $\chi_q := \chi_q^n$, we see that $n$-truss bundles $q$ over $P$ (up to structure preserving bundle isomorphism) are classified precisely by functors $\chi_q : P \to \kT^n$.
\end{obs}

\begin{lem}[Classification of $n$-mesh bundles] \label{lem:classnmeshbun} Bundle isomorphism classes of closed (or open) $n$-mesh bundles with base stratification $B$ bijectively correspond to functors $\Entr(B) \to \kT^n$. The correspondence maps $p \mapsto \chi_{\FTrs(p)}$.
\end{lem}

\nid We remark that the assumption that bundles are closed resp.\ open is used as it simplifies the types of cellular subrefinements one encounters.

\begin{proof}[Proof of \cref{lem:classnmeshbun}] Follows from \cite[Prop.\ 4.2.22]{fct} and \cite[Obs.\  4.2.68]{fct}.
\end{proof}


\begin{rmk}[The 1-constructible case] \label{rmk:base-cat-n} Analogous to \cref{rmk:base-cat}, the above lemma generalizes to the 1-constructible case: bundle isomorphisms classes of 1-constructible $n$-mesh bundles over a (not necessarily $0$-truncated) base stratifications $B$ correspond to functors from the entrance path $\infty$-category $\TEntr(B)$ into the 1-category $\kT^n$.
\end{rmk}

\subsection{Labelings and stratifications} \label{ssec:label-strat-nmesh}

\subsubsection*{Labelings and stratifications of $n$-trusses}

We now endow $n$-truss bundles with labelings and stratifications. Again, our definitions equally apply to the case of `non-bundled' trusses and meshes and we will not separately spell out either of these special cases.

\begin{defn}[Labeled truss bundles] \label{defn:labntrussbun} Let $\iC$ be a category. A \textbf{$\iC$-labeled $n$-truss bundle $(q,f)$} over $P$ is an `underlying' $n$-truss bundle $q$ over $P$, together with a `labeling' functor $f : \Totz {q} \to \iC$.
\end{defn}

\begin{obs}[Classification of labeled $n$-truss bundles] \label{obs:labeled-n-truss-bord} $\iC$-labeled $n$-truss bundles are classified by the category $\kT^n(\iC)$ of `$\iC$-labeled $n$-truss bordisms' obtained by $n$ times applying the functor $\kT^1(-)$ to the category $\iC$. This uses the same argument as \cref{obs:ntrussbord}, but with the category $\iC$ in place of $\ast$.
\end{obs}

\begin{defn}[Stratified $n$-truss bundles] A \textbf{stratified $n$-truss bundle} $(q,f)$ over a poset $P$ is a labeled $n$-truss bundle over $P$ such that
    \begin{enumerate}
        \item[(1)] $f$ stratifies the poset $\Totz q$ (in the sense of \cref{rmk:strat-posets}),
        \item[(2)] the map $f : \Totz q \to \Entr(f)$ factors through $q_{>0} : \Totz q \to P$ by a map $\Entr(f) \to P$ (where $q_{>0}$ is the composite $q_n \circ q_{n-1} \circ ... \circ q_1$). \qedhere
    \end{enumerate}
\end{defn}

\nid Note, condition (2) in the preceding definition guarantees that strata of $(\Totz q, f)$ live in the fibers of $q_{>0} : \Totz q \to P$.

\begin{rmk}[Trivial stratifications] Every $n$-truss bundle $q$ is trivially stratified $(q,f)$, by defining $f : \Totz q \to \ast$ to be the unique map to the terminal poset $\ast$.
\end{rmk}

\begin{defn}[Stratified maps] Given two stratified $n$-truss bundles $(q,f)$ and $(q',f')$, a \textbf{stratified map} $F : (q,f) \to (q',f')$ is a truss bundle map $F : q \to q'$ whose top component $F_n : \Totz q \to \Totz q'$ factors through $f$ and $f'$ by a (necessarily unique) map $\Entr(F) : \Entr(f) \to \Entr(f')$, that is, $f' \circ F_n = \Entr(F) \circ f$.
\end{defn}

\begin{term}[Stratified isomorphism] A stratified map $F : (q,f) \to (q',f)$ of two stratified $n$-truss bundles $(q,f), (q',f')$ over the same base $P$ (where $F_0 = \id_P$, see \cref{conv:fixing-the-base}) is called a `stratified isomorphism' if $F : q \to q'$ is a truss bundle isomorphism (see \cref{term:n-truss-bun-maps}) and $\Entr(F)$ is a poset isomorphism.
\end{term}

\begin{obs}[Essential gauntness of stratified truss bundles] \label{rmk:strat-truss-iso} Stratified truss bundle isomorphisms over a fixed base are unique if they exist. Thus, without harm we can pass to a skeleton, and work with equality $(q,f) = (q',f')$ in place of stratified isomorphisms $(q,f) \iso (q',f')$.
\end{obs}

\begin{term}[Stratified subtrusses] \label{term:stratified-subtruss} A stratified map $F : (q,f) \to (q',f')$ of stratified $n$-truss bundles is called a `stratified subtruss bundle' if $F : q \to q'$ is a subtruss bundle (see \cref{term:n-truss-bun-maps}) and $\Entr(F)$ is conservative.\footnote{This last condition mirrors the fact that substratifications are conservative maps on entrance path posets, see \cref{recoll:strat-maps}.}
\end{term}

\begin{rmk}[Stratified truss bundle pullback] \label{rmk:strat-truss-pullback} Given a stratified $n$-truss bundle $(q,f)$ over $P$, and a map $F : Q \to P$, then $(q,f)$ pulls back to a stratified $n$-truss bundle $F^*(q,f) \equiv (F^*q,F^*f)$ (formally, we first define the labeled truss bundle $(S,g)$ by the classifying map $\chi_{(q,f)} \circ F$, and then set $F^*q = S$ and $F^*f = \ccs g$, see \cref{rmk:strat-truss-from-pos-lab}). If $F$ is a subposet inclusion $Q \into P$, then we also write $\rest {(q,f)} Q \equiv (\rest q Q, \rest f Q)$ for this pullback. (Note $\rest {(q,f)} Q \into (q,f)$ is a stratified subtruss.)
\end{rmk}

Let us now discuss the corresponding definitions for meshes. Recall a strict coarsening is a coarsening of stratified spaces that is an identity on underlying topological spaces.

\begin{defn}[Stratified $n$-mesh bundles] \label{defn:strat-mesh-bun} A \textbf{stratified $n$-mesh bundle} $(p,f)$ an $n$-mesh bundle $p$ together with a strict coarsening $\Totz p \to f$ to a stratification $f$.
\end{defn}

\nid Note we may equivalently think of stratified mesh bundles as tuples $(p,f)$ where $p$ is a mesh bundle and $(\Entr (\Totz p),f)$ is a stratified poset (see \cref{obs:lab-strat-top}). This yields the following relation of stratified trusses and meshes.

\begin{rmk}[Relating stratified $n$-truss and $n$-mesh bundles] Consider a stratified $n$-mesh bundle $(p,f)$ over $B$, and a stratified $n$-truss bundle $(q,g)$ over $\Entr(B)$. If $q = \FTrs p$ and if there is (a necessarily unique) isomorphism $\Entr(g) \iso \Entr(f)$ that commutes with $g$, $f$ and $\Totz q \iso \Entr (\Totz p)$, then we call $(q,g)$ a `stratified fundamental truss bundle' of $(p,f)$, and conversely, say $(p,f)$ is a `stratified classifying mesh bundle' of $(q,g)$.
\end{rmk}

\begin{notn}[Stratified fundamental $n$-truss bundles] Stratified fundamental $n$-truss bundles of stratified $n$-mesh bundles $(p,f)$ are essentially unique, and we denote them by $\FTrs(p,f)$.
\end{notn}

\begin{notn}[Stratified classifying $n$-mesh bundles] Stratified classifying $n$-meshes bundles of stratified $n$-truss bundles $(q,g)$ are unique up to contractible choice (this will follow from \cref{cor:strat-mesh-truss-bundle}, at least for open/closed bundles), and we denote them by $\CMsh (q,g)$.
\end{notn}

\begin{defn}[Stratified mesh bundle maps] Given stratified mesh bundles $(p,f)$, $(p',f')$ a \textbf{stratified map} $F : (p,f) \to (p',f')$ is a mesh bundle map $F : p \to p'$ whose top component $F_n : \Totz p \to \Totz p'$ descends along the strict coarsenings $\Totz p \to f$ and $\Totz p' \to f'$ to a stratified map $f \to f'$.
\end{defn}

\begin{rmk}[Fundamental stratified truss bundle maps] Note, the construction of fundamental truss bundle maps (see \cref{notn:ftrs-bundle-map}) carries over to the setting stratified maps: any stratified map $F : (p,f) \to (p',f')$ determines a stratified map of truss bundles $\FTrs F : \FTrs (p,f) \to \FTrs (p',f')$.
\end{rmk}

We will be particularly interested in the following class of stratified maps.

\begin{defn}[Mesh bundle coarsenings] \label{defn:mesh-coarsening} Given stratified mesh bundles $(p,f)$, $(p',f')$ over the same base stratification $B$, a \textbf{mesh bundle coarsening} $F : (p,f) \to (p',f')$ is a stratified map whose top component is a coarsening $F_n : \Totz p \to \Totz p'$ that descends to a stratified homeomorphism $f \iso f'$. We call $F$ a \textbf{strict} mesh bundle coarsening if $F_n$ is a strict coarsening and $f = f'$.
\end{defn}

\begin{eg}[Mesh coarsenings] If the base stratification is trivial, the notion of mesh bundle coarsening specializes to that of `mesh coarsening'. An example of a strict mesh coarsening of stratified 2-meshes is given in \cref{fig:mesh-coarsening} (note that we omit the depiction of the tower of 1-mesh bundles).
\begin{figure}[ht]
    \centering
    \def\svgwidth{1\columnwidth}
    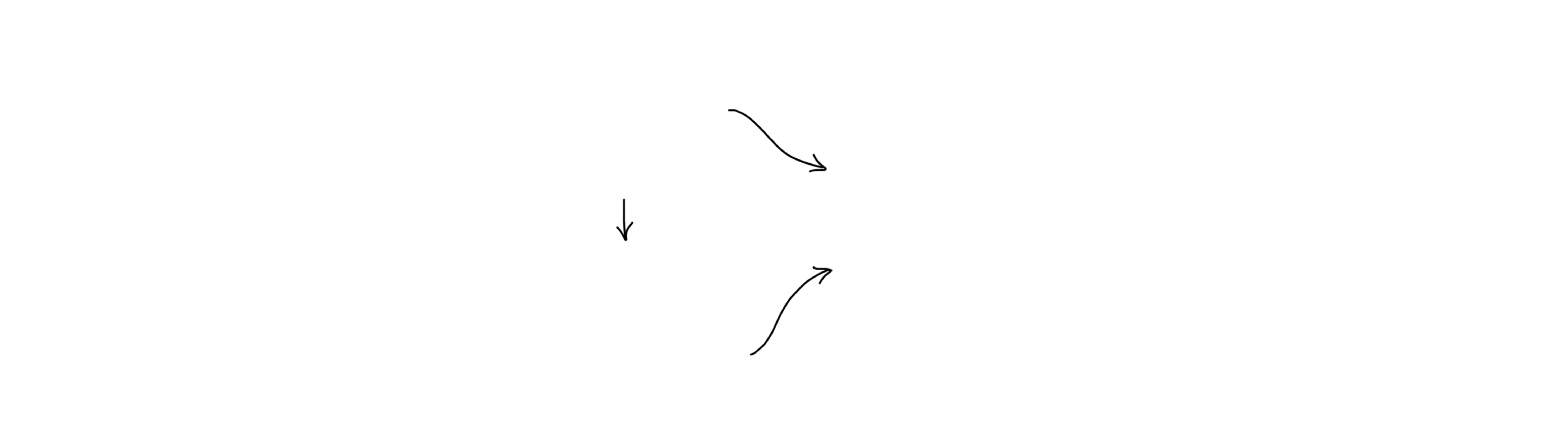

    \caption{A mesh coarsening between stratified 2-meshes}
    \label{fig:mesh-coarsening}
\end{figure}
\end{eg}

\nid The notion of mesh coarsenings, in turn, has the following combinatorial counterpart.

\begin{defn}[Truss bundle coarsenings] \label{defn:truss-coarsening} A stratified map $F : (q,f) \to (q',f')$ of stratified truss bundles is called a \textbf{truss bundle coarsening} if there is a `classifying' mesh bundle coarsening $G$ such that $F = \FTrs G$.
\end{defn}

\nid The definition also has a purely combinatorial phrasing, as follows.

\begin{altdefn}[Truss bundle coarsenings, combinatorially] \label{term:truss-coarsening} A stratified map $F : (q,f) \to (q',f')$ is a truss bundle coarsening if $\Entr(F)$ is an isomorphism and the underlying truss bundle map $F : q \to q'$ is a surjective cocellular map such that $F_i$, $i > 0$, preserves dimensions of endpoints of all 1-truss fibers. (Note, if $q$ and $q'$ are open truss bundles this last condition becomes redundant as it is implied by cocellularity.)
\end{altdefn}

\nid Note also that the definition equally applies in the (trivially stratified) case of truss bundles, allowing us to speak of truss coarsenings $F : q \to q'$.

\begin{rmk}[Relation of mesh and truss coarsenings] \label{obs:truss-and-mesh-coarsening} Given a stratified mesh bundle $(p,f)$ and with stratified fundamental truss bundle $(q,g) = \FTrs (p,f)$ then strict mesh coarsenings $(p,f) \to (p',f')$ are in 1-to-1 correspondence with truss bundle coarsenings $(q,g) \to (q',g')$: the correspondence takes a strict mesh coarsening $F$ to the fundamental truss bundle map $\FTrs F$.
\end{rmk}

\subsection{Framing structures of \texorpdfstring{$n$}{n}-meshes} \label{ssec:coordinates} Recall that, by definition, the underlying manifolds of 1-meshes carry a framing. We defined $n$-meshes as towers of bundles of 1-meshes. As a consequence, $n$-meshes too carry a type of framing structure. In this section, we give a precise description of this structure, which we will refer to as a `flat framing structure'. The structure is a special (namely, `flat') case of a more general notion of framings as follows.

\begin{term}[$n$-Framed space] \label{term:flat-framed-space} Let $X$ be a topological space.
\begin{enumerate}
    \item[$-$] An `$n$-framed chart' $(U,\gamma)$ in $X$ is an embedding $\gamma : U \into \lR^n$ of a subspace $U \subset X$. If $U' \subset U$, then $\gamma$ restricts to the `restricted chart' $(U',\gamma)$.
    \item[$-$] Given two $n$-framed charts $(U,\gamma)$, $(V,\rho)$ in spaces $X$ resp.\ $Y$, a map $F : X \to Y$ is said to be `framed' on the given charts, if it restricts to a map $F : U \to V$ such that, for all $i$, denoting by $\pi_{> i}$ the projection $\lR^i \times \lR^{n-i} \to \lR^i$, there is a continuous partial mapping $F_i : \lR^i \to \lR^i$ such that $\pi_{> i} \circ \rho \circ F = F_i \circ \pi_{> i} \circ \gamma$ commutes.
\item[$-$] Two $n$-framed charts $(U,\gamma)$, $(V,\gamma')$ in $X$ are `framed compatible' if $\id : X = X$ is a framed map on the restricted charts $(U \cap V, \gamma)$ resp.\ $(U \cap V,\gamma')$.
\item[$-$] An `$n$-framing structure' on $X$ is an `atlas' $\cA$ of framed compatible charts $\{(U_i,\gamma_i)\}$ such that the subspaces $U_i$ cover $X$.
\item[$-$] A `global $n$-framed chart' of $X$ is a chart $(X,\gamma)$ of all of $X$. A `flat $n$-framing structure' on $X$ is an atlas $\cA$ containing a global $n$-framed chart.
\item[$-$] Given spaces with $n$-framing structure $(X,\cA) \to (Y,\cB)$ then a map $F : X \to Y$ is said to be `framed' for each point $x \in X$, there are charts $(U_i,\gamma_i) \in \cA$, $(V_i,\rho_i) \in \cB$ and a neighborhood $U'_i \subset U_i$ of $x$, such that $F$ is framed on the charts $(U'_i,\gamma_i)$ and $(V_i,\rho_i)$. \qedhere
\end{enumerate}
\end{term}

\nid Note that, given a flat $n$-framed space $(X,\cA)$, any global $n$-framed chart $(X,\gamma) \in \cA$ determines $(X,\cA)$ up to framed homeomorphism; we therefore often denote flat framed spaces by giving a single global $n$-framed chart $(X,\gamma)$. Moreover, since we will, in fact, only ever endow spaces with flat framing structures, we usually refer to `global $n$-framed charts of $X$' simply as `$n$-framed charts of $X$'. The next remark points out an important convention regarding the indexing of coordinates when working with flat framed spaces.


\begin{rmk}[Categorical directions of $\lR^n$] \label{notn:coord-axis-reverse} We refer to the $(n-k)$th coordinate of $\lR^n$ as its `$k$th categorical direction'. We usually label the coordinate axes of $\lR^n$ by their categorical directions: that is, we label the coordinate axes of $\lR^n = (\lR \times \lR \times ... \times \lR)$ by $(n,n-1,\dots,1)$. We also often use `$k$-arrows' (arrows with $k$ parallel lines) with the convention that $k$-arrows point in the $k$th categorical direction. For $n = 2$ this is illustrated in \cref{fig:cube-axis-labeling} on the right. The reason for this convention lies in the `categorical' interpretation of framed space that we will discuss later in \cref{sec:cell-diagrams}.
\end{rmk}

\begin{eg}[Flat framed space] A space $X$ with a $2$-framed chart $\gamma$ into $\lR^2$ is shown in \cref{fig:cube-axis-labeling}. The axes of $\lR^2$ are labelled by categorical directions (that is, the first coordinate axis is labeled with `2', and the second with `1'). Grey lines indicate the preimages of the projection $\pi_{2} : \lR^2 \to \lR^1$ mapping $(x_1,x_2)$ to $x_1$. We indicate the standard framing $(e_1,e_2)$ of $\lR^2$ by a grid of coordinate frames ($e_1$ and $e_2$ being the constant vector fields $(1,0)$ resp.\ $(0,1)$; to depict vectors we use $k$-arrows as explained in \cref{notn:coord-axis-reverse}). We also show how these frames pull back to $X$.
    \begin{figure}[ht]
    \centering
    \def\svgwidth{1\columnwidth}
    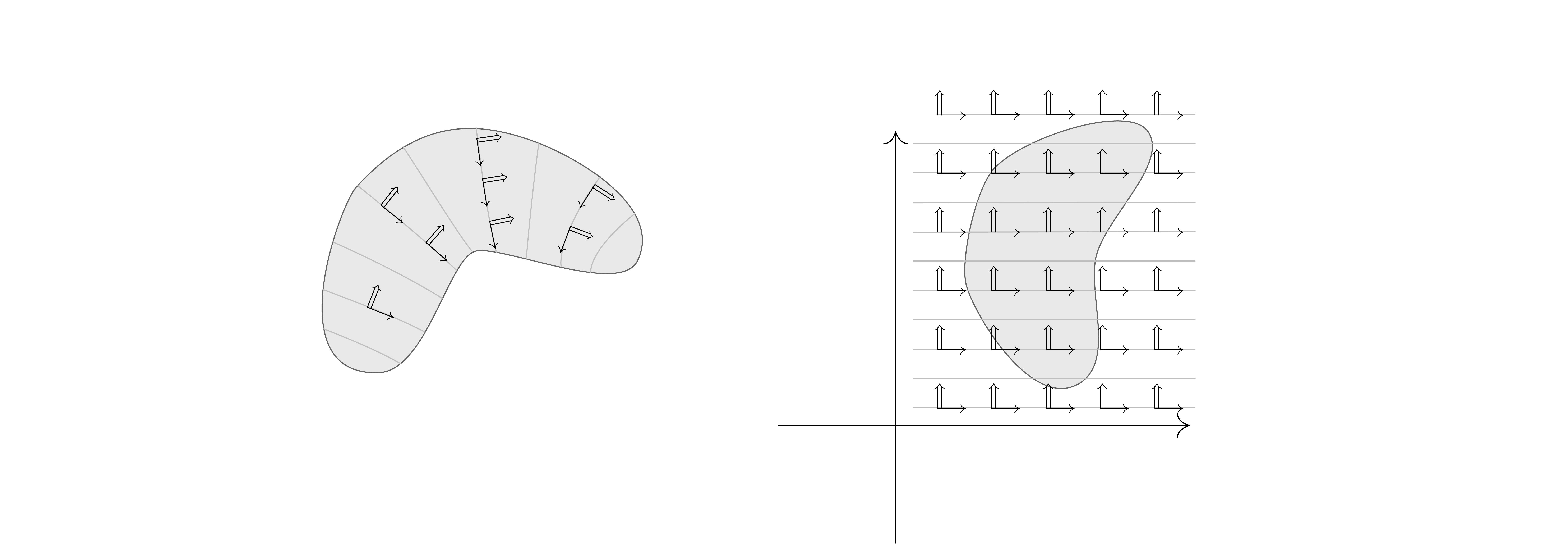

    \caption[Charts and categorical directions]{A flat framed space with a chart into standard framed euclidean space}
    \label{fig:cube-axis-labeling}
\end{figure}
\end{eg}

\begin{term}[Flat framed bundles] \label{term:flat-framed-bundles} Given a `base' space $B$, all notions from \cref{term:flat-framed-space} have immediate analogs for the case of `bundles over $B$': for this, replace $\lR^n$ with $B \times \lR^n$, and $\pi_{> i} : \lR^i \times \lR^{n-i} \to \lR^i$ with $B \times \pi_{> i} : B \times \lR^i \times \lR^{n-i} \to B \times \lR^i$, and require that $F_0 = \id_B$ for framed maps $F$. For instance, we speak of an `$n$-framed chart bundle $(U,\gamma)$ over $B$' to mean an embedding $\gamma : U \into B \times \lR^n$.
\end{term}

\begin{term}[Framed stratified notions] \label{term:framed-everything} We add the adjective `(flat) framed' to standard notions of stratification theory to indicate that the underlying topological spaces are flat framed spaces, and that the underlying maps are framed maps (this similarly applies in the case of bundles).
\end{term}

The relevance of flat framed spaces (and their framed maps) to us here stems from the next observations. These show that any $n$-mesh (and similarly, any $n$-mesh bundle) has a canonical flat $n$-framing structure. We will refer to charts in this canonical framing structure as the `coordinatizations' of the $n$-mesh.

\begin{obs}[Canonical framing structure of meshes and mesh bundles] \label{obs:mesh-bundle-framing-struct} Given an $n$-mesh $M = \{p_i : M_i \to M_{i-1}\}_{1 \leq i \leq n}$, a `coordinatization' is an embedding $\gamma_n : M_n \into \lR^n$, such that for all $i$, the composite projection $p_{> i} = p_{i+1} \circ ... \circ p_{n} : M_n \to M_i$ factors through $\pi_{> i} \circ \gamma_n : M_n \to \lR^i$ by an embedding $\gamma_i : M_i \into \lR^i$ and this embedding defines a coordinatizing embedding of the 1-mesh bundle $p_i$ in the sense of \cref{defn:1-mesh-bun} (up to identifying $M_{i-1} \iso \im(\gamma_{i-1})$ via $\gamma_{i-1}$). The fact that coordinatizations of $n$-meshes exist can be seen as follows. Inductively, construct a coordinatization $\gamma_{n-1} : M_{n-1} \into \lR^{n-1}$ of the $(n-1)$-mesh $M = \{p_i : M_i \to M_{i-1}\}_{1 \leq i \leq n-1}$. Pick a coordinatizing embedding $M_n \into M_{n-1} \times \lR$ for the 1-mesh bundle $p_n : M_n \to M_{n-1}$ and denote the mapping of this embedding by $x \mapsto (p_n(x),\gamma(x))$. Set $\gamma_n : M_n \to \lR^n$ to map $x \mapsto (\gamma_{n-1}\circ p_n(x), \gamma(x))$. This constructs a coordinatization $\gamma_n$ of $M$ as claimed. One can similarly define and construct coordinatizations $\gamma_n$ in $B \times \lR^n$ for $n$-mesh bundles over $B$. By construction, all coordinatizations of a given $n$-mesh (resp.\ $n$-mesh bundle) are framed compatible charts of $M_n$ and thus define a flat framing structure.
\end{obs}

\begin{obs}[Mesh maps are framed] Equipping meshes (and similarly mesh bundles) with their canonical framing structures, observe that all mesh maps are framed maps in the sense of \cref{term:flat-framed-space}.
\end{obs}

\nid We often prefer to work in coordinates instead of abstract topological spaces. For instance, in \cref{ch:mdiag} and \cref{ch:tangles-and-sings} our main interest will lie with the open $n$-cube $(-1,1)^n$ with flat framing inherited from the standard framing of $\lR^n$.

\begin{obs}[Coordinatizing open $n$-meshes] Any open $n$-mesh $M$ has a coordinatization $\gamma : M \into \lR^n$ whose image is the open cube $\II^n = (-1,1)^n \subset \lR^n$.
\end{obs}

\begin{note}[On the notion of framing] The notion of `framing' described in \cref{term:flat-framed-space} is not immediately recognizable as being related to the standard meaning of the term. A relation is explained in combinatorial terms in \cite[\S1]{fct}, and in more classical terms in \cite[App.\ A]{fct} where it is observed that the notion of framings described here is related to a `metric-free' generalization of orthonormal frames.
\end{note}

\section{Combinatorializibility of meshes and tame stratifications} \label{sec:meshtrusseqv}

In this section we summarize results about the `combinatorializability of flat framed stratified space' from \cite[\S5]{fct}. We first discuss the equivalence of meshes and trusses. We then add stratifications into the mix, showing that `tame' stratifications of flat framed space can be combinatorialized by so-called `normalized stratified trusses'.

\subsection{Combinatorialization of meshes} \label{ssec:comb-result-meshes}

We begin by recalling that the fundamental truss functor is a weak equivalence between the $\infty$-category of meshes (resp.\ mesh bundles) and the 1-category of trusses (resp.\ truss bundles). In \cite{fct}, this is proven in the following `closed-cellular' and `open-cocellular' cases.\footnote{One reason for only considering these cases is, as explained in \cite[\S4.2]{fct}, that they are the easiest to handle for $(\infty,1)$-categorical proof methods (in the other cases one encounters $(\infty,2)$-categories). Another reason is that these two cases will be the most meaningful to us.}

\begin{notn}[Cellular and cocellular subcategories] \label{notn:closed-cell-and-open-cocell-cat} Let $\ctmesh n$ (resp. $\otmesh n$) denote the subcategory of $\tmesh n$ of closed (resp. open) $n$-meshes with cellular (resp. cocellular) maps. Similarly, let $\ctruss n$ (resp. $\otruss n$) denote the subcategory of $\truss n$ of closed (resp. open) $n$-trusses with cellular (resp. cocellular) maps.
\end{notn}

\nid Recall from \cref{notn:ftrs-bundle-map} that $\FTrs$ is a functor of topologically enriched categories.

\begin{thm}[Equivalence of meshes and trusses, {\cite[Thm. 4.2.1]{fct}}] \label{thm:mesh-truss} The functors $\FTrs : \ctmesh n \to \ctruss n$ and $\FTrs : \otmesh n \to \otruss n$ are weak equivalences of $\infty$-categories.
\end{thm}

\begin{rmk}[The classifying mesh functor] \label{rmk:cmsh-functor-choice} Choosing a weak inverse to the functor $\FTrs$ gives rise to a `classifying mesh' functor $\CMsh : \ctruss n \to \ctmesh n$ (and similarly  $\CMsh : \otruss n \to \otmesh n$). We discuss a concrete construction later in \cref{rmk:cmsh-via-compactifiation}.
\end{rmk}

\nid \cref{thm:mesh-truss} immediately generalizes to the case of labeled (or stratified) meshes and trusses, since these structures are, by definition, in correspondence for meshes and their fundamental trusses. We record this generalization in the case of stratification structures (we add the prefix `Str' to indicate the straight-forward `stratified' analogues of the categories from \cref{notn:closed-cell-and-open-cocell-cat}).

\begin{cor}[Equivalence of stratified meshes and trusses] \label{cor:strat-mesh-truss} The functor $\FTrs$ gives weak equivalences between $\strctmesh n$ and $\strctruss n$, resp.\ $\strotmesh n$ and $\strotruss n$. \qed
\end{cor}

As an application of \cref{thm:mesh-truss}, we discuss the geometric dualization of meshes.

\begin{defn}[Truss dualization functor] \label{defn:dual-n-trusses} The \textbf{truss dualization} $\dagger : \truss n \to \truss n$ takes $n$-trusses $T = \{p_i : T_i \to T_{i-1}\}_{1 \leq i \leq n}$ to their \textbf{dual $n$-trusses} $T^\dagger = \{p^\dagger_i : T\op_i \to T\op_{i-1}\}_{1 \leq i \leq n}$ (where $p^\dagger_i$ is the dual bundle of $p_i$, see \cref{defn:dual-bundle}), and truss maps $F : T \to S$ to their \textbf{dual truss maps} $F^\dagger : T^\dagger \to S^\dagger$ whose mapping on objects equals that of $F$.
\end{defn}

\nid Note that $\dagger$ restricts to an isomorphisms of categories $\dagger : \ctruss n \iso \otruss n$. Together with \cref{thm:mesh-truss} this implies the following.

\begin{cor}[Dualization of $n$-meshes, {\cite[Cor. 4.2.4]{fct}}] \label{defn:dual-n-meshes}  There is an $\infty$-functor $\dagger : \ctmesh n \eqv \otmesh n$ determined by requiring $\FTrs{} \circ \dagger = \dagger \circ \FTrs$. The functor is a weak equivalence.
\end{cor}

\begin{eg}[Dualization of $n$-meshes] We illustrate the action of dualization of 2-meshes in an example in \cref{fig:dualization-of-2-meshes}.
\begin{figure}[ht]
    \centering
    \def\svgwidth{1\columnwidth}
    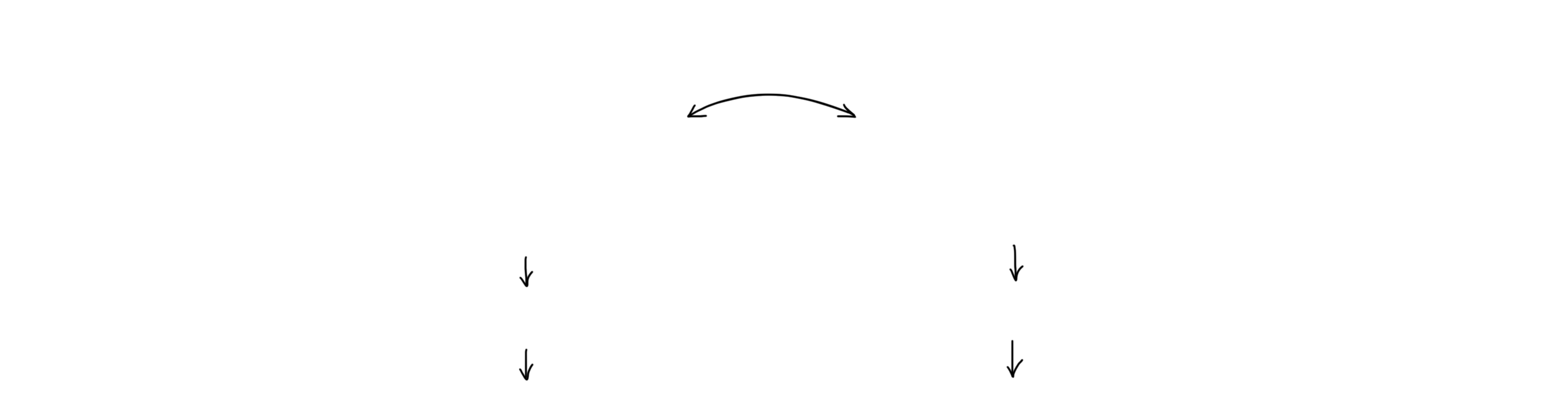

    \caption{Dualization of 2-meshes}
    \label{fig:dualization-of-2-meshes}
\end{figure}
\end{eg}

\cref{thm:mesh-truss} also generalizes to the case of mesh and truss bundles. (Our previous \cref{notn:closed-cell-and-open-cocell-cat} for closed-cellular and open-cocellular meshes and trusses generalizes to the case of bundles over a fixed base; notationally we indicate this by adding `$(B)$' for the base stratifications $B$ of mesh bundles resp.\ `$(P)$' for base posets $P$ of truss bundles.)

\begin{thm}[Equivalence of mesh bundles and truss bundles, {\cite[Thm. 4.2.2]{fct}}] \label{thm:mesh-truss-bundle} The functors $\FTrs : \ctmesh n(B) \to \ctruss n(\Entr B)$ and $\FTrs : \otmesh n (B) \to \otruss n (\Entr B)$ are weak equivalences of $\infty$-categories.
\end{thm}

\nid As before, the theorem also carries over to the stratified case, directly generalizing our previous \cref{cor:strat-mesh-truss} (with the evident notational changes for bundles).

\begin{cor}[Equivalence of stratified mesh and truss bundles] \label{cor:strat-mesh-truss-bundle} The functor $\FTrs$ gives weak equivalences between $\strctmesh n (B)$ and $\strctruss n (\Entr B)$, resp.\ between $\strotmesh n(B)$ and $\strotruss n(\Entr B)$. \qed
\end{cor}

\begin{note}[The 1-constructible case, cf.\ {\cite[Rmk.\ 4.2.4]{fct}}] Following the earlier \cref{rmk:base-cat} and \cref{rmk:base-cat-n}, \cref{thm:mesh-truss-bundle} has a yet further variation (with essentially the same proof): the $\infty$-category of closed-cellular (resp.\ open-cocellular) 1-constructible $n$-mesh bundles over a (not necessarily $0$-truncated) base stratification $B$ is weakly equivalent to corresponding category of $n$-truss bundles over the entrance path $\infty$-category $\TEntr(B)$ (here, an `$n$-truss bundle over an $\infty$-category $\cC$' describes the data of a functor $\cC \to \kT^n$, but we will omit a detailed discussion of the notion).
\end{note}

\subsection{Tame stratifications and coarsest refining meshes} \label{ssec:comb-tame-strat}

We now shift focus to stratifications of flat framed spaces. More precisely, we are interested in those stratifications that admit a framed refinement by a mesh in the following sense.

\begin{term}[Mesh refinements] \label{term:mesh-refinements} Given a stratification $(X,f)$ of a flat $n$-framed space $(X,\gamma)$, a `mesh refinement' of $f$ by an $n$-mesh $M$ is a framed strict coarsening $M_n \to f$, where $M_n$ is the total stratification of $M$. We will denote mesh refinements by writing $M \mshar f$.
\end{term}

\nid Note that any stratified mesh $(M,f)$ yields a stratification $f$ together with a mesh refinement $M \mshar f$. (And conversely, any mesh refinement $M \mshar f$ of a stratification $f$ defines a stratified mesh $(M,f)$.)

\begin{defn}[Tame stratifications] \label{defn:flat-framed-stratification} A \textbf{tame stratification} $(X,f,\gamma)$ is a stratification $(X,f)$ of a flat framed space $(X,\gamma)$ that has a mesh refinement.
\end{defn}

\begin{rmk}[`Tame' vs.\ `flat framed' stratifications] Tame stratifications of flat framed spaces were simply called `flat framed stratifications' in \cite[\S5]{fct}. For our purposes here it will be useful to distinguish both `tame' and `not necessarily tame' stratifications of flat framed spaces.
\end{rmk}

\nid To simplify notation, we will usually work with subspaces of $\lR^n$.

\begin{notn}[Working with subspaces] \label{obs:stratified-meshes-vs-ref-of-tame-strat} We henceforth drop framed charts $\gamma$ from our notation, and work with subspaces $X \subset \lR^n$ instead. In particular, we abbreviate tame stratifications $(X,f,\gamma)$ by $(X,f)$ or $f$.
\end{notn}

\nid Recall that a `framed stratified map' is a stratified map whose underlying map is a framed map of framed spaces (see \cref{term:framed-everything}).

\begin{defn}[Tame framed stratified maps] \label{term:tame-maps} A \textbf{tame framed stratified map} $F : (X,f) \to (Y,g)$ of tame stratifications is a framed stratified map such that there exist mesh refinements $M \mshar f$ and $M' \mshar f'$ (see \cref{term:mesh-refinements}) through which $F$ factors by an $n$-mesh map $G : M \to M'$, i.e.\ the following commutes
\[\begin{tikzcd}[baseline=(W.base)]
	{(X,f)} & {(Y,f')} \\
	{M_n} & |[alias=W]| {M'_n}
	\arrow["F", from=1-1, to=1-2]
	\arrow[from=2-2, to=1-2]
	\arrow["G_n", from=2-1, to=2-2]
	\arrow[from=2-1, to=1-1]
\end{tikzcd} . \qedhere
\]
\end{defn}

A priori, a tame stratification can have many refining meshes. The next result observes that there is in fact a canonical choice. The result plays a central role for the combinatorializibility of tame stratifications.

\begin{thm}[Coarsest refining meshes, {\cite[Thm. 5.1.17]{fct}}] \label{thm:coar-ref-mesh} Every tame stratification $f$ has a unique `coarsest refining mesh' $M \mshar f$, satisfying that any other mesh refinement $M' \mshar f$ factors through $M \mshar f$ by a unique strict $n$-mesh coarsening $M' \to M$.
\end{thm}

\nid The notion of `coarsest refining meshes' has the following combinatorial counterpart. 

\begin{defn}[Normal forms] \label{defn:normal-forms} A stratified truss $(T,f)$ is said to be in \textbf{normal form} (or `normalized') if any truss coarsening $F : (T,f) \to (S,g)$ is an identity.
\end{defn}

\begin{obs}[Relation of normal forms to coarsest refining meshes] \label{obs:normal-form-vs-crs-ref-mesh} Consider a stratified $n$-mesh $(M,f)$ with stratified fundamental $n$-truss $(T,g) = \FTrs (M,f)$. Using \cref{obs:truss-and-mesh-coarsening}, note $(T,g)$ is in normal form iff $M \mshar f$ is the coarsest refining mesh.
\end{obs}

\begin{cor}[Normalization] \label{cor:normalization} Every stratified truss $(T,f)$ has a unique truss coarsening to a stratified truss in normal form, denoted $(T,f) \to \NF{T,f}$.
\end{cor}

\begin{proof} Pick any stratified classifying mesh $(M,g) \iso \CMsh (T,f)$ and use \cref{thm:coar-ref-mesh} for the tame stratification $g$.
\end{proof}

Let us briefly discuss a purely combinatorial approach to the existence of normal forms. First, note that the notion of truss coarsenings can be straight-forwardly generalized to deal with labelings in some category $\iC$.

\begin{term}[Truss coarsenings for labeled trusses] \label{term:truss-coarsenings-gen} Given $\iC$-labeled trusses $(T,f)$ and $(T',f')$, a `truss coarsening' $F : (T,f) \to (T',f')$ is a truss coarsening $F : T \to T'$ such that $f' \circ F_n = f$ commutes.
\end{term}

\nid \cref{cor:normalization} now generalizes as follows.

\begin{term}[Normalized labeled trusses] A $\iC$-labeled truss $(T,f)$ is said to be `normalized' (or in `normal form') if no non-identity coarsening applies to it.
\end{term}

\begin{obs}[Existence of normal forms for labeled trusses] \label{rmk:normal-forms-for-gen-labelled-trusses} Any $\iC$-labeled truss $(T,f)$ has a unique coarsening $(T,f) \to \NF{T,f}$ to a normalized labeled truss $\NF{T,f}$ to which no non-identity coarsening applies. For open trusses $T$, this was shown in \cite[Thm. 5.2.2.11]{thesis}, and the given proof applies to trusses as well. That proof is purely combinatorial and does not rely on any results about meshes.
\end{obs}

\nid Further generalizations of \cref{cor:normalization} can be obtained, for instance, by weakening the condition $f' \circ F_n = f$ in \cref{term:truss-coarsenings-gen} to hold only up to natural isomorphism; or yet more generally, by considering labelings in higher categories.

\begin{rmk}[Normal forms are computable] Observe that normal forms are algorithmically computable: indeed, we may simply search through all truss coarsenings of a labeled $n$-truss $(T,f)$ to find its normal form $\NF{T,f}$.\footnote{There are more efficient algorithms to compute normal forms $\NF{T,f}$: for instance, inductively in descending level $i$, one can compute the normalized labeled 1-truss bundles $\NF{(q_i,f_i)}$, where $p_i : T_i \to T_{i-1}$ is the $i$th 1-truss bundle in $T$, the labeling $f_i$ is classifying functor $\chi_{\NF{q_{i+1},f_{i+1}}}$, and $f_n = f$ is the labeling of $T$. This computes the functor $f_0$, which classifies $\NF{T,f}$. The fact that normal forms can be computed in this way follows from \cite[Lem.\ 5.2.2.8]{thesis}, but an explicit description of the algorithm was omitted in \emph{loc.cit.}. A related algorithm was recently described in detail in \cite{heidemann2022zigzag}.}
\end{rmk}

\subsection{Combinatorialization of tame stratifications} The central `combinatorialization' theorem for tame stratifications can now be stated as follows.

\begin{thm}[Combinatorialization of tame stratifications, {\cite[Thm.\ \S5.0.4]{fct}}] \label{thm:tame-strat-comb} Tame stratifications up to framed stratified homeomorphism are in 1-to-1 correspondence with normalized stratified trusses.
\end{thm}

\begin{proof}[Proof of \cref{thm:tame-strat-comb}] Given a tame stratification $f$, first construct its coarsest refining mesh $M \mshar f$. The normalized stratified truss corresponding to $f$ is then given by $\FTrs (M,f)$ (see \cref{obs:normal-form-vs-crs-ref-mesh}). Equivalently, this stratified truss is obtained as the normal form $\NF{T,g}$ of the fundamental truss $(T,g) = \FTrs(M,f)$ of \emph{any} mesh refinement $M \mshar f$.
\end{proof}

\begin{notn}[Combinatorialization of tame stratifications] \label{notn:tame-strat-comb} Given a tame stratification $(X,f)$, we denote its coarsest refining mesh by $\iM^f$. We abbreviate $\FTrs (\iM^f,f)$ simply by $\NFTrs f$ (or by $\NFTrs (X,f)$ if we want to make the underlying flat framed space $X$ explicit), and refer to it as the `normalized stratified fundamental truss' of $f$. Conversely, we also call $f$ a `classifying tame stratification' of $\NFTrs f$ (or any stratified truss isomorphic to it).
\end{notn}

\begin{eg}[Combinatorializing tame stratifications] In \cref{fig:combinatorializing-tame-stratifications} on the left we depict a tame stratification $f$ of the open 2-cube $\II^2 = (-1,1)^2$; to its right, we depict its coarsest refining mesh $\iM^f$. On the further right, we depict the normalized stratified truss $(T,g) = \FTrs(\iM^f,f)$ consisting of a 2-truss $T$ and the characteristic poset map $g = \Entr(\iM^f \to f)$ (we depict $g$ by coloring images $x$ and preimages $g\inv(x)$ in the same color).
\begin{figure}[ht]
    \centering
    \def\svgwidth{1\columnwidth}
    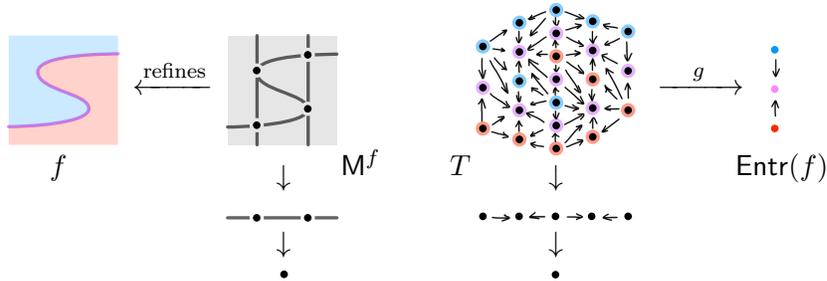

    \caption{Combinatorializing tame stratifications}
    \label{fig:combinatorializing-tame-stratifications}
\end{figure}
\end{eg}

Let us also briefly address how our discussion of tame stratification, coarsest refining meshes, and normalized stratified trusses generalizes to the case of bundles.

\begin{term}[Tame stratified bundles] \label{rmk:tame-bundles} Note that replacing meshes by mesh bundles over a stratification $B$, and flat framed spaces by flat framed bundles over $B$ (see \cref{term:flat-framed-bundles}), one readily defines `tame stratified bundles' $f$ over $B$ to be stratified bundles $(X,f) \to B$ with underlying bundle map $X \subset B \times \lR^n \to B$, such that $f$ can be strictly refined by some mesh bundle $p$ over $B$; as before, we write this refinement as $p \mshar f$.
\end{term}

\nid Note, a mesh refinement $p \mshar f$ defines and is defined by a stratified mesh bundle $(p,f)$.

\begin{term}[Normalized stratified truss bundles] On the combinatorial side, the notion of normal forms carries over verbatim: a stratified truss bundles $(q,f)$ is `normalized' if no non-trivial truss bundle coarsening applies to it.
\end{term}

\begin{thm}[Coarsest refining mesh bundles, {\cite[Thm.\ 5.2.24]{fct}}] Every tame stratified bundle $f$ has a `coarsest refining $n$-mesh bundle' $p \mshar f$, such that any other refining $n$-mesh bundle $p' \mshar f$ factors through $p \mshar f$ by a strict $n$-mesh bundle coarsening $F : p' \to p$.
\end{thm}

\begin{cor}[Normalization for bundles] Every stratified truss bundle $(q,f)$ has a unique truss coarsening $(q,f) \to \NF{q,f}$ to a normalized stratified truss bundle. \qed
\end{cor}

\begin{thm}[{\cite[Thm.\ 5.2.25]{fct}}] \label{thm:tame-bundle-class} Tame bundles over $B$ up to framed stratified homeomorphism are in 1-to-1 correspondence with normalized stratified trusses over $\Entr(B)$.
\end{thm}

\begin{proof} Given a tame stratified bundle $f$ over $B$, the corresponding normalized stratified truss is constructed as $\FTrs (p,f)$ for the coarsest refining mesh bundle $p \mshar f$. Equivalently, this can be obtained as the normal form $\NF{q,g}$ of the fundamental truss $(q,g) = \FTrs(p,f)$ of \emph{any} mesh bundle refinement $p \mshar f$.
\end{proof}

\begin{term}[Combinatorialization of tame stratified bundles] Given a tame stratified bundle $f$, we denoted its coarsest refining mesh bundle by $\ip^f$. As before, we abbreviate $\FTrs(\ip^f,f)$ by $\NFTrs f$, and speak of the `normalized fundamental stratified truss bundle' of $f$. Conversely, we also call $f$ a `classifying tame stratified bundle' of $\NFTrs f$ (and of any stratified truss bundle isomorphic to it).
\end{term}





\chapter{Manifold diagrams} \label{ch:mdiag}

We give two definitions of manifold diagrams; one in more familiar geometric terms, and one in purely combinatorial terms. Using our results about the combinatorializability of tame stratifications in the previous section, we will deduce that these definitions are equivalent. We use our construction of dualization functors in the previous section to relate manifold diagrams to so-called `cell diagrams', which can be thought of as pasting diagrams of higher morphisms in the familiar higher categorical sense.

\section{Topological definition}

\subsection{Definition on the open cube} Giving a definition of manifold diagrams in topological terms will require two ingredients: \emph{framed conicality} and \emph{tameness}. It will be convenient to fix the following `framed background' space for manifold diagrams.

\begin{notn}[Cubes] \label{notn:cubes} We denote by $\II^n$ the open $n$-cube $(-1,1)^n$ in $\lR^n$, and by $\bI^n$ the closed $n$-cube $[-1,1]^n$ in $\lR^n$ (we consider both as flat $n$-framed spaces leaving the inclusion $\gamma$ into $\lR^n$ implicit, see \cref{term:flat-framed-space}). We write $\partial \bI^n = \bI^n \setminus \II^n$ for the $n$-cube boundary.
\end{notn}

\begin{term}[Cubical links and cones] \label{notn:cube-links} A `cubical link' (or simply, a `link') is a stratification $(\partial\bI^n,l)$ of the $n$-cube boundary. We identify the open cone $\cone(\partial \bI^n) = \partial \bI^n \times [0,1) \slash \partial \bI^n \times \Set{0}$ with the open cube $\II^n$ by mapping $(x,\lambda) \in \partial \bI^n \times [0,1)$ to $\lambda x \in \II^n$. Similarly we identify the closed cone $\overline\cone(\partial \bI^n)$ with $\bI^n$. In particular, given a link $(\partial \bI^k, l)$ we write $(\II^k,\cone(g))$ for the open cone of $l$ and $(\bI^k,\overline\cone(l))$ for the closed cone (see \cref{rmk:cub-cone-strat}). We say $l$ is a `tame' link, if $(\bI^n,\overline\cone(l))$ is a tame stratification (which implies $(\II^n,\cone(l))$ too is tame).
\end{term}

\nid Recall `conical' stratification require tubular neighborhoods of the form $U \times (\cone(L),\cone(l))$ (see \cref{recoll:conical-strat}). Recall the notions of framed maps (see \cref{term:flat-framed-space}), and of tame framed stratified maps (see \cref{term:tame-maps}).

\begin{defn}[Framed conical stratifications] \label{defn:framed-conical} Given a stratification $(\II^n,f)$ and a point $x \in \II^n$ we say $f$ is \textbf{(tame) framed conical at $x$} if there is a (tame) link $(\partial \bI^{n-k}, l_x)$ and a (tame) framed stratified neighborhood $\phi : \II^{k} \times (\II^{n-k}, \cone(l_x)) \into (\II^n,f)$ such that $x \in \II^k \times \{0\}$, where $0$ is the cone point of $\cone(l_x)$. We say $(\II^n,f)$ is \textbf{(tame) framed conical} if it is (tame) framed conical at all $x \in \II^n$.
\end{defn}

\nid We usually refer to the neighborhood $\phi$ in the preceding definition as a `framed tubular neighborhood around $x$'.

\begin{eg}[Framed conicality condition] In the middle of \cref{fig:framed-conicality-condition-and-failure} we depict a tame stratification in $\II^2$ (we use \cref{notn:coord-axis-reverse} to indicate the framing of $\II^2$). Note that the `bifurcating' line is a single stratum. The stratification is framed conical at the blue point as shown, but it fails to be framed conical at the red `bifurcation' point. Another illustration of the framed conicality condition was given earlier in \cref{fig:framed-conicality-condition}.
\begin{figure}[ht]
    \centering
    \def\svgwidth{1\columnwidth}
    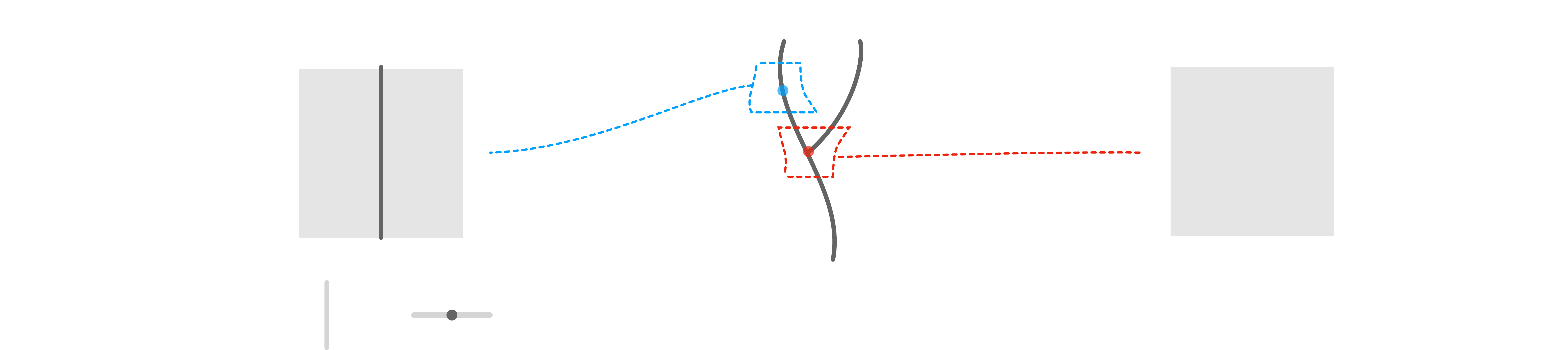

    \caption{Tame stratification with framed conical and non-conical points}
    \label{fig:framed-conicality-condition-and-failure}
\end{figure}
\end{eg}

We can now define manifold diagrams as follows.

\begin{defn}[Manifold $n$-diagrams] \label{defn:mdiag} A \textbf{manifold $n$-diagram} $(\II^n,f)$ is a tame stratification of the open cube that is tame framed conical.
\end{defn}

\begin{rmk}[Strata are manifolds endowed with local euclideanization] \label{rmk:strata-in-mdiag} The framed conicality condition guarantees the following. Strata in manifold $n$-diagrams are $k$-manifolds ($0 \leq k \leq n$). Each $k$-manifold stratum $M$ in a manifold $n$-diagram has a local homeomorphism to $\II^k$, obtained by restricting the projection $\II^n \to \II^k$ to $M$.
\end{rmk}

\begin{eg}[Manifold diagrams] A range of examples in dimensions $n = 1$, $2$, and $3$ were given in \cref{fig:manifold-diagrams-inductive-idea}. In \cref{fig:manifold-diagrams-in-dim-4} we give two examples in dimension 4, illustrated by sampling at slices $\{t_i\} \times \II^3$ of the 4-cube $\II^4$ for `times' $t_i \in \II$: small dots are restriction of 1-manifold strata to these slices, while big dots represent 0-manifold strata; the evolution of 1-manifold strata across time slices is indicated in color.
    \begin{figure}[ht]
    \centering
    \def\svgwidth{1\columnwidth}
    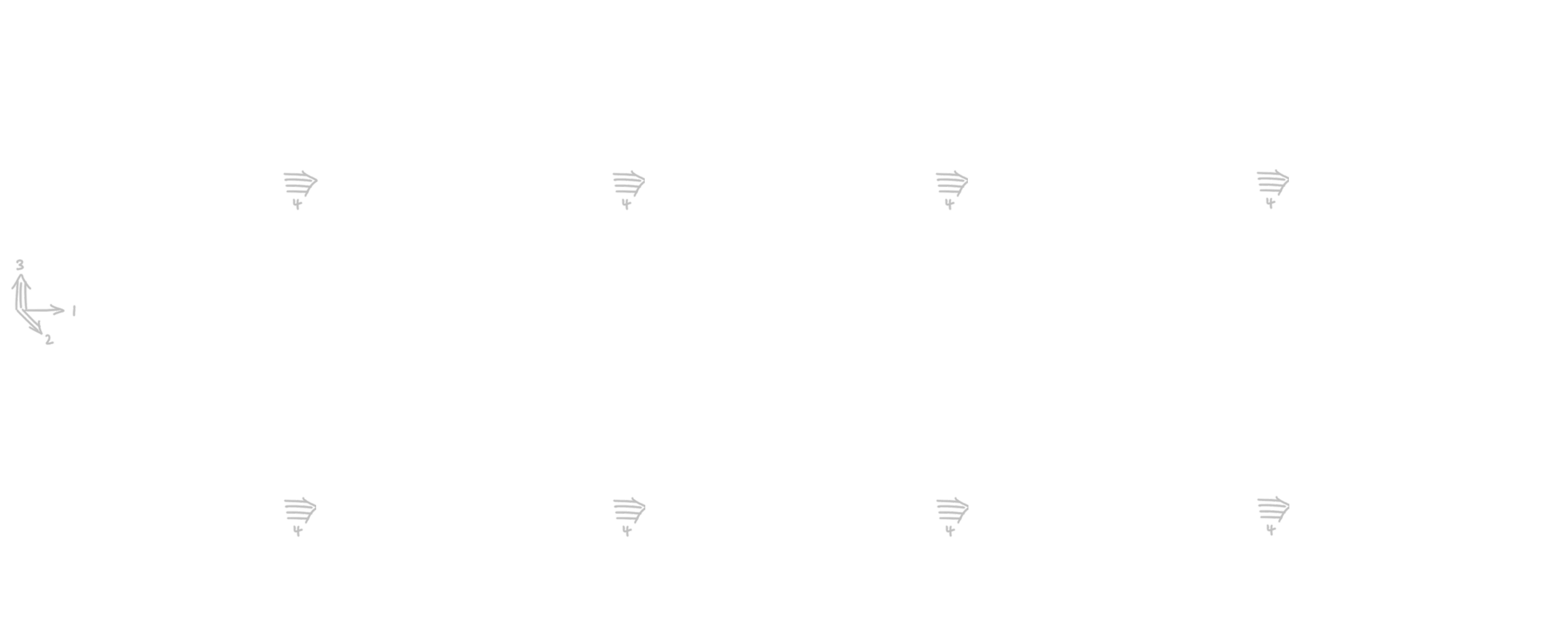

    \caption{Manifold 4-diagrams}
    \label{fig:manifold-diagrams-in-dim-4}
\end{figure}
\end{eg}

\begin{noneg}[Non-tame diagram] In \cref{fig:a-non-tame-diagram} we depict a stratification in dimension 3 that is not a manifold 3-diagram. Note that the given stratification $(\II^3,g)$ is framed conical, but is itself not tame---indeed, as shown its projection to $\II^2$ has infinitely many intersections which cannot admit a mesh refinement (since meshes are finite stratifications).
\begin{figure}[ht]
    \centering
    \def\svgwidth{1\columnwidth}
    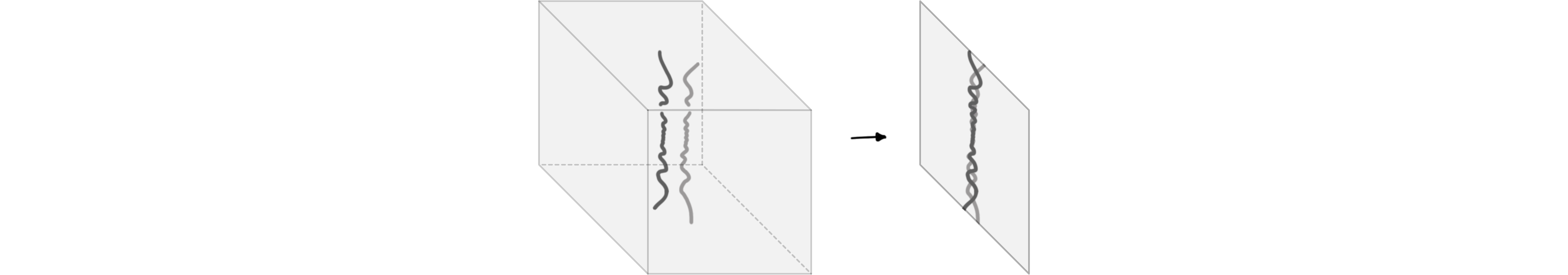

    \caption{A non-tame diagram}
    \label{fig:a-non-tame-diagram}
\end{figure}
\end{noneg}

\nid We remark that is possible to weaken `tame framed conicality' to `framed conicality' in the definition of manifold diagrams, without changing the class of tame stratifications that are manifold diagrams.

\begin{rmk}[Weakening the tame framed conicality condition] \label{obs:redundancy-tame-con} Given a tame stratification $(\II^n,f)$ which is framed conical, then it is also tame framed conical and thus a manifold $n$-diagram. In other words, weakening \cref{defn:mdiag} by replacing tame framed conicality by framed conicality yields and equivalent definition. \qedhere

\end{rmk}

\subsection{Definition on the closed cube} Instead of defining manifold diagrams on the open cube $\II^n$ we may also consider manifold diagrams on the closed cube $\bI^n$. As we will see, the two notions are (almost) in correspondence by a notion of compactification.

\begin{term}[Corner neighborhoods] \label{term:corner-neighborhoods} Let $\kP = \{\emptyset,-1,+1\}$, and recall that $\II = (-1,1)$ and $\bI = [-1,1]$. For $\sigma \in \kP$, denote by $\II^\sigma$ the (open or half-open) interval $\II \cup \sigma$ (which is a subinterval of $\bI$). Now let $\sigma = (\sigma_1,\sigma_2, ..., \sigma_k) \in \kP^k$ be a $\kP$-valued $k$-tuple. Denote by $\II^\sigma$ the `$\sigma$-corner' obtained as the $k$-fold product $\II^{\sigma_1} \times \II^{\sigma_2} \times ... \times \II^{\sigma_k}$.
\end{term}

\begin{defn}[Compact framed conical stratifications] A tame stratification $(\bI^n,f)$ is \textbf{(tame) compact framed conical at $x \in \bI^n$} if $x$ has a (tame) framed stratified neighborhood $\II^\sigma \times (\II^k,\cone(l_x)) \into (\bI^n,f)$ where $\sigma \in \kP^{n-k}$ and $x \in \II^{\sigma} \times \{0\}$. We say $f$ is \textbf{(tame) compact framed conical} if it is so at all points $x \in \bI^n$.
\end{defn}

\begin{defn}[Compact manifold diagrams] \label{obs:compact-diagrams} A \textbf{compact manifold $n$-diagram} $(\bI^n, f)$ is a tame stratification of the closed cube that is tame compact framed conical.
\end{defn}

\nid The relation of the compact and `open' definitions of manifold diagrams will be addressed in \cref{constr:compactifying-diagrams} and \cref{constr:compactifying-diagrams}.

It is reasonable to conjecture that compact manifold $n$-diagrams can approximate general (not necessarily tame) stratifications as long as these stratifications are compact framed conical.\footnote{Here, a stratification $(X,f)$ of a metric space $X$ is said to be an `$\eps$-close approximation' of another stratification $(X,g)$, if the two stratifications are stratified homeomorphic by a self-homeomorphism of $X$ that moves points by at most the distance $\eps$.}

\begin{conj}[Tame diagrams approximate all diagrams] \label{conj:mfld-diag-approx} Any compact framed conical stratification $(\bI^n,f)$ has an arbitrarily close approximation by a compact manifold $n$-diagram.
\end{conj}

\nid As an example, consider the stratification in \cref{fig:a-non-tame-diagram} (suitably compactified to a stratification of $\bI^3$): while it is compact framed conical, it is not tame itself; however, it may be approximated with arbitrary precision by a compact manifold diagram that resolves the infinite number of braid crossings by a finite number of such crossings. Note that it is natural to state this conjecture in the compact case: in the open case, non-tame stratifications may have `non-tame behavior at infinity' which cannot be tamely approximated, but this is automatically excluded in the compact case.

\subsection{Framed piecewise-linear structure} \label{ssec:mdiag-PL} We end this section with observations contrasting the framed topological setting with the classical topological one. Our first observation concerns the choices of framed link stratifications in the definition of manifold diagrams.

\begin{rmk}[Canonical choices of links] \label{rmk:links-well-def-in-mdiag} In the case of (classical) conical stratifications, one generally cannot speak of `the' link of a stratum; that is, strata may have several non-homeomorphic stratified links (see e.g.\ \cite[Eg.\ 2.3.6]{friedman2020singular}). In contrast, framed links around points in a manifold diagram can be canonically chosen: this will follow from the combinatorialization of manifold diagrams and is further discussed in \cref{rmk:links-are-well-defined} and \cref{rmk:links-are-strat-homeo}.
\end{rmk}

\nid We next compare framed TOP and framed PL structures on manifold diagrams.

\begin{term}[PL stratifications in $\lR^n$] A `PL stratification' $(U,f)$ is a stratification $f$ of a bounded subspace $U$ of $\lR^n$ that has a `simplicial subrefinement': this means there exists a span $f \leftepi g \into K$ where $K$ is a simplicial complex, $g \into K$ is a constructible substratification, and $g \epi f$ is a refinement of $f$ by $g$ (which is linear on each simplex w.r.t.\ the standard linear structure of $\lR^n$).
\end{term}

\nid Note that working with subrefinements $f \ot g \into K$ in place of refinements $f \ot K$ allows our notion of PL stratifications to accommodate non-compact stratifications.

\begin{term}[Framed PL structures] \label{term:framed-pl-struct} Given a stratification $(U,f)$ of a subspace $U \subset \lR^n$, a `framed triangulation' of $(U,f)$ is a framed stratified homeomorphism $\alpha : (U,f) \iso (V,g)$ to a PL stratification $(V,g)$. Two framed triangulations $\alpha : (U,f) \iso (V,g)$, $\beta : (U,f) \iso (W,h)$ are `framed equivalent' if there is a framed stratified PL homeomorphism $\rho : (V,g) \iso (W,h)$.\footnote{One may further require this to satisfy $\rho \circ \alpha \eqv \beta$ up to a homotopy through framed stratified maps (or even framed stratified homeomorphisms)---but the resulting notions of framed PL structures are in fact all equivalent.} A `framed PL structure' on $(U,f)$ is a framed equivalence class of framed triangulations.
\end{term}

\nid Note that any \emph{tame} stratification $f$ has a canonical framed PL structure since $f \iso \CMsh (\NFTrs f)$ and, analogous to classifying stratifications of posets having canonical PL structures, the classifying stratified mesh $\CMsh T$ of any truss $T$ has canonical framed PL structure (see \cref{rmk:cmsh-via-compactifiation} for further explanation, and see \cite[\S4.2.4]{fct} for full details of the construction). In fact, this is the unique framed PL structure of $f$ as the next theorem records.

\begin{thm}[Uniqueness of framed PL structures, {\cite[Cor.\ 5.7]{fct}}] \label{thm:framed-PL-struct-unique} All tame stratifications $f$ have unique framed PL structures. This structure is represented by $\CMsh \NFTrs f$.
\end{thm}

\nid In contrast, classical topological stratifications and PL stratifications aren't compatible in this way, that is, a given topological stratification may have several inequivalent PL structures or none at all.

In the case of manifold diagrams, the canonical PL structure of \cref{thm:framed-PL-struct-unique} endows strata with PL manifold structures as the next observation shows. The following terminological distinction may be useful to highlight.

\begin{rmk}[PL manifold structures] \label{term:PL-struct-vs-PL-manifold} A `PL manifold structure' on a space $X$ is a PL structure on $X$ (i.e.\ a PL equivalence class of triangulations of $X$) such that, for any triangulation in the structure, links of that triangulations are  PL homeomorphic to the PL standard sphere. A space $X$ equipped with a PL manifold structure is also called a PL manifold (in this case $X$ is in particular a topological manifold.) Equivalently, a PL manifold structure can be given by an atlas with PL transition maps \cite[\S3]{hudson1969piecewise}.
\end{rmk}

\begin{obs}[Manifold diagrams have unique PL structures] \label{obs:PL-struct-of-mdiag} By \cref{thm:framed-PL-struct-unique}, any manifold diagram $(\II^n,f)$ has a  unique framed PL structure $(\II^n,f) \iso (\II^n,f_{\text{PL}})$. This restricts on each stratum in $f_{\text{PL}}$ to the PL structure of a (non-compact) PL manifold: indeed, a PL atlas can be constructed using \cref{rmk:strata-in-mdiag}.
\end{obs}

\nid Importantly, PLness in fact implies tameness.

\begin{thm}[PL implies tame] \label{thm:plness-implies-tameness} Any PL stratification of $\II^n$ (or $\bI^n$) is tame.
\end{thm}

\begin{proof} This is a slight variation of \cite[Prop.\ 5.8]{fct}.
\end{proof}

\begin{obs}[Framed conical PL stratifications are manifold diagrams] \label{obs:PL-diagrams-are-mdiag} Any tame framed conical PL stratification of $\II^n$ is a manifold diagram (and similar, tame compact framed conical PL stratifications of $\bI^n$ are compact manifold diagrams): this follows since PL stratifications are tame by \cref{thm:plness-implies-tameness}. Taking \cref{obs:redundancy-tame-con} into account (which also applies in the compact case), one may further generalize this statement by weakening tame framed conicality to framed conicality.
\end{obs}

\begin{rmk}[PL implies tame for maps] \label{rmk:pl-map-implies-tame-map} Analogous to the proof of \cref{thm:plness-implies-tameness} one shows that any PL map of PL stratifications $(\bI^n,f) \to (\bI^n,g)$ is tame (this also holds in the open case for stratifications on $\II^n$ as long as maps can be made simplicial by finite triangulations).
\end{rmk}

\nid Last but not least, we briefly sketch how strata in manifold diagrams can be endowed with canonical smooth structures as well.

\begin{rmk}[Smooth structures on manifold diagram strata] Using \cref{rmk:strata-in-mdiag} and \cref{obs:PL-struct-of-mdiag}, each $k$-manifold stratum $M$ in a manifold $n$-diagram has a local PL homeomorphism to $\II^k$. The standard framing of $\II^k$ can now be pulled back to a framing of the PL manifold $M$. Using classical smoothing theory \cite{hirsch1974smoothings} (the folklore result stating that `framed PL equals framed DIFF'), one can show that strata are endowed with canonical smooth structures.
\end{rmk}

\nid The more interesting, and harder, question concerns how smooth structures of different strata interact (the task of defining `smooth conical stratified space' with good properties is difficult in general, see \cite{ayala2017local} for discussion); a version of this question crops up naturally in the context of tangles and singularities, which we will return to in \cref{ch:tangles-and-sings}.

\section{Combinatorial definition}

We give a purely combinatorial definition of manifold diagrams. As in the topological case, the notion has both an `open' and a `compact' variation. We will relate the two by a combinatorial process of compactification.

\subsection{Neighborhoods, cones, and products for stratified trusses}

A first key observation is that topological operations (taking neighborhoods, cones, products, etc.) have natural combinatorial counterparts. Recall the notion of subtrusses from \cref{term:n-truss-bun-maps}.

\begin{defn}[Truss neighborhoods] Let $T$ be an $n$-truss and $x \in T_n$ an object of its top poset. There is a unique subtruss $T^{\leq x} \into T$, called the \textbf{truss neighborhood} of $x$, whose top component $T^{\leq x}_n \into T_n$ is the downward closure of $x$ in $T_n$.
\end{defn}

\nid Recall the notion of stratified subtrusses from \cref{term:stratified-subtruss}.

\begin{defn}[Stratified truss neighborhoods] \label{defn:strat-truss-nbhd} Let $(T,f)$ be a stratified $n$-truss. For any $x \in T_n$, there is a unique stratified subtruss $(T^{\leq x},f^{\leq x}) \into (T,f)$ called the \textbf{stratified truss neighborhood} of $x$.
\end{defn}

\begin{eg}[Stratified truss neighborhoods] In \cref{fig:stratified-truss-neighborhoods} we depict two 2-trusses $T$ and $S$ (note, we omit the bundles $T_2 \to T_1 \to T_0$ resp.\ $S_2 \to S_1 \to S_0$---in each case these are determined by first projecting vertically). We select objects $x \in T_2$ and $y \in S_2$ as shown, and depict the resulting stratified subtrusses $T^{\leq x} \into T$ resp.\ $S^{\leq y} \into S$. Observe that the map $\Entr(g^{\leq y}) \to \Entr(g)$ is not a subposet inclusion.
\begin{figure}[ht]
    \centering
    \def\svgwidth{1\columnwidth}
    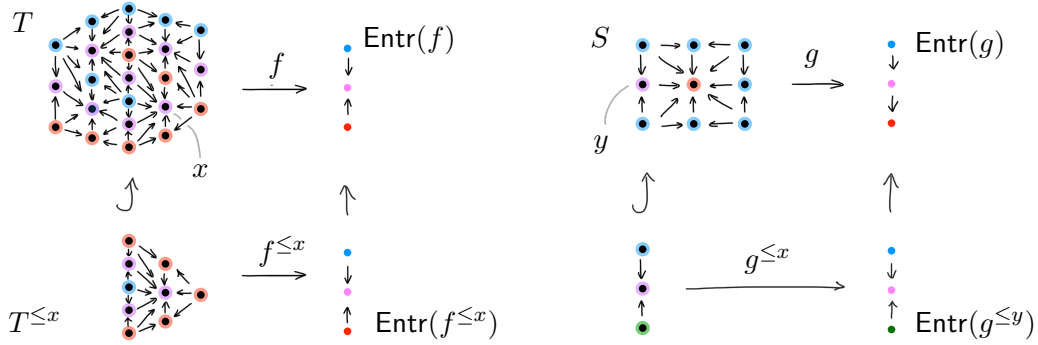

    \caption{Stratified truss neighborhoods in 2-trusses}
    \label{fig:stratified-truss-neighborhoods}
\end{figure}
\end{eg}

\nid Recall from \cref{rmk:n-truss-dim} that $n$-trusses $T$ come with a map $\dim : T_n \to [n]\op$.

\begin{defn}[Cone trusses] \label{defn:truss-cones} A \textbf{open cone $n$-truss} $T$ is an open $n$-truss whose top poset $T_n$ has a maximal element $\top \in T_n$ with $\dim(\top) = 0$ (which we call the `cone point' of $T$).
\end{defn}

\begin{defn}[Stratified cone trusses] \label{defn:strat-truss-cone} A \textbf{stratified open cone $n$-truss} $(T,f)$ is a stratified $n$-truss whose underlying truss is an cone $n$-truss, and such that the cone point is its own stratum (i.e.\ $f\inv \circ f (\top) = \{\top \}$).
\end{defn}

\begin{eg}[Cone trusses] In \cref{fig:truss-cones} we depict two stratified 2-trusses: both have a maximal element $\top$ with $\dim(\top) = 0$, however, only the second example stratifies $\{\top\}$ as its own stratum. Consequently, only the second stratified truss is a stratified cone truss.
\begin{figure}[ht]
    \centering
    \def\svgwidth{1\columnwidth}
    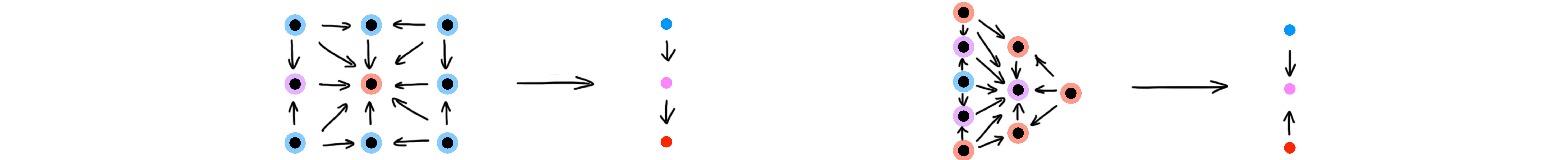

    \caption{An example and a non-example of a cone truss}
    \label{fig:truss-cones}
\end{figure}
\end{eg}

\nid We state the next definition in the generality of labeled trusses (it certainly can be specialized to the stratified case).

\begin{defn}[Truss products] Let $P$ be a poset and $(T,l)$ a $\iC$-labeled $k$-truss. Define the `poset-truss product' $P \times (T,l)$ (equivalently written $(P \times T,P \times l)$) to be the trivial $\iC$-labeled $k$-truss bundle over $P$ whose fiber is $(T,l)$. Now let $S$ be an $m$-truss bundle. The \textbf{truss product} $S \times (T,l)$ (also written $(S \times T, S \times l)$) is the $\iC$-labeled $(m+k)$-truss bundle over $P$ obtained by concatenating $S_m \times T$ and $S$.
\end{defn}

\nid Note that for $m = 0$ truss products and poset-truss products coincide (recall from \cref{ssec:truss-mesh-bundles} that $0$-truss bundles are simply posets).


\subsection{Defining open and compact combinatorial diagrams} We now give definitions of combinatorial manifold diagrams. Recall the construction of normal forms $\NF{T,f}$ of stratified trusses $(T,f)$ from \cref{defn:normal-forms} and \cref{cor:normalization} (it is also worth recalling that normal forms can be constructed in purely combinatorial terms, see \cref{ssec:comb-tame-strat}).

\begin{term}[The open cube truss] \label{term:open-cube-truss} Define the `open $1$-cube truss' $\OTT^1$ to be the open 1-truss with a single object. The `open $k$-cube truss' $\OTT^k$ is the $k$-truss obtained as the $k$-fold product truss $\OTT^1 \times \OTT^1 \times ... \times \OTT^1$.
\end{term}

\begin{defn}[Combinatorially conical stratified trusses] A stratified $n$-truss $(T,f)$ is \textbf{combinatorially conical at $x \in T_n$} if the stratified truss neighborhood around $x$ normalizes to the product of an open truss cube and a stratified open cone truss; that is, there is $k \leq n$ and a stratified open cone $(n-k)$-truss $(C_x,c_x)$ (called the `cone at $x$') such that
\[
	\NF{T^{\leq x},f^{\leq x}} = \OTT^k \times (C_x,c_x).
\]
If $(T,f)$ is combinatorially conical at all $x \in T_n$ then it is said to be \textbf{combinatorially conical}.
\end{defn}

\begin{defn}[Combinatorial manifold diagrams] \label{defn:comb-diag} A \textbf{combinatorial manifold $n$-diagram} (also called a `manifold diagram $n$-truss') is a normalized open stratified $n$-truss $(T,f)$ that is combinatorially conical.
\end{defn}

\begin{eg}[Combinatorial manifold 2-diagram] In \cref{fig:combinatorial-manifold-2-diagram} we depict a stratified truss $(T,f)$ that is a combinatorial manifold 2-diagram. To illustrate that it satisfies the combinatorial conicality condition we picked an element $x \in T_2$. To the right of $(T,f)$ we illustrate the neighborhood $(T^{\leq x}, f^{\leq x})$. On the right of that neighborhood we illustrate the normalization $\NF{T^{\leq x}, f^{\leq x}}$. This is of the required form $\OTT^k \times (C_x,c_x)$ for $k = 1$.
\begin{figure}[ht]
    \centering
    \def\svgwidth{1\columnwidth}
    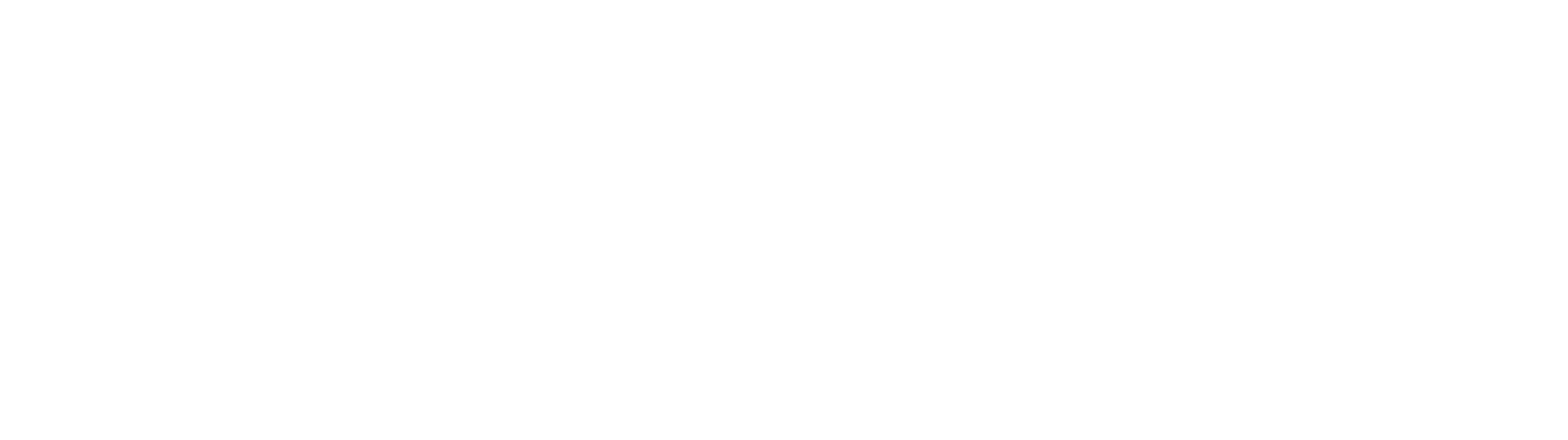

    \caption{A combinatorial manifold 2-diagram, a neighborhood in it, and the neighborhood's normal form}
    \label{fig:combinatorial-manifold-2-diagram}
\end{figure}
\end{eg}

\begin{eg}[Combinatorial manifold 3-diagram] In \cref{fig:combinatorial-manifold-3-diagram} we depict a stratified truss $(T,f)$ that is a combinatorial manifold 3-diagram. Note that for simplicity we only depict `essential' poset arrows: all other arrows can be generated from essential arrows by composition.\footnote{One can give a precise description of essential arrows in trusses, see \cite[Constr.\ 2.3.69]{fct}.} Let us verify the combinatorial conicality condition at the shown element $x \in T_2$. In \cref{fig:normalization-of-a-neighborhood} we illustrate the neighborhood $(T^{\leq x}, f^{\leq x})$ around $x$. Underneath, we illustrate the normalization $\NF{T^{\leq x}, f^{\leq x}}$. This normalization is of the required form $\OTT^k \times (C_x,c_x)$.
\begin{figure}[ht]
    \centering
    \def\svgwidth{1\columnwidth}
    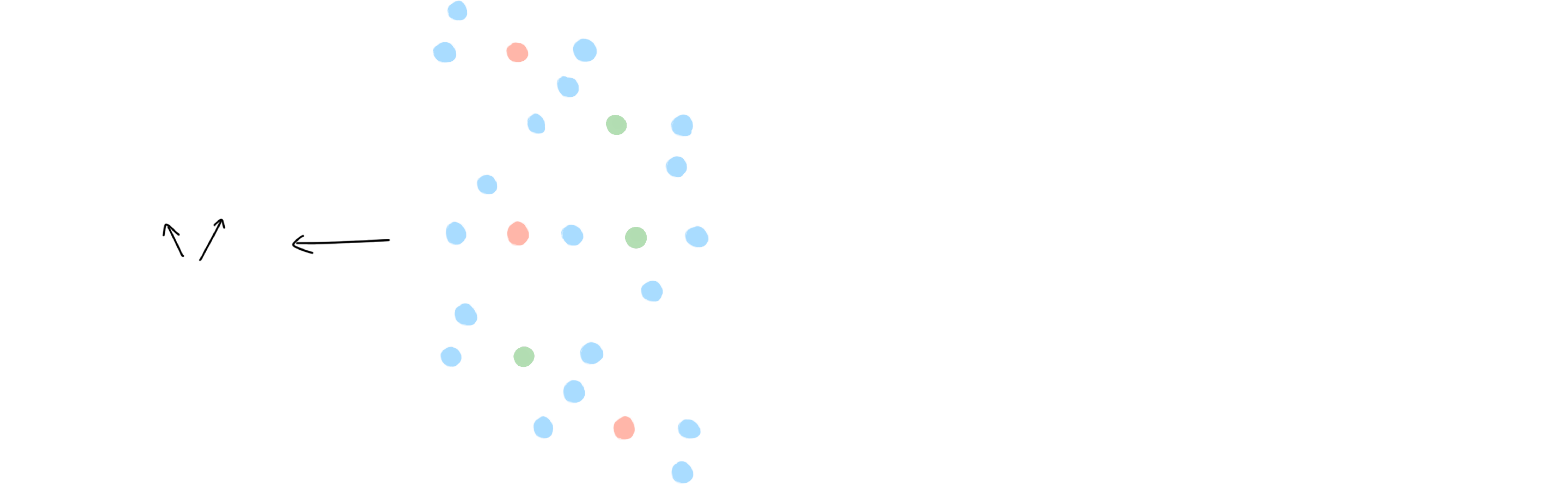

    \caption{A combinatorial manifold 3-diagram}
    \label{fig:combinatorial-manifold-3-diagram}
\end{figure}
\begin{figure}[ht]
    \centering
    \def\svgwidth{1\columnwidth}
    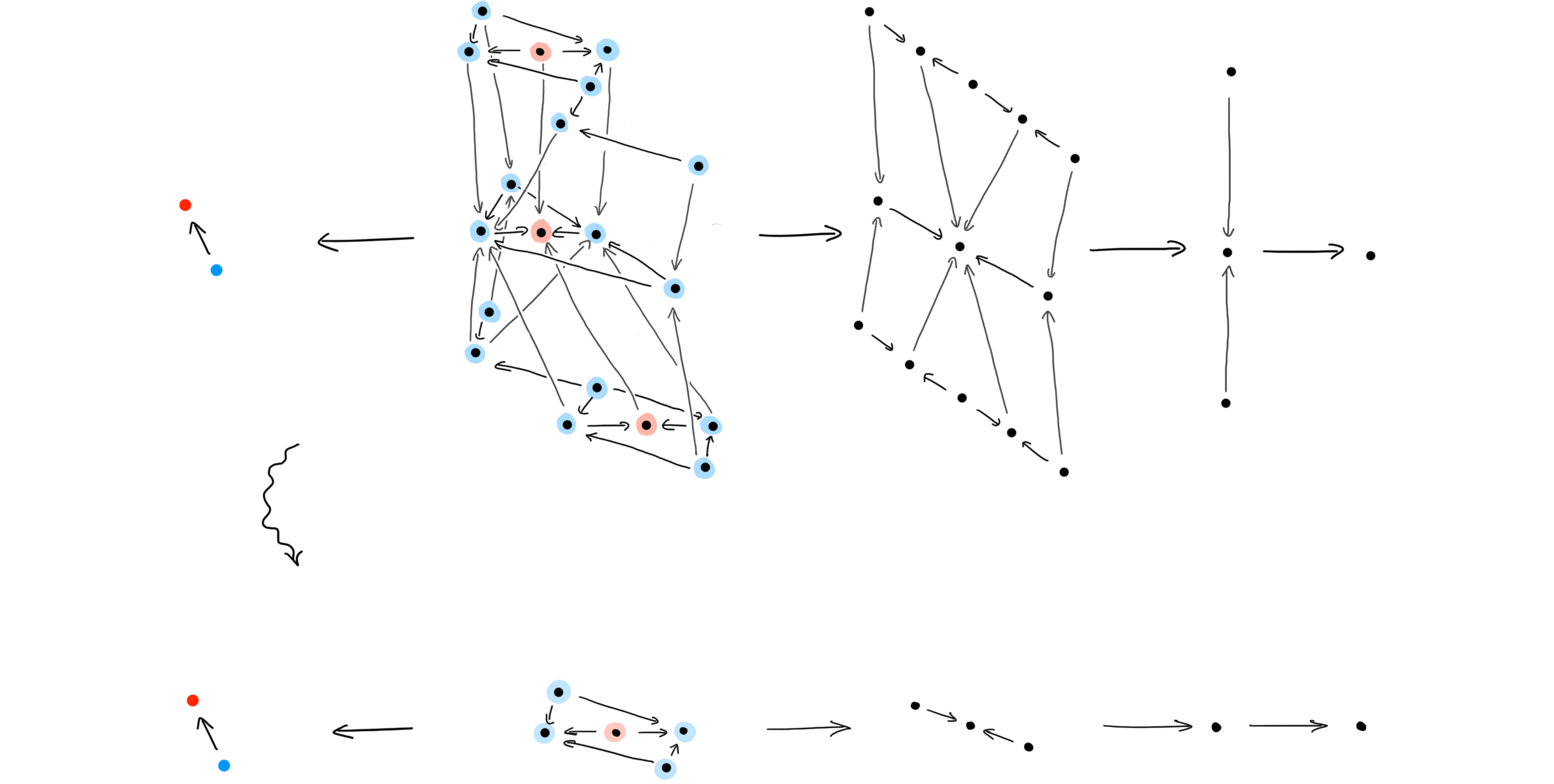

    \caption[Normalizing truss neighborhoods]{Normalization of a neighborhood in a combinatorial manifold 3-diagram}
    \label{fig:normalization-of-a-neighborhood}
\end{figure}
\end{eg}

\begin{noneg}[Failure of combinatorial conicality condition] The stratified 2-truss $(T,f)$ shown on the left in the earlier \cref{fig:stratified-truss-neighborhoods} is not a combinatorial manifold diagram: it fails the conicality condition at the indicated point $x$.
\end{noneg}

Combinatorial manifold diagrams stratify open trusses, and we consequently sometimes refer them as `open' combinatorial manifold diagrams. We now define a notion of `compact' combinatorial manifold diagram, which are combinatorial analogues of our earlier topological definition of compact manifold diagrams.

\begin{term}[Corner trusses] \label{term:corner-trusses} Define $\CTT_1 = \{- \ot 0 \to +\}$ to be the unique closed truss with $3$ elements. Let $\kP = \{\emptyset,-,+\}$. For $\sigma \in \kP$, denote by $\TT^\sigma$ the unique subtruss of $\CTT_1$ containing both $0$ and $\sigma$. Now let $\sigma = (\sigma_1,\sigma_2, ..., \sigma_k) \in \kP^k$ be a $\kP$-valued $k$-vector. Denote by $\TT^\sigma$ the `$\sigma$-corner truss' obtained as the $k$-fold product $\TT^{\sigma_1} \times \TT^{\sigma_2} \times ... \times \TT^{\sigma_k}$. Note that for $\sigma = (\emptyset,\emptyset,...,\emptyset)$ we have $\TT^\sigma = \OTT^k$.
\end{term}

\begin{defn}[Compact combinatorially conical stratified trusses] A stratified $n$-truss $(T,f)$ is \textbf{compact combinatorially conical at $x \in T_n$} if the stratified truss neighborhood around $x$ normalizes to the product of a corner truss and a stratified open cone truss; that is, there is $k \leq n$, $\sigma \in \kP^{k}$, and a stratified open cone $(n-k)$-truss $(C_x,c_x)$ such that
\[
	\NF{T^{\leq x},f^{\leq x}} = \TT^{\sigma} \times (C_x,c_x)
\]
If $(T,f)$ is combinatorially conical at all $x \in T_n$ then it is said to be \textbf{compact combinatorially conical}.
\end{defn}

\begin{defn}[Compact combinatorial manifold diagrams] \label{defn:compact-comb-diag} A \textbf{compact combinatorial manifold $n$-diagram} $(T,f)$ is a normalized closed stratified $n$-truss that is compact combinatorially conical.
\end{defn}

\begin{eg}[Compact combinatorial manifold diagrams] In \cref{fig:a-compact-combinatorial-manifold-diagrams} we depict a compact combinatorial manifold 2-diagram $(T,f)$ and, for chosen $x,y \in T_n$, illustrate how stratified neighborhoods normalize to products of corner trusses and stratified open cone trusses (note, for our choice of $x$, the neighborhood is already normalized).
\begin{figure}[ht]
    \centering
    \def\svgwidth{1\columnwidth}
    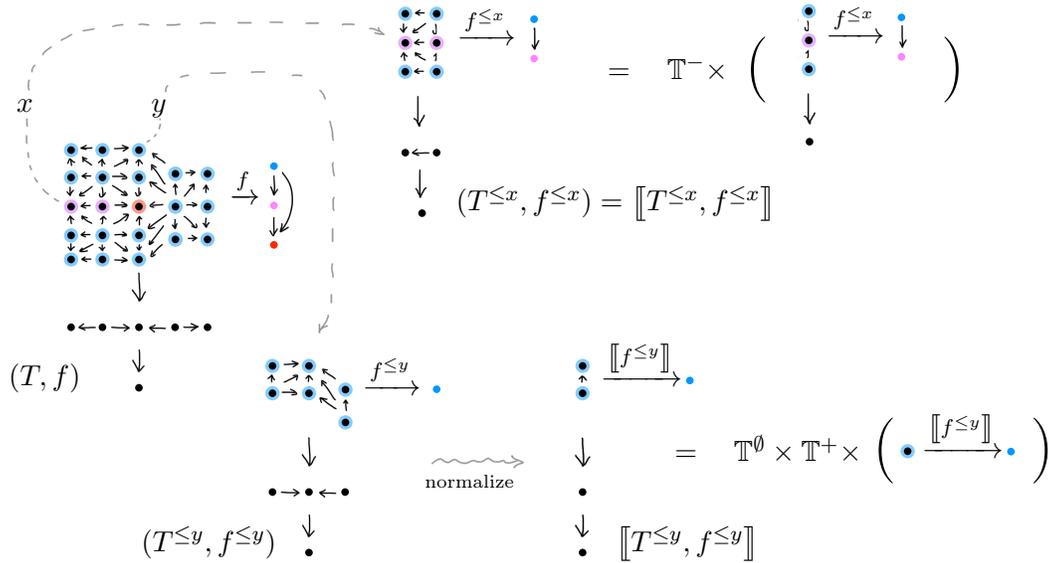

    \caption{A compact combinatorial manifold 2-diagram, and two neighborhoods in it, together with their normal forms}
    \label{fig:a-compact-combinatorial-manifold-diagrams}
\end{figure}
\end{eg}

\subsection{Compactification} \label{ssec:compactification} Let us now describe the formal relation of open and compact combinatorial manifold diagrams. This is given by a process of cubical compactification. The construction is based on the following topological intuition. Given a stratification $(\II^n,f)$ of the open $n$-cube we may define a compactification simply as a extension $(\II^n,f) \into (\bI^n,\tilde f)$ of that stratification to the closed $n$-cube. However, such an extension $\tilde f$ may arbitrarily add new strata in addition to those in $f$; to further control $\tilde f$, we usually require $\tilde f$ to `retract' back into $f$: a `retractable' compactification $i : f \into \tilde f$ is a compactification such that there exists a framed substratification $r : \tilde f \into f$ and both $\Entr(r \circ i)$ and $\Entr(i \circ r)$ are identities. Our notion of compactification describes the universal such retractable compactification, which we will refer to as `cubical compactification'.

\begin{term}[Retractable compactifications] Given a stratified $n$-truss $(T,f)$, a `retractable compactification' is a pair of cocellular stratified $n$-truss maps $i : (T,f) \toot (\tilde T,\tilde f) : r$ where $\tilde T$ is a closed stratified truss, $i$ is a dense stratified subtruss (`dense' meaning that the closure of $i_n(T_n)$ is $\tilde T_n$), and $r \circ i = \id$.
\end{term}

\begin{defn}[Cubical compactifications of trusses] \label{constr:comp-comb-diagrams}
    The \textbf{cubical compactification $(\overline T,\overline f)$} of an stratified $n$-truss $(T,f)$ is the retractable compactification $\cint : (T,f) \toot (\overline T,\overline f) : \cret$ with the following universal property: for any other retractable compactification $i : (T,f) \toot (\tilde T,\tilde f) : r$ there exists a unique stratified truss bordism $R : (\overline T,\overline f) \proto{} (\tilde T,\tilde f)$ (i.e.\ a morphism in $\kT^n(\Entr(f))$, see \cref{obs:labeled-n-truss-bord}, or equivalently an $\Entr(f)$-labeled bundle over $[1]$ such that $\rest R 0 = (\overline T,\overline f)$ and $\rest R 1 = (\tilde T,\tilde f)$) that on the images of $\cint$ resp.\ $i$ restricts to an identity bordism $\id : (T,f) \proto{} (T,f)$.
\end{defn}

\nid The construction of cubical compactifications $\overline T$ on underlying trusses $T$ is given in \cite[Constr.\ 4.2.57]{fct}. The stratified case described here adds to this the labeling $\overline f$ determined as the composite $f \circ \cret$. Often we simply speak of `compactification' instead of `cubical compactification'.

\begin{eg}[Retractable and cubical compactifications] Consider the stratified truss $(T,f)$ shown in the middle of \cref{fig:compactification-of-a-stratified-truss}: to its left, we depict its cubical compactification $(\overline T, \overline f)$, and to its right, we depict a retractable but non-cubical compactification. One verifies that there is a unique stratified bordism $\overline T \proto{} \tilde T$ that restricts to an identity bordism on $T$.
\begin{figure}[ht]
    \centering
    \def\svgwidth{1\columnwidth}
    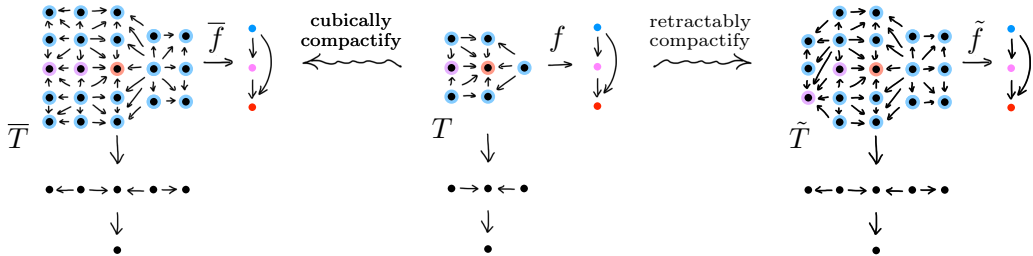

    \caption{Cubical and non-cubical compactification of a stratified truss}
    \label{fig:compactification-of-a-stratified-truss}
\end{figure}
\end{eg}

\nid Note that in the preceding example $(T,f)$ is a combinatorial manifold diagram, and $(\overline T, \overline f)$ is a compact combinatorial manifold diagram. This observation holds in general, as we shortly explain. We first remark that compactification has a `left inverse', and that it can be used in the construction of classifying meshes.

\begin{rmk}[Interiors of closed stratified trusses] \label{term:interior-compactifying} Given a closed stratified truss $(T,f)$ its `interior' $(T\intrr,f\intrr)$ is the maximal stratified open subtruss of $(T,f)$. Note that, given a stratified open truss $(S,g)$, then $(\overline S\intrr, \overline g\intrr) = (S,g)$. The converse need not hold. If for a stratified closed truss $(T,g)$ we have $(\overline {T\intrr}, \overline {f\intrr}) = (T,f)$ then we call $(T,f)$ `interior-compactifying'.
\end{rmk}

\begin{rmk}[Compactification in the construction of classifying meshes] \label{rmk:cmsh-via-compactifiation} Given a stratified $n$-truss $(T,f)$, a classifying mesh $(M,g) = \CMsh (T,f)$ can be constructed by first constructing the classifying stratification $\CStr {\overline T}_i$ (see \cref{recoll:classifying-strat}), and then defining $M_i$ to be the constructible substratification of $\CStr {\overline T}_i$ determined on entrance path posets by the inclusion $T_i \into {\overline T}_i$ (see \cref{recoll:strat-maps}); the mesh bundles $p_i : M_i \to M_{i-1}$ are obtained by restricting the maps $\CStr {(\overline T_i \to \overline T_{i-1})}$, and framings of fibers are induced canonically by the frame orders of fibers in $T$ (the full construction of $M = \CMsh T$ is spelled out in \cite[\S4.2.4]{fct}). Note that $M$ inherits framed PL structure from $\CStr {\overline T}$. The coarsening $M_n \to g$ is determined on entrance paths by $f : T_n = \Entr M_n \to \Entr(f)$.
\end{rmk}

\begin{rmk}[Compactification for truss bundles] \label{rmk:compact-truss-bun} Note that our discussion immediately carries over to the case of stratified truss bundles over a base poset $P$ (for unstratified truss bundles, see \cite[Defn.\ 4.2.53ff]{fct}). This yields notions of `cubical compactifications' $(\overline q, \overline f)$ of stratified open truss bundles $(q,f)$, and of `interiors' $(q\intrr, f\intrr)$ of stratified closed truss bundles $(q,f)$. (We omit spelling out these definitions, since they almost verbatim repeat our earlier discussion.)
\end{rmk}

\nid We now specialize our discussion to the case of manifold diagrams.

\begin{obs}[Relating of open and compact combinatorial diagrams via compactification] \label{obs:open-compact-comb-corr} We state the following observations (without proof).
\begin{enumerate}
\item Given a compact combinatorial manifold diagram $(T,f)$, its interior $(T\intrr,f\intrr)$ is an open combinatorial manifold diagram.
\item Conversely, given an open combinatorial manifold diagram $(T,f)$, then its cubical compactification $(\overline T, \overline f)$ is a compact combinatorial manifold diagram.
\end{enumerate}
Together, these observations establish a 1-to-1 correspondence between open combinatorial manifold diagrams and interior-compactifying compact combinatorial manifold diagrams.
\end{obs}

\nid Examples of interior-compactifying and non-interior-compactifying compact manifold diagrams are illustrated later in terms of (topological) manifold diagrams in \cref{fig:interior-compactifying-diagram}. While not all compact combinatorial diagrams are interior-compactifying, it is possible to show that they are so `up to a (canonical) perturbation'---in this sense, notions of open and compact combinatorial manifold diagrams are closely related. Our focus will mostly lie on the open case, but compactification proves to be a useful tool in many instances.



\section{Equivalence of the definitions}

We now compare our topological and combinatorial definitions of manifold diagrams. The central result used in the comparison will be the correspondence of tame stratifications up to framed stratified homeomorphism with normalized stratified trusses (see \cref{sec:meshtrusseqv}).

\subsection{Translating products and cones}

Before stating and proving the comparison, we will show that both products and cones can be faithfully translated from their topological definitions into combinatorial ones.

\begin{lem}[Combinatorialization preserves products and cones] \label{lem:NFTrs-preserves-prod-and-cone} The construction of $\NFTrs$ preserves products and cones as follows.
    \begin{enumerate}
        \item Given a tame stratification $(\II^n,f)$, then $\NFTrs (\II^k \times f) = \OTT^k \times \NFTrs f$.
        \item Given a tame link $(\partial\bI^n,l)$ then $\NFTrs \cone(l)$ is an stratified open cone truss.
    \end{enumerate}
\end{lem}

\begin{proof} The first statement follows since each mesh refining $\II^k \times f$ restricts on the preimage $x \in \II^k$ under $\II^k \times \II^n \to \II^k$ to a mesh refining $f$. Thus $\II^k \times \iM^f$ (where $\iM^f$ is the coarsest refining mesh of $f$) must be the coarsest refining mesh of $\II^k \times f$.

    For the second statement, first observe that `scaling' $\lambda : \cone(l) \to \cone(l)$, which on underlying spaces maps $x \mapsto \lambda x$ with $0 < \lambda \leq 1$, is a framed stratified map. Let $(T,g) = \NFTrs \cone(l)$ and let $x \in T_n$ correspond to the cone point $0$ of $\cone(l)$. Then $(T^{\leq x},g^{\leq x}) \into (T,g)$ has a classifying mesh map $\phi : (N,e) \into (\iM^{\cone(l)},\cone(l))$, whose image is an open neighborhood around $0$. For small $\lambda$, we get a tame substratifications $\lambda : \cone(l) \into e$ (to see tameness, one can pass to appropriate PL replacements and use that $(\bI^n, \overline\cone(l))$ is assumed to be tame, cf.\ \cref{rmk:pl-map-implies-tame-map}).
    Since $\lambda : \cone(l) \into e$ is tame, pick mesh refinements $L \mshar \cone(l)$ and $E \mshar e$ such that $\lambda : (L,\cone(l)) \into (E,e)$ is a stratified mesh map. Pass to trusses to obtain a stratified truss map $i = \FTrs \lambda : \FTrs (L,\cone(l)) \into \FTrs (E,e)$. Note, there is a truss coarsening $c : \FTrs (E,e) \to \NF{T^{\leq x},g^{\leq x}}$. The composite $c \circ i$ uniquely factors as
\[\begin{tikzcd}
	{\FTrs(E,e)} & {\NF{T^{\leq x},g^{\leq x}}} \\
	{\FTrs (L,\cone(l))} & {(S,h)}
	\arrow["i",from=2-1, to=1-1]
	\arrow[hook, from=2-2, to=1-2]
	\arrow["c",two heads, from=1-1, to=1-2]
	\arrow[two heads, from=2-1, to=2-2]
\end{tikzcd}\]
where `$\epi$' arrows are stratified coarsenings and `$\into$' arrows are open stratified subtrusses (this can be shown by `(epi,mono)-factorizing' the map $c \circ i$, see e.g.\ \cite[Lem.\ 2.3.101]{fct}). Since open cone trusses normalize to open cone trusses, observe $\NF{T^{\leq x},g^{\leq x}}$ is an open cone truss. But we must have $(S,h) = \NF{T^{\leq x},g^{\leq x}}$ since $S \into T$ contains the cone point in its image. Thus $\NFTrs \cone(l) = (S,h)$ is an open cone truss as claimed.
\end{proof}

\nid The converse of \cref{lem:NFTrs-preserves-prod-and-cone} is easier and we simply state it as an observation.

\begin{obs}[Combinatorialization reflects products and cones] \label{obs:CMsh-preserves-prod-and-cone} Given a stratified $n$-truss $(T,g)$ and a tame stratification $f$ with $\NFTrs f = \OTT^k \times (T,g)$, then $f \iso \II^n \times h$ where $\NFTrs h = (T,g)$. Given a stratified cone $n$-truss $(T,g)$ and a tame stratification $f$ with $\NFTrs f = (T,g)$, then $f \iso \cone(l)$ for a tame link $l$. Both statement can be seen from the explicit construction of classifying meshes of trusses (see e.g.\ \cref{rmk:cmsh-via-compactifiation}).
\end{obs}

\subsection{Statement and proof of the equivalence}

We start with the open case, and later adapt our discussion to the compact case.

\begin{thm}[Combinatorialization of manifold diagrams] \label{thm:comb-mdiag} Framed stratified homeomorphism classes of manifold $n$-diagrams are in 1-to-1 correspondence with combinatorial manifold $n$-diagrams: the correspondence takes $(\II^n,f)$ to $\NFTrs(\II^n,f)$.
\end{thm}

\begin{rmk}[A higher categorical classification] There are also higher versions of this statement: for instance, the $\infty$-groupoid of manifold $n$-diagrams (modeled by the topological category of manifold diagrams and framed stratified maps that have weak inverses) is equivalent to the \emph{set} of combinatorial manifold $n$-diagrams.
\end{rmk}

\begin{proof}[Proof of \cref{thm:comb-mdiag}] By \cref{thm:tame-strat-comb}, mapping $(\II^n,f)$ to $\NFTrs(\II^n,f)$ injects framed stratified homeomorphism classes of manifold diagrams (see \cref{defn:mdiag}) into the set of normalized stratified trusses. To prove the theorem we need to show (1) that for each manifold diagram $f$, its fundamental stratified truss $\NFTrs f$ is combinatorially conical, and (2) that given a tame stratification $f$ such that $\NFTrs f$ is a combinatorial manifold diagram, then $f$ is framed transversal.

    We first prove (1). Let $(\II^n,f)$ be a manifold diagram, denote by $\iM^f$ the coarsest refining mesh of $f$ and set $(T,g) := \NFTrs f$. We want to show that $(T,g)$ is a combinatorial manifold diagram. Take $x \in T_n$. Take any point $z$ in the corresponding stratum $x$ of $\iM^f$. Choose a tame framed tubular neighborhood $\phi : \II^{k} \times \cone(l_z) \into f$ of $z$ (which exists because $f$ is a manifold diagram). Since the neighborhood is tame, there are mesh refinements $N \mshar \II^{k} \times \cone(l_z)$ and $M \mshar f$ such that $\phi$ descends to a stratified mesh map $\psi : (N,e) \to (M,f)$ (where we abbreviate $e \equiv \II^{k} \times \cone(l_z)$).
    Passing to trusses, this implies $i = \FTrs \psi : \FTrs (N,e) \into \FTrs (M,f)$ is a stratified subtruss. We have a stratified coarsening $c : (M,f) \epi \NF{\FTrs (M,f)} = (T,g)$ since $(T,g) = \FTrs (\iM^f,f)$ and $\iM^f$ is the coarsest refining mesh of $f$. The composite $c \circ i$ uniquely factors as
\[\begin{tikzcd}
	{\FTrs(M,f)} & {(T,g)} \\
	{\FTrs (N,e)} & {(S,h)}
	\arrow["i",from=2-1, to=1-1]
	\arrow[hook, from=2-2, to=1-2]
	\arrow["c",two heads, from=1-1, to=1-2]
	\arrow[two heads, from=2-1, to=2-2]
\end{tikzcd}\]
By uniqueness of normal forms observe $\NFTrs e = \NF{\FTrs (N,e)} = \NF{S,h}$. Since $e$ is a stratification of the form $\II^{k} \times \cone(l_z)$, and using \cref{lem:NFTrs-preserves-prod-and-cone}, observe that $\NFTrs e$ is of the form $\OTT^{k} \times (C,c)$ where $(C,c)$ is a stratified open cone truss. From this we obtain the commutative diagram:
\[\begin{tikzcd}
	{(S,h)} & {\OTT^{k} \times (C,c)} \\
    {(T^{\leq x},g^{\leq x})} & {(\tilde T^{\leq x},\tilde g^{\leq x})}
	\arrow[hook, from=2-1, to=1-1]
	\arrow[hook, from=2-2, to=1-2]
	\arrow[two heads, from=1-1, to=1-2]
	\arrow[two heads, from=2-1, to=2-2]
\end{tikzcd}\]
Observe that the right-hand subtruss must be an equality since it contains the cone point in its image. This verifies that $(T,g)$ is a combinatorial manifold diagram as claimed (see \cref{defn:comb-diag}).

It remains to show (2). Let $(T,g)$ be a combinatorial manifold $n$-diagram, and $(\II^n,f)$ a tame stratification with $(T,g) = \NFTrs(\II^n,f)$. We claim that $(\II^n,f)$ is tame framed conical (and thus a manifold $n$-diagram). Denote by $\iM^f$ the coarsest refining mesh of $f$. Let $z \in \II^n$. Then $z$ lies in a stratum $x$ of $\iM^f_n$. Since $T = \FTrs \iM^f$, consider $x$ as an element of $T_n$. Consider the neighborhood $\phi : T^{\leq x} \into T$ around $x$. By our choice of $x$, it is possible to choose the classifying submesh $\psi : N \into \iM^f$ of $\phi$ (i.e.\ $\FTrs \psi = \phi$) such that $\im(\psi)$ contains $z$. Coarsening this submesh by the stratified coarsening $\iM^f \to f$, $\psi$ now becomes a neighborhood $\psi^f : (\II^n,e) \into (\II^n,f)$. Since $\NF{T^{\leq x},g^{\leq x}}$ is of the form $\OTT^n \times (C_x,c_x)$, and since $\NFTrs e = \NF{T^{\leq x},g^{\leq x}}$, we can use \cref{obs:CMsh-preserves-prod-and-cone} to see that $\psi^f$ is in fact a tame framed tubular neighborhood as required in \cref{defn:mdiag}.
\end{proof}

\begin{term}[Fundamental and classifying diagrams] Given a manifold $n$-diagram $(\II^n,f)$, we call $\NFTrs (\II^n,f)$ the `fundamental combinatorial manifold diagram' of $f$. Conversely, we say $(\II^n,f)$ is a `classifying manifold diagram' of $(T,g) = \NFTrs (\II^n,f)$.
\end{term}

\begin{eg}[Combinatorializing manifold 2-diagrams] In \cref{fig:combinatorializations-of-manifold-2-diagrams} on the left, we depict (two) manifold 2-diagrams $f$; to their right, we depict their coarsest stratifying 2-mesh $M$. On the right, we depict the corresponding combinatorial manifold 2-diagram $(T,g) = \FTrs(M,f)$.
\begin{figure}[ht]
    \centering
    \def\svgwidth{1\columnwidth}
    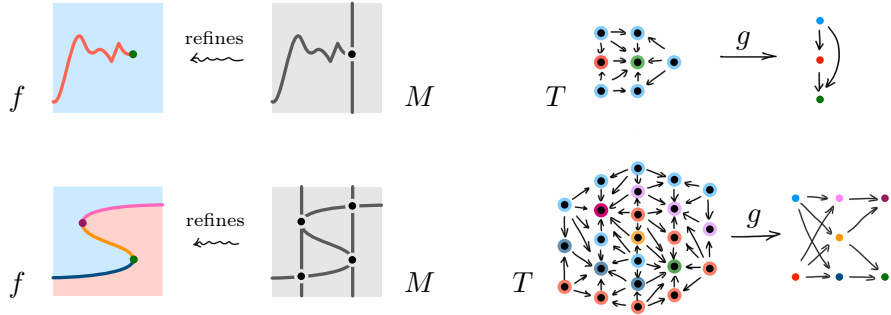

    \caption{Combinatorializations of manifold 2-diagrams}
    \label{fig:combinatorializations-of-manifold-2-diagrams}
\end{figure}
\end{eg}

\begin{eg}[Combinatorializing a manifold 3-diagram] Consider the manifold 3-diagram $(\II^3,f)$ shown on the left in \cref{fig:combinatorialization-of-a-3-diagram}. A refining 3-mesh of this diagram (up to framed stratified homeomorphism) was given in \cref{fig:meshes-cellulate-manifold-diagrams}. On the right of \cref{fig:combinatorialization-of-a-3-diagram} we depict the corresponding combinatorial manifold 3-diagram $(T,g) = \NFTrs f$ (the full 3-truss $T_3 \to T_2 \to T_1$ was shown earlier in \cref{fig:combinatorial-manifold-3-diagram}).
\begin{figure}[ht]
    \centering
    \def\svgwidth{1\columnwidth}
    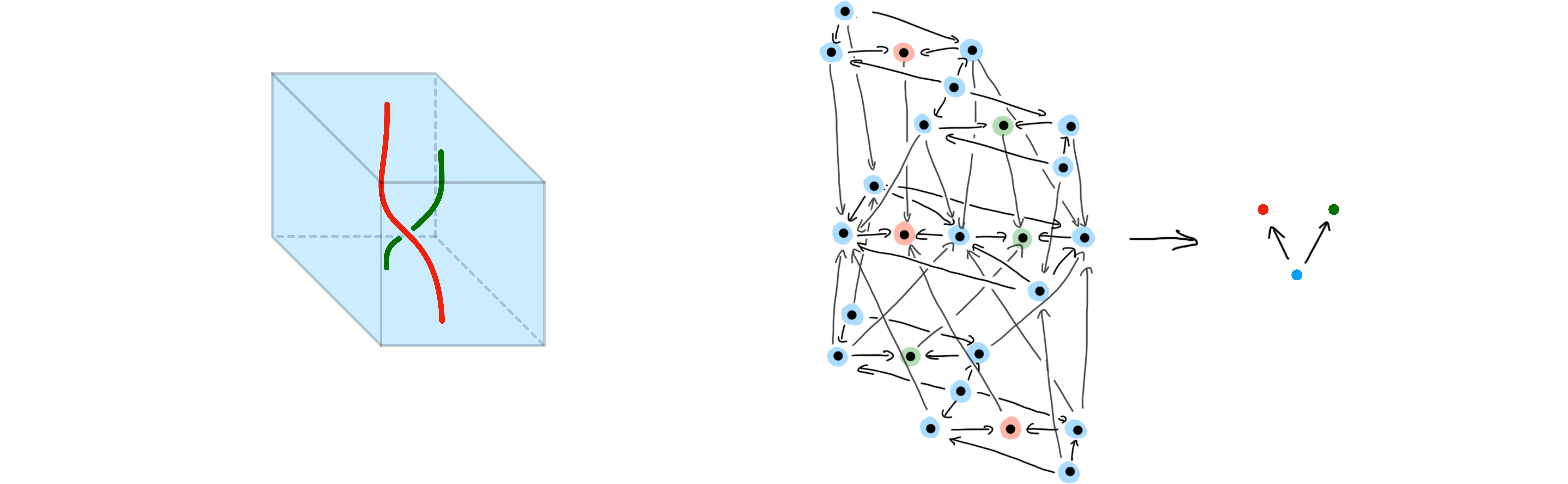

    \caption{Combinatorialization of a manifold 3-diagram}
    \label{fig:combinatorialization-of-a-3-diagram}
\end{figure}
\end{eg}

\subsection{The compact case}

Analogous to the open case, compact manifold diagrams and compact combinatorial manifold diagrams are related as follows.

\begin{thm}[Combinatorialization of compact manifold diagrams]\label{lem:class-mdiag-compact} Framed stratified homeomorphism classes of compact manifold $n$-diagrams are in 1-to-1 correspondence with compact combinatorial manifold $n$-diagrams: the correspondence takes $(\bI^n,f)$ to $\NFTrs (\bI^n,f)$.
\end{thm}

\begin{proof} The proof of \cref{thm:comb-mdiag} carries over to the compact case almost verbatim (with evident adjustments accounting for `corner neighborhoods', see \cref{term:corner-neighborhoods} and \cref{term:corner-trusses}).
\end{proof}

\begin{eg}[Combinatorialization of a compact manifold diagrams] In \cref{fig:combinatorialization-of-a-compact-manifold-diagrams} on the left we depict a compact manifold 2-diagram, together with its coarsest refining mesh $M$. On the right we depict the corresponding stratified closed truss $(T,g) = \FTrs(M,f)$ which is a compact combinatorial manifold 2-diagram.
\begin{figure}[ht]
    \centering
    \def\svgwidth{1\columnwidth}
    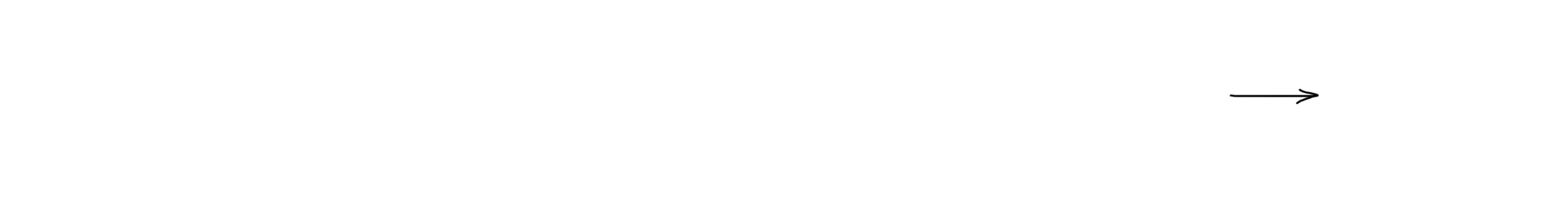

    \caption{Combinatorialization of a compact manifold diagram}
    \label{fig:combinatorialization-of-a-compact-manifold-diagrams}
\end{figure}
\end{eg}

Having `combinatorialized' topological manifold diagrams, we may now describe the relation of compact and open topological manifold diagrams using our combinatorial construction of compactifications.

\begin{obs}[Relation of open and compact manifold diagrams] \label{constr:compactifying-diagrams} Given a compact manifold diagram $(\bI^n,f)$, its `interior' $(\II^n,f\intrr)$ is the substratification of $f$ on $\II^n \subset \bI^n$. One checks that $(\II^n,f\intrr)$ is in fact an open manifold diagram. The converse is slightly more subtle in the topological case. Given an open manifold diagram $(\II^n,f)$ it might not make sense to `compactify it on the nose' since manifold strata of $f$ in $\II^n$ need not continuously extend to $\bI^n$---however, they always do so up to framed stratified homeomorphism. Indeed, we may define the  `cubical compactification' of $f$ to be a compact manifold diagram $(\bI^n,\overline f)$ (up to stratified homeomorphism) such that $\NFTrs \overline f = \overline {\NFTrs f}$ (the latter being the cubical compactification of $\NFTrs f$, see \cref{obs:open-compact-comb-corr}). If a compact manifold diagram cubically compactifies its interior in this way we say it is `interior-compactifying'.
\end{obs}

\begin{eg}[Interior-compactifying diagrams] In \cref{fig:interior-compactifying-diagram} we depict two compact topological manifold diagrams: the left does not cubically (and in fact, neither retractably) compactify its interior while the right one does.
\begin{figure}[ht]
    \centering
    \def\svgwidth{1\columnwidth}
    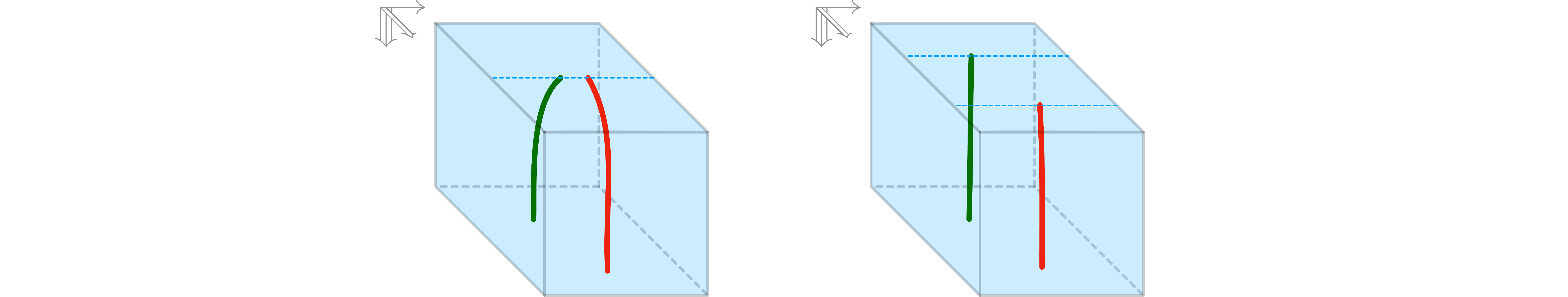

    \caption{A non-interior-compactifying and an interior-compactifying diagram}
    \label{fig:interior-compactifying-diagram}
\end{figure}
\end{eg}


\subsection{Canonicity of links}

As an application of the preceding two theorems, we observe that links in manifold diagrams can be canonically chosen and are well-defined up to stratified homeomorphism (recall from \cref{rmk:links-well-def-in-mdiag} that generally this need not be the case in topological conical stratifications). We argue in the open case; the compact case is analogous.

\begin{constr}[Canonical links] \label{rmk:links-are-well-defined} Consider a manifold diagram $(\II^n,f)$ and a stratum $s$ of $f$. The \textbf{canonical link} $\link_s$ of $s$ can be constructed as follows. Set $(T,g) := \NFTrs f$.
    Take any $x \in T_n$ with $g(x) = s$ (using $\Entr(g) = \Entr(f)$). Since $(T,g)$ is a combinatorial manifold diagram, the normal form $\NF{T^{\leq x},g^{\leq x}}$ is of the form $\OTT^k \times (C_x,c_x)$. Compactify $(C_x,c_x)$ and obtain the closed stratified mesh $(N,e) = \CMsh (\overline C_x, \overline c_x)$. Restrict the stratification $(\bI^n,e)$ to $\partial \bI^n$ to construct the canonical link $(\partial \bI^n, \link_s)$.
    Up to framed stratified homeomorphism, the construction is independent of any choices made. One shows that each point $z \in s$ has a framed tubular neighborhood of the form $\II^k \times \cone(\link_s)$.
\end{constr}


\begin{rmk}[All links are stratified homeomorphic] \label{rmk:links-are-strat-homeo} Given a manifold diagram $(\II^n,f)$, a stratum $s$ of $f$, and $z \in s$, pick any tame tubular neighborhood $\II^k \times \cone(l_z)$ with tame link $l_z$. One can show that $\overline\cone(l_z)$ is a retractable compactification of $\cone(l_z)$ (in the sense that $\FTrs \overline\cone(l_z)$ is a retractable compactification of $\FTrs \cone(l_z)$): a retraction can be produced by down-scaling the closed cone (cf.\ the proof of \cref{lem:NFTrs-preserves-prod-and-cone}(2)). Using that $\NFTrs \cone(l_z) \iso \NFTrs \cone(\link_s)$, one can produce a stratified homeomorphism between $l_z$ and $\link_s$---thus all choices of links $l_z$ are stratified homeomorphic. However, if $\overline\cone(l_z)$ non-cubically compactifies its interior, there need not be a \emph{framed} stratified homeomorphism between the two. An example of this situation can be given by topological analog of the retractable compactification in \cref{fig:compactification-of-a-stratified-truss}.
\end{rmk}

We end this section with the remark that more generally manifold diagrams can be studied in bundles.

\begin{rmk}[Manifold diagram bundles] Just as tame stratifications can be considered in bundles (see \cref{rmk:tame-bundles}), so can manifold diagrams. We will not spell this out here, but details are straight-forward to work out, and inspiration can be taken from the (closely related) later discussion of tangle bundles in \cref{ch:tangles-and-sings}.
\end{rmk}

\section{Cell diagrams} \label{sec:cell-diagrams}

In this section we discuss the geometric duals of manifold diagrams: this will lead us to a very general class of diagrams of directed cells. We discuss how they can be thought of as `pasting diagram' (where we use the term with its familiar, but informal, meaning from higher category theory).

\subsection{Topological and combinatorial definitions}

Recall the duality of truss and meshes (see \cref{defn:dual-n-trusses} and \cref{defn:dual-n-meshes}). This extends to the case of stratified trusses and tame stratifications as follows.

\begin{defn}[Duality for stratified trusses] Given a stratified $n$-truss $(T,f)$, its \textbf{dual stratified truss} $(T,g)^\dagger$ is the stratified $n$-truss $(T^\dagger, f^\dagger)$ where $f^\dagger := f\op : T_n\op \to \Entr(f)\op$ is the dual of $f$.
\end{defn}

\begin{defn}[Duality for tame stratifications] Given tame stratification $(X,f)$ and $(Y,g)$ we say $g$ is a \textbf{dual} of $f$ if $\NFTrs g = (\NFTrs f)^\dagger$.
\end{defn}

\nid Applying this to the notion of manifold diagrams yields the following.

\begin{defn}[Cell diagrams] A \textbf{cell $n$-diagram} $(X,f)$ is a tame stratification that is dual to a manifold $n$-diagram.
\end{defn}

\begin{eg}[Cell diagrams] We illustrate four cell diagrams in the upper row of \cref{fig:pasting-diagrams-and-their-dual-manifold-diagrams}. The third example illustrates that the underlying $n$-framed space $X$ of a cell $n$-diagram $(X,f)$ need not be homeomorphic to the closed cube. Note, the fourth example is a cell 3-diagram consisting of two 3-cells, glued together as indicated to the right (note, the cells are 3-dimensional balls: we highlight this by indicating a cross-sections of the respective cells). Underneath each cell diagram we illustrate a dual manifold diagrams: in particular, the fourth example recovers our earlier example of the braid homotopy.
\begin{figure}[ht]
    \centering
    \def\svgwidth{1\columnwidth}
    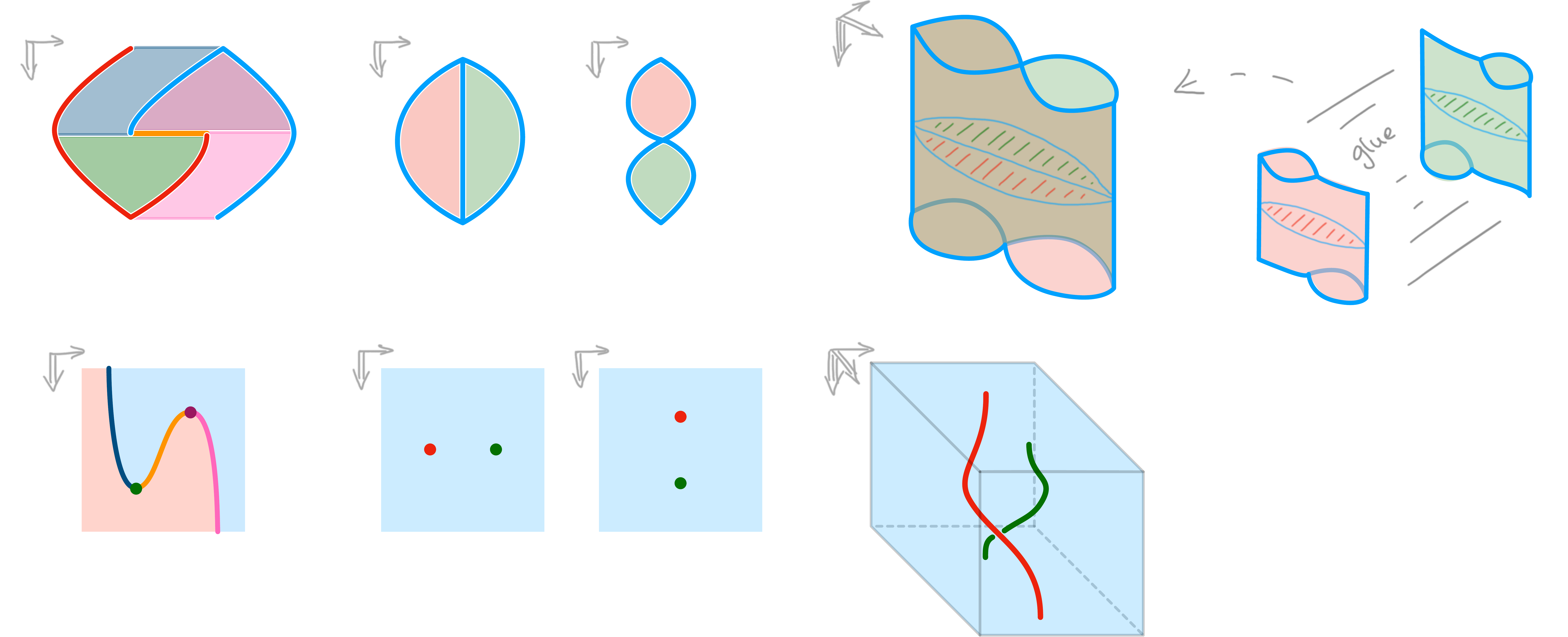

    \caption{Cell diagrams and their dual manifold diagrams}
    \label{fig:pasting-diagrams-and-their-dual-manifold-diagrams}
\end{figure}
\end{eg}

\begin{obs}[Duality between manifold and cell diagrams] \label{cor:duality-m-p-diag} Framed stratified homeomorphisms classes of manifold $n$-diagrams and cell $n$-diagrams correspond bijectively by dualization.
\end{obs}

One can similarly define combinatorial cell diagrams as those stratified trusses that are dual to combinatorial manifold diagrams. We take a different path: based on the observation that normalization is compatible with dualization (that is, $\NF{(T,g)^\dagger} = \NF{T,g}^\dagger$), we dualize the condition of combinatorial conicality and use it to define combinatorial cell diagrams.

\begin{term}[Stratified truss closures] Given a stratified $n$-truss $(T,f)$, the `stratified truss closure' $(T^{\geq x},f^{\geq x})$ is the stratified subtruss of $(T,f)$ whose support is the upper closure $T^{\geq x}$ of $x$. This is dual to the notion of `stratified truss neighborhoods' from \cref{defn:strat-truss-nbhd}.
\end{term}

\begin{term}[Stratified facet truss] A `closed facet $n$-truss' $(T,f)$ is a closed stratified $n$-truss with minimal element $\bot$ in $T_n$ such that $\dim(\bot) = n$ and $\{\bot\}$ is its own stratum. This is dual to the notion of `open cone trusses' from \cref{defn:truss-cones}.
\end{term}

\begin{term}[The truss point] The `closed 1-truss point' $\CTT^1$ is the closed 1-truss with one element. The `closed $k$-truss point' $\CTT^k$ is the closed $k$-truss obtained as the $k$-fold product truss $\CTT^1 \times \CTT^1 \times ... \times \CTT^1$. This is dual to the notion of `open $k$-truss cube' from \cref{term:open-cube-truss} (note that $(\CTT^k)^\dagger = \OTT^k$).
\end{term}

\nid Combinatorial conicality now dualizes to the following condition.

\begin{defn}[Combinatorially facetal stratified trusses] A closed $n$-truss $(T,f)$ is \textbf{combinatorially facetal at $x \in T_n$} if the stratified truss closure of $x$ normalizes to a product of a truss point and a truss facet, that is, there is $0 \leq k \leq n$ and a closed facet $(n-k)$-truss $(D_x,d_x)$ such that:
\[
	\NF{ T^{\geq x}, f^{\geq x} } = \CTT^{k} \times (D_x,d_x).
\]
We say $(T,f)$ is \textbf{combinatorially facetal} if it is so at all $x \in T_n$.
\end{defn}

\begin{defn}[Combinatorial cell diagrams] \label{defn:comb-past-diag} A \textbf{combinatorial cell $n$-diagram} $(T,f)$ is a normalized closed stratified $n$-truss that is combinatorially facetal.
\end{defn}

\begin{rmk}[A self-contained definition of cell diagrams] With enough care it is equally possible to define (topological) cell diagrams in self-contained topological terms, using an appropriate topological notions of `framed facetal' tame stratification.
\end{rmk}

\begin{obs}[Duality of combinatorial manifold and cell diagrams] \label{obs:duality-diagrams-comb} A stratified $n$-truss $(T,f)$ is a combinatorial cell diagram if and only if its dual $(T,f)^\dagger$ is a combinatorial manifold diagram.
\end{obs}

\begin{obs}[Correspondence of cell diagram and combinatorial cell diagrams] A tame stratification $(X,f)$ is a cell $n$-diagram if and only if $\NFTrs f$ is a combinatorial cell $n$-diagram.
\end{obs}

\subsection{On the categorical meaning of cell diagrams} \label{ssec:cell-diag-interpretation}

We briefly outline how to think about cell diagrams in more familiar higher categorical terms, as `pasting diagrams of higher morphisms'.

\begin{term}[Non-degenerate and degenerate cells, and cell shapes] Consider a combinatorial cell diagram $(T,f)$, take an object $x \in T_n$ with $\dim(x) = k$. We say $x$ is a `non-degenerate $k$-cell' in $(T,f)$ if $\{x\}$ is its own stratum in $(T^{\leq x},f^{\leq x})$. Otherwise we say $x$ is a `degenerate $k$-cell'. The terminology similarly applies to cells in the coarsest refining mesh of a topological cell diagram.
\end{term}

\begin{rmk}[Relating cell diagrams to classical pasting diagrams] \label{rmk:cell-diagram-interpretation} Cell diagrams can be depicted as classical `point and arrow diagrams' by replacing $k$-cells by $k$-arrows if they are non-degenerate, and by `$k$-identities' (i.e.\ headless $k$-arrows) if they are degenerate. For our earlier examples in \cref{fig:pasting-diagrams-and-their-dual-manifold-diagrams} this translation is given in \cref{fig:interpreting-cell-diagrams-as-pasting-diagrams}. (Here, a `$k$-arrow' is drawn with $k$ parallel lines, the exception being an identity $1$-arrow, which is drawn with $2$ lines as `$=$'.)
\begin{figure}[ht]
    \centering
    \def\svgwidth{1\columnwidth}
    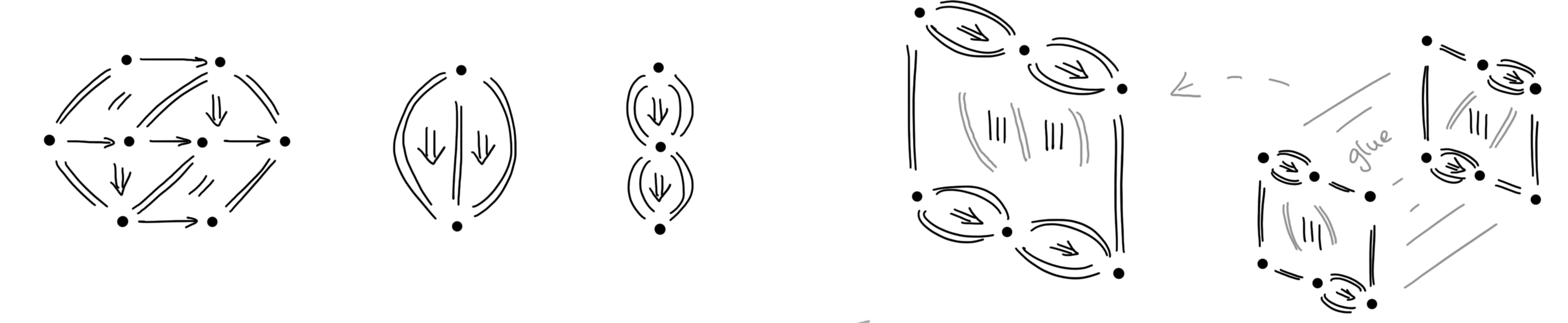

    \caption{Interpreting cell diagrams as pasting diagrams}
    \label{fig:interpreting-cell-diagrams-as-pasting-diagrams}
\end{figure}
\end{rmk}

The previous remark explains how to think of cell diagrams as diagrams of points and arrows, and we may interpret these as traditional categorical pasting diagrams in which arrows represent higher morphisms. Let us next comment on the type of `morphism shapes' that can appear in such diagrams.

\begin{term}[Framed cell shapes] A cell $n$-diagram $(S,g)$ with minimal object $x \in S_n$ is called a `$n$-framed cell shape'.\footnote{A `framed cell shape' is different from a `framed regular cell' as defined in \cite[\S1]{fct}, due to the additional stratification structure. Moreover, after removing degeneracies (i.e.\ after quotienting it by the degenerate cells in the shape's boundary) a framed cell shape need not be a regular cell.}
\end{term}

\nid Note, every cell $(T^{\leq x},f^{\leq x})$ in a cell diagram $(T,f)$ normalizes to a framed cell shape. Framed cell shapes generalize many other `categorical cell shapes' in use, such as globes, (directed) simplices, cubes, opetopes, etc.---however, in order to allow for a `side-by-side' comparison, framed cell shapes need to first be `quotiented' by the degenerate cells in their boundary: we illustrate this in \cref{fig:blow-up-of-classical-cells-to-computadic-cells}, where the classical shapes shown in the top row are obtained by quotienting along the identities in the framed cell shapes shown in the bottom row. (We also schematically indicate how this quotient acts on the framing of the underlying framed space.)
\begin{figure}[ht]
    \centering
    \def\svgwidth{1\columnwidth}
    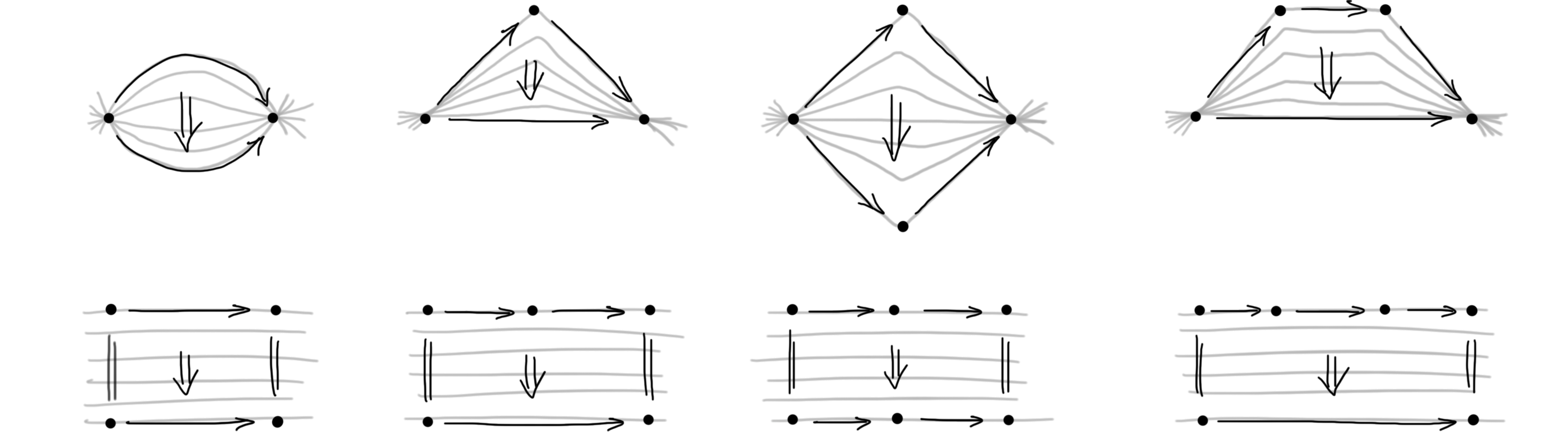

    \caption{Comparison of common cell shapes and framed cell shapes}
    \label{fig:blow-up-of-classical-cells-to-computadic-cells}
\end{figure}

The central feature of cell diagrams that distinguishes them from existing notions of pasting diagrams is that cell diagrams can record certain categorical coherences (namely, by diagrams that are geometrically dual to isotopies of manifold diagrams, such as the braid). This can be used to resolve classical problems of `strict computads' not having a well-defined category of shapes: roughly speaking, when working with `strict' shapes one encounters the problem that composing higher cells with degenerate boundaries in pasting diagrams may lead to an ill-defined notion of cells in those diagrams (see e.g.\ \cite[\S7.6]{leinster2004higher} \cite{henry2017nonunital} \cite{hadzihasanovic2020diagrammatic}). In contrast, cell diagrams deal well with degenerate boundaries in all dimensions (see e.g.\ the third example in \cref{fig:interpreting-cell-diagrams-as-pasting-diagrams}). However, an in-depth discussion of this topic would lead beyond the scope of the present paper.

\addtocontents{toc}{\protect\newpage}
\chapter{Tame tangles} \label{ch:tangles-and-sings}

We introduce the notion of \emph{tame tangles}, which are manifolds with corners that are `transversally' embedded in the framed $n$-cube. We show that tame tangles may also be understood in purely combinatorial terms, which parallels the case of manifold diagrams. (In fact, tame tangle and manifold diagrams are closely related: tame tangles can be canonically refined to manifold diagrams by stratifying the tangle manifold by the `loci of its critical points' as we will explain.) Tame tangles can naturally be studied in bundles, and this will lead us to definitions of perturbations and perturbation stability of tame tangles. Locally, by studying the perturbation stable `germs' of tames tangles around a point, this gives rise to a theory of singularities. We will end with a discussion of conjectures and directions for future work.

\section{Definitions}

\subsection{Topological definition} We define tame tangles as manifolds embedded in $\II^n$ that satisfy a `framed transversality' condition: the condition will ensure that `loci of critical points' are in general position with respect to the ambient framing.

\begin{rmk}[Embeddings as stratifications] \label{rmk:emb-as-strat} Given an embedding $f : W \into X$ of a closed subspace $W$ in a topological space $X$, the embedding $f$ defines a stratification of the space $X$ whose strata are the connected components of $W$ and the connected components of the complement $X \setminus W$. Abusing notation we write $f$ for the stratification that this defines, and $f : X \to \Entr(f)$ for the stratification's characteristic map. (Whenever we consider an embedding $f : W \into X$ as a stratification in this way, we always tacitly assume the subspace $W$ is closed in $X$.)
\end{rmk}

\nid Recall that a tame stratification is a stratification of a flat framed space that can be refined by a mesh (see \cref{defn:flat-framed-stratification}).

\begin{defn}[Tame embeddings] Given a flat framed space $X$, an embedding $f : W \into X$ of topological spaces is called a \textbf{tame embedding} if it defines a tame stratification of $X$.
\end{defn}

Recall the notions of (tame) links and cones from \cref{notn:cube-links}.

\begin{term}[Cones of embeddings] Given an embedded space $f = (W \into \partial \bI^n)$ its `cone embedding' $\fcone(f)$ is the embedding of cones $\cone(W) \into \cone(\partial \bI^n) \iso \II^n$. (Note $W$ can be empty, in which case $\cone(W)$ is a point.) One similarly defines the `closed' cone embedding $\overline\fcone(f)$.
\end{term}

\nid Note well that cone embeddings $\fcone(f)$ differ from stratified cones $\cone(f)$ in that $\fcone(f)$ doesn't have a separate cone point stratum; we illustrate the difference in \cref{fig:coned-embeddings-vs-stratified-cones}.

\begin{figure}[ht]
    \centering
    \def\svgwidth{1\columnwidth}
    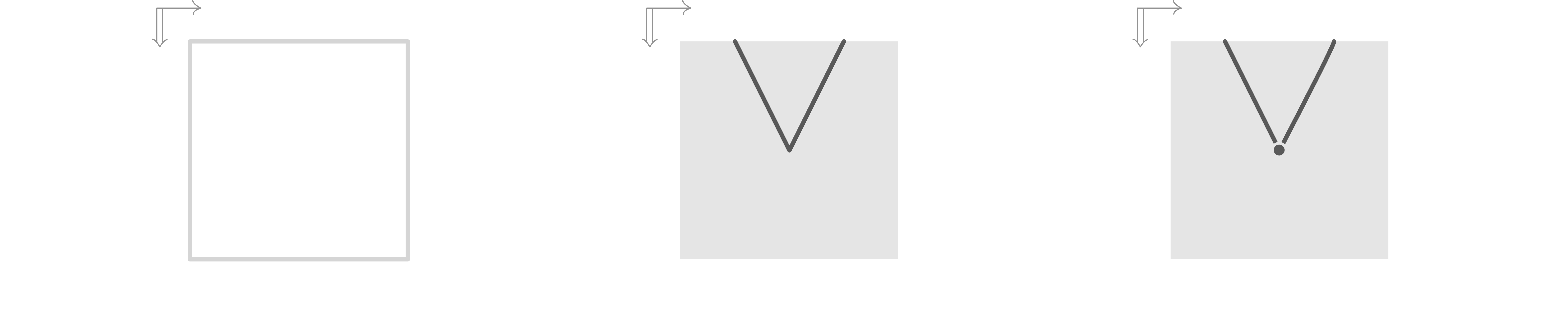

    \caption{Cone embeddings vs stratified cones}
    \label{fig:coned-embeddings-vs-stratified-cones}
\end{figure}

\begin{defn}[Framed transversality] \label{defn:framed-transv} Given an embedded $m$-manifold $f = (W \into \II^n)$, inductively define a notion of \textbf{$k$-transversal points $x$ of $W$} as follows.
    \begin{enumerate}
        \item[$-$] If $x \in W$ has a tame framed stratified neighborhood of the form $\II^{m} \times (\II^{n-m},\fcone(l))$ with $l = (\emptyset \into \partial\bI^{n-k})$ the empty link, and $x \in \II^m \times 0$, then we say $x$ is an \textbf{$m$-transversal} point of $W$.
        \item[$-$] Now take $k < m$. Assume $x$ is not $j$-transversal for any $j > k$. If $x$ has a framed stratified neighborhood of the form $\II^{k} \times (\II^{n-k},\fcone(l))$ for $l$ a tame link of the form $l = (S^{m-k-1} \into \partial\bI^{n-k})$, with $x \in \II^k \times 0$, and all $y \in \II^k \times \cone(S^{m-k-1}) \setminus \II^k \times 0$ being $j_y$-transversal for $j_y > k$, then $x$ is a \textbf{$k$-transversal} point of $W$.
    \end{enumerate}
We say $f$ is a \textbf{tame framed transversal stratification} if all points $x \in W$ are $k_x$-transversal points of $W$ for some $0 \leq k_x \leq m$.
\end{defn}


\begin{notn}[Transversal dimension] \label{notn:transv-dim} If $f$ is $k$-transversal at $x$, we also call $k$ the `transversal dimension' of $x$ and write $k = \tdim (x)$.
\end{notn}

\begin{term}[Regular and singular points] \label{term:reg-and-sing-points} Given an embedded manifold $f = (W \into \II^n)$, and a point $x \in W$ with transversal dimension $k$ we say $x$ is `regular' if $k = \dim(W)$, and `singular' if $k = 0$.
\end{term}

\begin{eg}[Transversality and non-transversality] In \cref{fig:framed-transversality-condition-and-failure}, for the 1-manifold embedded in the 2-cube shown in the center, we indicate the transversal dimension of three points on the manifold. The blue point is regular, the green point is singular, while the red point fails to be transversal (while it does have a neighborhood of the form $\II^0 \times \cone(l)$ for $l = (S^0 \into \partial \bI^2)$, the condition that points away from the cone stratum are $k$-transversal with $k > 0$ is not satisfied).
\end{eg}

\begin{figure}[ht]
    \centering
    \def\svgwidth{1\columnwidth}
    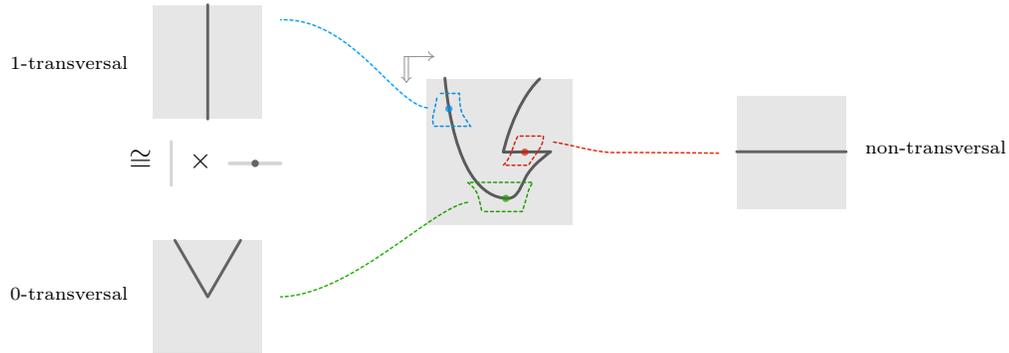

    \caption{Points of transversality and of non-transversality}
    \label{fig:framed-transversality-condition-and-failure}
\end{figure}

\nid The previous example illustrates that the inductive phrasing of framed transversality is crucial. We will see later in \cref{ssec:tangle-vs-mdiag} that the definition of framed transversality ensures properties that are implicit in our earlier definition of framed conicality.

\begin{defn}[Tame tangles] \label{defn:tame-tangle} A \textbf{tame $m$-tangle $f = (W \into \II^{n})$} is a tame embedded $m$-manifold such that $f$ is a tame framed transversal stratification.\footnote{\label{foot:non-tame-nbh-for-tangles} As in the case of manifold diagrams (see \cref{obs:redundancy-tame-con}) an equivalent definition of tame tangles can be obtained by replacing the condition of `tame framed transversality' by `framed transversality' (defined by omitting the adjective `tame' throughout \autoref{defn:framed-transv})---it is sufficient to require the stratification $f$ itself to be tame.}
\end{defn}



\begin{eg}[Tame tangles] Several tame tangles $W \into \II^n$ are illustrated in \cref{fig:tame-tangle-in-dim-2-and-3}.
\begin{figure}[ht]
    \centering
    \def\svgwidth{1\columnwidth}
    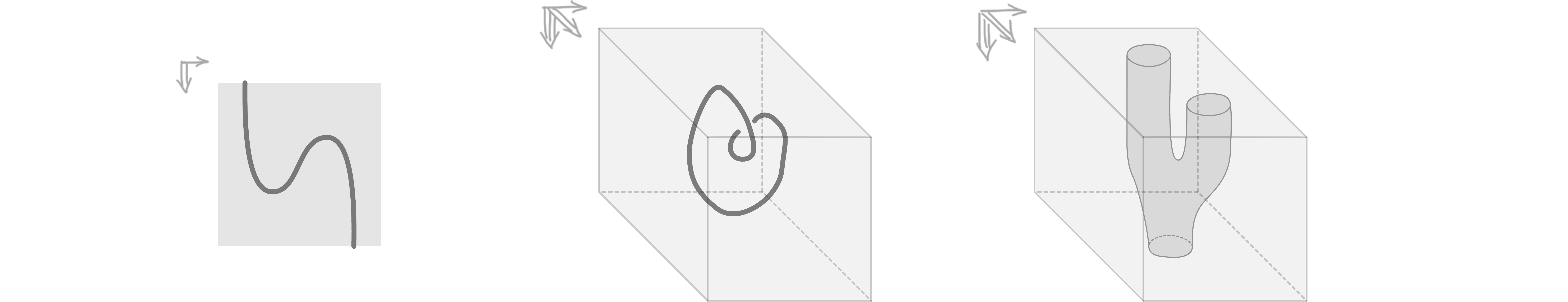

    \caption{Tame tangles in dimension 2 and 3}
    \label{fig:tame-tangle-in-dim-2-and-3}
\end{figure}
\end{eg}

\begin{noneg}[Tame tangles] The embedded manifold given in \cref{fig:framed-transversality-condition-and-failure} does not define a tame tangle, since it fails the tame framed transversality condition. The embedded manifold in \cref{fig:a-non-tame-diagram} (thought of as an embedding of a manifold consisting of two disconnected intervals) does not define a tame tangle, since it fails to be a tame stratification.
\end{noneg}

\begin{rmk}[Tangles with defects] By considering $k$-filtrations $f$ of closed subspaces $W_0 \into W_1 \into ... \into W_k \into X$ in place of simple embeddings $W \into X$ the preceding discussion can be generalized to define a notion of `tame tangles with defects'.
\end{rmk}

Parallel to our earlier discussion of compact manifold diagrams, one may adapt the definitions of tame tangles to the case of the \emph{closed} $n$-cube $\bI^n$.

\begin{rmk}[Compact tame tangles] \label{rmk:compact-tame-tangles} Replacing the open cube $\II^{k}$ by corner neighborhoods $\II^\sigma$ (see \cref{term:corner-neighborhoods}) one readily defines a condition of \textbf{tame compact framed transversality} (which requires neighborhoods of the form $\II^\sigma \times (\II^{n-k},\fcone(l))$, $\sigma \in \kP^k)$. A \textbf{compact tame $m$-tangle} $g = (W \into \bI^n)$, where $W$ is an $m$-manifold, is a tame stratification that is tame compact framed transversal.
\end{rmk}

\nid As a straight-forward generalization of \cref{conj:mfld-diag-approx}, we mention the following conjecture which relates our tame notion of tangles with the following classical notions of tangles.

\begin{term}[Tangles] \label{term:transversal} An `$m$-tangle' is an embedded $m$-manifold $f = (W \into \bI^{n})$. We say $f$ is a PL (resp.\ smooth) $m$-tangle if it is a PL embedding of a PL manifold (resp.\ a smooth embedding of a smooth manifold).
\end{term}

\begin{rmk}[Tangles with nice corners] \label{rmk:what-about-corners} Notions of $m$-tangles $f = (W \into \bI^{n})$ usually include conditions that guarantee that $W$ is well-behaved with respect to the corners of $\bI^n$. For instance, a `tangle with nice corners' $f$ can be obtained by requiring, firstly, that $\im(f) \subset \bI^m \times \II^{n-m}$ and secondly, that maps $\pi_{>i} : f \to \bI^i$ (for $1 \leq i \leq m$) restrict on $\II^{i-1} \times \{\pm 1\}$ to trivial framed stratified bundles of lower dimensional tangles with nice corners (with the latter notion being defined inductively). We usually need not impose such extra conditions, since corner behavior is already taken care of by enforcing compact framed transversality.
\end{rmk}

\begin{conj}[Tame tangles approximate transversal tangles] \label{conj:transv-tang-approximation} Any compact framed transversal $m$-tangle can be approximated by compact tame $m$-tangles.
\end{conj}

\nid The conjecture holds in the case of PL tangles (we will return to this point later in \cref{obs:pl-tang-gen-tame}).


\subsection{Combinatorial definition} Next, let us discuss the analogous combinatorial definition of tangles. Recall, from \cref{rmk:strat-posets} that a poset $P$ can be endowed with a topology whose basic opens are the downward closures $P^{\leq x}$ of its elements $x \in P$. A subposet $Q \into P$ is called `closed' if it is closed as a topological subspace; equivalently, this means that for any $x \in Q$ and $y \in P$ with an arrow $x \to y$ we have $y \in Q$.

\begin{rmk}[Poset embeddings as stratifications] \label{rmk:poset-emb-strat} Paralleling the translation of embeddings to stratifications in \cref{rmk:emb-as-strat}, note that any embedding $f : Q \into P$ of a closed subposet $Q$ of $P$, determines a poset map $\theta_Q : P \to [1]\op$ (mapping $x \in P$ to $0$ iff $x \in Q$) called the `indicator map' of $Q$. From this we obtain a stratification $(P,\ccs {\theta_Q})$ of $P$ by connected component splitting of $\theta_Q$ (see \cref{obs:lab-strat-top}). Explicitly, strata of $\ccs {\theta_Q}$ are the connected components of $Q$ and of $P \setminus Q$. Abusing notation, we also write $f : P \to \Entr(f)$ in place of $\ccs{\theta_Q} : P \to \Entr(\ccs{\theta_Q})$.
\end{rmk}

\nid Given an $n$-truss $T$ and a closed subposet $f : Q \subset T_n$, we write $(T,f)$ for the stratified truss whose labeling is obtained from $f$ via the previous remark.

\begin{term}[Poset spheres] A poset $P$ `is an $m$-sphere' if its classifying space (i.e. the geometric realization of its nerve) is a topological $m$-sphere.
\end{term}

\nid Recall the definition of cone trusses (see \cref{defn:truss-cones}); as before, we usually denote maximal elements in total posets of cone trusses by $\top$. The combinatorial analog of the notion of tame framed transversality is given in the next definition.

\begin{defn}[Combinatorial transversality] \label{defn:comb-transv} For fixed $m < n$, a stratified open $n$-truss $(T,f : Q \into T_n)$ is said to be \textbf{combinatorially $m$-dim $k$-transversal at $x \in T_n$} ($k \leq m$) if there exists an open cone $(n-k)$-truss $C$ with a subposet $D \into C_{n-k}$ containing the cone point $\top$, with $D \setminus \top$ being an $(m-k-1)$-sphere, and
\[
    \NF{T^{\leq x}, f^{\leq x}} = \OTT^{k} \times (C,D \into C_{n-k}).
\]
If all $x \in T_n$ are $k_x$-transversal points for some $k_x$, then we say that $(T,f)$ is \textbf{combinatorially $m$-dim transversal}.
\end{defn}

\nid Observe that $(C,D)$ is necessarily normalized in the preceding definition. Thus, since $C$ is assumed to be a cone truss, it follows that $(C,D)$ is not a product itself (i.e.\ not of the form $\II^i \times (S,g)$ for $i > 0$).

\begin{rmk}[Comparison of combinatorial and tame framed transversality] We point out two important differences between the definition of combinatorial transversality above, and our earlier definition of tame framed transversality.
\begin{enumerate}
    \item[$-$] In the definition of combinatorial transversality, we require an explicit `dimension' parameter $m$, whereas in our earlier definition of tame framed transversality this parameter was determined as the dimension of the embedded manifold. (Morally, this roots in the fact that the subposet $Q \into T_n$ in general need have an intrinsically definable notion of dimension; for instance $Q$ could be a singleton, but still describe a combinatorially $m$-dim $k$-transversal point.)
    \item[$-$] Secondly, note that the definition of combinatorial transversality, unlike our earlier definition of framed transversality, need not be phrased inductively (that is, combinatorial $k$-transversality need to refer to the $(l > k)$-transversality of neighboring points)---the required `filtration by transversality dimension' instead follows from the underlying combinatorics of trusses. \qedhere
\end{enumerate}
\end{rmk}

\begin{notn}[Transversal dimension, combinatorially] \label{notn:transv-dim-comb} Given an $m$-dim transversal stratified $n$-truss $(T,f : Q \into T_n)$, for $x \in Q$, we write $k = \tdim(x)$ if $(T,f)$ is $m$-dim $k$-transversal at $x$.\end{notn}

\begin{eg}[Combinatorial transversality] In the middle of \cref{fig:a-stratified-truss-with-combinatorially-transversal} we illustrate a stratified 2-truss $(T,f : Q \into T_2)$ (corresponding to the tame stratification given earlier in \cref{fig:framed-transversality-condition-and-failure}). The subposet $Q$ is highlighted in gray. We mark three points in $T_2$ in blue, green and red respectively: the stratified truss $(T,f)$ is combinatorially 2-transversal at the blue and green points, while it fails to be combinatorially transversal at the red point (while the neighborhood around the red point is normalized and of the form $\OTT^0 \times (C,W \into C_2)$, the stratified truss $(C,W \into C_2)$ is not a cone truss; the maximal element of $C_2$ has dimension 1!).
\begin{figure}[ht]
    \centering
    \def\svgwidth{1\columnwidth}
    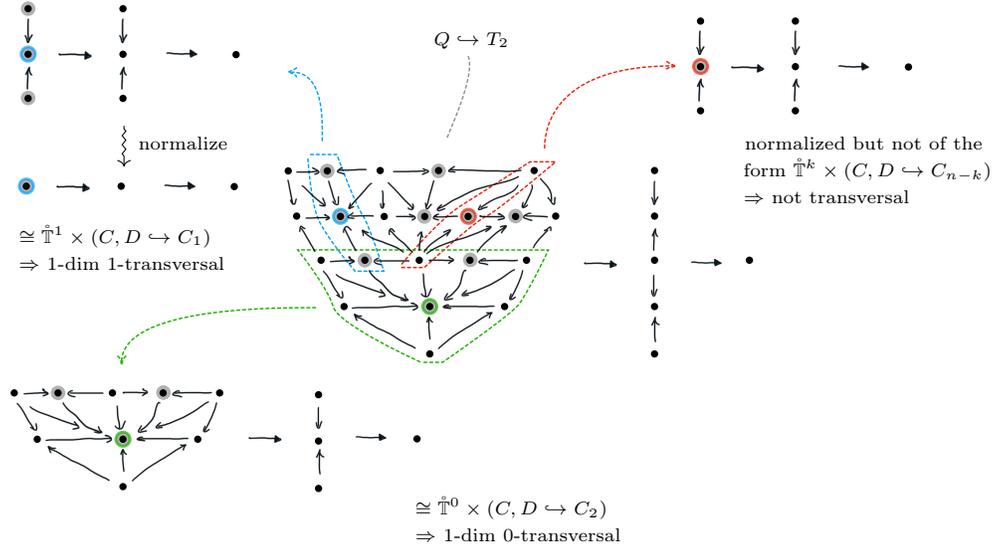

    \caption{A stratified truss with combinatorially transversal and non-transversal points}
    \label{fig:a-stratified-truss-with-combinatorially-transversal}
\end{figure}
\end{eg}

\begin{defn}[Tangle trusses] \label{defn:comb-tangle} An \textbf{$m$-tangle $n$-truss} $(T,f : Q \into T_n)$, is a normalized stratified open $n$-truss that is combinatorially $m$-dim transversal.
\end{defn}

\nid We also refer to tangle trusses as `combinatorial tangles'.

\begin{eg}[Tangle truss] The stratified truss depicted earlier in \cref{fig:combinatorializing-tame-stratifications} is a 1-tangle 2-truss.
\end{eg}

\begin{term}[Tangle truss manifolds] \label{term:comb-filtration} Given an $m$-tangle $n$-truss $(T,f : Q \into T_n)$ we call $Q$ the `tangle truss manifold' of $(T,f)$. 
\end{term}

Our arguments for the comparison of topological and combinatorial manifold diagrams carry over to the following comparison of topological and combinatorial tangles.

\begin{thm}[Combinatorializing tame tangles] \label{obs:combinatorializing-tame-tangles} Framed stratified homeomorphism classes of tame $m$-tangles in $\II^n$ are in 1-to-1 correspondence with $m$-tangle $n$-trusses, by taking tame tangles $f = (W \into \II^n)$ to their fundamental stratified trusses $\NFTrs f$.
\end{thm}

\begin{proof} The argument closely follows the proof of \cref{thm:comb-mdiag}, and we therefore only give a sketch proof here. Recall the framed transversality condition given in \cref{defn:framed-transv} resp.\ the combinatorial transversality condition in \cref{defn:comb-transv}. Analogous to the proof of \cref{thm:comb-mdiag}, we need to show (1) that for each tame $m$-tangle $f$, its fundamental stratified truss $\NFTrs f$ is combinatorially $m$-dim transversal, and (2) that given a tame stratification $f$ such that $\NFTrs f$ is an $m$-tangle truss, then $f$ is framed transversal.  The only novelty is the verification of the inductive condition in \cref{defn:framed-transv} in step (2) (recall, the condition requires us to show that $k$-transversal points have tubular neighborhoods in which all points away from the cone stratum are $j$-transversal, for $j > k$): but this follows from the combinatorial transversality condition satisfied by $\NFTrs f$, using the fact that stratified truss coarsenings restricts to stratified subtrusses (note, any arrow $x \to y$ in $T_n$, yields a subtruss inclusion $T^{\leq y} \into T^{\leq x}$ of the respective truss neighborhoods).
\end{proof}

\nid We call $\NFTrs f$ the `fundamental tangle truss' of the tame tangle $f$  and, conversely, say $f$ is a `classifying tame tangle' of $\NFTrs f$.

\begin{eg}[Combinatorializing tame tangles] The tame stratification and its combinatorialization shown earlier \cref{fig:combinatorializing-tame-stratifications} is also an example of a tame tangle and its combinatorialization. Another example can be constructed from the combinatorialization of the `braid' manifold diagram in \cref{fig:combinatorialization-of-a-3-diagram}: indeed, we may consider this diagram as a tame tangle whose tangle manifold is the union of the red and green interval strata (similarly, its combinatorialization may be considered as tangle truss whose tangle truss manifold is the union of the red and green poset strata).
\end{eg}

Parallel to our earlier discussion of compact combinatorial manifold diagrams, we may adapt the notion of combinatorial tangles to the case \emph{closed} trusses.

\begin{rmk}[Compact tangle trusses] \label{rmk:compact-tame-tangles-trusses} Replacing the open cube truss $\OTT^{k}$ by corner trusses $\TT^\sigma$ (see \cref{term:corner-trusses}) one readily defines a condition of \textbf{compact combinatorial $m$-dim transversality} for stratified closed trusses $(T,f)$ (the condition requires truss neighborhoods to normalizes to products of the form $\TT^\sigma \times (C,D \into C_{n-k})$ where $C$ is a cone truss and $D \setminus \top$ and is an $(m-k-1)$-sphere). A \textbf{compact $m$-tangle $n$-truss} $(T,f : Q \into T_n)$, is a normalized stratified closed $n$-truss that is compact combinatorially $m$-dim transversal.
\end{rmk}

\nid Analogous to the case of manifold diagrams, definitions of open tangle trusses and compact tangle trusses are related by cubical compactification (see \cref{ssec:compactification}). Via the combinatorialization of tame tangles in \cref{obs:combinatorializing-tame-tangles}, this compactification construction also finds a topological counterpart.

\begin{rmk}[Compactifying tangle trusses and tame tangles] \label{notn:compactifying-tangle-trusses}
    Given an $m$-tangle $n$-truss $(T,f : Q \into T_n)$ then its cubical compactification is a compact $m$-tangle $n$-truss which will be denoted by $(\overline T,\overline f : \overline Q \into \overline T_n)$ (see \cref{constr:comp-comb-diagrams}).
    Analogously, given an $m$-tangle $f = (W \into \II^n)$, its compactification is the compact tame tangle $\overline f = (\overline W \into \bI^n)$ (defined up to framed stratified homeomorphism) constructed as the classifying tangle of the cubical compactification $\overline {\NFTrs f}$ of the fundamental tangle truss $\NFTrs f$. Note that a compact tame tangle $(\bI^n,g)$ is the compactification of a tame tangle if and only if $\NFTrs g$ is `interior-compactifying' (cf.\ \cref{constr:compactifying-diagrams}).
\end{rmk}

\subsection{Refining tame tangles to manifold diagrams} \label{ssec:tangle-vs-mdiag} The definition of tame tangles is closely related to earlier definition of manifold diagrams. Indeed, as we now explain, any tame tangle can be refined to a manifold diagram whose strata are the components of the subspace of $k$-transversal points of the tame tangle (for all $k \geq 0$).

\begin{defn}[Transversal stratification of tangles] Given a tame $m$-tangle $f = (W \into \II^n)$, the \textbf{transversal stratification} $(\II^n,\tstr(f))$ of $f$ is the stratification whose strata are connected components of the subspaces $\tdim\inv(k) \subset W$ of $k$-transversal points (see \cref{notn:transv-dim}), and the connected components of the complement $\II^n \setminus W$ of $W$.
\end{defn}

\nid Note that the transversal stratification of a tame tangle is a \emph{refinement} of the tangle.

\begin{eg}[Transversal stratification of tangles] In \cref{fig:manifold-diagram-refinements} we depict two tame tangles refined by their transversal stratifications.
\begin{figure}[ht]
    \centering
    \def\svgwidth{1\columnwidth}
    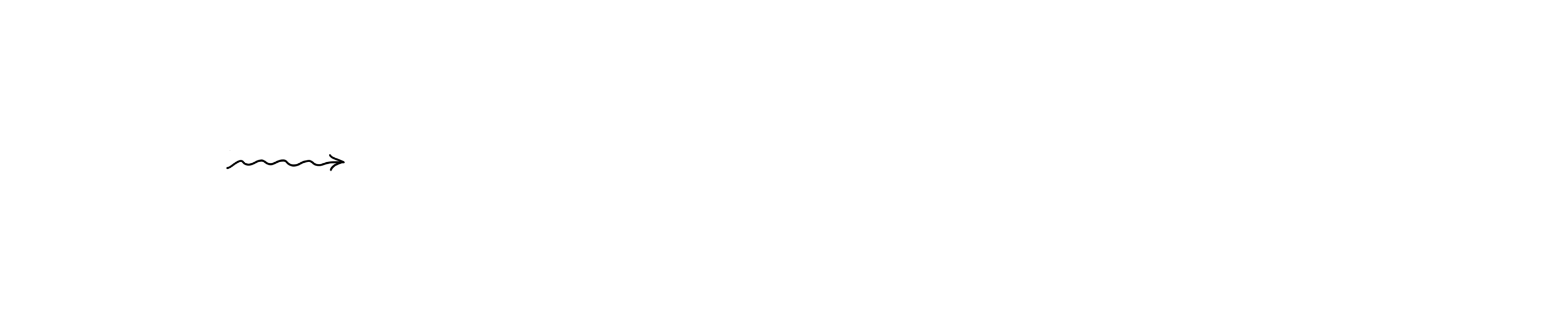

    \caption{Transversal stratification of tangle}
    \label{fig:manifold-diagram-refinements}
\end{figure}
\end{eg}

The definition can of course also be mirrored combinatorially.

\begin{defn}[Combinatorial transversality stratification] Given a $m$-tangle $n$-truss $(T,f : Q \into T_n)$, the \textbf{transversal stratified truss} $(T,\tstr(f))$ (also written $\tstr(T,f)$) is the stratified $n$-truss whose strata are connected components of the subposets $\tdim\inv(k)$ and the connected components of the subposet $T_n \setminus Q$ of $T_n$ (here, $\tdim$ is the combinatorial transversal dimension, see \cref{notn:transv-dim-comb}).
\end{defn}

\nid Note that, if $(T,g) = \NFTrs f$ for a tame tangle $f$, then $(T,\tstr(g)) = \NFTrs \tstr(g)$.

\begin{thm}[Tangles refine to manifold diagrams] Given a tame tangle $f$, then the transversal stratification $\tstr(f)$ of the tangle is a manifold diagram.
\end{thm}

\begin{proof} Equivalently, we may show that $(T,\tstr(g))$ is a combinatorial manifold diagram for any $m$-tangle $n$-truss $(T,g)$. By definition of combinatorial manifold diagrams we need to verify the combinatorial conicality condition for $\tstr(g)$. Take $x \in T_n$. Since $(T,g)$ is combinatorially transversal, we have $\NF{T^{\leq x}, g^{\leq x}} = \OTT^{k} \times (C,D \into C_{n-k})$. Observe, firstly, that $\NF{T^{\leq x}, \tstr(g)^{\leq x}} = \tstr(\NF{T^{\leq x}, g^{\leq x}})$, secondly, that $\tstr(\OTT^{k} \times (C,D \into C_{n-k})) = \OTT^{k} \times \tstr(C,D \into C_{n-k})$, and, thirdly, that $\tstr(C,D \into C_{n-k})$ is a stratified cone truss (all three observations are left as an exercise). Taken together these observations imply that $(T,\tstr(g))$ is combinatorially conical at $x$, which verifies the claim.
\end{proof}

\begin{obs}[Coarsest refining manifold diagram] Given a tame tangle $f$ then any refinement $g \to f$ by a manifold diagram $g$ factors uniquely through the refinement $\tstr f \to f$ by a refinement $g \to \tstr f$. In this sense, $\tstr f$ is the \emph{coarsest refining} manifold diagram of a tame tangle.
\end{obs}

\nid An immediate upshot of the construction of the transversal stratification is the following.

\begin{obs}[Tangles canonically yield cell diagrams] \label{obs:tang-to-pd} Given a tame tangle $f$ we obtain a cell diagram $\tstr f^\dagger$. Via \cref{rmk:cell-diagram-interpretation}, this allows us to interpret tame tangles as higher categorical diagrams of higher morphisms.
\end{obs}

\subsection{Cell structures of tame tangles}

We now show that tame tangles have canonical regular cell structures. Recall, a poset $P$ is called cellular when for each $x \in P$ the strict upper closure $P^{>x}$ is a sphere (see \cref{recoll:cellular-poset}).

\begin{constr}[Canonical cell structures of tangles] \label{contr:cell-struct-of-tang} Given a tame tangle $f = (W \into \II^n)$, construct its coarsest refining mesh $\iM^f$ and its fundamental tangle truss $(T,g : Q \into T_n) = \NFTrs f$.
Compactify the latter to obtain the closed stratified truss $(\overline T, \overline g : \overline Q \into \overline T_n)$ (see \cref{notn:compactifying-tangle-trusses}). For each $x \in \overline T_n$, the strict upper closure $\overline T_n^{>x}$ is a sphere (for instance, by \cite[Lem.\ 3.2.1]{fct}). It follows that $\overline T_n$ is a cellular poset, and its classifying mesh $\iM^{\overline f} \equiv \CMsh \overline T$ gives a corresponding regular cell complex $\iM^{\overline f}_n$.
Consider the constructible substratification $\cellstr_f \into \iM^{\overline f}_n$ consisting of those cells of $\iM^{\overline f}_n$ which lie in the subspace $W \into \bI^n$. We refer to $\cellstr_f$ as the `canonical cell structure of $W$' (note that $W \ot \cellstr_f \into \iM^{\overline f}_n$ is a cellular subrefinement of $W$ in the sense of \cref{rmk:cellulable-strat}). We also define $\cellstr_{\overline f} := \CStr {\overline Q}$, which cellulates the compactification $\overline W$ of $W$ (see \cref{notn:compactifying-tangle-trusses}).
\end{constr}

\nid We can also construct the following dual cell structure.

\begin{constr}[Canonical dual cell structures of tangles] \label{contr:dual-cell-struct} Given a tame tangle $f = (W \into \II^n)$, construct its coarsest refining mesh $\iM^f$ and its fundamental tangle truss $(T,g : Q \into T_n) = \NFTrs f$ as before. For each $x \in Q$, verify that $Q^{<x}$ is a sphere.\footnote{By definition of tangle trusses, $(T_n^{\leq x},g)$ normalizes to $\OTT^k \times (C,D)$ where $D \setminus \top$ is a sphere. Thus $D$ itself is a disk. Thus, the restriction of the projection $T_n^{\leq x} \to T_k^{\leq x_k}$ to $Q$ gives a poset map $Q^{\leq x} \to T_k^{\leq x_k}$ whose fibers are disks. The `base' $T_k^{\leq x_k}$ of this bundle is a disk (by the dual version of \cite[Lem.\ 3.2.1]{fct}. Thus $Q^{\leq x}$ is a disk itself. It follows that $Q^{\leq x} \setminus x = Q^{<x}$ is a sphere.}
    Thus $Q\op$ is a cellular poset. The classifying stratification $\CStr {Q\op}$ will be written $\cellstr^\dagger_f$ and called the `canonical dual cell structure' of $W$. (Note, if $\overline W$ has boundary then the classifying space $\abs{Q\op}$ need not be homeomorphic to $W$; but there is always an inclusion $\abs{Q\op} \into W$ which is an homotopy equivalence, which can be constructed by composing the identification $\abs{Q\op} \iso \abs{Q}$ with the realization $\abs{Q \into \overline T_n}$, noting that $\abs{\overline T_n}$ is the underlying space of $\CStr {\overline T_n} \iso \iM^{\overline f}_n$).
\end{constr}

\begin{eg}[Cell and dual cell structures of tame tangles] We illustrate cell and dual cell structures of tame tangles in \cref{fig:cell-and-dual-cell-structures-of-tame-tangles}: note that for the second example (the torus) we only give the 1-skeleton together with a couple of 2-cells---the reader will readily be able to fill in the rest of the 2-skeleton.
\begin{figure}[ht]
    \centering
    \def\svgwidth{1\columnwidth}
    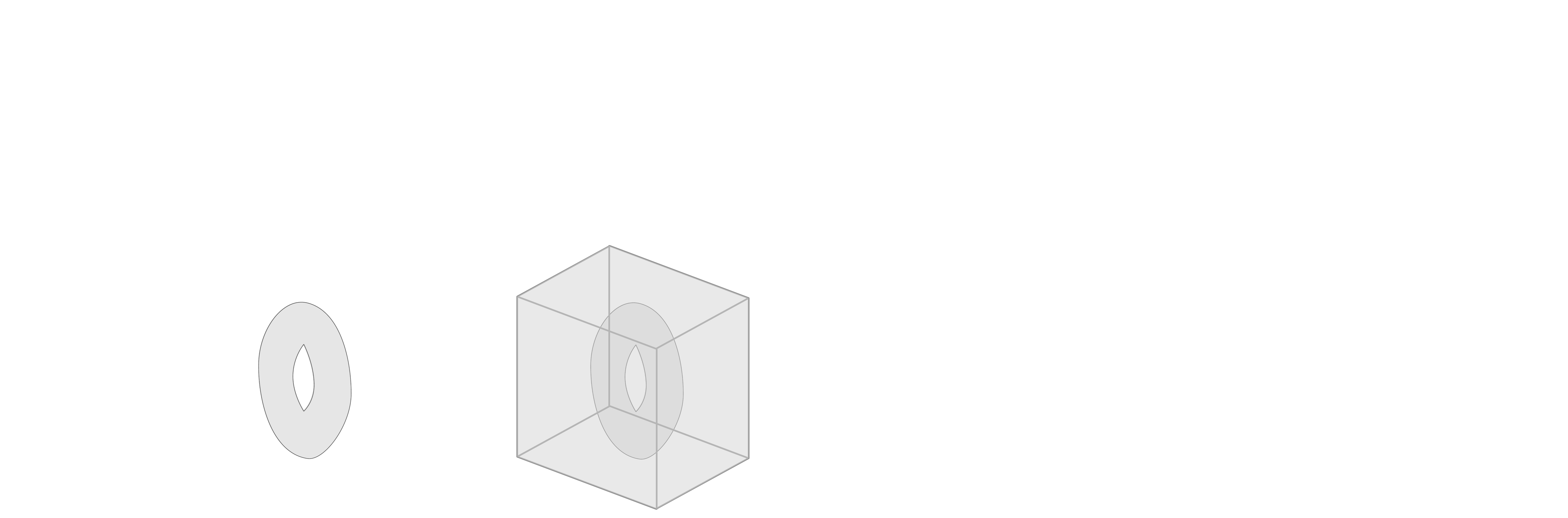

    \caption[Canonical cell and dual cell structures]{Canonical cell and dual cell structures of two tame tangles}
    \label{fig:cell-and-dual-cell-structures-of-tame-tangles}
\end{figure}
\end{eg}

\begin{rmk}[Non-existence of classical dual cell structures] In general, given a cellular poset $P$ whose classifying space is a closed manifold (i.e. compact without boundary) then $P\op$ need not be a cellular poset. In contrast, if $(T, f : Q \into T_n)$ is a combinatorial tangle, with the classifying space of $Q$ being a closed manifold, then the above constructions show that both $Q$ and $Q\op$ are cellular posets. Moreover, in the non-closed case, cellulations dualize up to compactification (that is, $\overline Q$ and $Q\op$ are cellular).
\end{rmk}

\subsection{PL structure and recognizability of tame tangles} \label{ssec:PL-struct} We address three foundational questions about our definition of tame tangles. Firstly, are tame tangle manifolds (canonically) PL manifolds? Secondly, are PL tangles generically tame tangles? And thirdly, given a tame stratification, can it be computably verified whether it is a tame tangle? The answers will be `it depends on the smooth Poincar\'e conjecture in dimension 4', `yes', and `no' respectively.

Note, while tame tangles have been defined in purely topological terms, as in the case of manifold diagrams, they carry canonical framed PL structure (as stratifications, see \cref{term:framed-pl-struct}). Indeed, this follows from \cref{thm:framed-PL-struct-unique}. Restricted to tangle manifolds, this structure gives a PL structure but a priori not necessarily a PL manifold structure (see \cref{term:PL-struct-vs-PL-manifold}). In contrast to \cref{obs:PL-struct-of-mdiag}, which noted that strata in manifold diagrams are canonically PL manifolds, the question of whether the canonical PL structure endows tangle manifolds with the structure of PL manifolds is more subtle.

A standard way of producing triangulated topological manifolds that are not PL manifolds uses the Edward---Cannon double suspension theorem \cite{cannonshrinking}: namely, given a triangulated non-trivial homology sphere $H$ its double suspension $\Sigma^2 H$ is a triangulated topological sphere that is not a PL manifold. One could hope to build a tame tangle by PL embedding $\Sigma^2 H \into \II^n$, which would yield a triangulated tangle manifold that is not a PL manifold. However, no such embedding as the next observation shows.

\begin{term}[Weak PL manifolds] A PL structure is called a `weak PL manifold structure' if all its links are homeomorphic to spheres (but not necessarily PL homeomorphic to the standard PL sphere as required in the definition of PL manifold structures, see \cref{term:PL-struct-vs-PL-manifold}).
\end{term}

\nid While our construction of $\Sigma^2 H$ yields a topological manifold with PL structure, it this fails to be a weak PL manifold, since (any triangulation in) the constructed PL structure contains links homeomorphic to $\Sigma H$ (namely, the links around the suspension points of $\Sigma^2 H$) and these links therefore fail to be homeomorphic to spheres.

\begin{obs}[Tame tangle manifolds are weak PL manifolds] The canonical PL structures of (open or compact) tame $m$-tangles restrict to weak PL $m$-manifold structures on tame tangle manifolds. (This can be shown for instance via the combinatorialization of tame tangles as tangle trusses, and the definition of combinatorial transversality.)
\end{obs}

\begin{conj}[Tame tangle manifolds are PL manifolds] \label{conj:links-are-standard-PL} The canonical PL structures of (open or compact) tame $m$-tangles restrict to PL $m$-manifold structures on tame tangle manifolds.
\end{conj}

\nid The conjecture depends on the following fundamental open problem.

\begin{rmk}[PL Poincar\'e conjecture in dim 4] \label{rmk:SPC4} If $m \neq 4$, all PL $m$-spheres are standard by the PL Poincar\'e conjecture  \cite[\S I.4]{buoncristiano2003fragments}. In contrast, it is an open problem whether non-standard PL 4-spheres exists, which is equivalent to the smooth Poincar\'e conjecture in dimension 4 (SPC4), see \cite{freedman2010man}. \cref{conj:links-are-standard-PL} is equivalent to SPC4: if no non-standard PL spheres exists than an inductive argument can be given to prove the conjecture, and conversely, if SPC4 is false then a counterexample to the conjecture can be given.\footnote{If SPC4 turns out to be false, it may make sense to change the definition of tame tangles to require links to be standard PL spheres, in order to ensure tangle manifolds carry canonical PL manifold structures.}
\end{rmk}

\begin{obs}[Canonical PL and smooth structures of tame $m$-tangles, $m \leq 4$] \label{rmk:tame-4-tang-PL} The conjecture is true up to and including dimension $4$ (a simpler proof can be given in this case, since the Hauptvermutung holds in dimension $k \leq 3$, see \cite{moise2013geometric}, which means that all triangulations of $k$-spheres are standard PL spheres; and thus weak PL manifolds are PL manifolds in dimension $\leq 4$). Moreover, since PL and smooth structures are equivalent up to and including dimension 6 (see \cite{hirsch1974smoothings}), tame $m$-tangles also carry canonical smooth structures on their tangle manifolds for $m \leq 4$.
\end{obs}

We have just seen that potentially not all tame tangle manifolds are PL manifolds; thus, a tame tangle may not necessarily be framed stratified homeomorphic to a PL tangle (see \cref{term:transversal}). Conversely however, the next two observations explain that PL tangles are `generically' tame tangles, and that if PL tangles are tame then their given PL structures coincide with their canonical PL structures as tame tangles. 

\begin{obs}[PL tangles are generically tame tangles] \label{obs:pl-tang-gen-tame} Given a closed PL tangle (i.e.\ a closed PL manifold that PL embeds into the open cube $\II^n$), in order to verify that it is a tame tangle it is sufficient to check that it is tame compact framed transversal (this is analogous to \cref{obs:PL-diagrams-are-mdiag}). Moreover, the transversality condition can be regarded as a `genericity' condition in the following sense: any closed PL tangle given by an embedded simplicial complex $W \into \II^n$ can be made framed transversal by an arbitrarily small perturbation of the vertices of $W$. This observation generalizes to the case of non-closed compact PL tangles $W \into \bI^n$ with appropriate care on boundaries.
\end{obs}

\begin{obs}[Tame PL tangles carry the canonical PL structures] \label{obs:pl-tang-inherit-pl-struct} We refer to a PL tangle that is a tame tangle as a `tame PL tangle'. Note, in a tame PL tangle $(W \into \bI^n)$ the canonical PL structure of the tangle (see \cref{thm:framed-PL-struct-unique}) restricts to the given PL structure of the PL manifold $W$. Moreover, given two tame PL tangles $(W \into \bI^n)$ and $(W' \into \bI^n)$ that are framed stratified homeomorphic, then they are framed stratified PL homeomorphic (this follows from the `flat framed Hauptvermutung', see \cite[Cor.\ 5.0.7]{fct}). Restricting this homeomorphism to tangle manifolds shows that the PL manifolds $W$ and $W'$ (with the PL structure given to them as PL tangles) are PL homeomorphic. In particular, given a tame PL tangle $(W \into \bI^n)$, the given PL structure of $W$ coincides with the canonical PL structure of the tame tangle restricted to $W$.
\end{obs}

\begin{obs}[All PL structures appear as tame PL tangles] \label{obs:realizing-PL-struct} By \cref{obs:pl-tang-gen-tame} and \cref{obs:pl-tang-inherit-pl-struct} it follows that \emph{all} PL $m$-manifolds $W$ can be realized as tame PL tangles $W \into \II^n$.
\end{obs}


Unfortunately, checking the framed transversality of an arbitrary (non-PL) tangle, or more generally of some given tame stratification, may not be easy, as follows.

\begin{rmk}[Unrecognizability of tame tangles] Given a tame stratification, we cannot decide whether it is a tame tangle. Indeed, it is impossible to write an algorithm that, given a simplicial complex $K$, decides whether $\abs{K}$ is a manifold \cite{markov1958insolubility} \cite{volodin1974problem} \cite{weinberger2004computers} \cite{poonen2014undecidable}, and this impossibility remains in place when trying to check the framed transversality condition.
\end{rmk}

\nid This stands in contrast to the case of manifold diagrams (where it is possible to algorithmically check the framed conicality condition) and a priori gives the notion of tame tangle a less `computationally tractable' feel. Several remedies are imaginable---including, for instance, a tractable classification of perturbation stable singularities of tangles (as discussed in the next sections) which would make at least the class of perturbation stable tame tangles computationally tractable.

\section{Tangle stability} \label{sec:fam-and-stability} We now discuss bundles of tangles, and notions of `perturbations' and `paths' of tangles, which in turn lead to definitions of `stable' tangles.

\subsection{Tangle bundles} We give the basic definitions of stratified bundles of tame tangles and tangle trusses. We also discuss a further strengthening of the notion, which describes `fiber bundles' of tame tangles.


\begin{rmk}[Stratified bundles from bundled embeddings] Given a stratification $B$, a space $X$, and a closed subspace $W \into B \times X$, we obtain a stratification $f$ of $X$ whose strata are the intersections of the connected components of $W$ and its complement with the preimages $\pi\inv(s)$ of strata $s$ in $B$ (where $\pi$ is the projection $B \times X \to B$). We will speak of the `stratified bundle $p = (W \into B \times X)$' to refer to the stratified bundle $\pi : (B \times X, f) \to B$.
\end{rmk}

\begin{defn}[Tame tangle bundle] \label{defn:tame-tang-bun} A \textbf{tame $m$-tangle bundle} $p = (W \into B \times \II^n)$ is a tame stratified bundle whose fibers are tame $m$-tangles.
\end{defn}

\nid The definition has the following combinatorial counterpart.

\begin{notn}[Stratified truss bundles from bundled embeddings] Given a $n$-truss bundle $q$ over $P$, and a closed subposet $Q \into \Totz q$, we write $(T,f : Q \into \Totz q)$ for the stratified $n$-truss bundle whose strata are the intersections of the connected components of $Q$ and its complement with the preimages $q_{>0}\inv(b)$ of elements $b \in P$ (where $q_{>0}$ is the composite $q_n \circ q_{n-1} \circ...\circ q_1$).
\end{notn}

\begin{defn}[Stratified bundles of tangle trusses] \label{defn:tang-truss-bun} Given a poset $P$, a \textbf{stratified bundle of $m$-tangle $n$-trusses} $(q,f : Q \into \Totz q)$ over $P$ is a stratified $n$-truss bundle over $P$ whose fibers $(\rest q b,\rest f b)$ over objects $b \in P$ are $m$-tangle $n$-trusses.
\end{defn}

\nid Recall from \cref{thm:tame-bundle-class} the correspondence of tame stratified bundles and normalized stratified truss bundles. Given a stratified bundle of tame $m$-tangles $p$ over $B$, then $\NFTrs p$ is a stratified bundle of $m$-tangle trusses, which we refer to as the `fundamental tangle truss bundle' of $p$. Conversely, given a tangle truss bundle $q$ over $\Entr B$ such that $q \iso \NFTrs p$, we say that $p$ is a `classifying tame $m$-tangle bundle' of $q$. Both the notions of tangle truss bundles and of tame tangle bundles further have `compact' variations to which they relate by `compactification', which is analogous to the unbundled case discussed in \cref{rmk:compact-tame-tangles} and \cref{notn:compactifying-tangle-trusses} (note, this uses cubical compactification for stratified truss \emph{bundles}, see \cref{rmk:compact-truss-bun}).

\begin{eg}[Tame tangle bundles] We illustrate several tame tangle bundles in \cref{fig:tame-tangle-bundles}. The first two examples are bundles of 0-tangles in dim 1 over the stratification $\CStr [1]$; the third example is a bundle of 1-tangles in dim 2 over the same stratification; the fourth example is a bundle of 0-tangle in dim 2 over a cellulated circle; and the last example is a bundle of 1-tangle in dim over a cellulated open interval. Passing to the fundamental stratified truss bundle $\FTrs p$ yields corresponding examples of tangle truss bundles.
\begin{figure}[ht]
    \centering
    \def\svgwidth{1\columnwidth}
    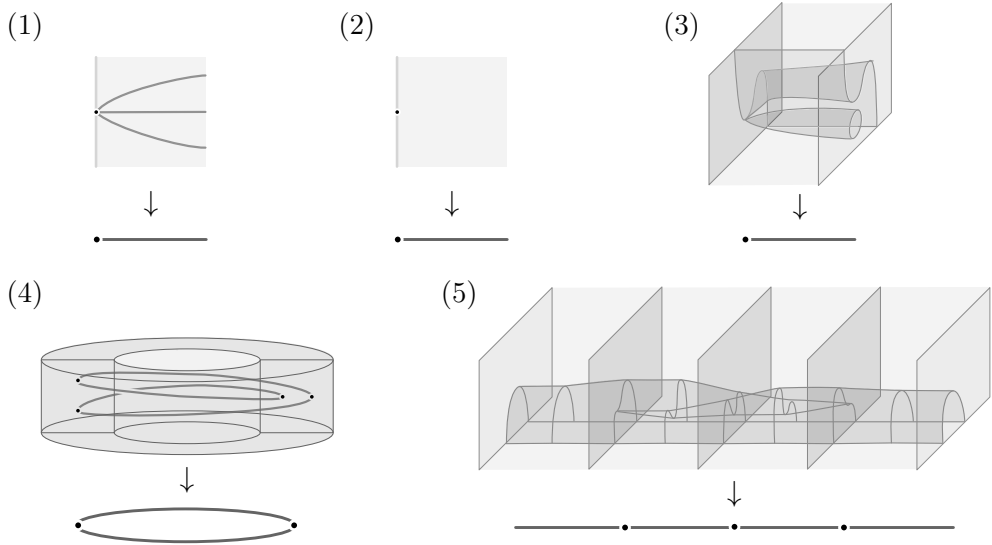

    \caption{Stratified bundles of tame tangle}
    \label{fig:tame-tangle-bundles}
\end{figure}
\end{eg}

\begin{noneg}[Tame tangle bundles] In \cref{fig:tang-bun-non-eg} we give two examples of tame stratified bundles that fail to be tame tangle bundles: the first example fails since the subspace $W \into B \times \II^1$ is not closed; the third example fails since the fiber over the point stratum in the base stratification is not a tame tangle.
\begin{figure}[ht]
    \centering
    \def\svgwidth{1\columnwidth}
    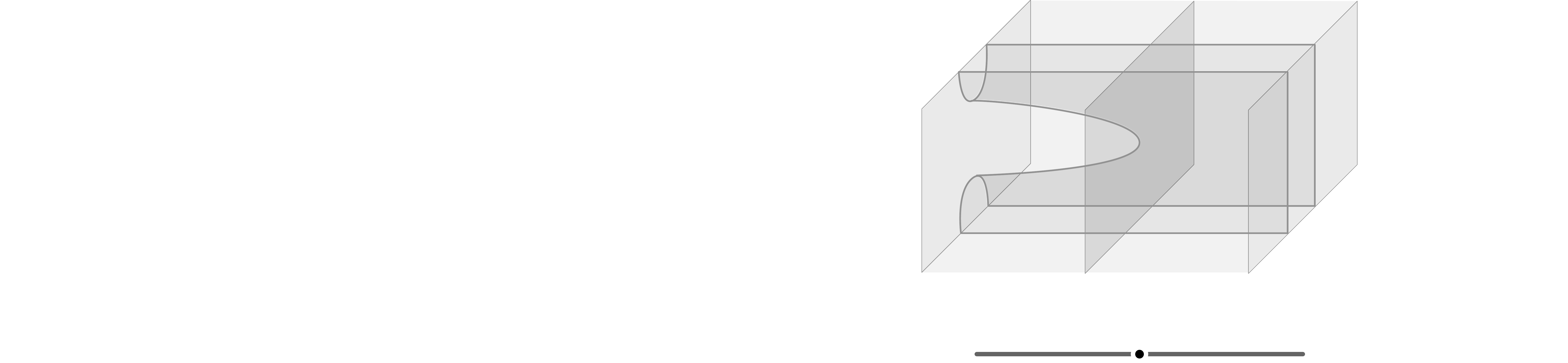

    \caption{Non-examples of stratified bundles of tame tangle}
    \label{fig:tang-bun-non-eg}
\end{figure}
\end{noneg}

Observe that, in \cref{fig:tame-tangle-bundles}, the first three examples of tame tangle bundles $p = (W \into B \times \II^n)$ differ from the last two examples in that the latter bundles restrict to a fiber bundles $W \to B$ while the former bundles do not. The property of being a `tangle fiber bundle' in this sense is captured by the next definition.

\begin{term}[Tangle disks] An $m$-tangle $n$-truss $(T,f : Q \into T_n)$ is called an `$m$-disk' if the classifying space of the compactification $\overline Q$ of $Q$ (see \cref{notn:compactifying-tangle-trusses}) is a closed $m$-disk.
\end{term}

\begin{defn}[Fiber bundles of tangle trusses] \label{defn:tang-truss-bun-fiber} A  \textbf{fiber bundle of $m$-tangle $n$-trusses} $(q,f : Q \into \Totz q)$ over a poset $P$ is a stratified bundle of $m$-tangle $n$-trusses over $P$ that satisfies the following `fiber transition condition': fibers $(\rest q {c \to b}, \rest f {c \to b})$ over arrows $c \to b$ in $P$ satisfy, for $x \in \rest Q b$, that the `generic fiber' $(\rest r c, \rest g c)$ of the lower closure $(r,g) := (\restup q {c \to b} {\leq x},\restup f {c \to b} {\leq x})$ is an $m$-disk.
\end{defn}

\nid In parallel topological terms, we define the following.

\begin{defn}[Fiber bundles of tame tangles] A \textbf{fiber bundle of tame $m$-tangles} $p = (W \into B \times \II^n)$ is a tame tangle bundle whose normalized fundamental tangle truss bundles $\NFTrs p$ is a fiber bundle. 
\end{defn}

\nid The core observation about the previous two definitions is that the fiber transition condition guarantees that tame tangle fiber bundles $(p, W \into B \times \II^n)$ restrict to topological fiber bundles $W \to B$. We omit a proof of this observation.

\begin{term}[Stratified bundles by default] We henceforth often refer to `stratified bundles' of tames tangles (or tangle trusses) simply as `bundles'.
\end{term}

\subsection{Paths of tangles}

We introduce the notion of `paths' of tangles: the notion will be central for defining when singularities of tangles are equivalent. In a brief excursion, we also sketch notions of `higher paths', and how these form a `higher groupoid' of tangles.

In combinatorial terms, paths of tangles are tangle trusses, which are fiber bundles over their 1-level projections. We make this precise as follows.

\begin{notn}[Truncations] Given an $n$-truss $T = \{q_i : T_i \to T_{i-1}\}_{1 \leq i \leq n}$ write $T_{>k}$ for the $(n-k)$-truss bundle $\{q_i : T_i \to T_{i-1}\}_{k < i \leq n}$ over the poset $T_k$, and $T_{\leq k}$ for the $k$-truss $\{q_i : T_i \to T_{i-1}\}_{1 \leq i \leq k}$.
\end{notn}

\begin{term}[Fiber bundle trusses] An $m$-tangle $n$-truss $(T,f : Q \into T_n)$ is said to be an `$(n-k)$-fiber bundle' if the stratified $(n-k)$-truss bundle $(T_{>k},f)$ over the poset $T_k$ is a fiber bundle of $(m-k)$-tangles. In this case, we denote $(T,f)$ by $(T_{>k},f);T_{\leq k}$. 
\end{term}

\begin{defn}[Paths of tangle trusses] A \textbf{path $(S,g)$ of $m$-tangle $n$-trusses} is a 1-fiber bundle $(m+1)$-tangle $(n+1)$-truss, that is, $(S,g)$ of the form $(p,g);U$ where $(p,g)$ is a fiber bundle of $m$-tangle $n$-trusses, and $U$ is a 1-truss.
\end{defn}

\nid In topological terms, one analogously defines `paths of tame tangles' to be a tame tangles whose fundamental tangle truss is a path of tangle trusses.

\begin{eg}[Paths of tame tangles] The last example in \cref{fig:tame-tangle-bundles} is (up to endowing the base poset with a framing, making it a 1-mesh) is an example of a path of tame tangles.
\end{eg}

To describe composition of paths, the following notion of gluings will be needed (note that construction of gluings we give here is more general than what is needed for gluing paths; it will also allow us to glue `higher paths' later on). Recall that we refer to the $(n-k)$th coordinate of $\II^n$ as the `$k$th categorical direction' (see \cref{notn:coord-axis-reverse}). All trusses are assumed to be open.

\begin{constr}[Sides and gluings along categorical directions] \label{constr:gluings} Consider a stratified open $n$-truss $(T,f)$. Denote by $\gamma^{i}_\pm : T_{i} \into T_{i+1}$ the section of the 1-truss bundle $p_{i+1} : T_{i+1} \to T_{i}$ in $T$, that maps $x$ to the upper resp.\ lower endpoint of the fiber $p_i\inv(x)$. For $1 \leq k \leq n$, the \textbf{$k$th categorical upper} resp.\ \textbf{lower side} $(\partial^\pm_k T, \partial^\pm_k f)$ of $T$ is the maximal stratified subtruss of $(T,f)$ such that the image of the inclusion $(\partial^\pm_k T)_{n-k+1} \into T_{n-k+1}$ equals the image of the section $\gamma^{n-k}_\pm$. Given another stratified $n$-truss $(S,g)$ such that $(\partial^+_k T, \partial^+_k f)$ and $(\partial^-_k S, \partial^-_k g)$ both equal the same stratified truss $(W,h)$, their \textbf{$k$-side gluing} $(T \stack k S, f \stack k g)$ (also written as $(T,f) \stack k (S,g)$) is defined as the pushout of $(T,f) \hookleftarrow (W,h) \into (S,h)$ (in the category $\strotruss n$ of stratified open trusses and their stratified cocellular maps).\footnote{On underlying trusses, this colimit can be computed by level-wise pushouts of posets.}
\end{constr}

\begin{eg}[Gluings of stratified trusses] Let $T$ and $S$ be the (normalized) stratified trusses shown in \cref{fig:gluings-of-stratified-trusses}. To their right, we illustrate four possible gluings of $T$ and $S$ along their $2$-sides resp.\ along their $1$-sides (the mutual side along which we glue is encircled in blue). \begin{figure}[ht]
    \centering
    \def\svgwidth{1\columnwidth}
    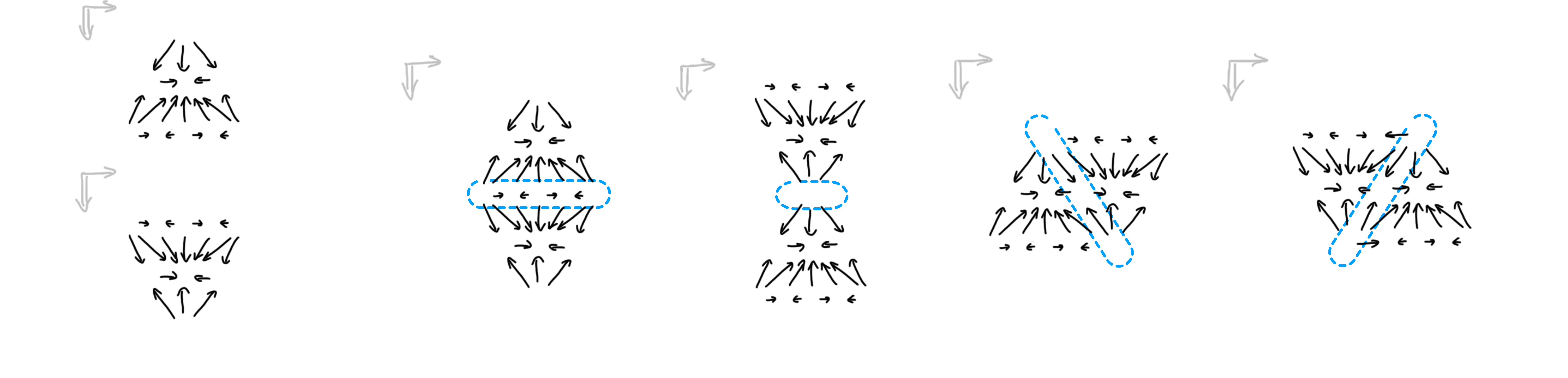

    \caption{Gluings of stratified trusses $T$ and $S$}
    \label{fig:gluings-of-stratified-trusses-comb}
\end{figure}
\end{eg}

\begin{rmk}[Topological gluings] The construction of gluings of course has an immediate topological counterpart which can be obtained by passing back and forth between topology and combinatorics using $\FTrs$ and $\CMsh$ (or else can also be defined in purely topological terms). For example, the topological gluings of tame stratification corresponding to the gluings of stratified trusses in \cref{fig:gluings-of-stratified-trusses-comb} are shown in \cref{fig:gluings-of-stratified-trusses}.
\begin{figure}[ht]
    \centering
    \def\svgwidth{1\columnwidth}
    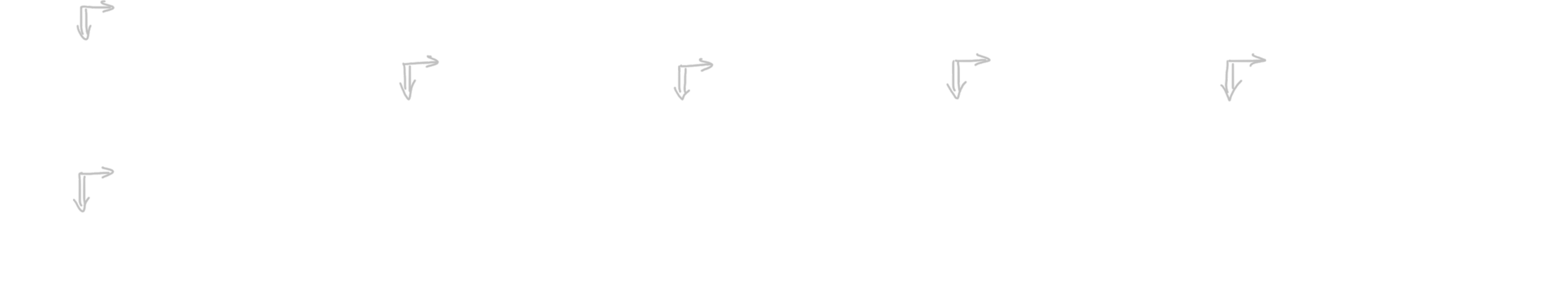

    \caption{Gluings of tame stratifications $f$ and $g$}
    \label{fig:gluings-of-stratified-trusses}
\end{figure}
\end{rmk}

\begin{rmk}[Gluings up to normalization] \label{rmk:gluing-up-to-norm} When working with normalized trusses, sometimes sides do not coincide `on the nose' but only up to stratified coarsening: given normalized stratified $n$-trusses $(T,f)$ and $(S,g)$ we may write $(T,f) \stack k (S,g)$ to mean the gluing $(\tilde T, \tilde f) \stack k (\tilde S, \tilde g)$ if there are canonical (i.e.\ smallest) stratified refinements $(\tilde T, \tilde f)$ resp.\ $(\tilde S, \tilde g)$ of $(T,f)$ resp.\ $(S,g)$ for which the gluing is defined.
\end{rmk}

\begin{constr}[Domains and codomains] \label{term:src-and-tgt} Given a stratified $n$-truss $(T,f)$ where $T = \{q_i : T_i \to T_{i-1}\}_{1 \leq i \leq n}$, consider its $n$th categorical sides $(\partial^\pm_n T, \partial^\pm_n f)$ and write $\partial^\pm_n T = \{\partial^\pm_n q_i : \partial^\pm_n T_i \to \partial^\pm_n T_{i-1}\}_{1 \leq i \leq n}$.
    The \textbf{domain} $\partial_-(T,f)$ (or `source') and \textbf{codomain} $\partial_+(T,f)$ (or `target') of $(T,f)$ are the stratified $(n-1)$-trusses  and with underlying trusses $\{\partial^\pm_n q_i : \partial^\pm_n T_i \to \partial^\pm_n T_{i-1}\}_{2 \leq i \leq n}$ and labeling $\partial^\pm_n f$. (In other words, $\partial_\pm (T,f)$ are obtained from $(\partial^\pm_n T, \partial^\pm_n f)$ by forgetting the first truss bundle $\partial^\pm_n q_1$.)
\end{constr}

\begin{notn}[Function type notation] We denote a stratified $n$-truss $(T,f)$ together with its domain and codomain using `function type' notation, writing $(T,f) : \partial_-(T,f) \to \partial_+(T,f)$.
\end{notn}

\nid In the case of paths, function type notation specializes as follows.

\begin{term}[Start and end points of paths] Given a path $(S,g) : (T,f) \to (T',f')$, we refer to its domain $(T,f)$ also as its `start point' and to its codomain $(T',f')$ as its `end point'.
\end{term}

\nid Our construction of gluings can now be used to construct compositions of paths.

\begin{obs}[Paths compose] Given two paths $(S,g) : (T,f) \to (T',f')$ and $(S',g') : (T',f') \to (T'',f'')$ of $m$-tangle $n$-trusses, we obtain a `composite' path $(S,g) \stack {n+1} (S',g') (T,f) \to (T'',f'')$ by gluing (see \cref{constr:gluings}).
\end{obs}

In fact, we may also describe `higher paths' of tame tangles which organize into the following construction of a `higher groupoid' of tame tangles.

\begin{rmk}[The higher groupoid of tangles] \label{constr:higher-gpd-tangles} A \textbf{$k$-path of $m$-tangles $n$-trusses} $(S,g) : (T,f) \to (Y',e')$ is a $k$-fiber bundle $(m+k)$-tangle $(n+k)$-trusses. Thus $(S,g)$ is of the form $(p,g);U$ for $U$ an $k$-truss. (It may be convenient, but not necessary, to also impose globularity: this requires $(S,g)$ to be constant over the image of any $\gamma^i_\pm : U_{i} \into U_{i+1}$.) Note that for $k = 0$, a $k$-path is simply an $m$-tangle $n$-truss, and for $k = 1$, it is a path of $m$-tangle $n$-trusses. The collection of all $k$-paths of $m$-tangle $n$-trusses will be referred to as the `higher groupoid of $m$-tangle $n$-trusses (if globularity is imposed, this data organizes into a globular set \cite{leinster2004higher}).

    Notions of compositions of $k$-paths can be easily defined by gluings. While this definition of composition is strictly associative and unital, it turns out to not strictly satisfy the interchange law. Thus the above data does not define a strict higher groupoid, but exhibits some weakness. We consciously omit any further discussion of the type of structure it defines,\footnote{The term `free associative higher category' may be applicable; see \cite{thesis}.} as we shall only need the notion of 1-paths.
\end{rmk}

A final useful notion is that of tangle coherences: these are paths that do not contain singularities.

\begin{defn}[Tangle coherences] \label{term:coherence} A \textbf{coherence of tame $m$-tangle} $(T,f)$ is a path of tame $m$-tangles such the $(m+1)$-tangle $(T,f)$ does not contain singular points (see \cref{term:reg-and-sing-points}).
\end{defn}

\nid In fact, any tame $(m+1)$-tangle without singular points is automatically a path and thus a coherence. The standard example of a coherence is the braid, which we saw already in the introduction; more specifically, the braid is a $0$-tangle $1$-coherence. In contrast, the path in \cref{fig:tame-tangle-bundles}(5) is not a coherence, as it contains singular points. More generally, one may define a `$k$-coherence of tame $m$-tangles' to be a $k$-path without points of transversal dimension less than $k$.  For instance, the Reidemeister III move is a $0$-tangle $2$-coherence (see e.g. \cite[Fig.\ I.6]{fct}).

The next remark summarizes the four notions defined in this and the previous section.

\begin{rmk}[Summary of bundle and path notions] Conceptually, notions of bundles and paths may be organized by `increasing specialization' as follows.
    \begin{itemize}
        \item \emph{Bundles} of tame tangles are stratified bundles whose fibers are tame tangles.
        \item \emph{Fiber bundles} of tame tangle are bundles that are topological fiber bundles when restricting to tangle manifolds in the total stratification.
        \item \emph{Paths} are tame tangles whose projections to $\lR$ are fiber bundles.
        \item \emph{Coherences} of tangles are paths without singularities. \qedhere
    \end{itemize}
\end{rmk}

\subsection{Singularities, perturbation, and stability} Finally, we discuss tangle singularities and define notions of `stability' for them. Throughout this section we freely make use of the correspondence between tame tangles and tangle trusses (resp.\ tame tangle bundles and tangle truss bundles).

\begin{defn}[Tangle singularities] An $m$-tangle $n$-truss $(T,f)$ is called an \textbf{$m$-singularity} if $T$ is an open cone truss. (Note this implies $\tdim(\top) = 0$ where $\top$ is the maximal object of $T_n$.)  Similarly, a tame $m$-tangle $(\II^n,f)$ is an `$m$-singularity' when $\NFTrs f$ is an $m$-singularity.
\end{defn}

\nid Alternatively, a tame $m$-tangle is an $m$-singularity when it is the cone of a tame link whose cone point has transversal dimension $0$.

\begin{eg}[Basic tangles singularities] \label{eg:basic-tang-sing} Up to ambient dimension 2 there are finitely many tangle singularities as follows: there is a single 0-tangle singularities in dimension 1 (also denoted by the symbol $\iA_1^{\numovar 0}$, or more simply, by $\pts$); there are two 1-tangle singularities in dimension 2 (which together are denoted by the symbol $\iA_1^{\numovar 1}$, or more simply, by $\iA_1$); finally there is also a single 0-tangle singularity in dimension 2 (also denoted by the symbol $\iA_1^{\numovar 0} \oplus \eps$). These singularities are shown in \cref{fig:basic-tangles-singularities}.
\begin{figure}[ht]
    \centering
    \def\svgwidth{1\columnwidth}
    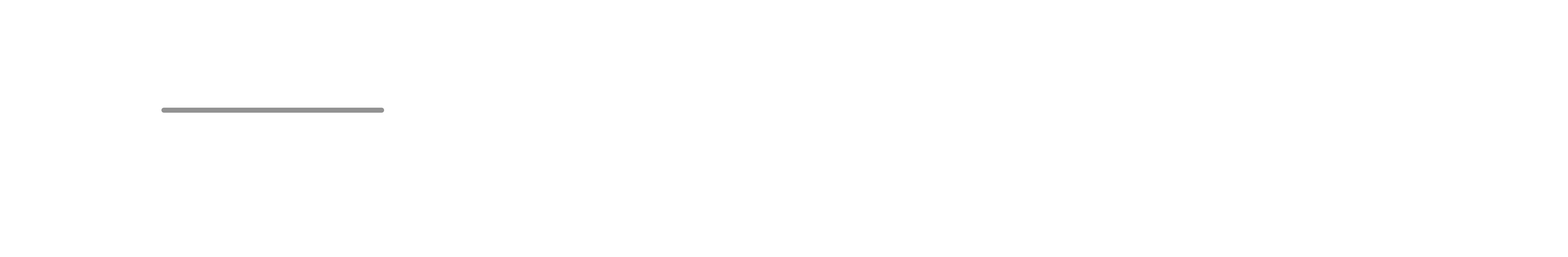

    \caption[Basic tangles and symbols]{Basic tangles singularities up to ambient dimension 2}
    \label{fig:basic-tangles-singularities}
\end{figure}
\end{eg}

\begin{rmk}[Normal singularities] \label{rmk:normal-singularities} Given an $m$-tangle $n$-truss $(T,f)$ and $x \in T_n$, by definition the neighborhood around $x$ normalizes to a stratified truss of the form $\OTT^k \times (C,D \into C_{n-k})$. The $(m-k)$-tangle $(n-k)$-truss $(C,D \into C_{n-k})$ is an $(m-k)$-singularity, called the `normal singularity' at $x$. An illustration is given in \cref{fig:normal-singularity}.
\begin{figure}[ht]
    \centering
    \def\svgwidth{1\columnwidth}
    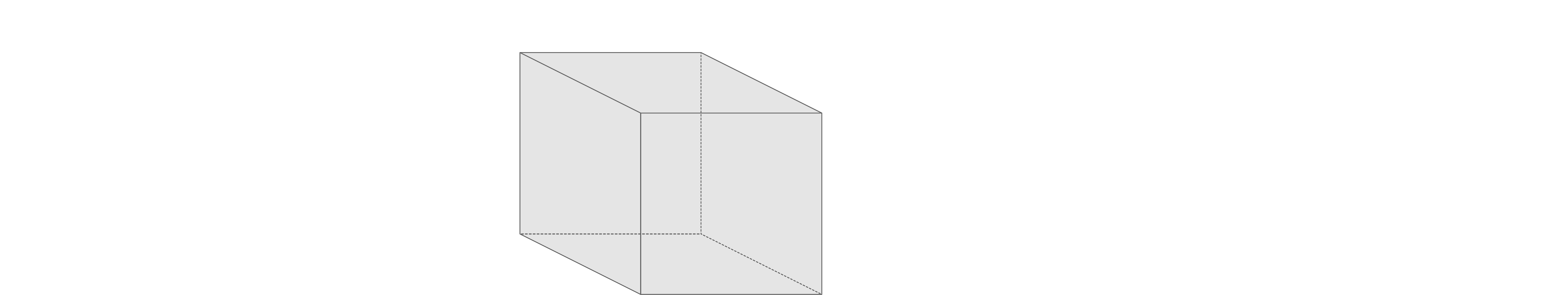

    \caption{A normal singularity at a point in a tangle manifold}
    \label{fig:normal-singularity}
\end{figure}
\end{rmk}

\nid We will study the behavior of singularities under `perturbations', which are a special instance of tangle truss bundles over the 1-simplex, whose generic fiber `surjects' onto its special fiber as follows.

\begin{defn}[Perturbations] A \textbf{perturbation $(q, f : Q \into \Totz q)$ of $m$-tangle $n$-trusses} is a $m$-tangle $n$-bundle over the 1-simplex $[1]\op = (0 \ot 1)$, such for each object $x \in \rest Q 0$ in the tangle manifold of the `special fiber' $(\rest q 0, \rest f 0 : \rest Q 0  \into \Totz {\rest q 0})$ there exists an object $y \in \rest Q 1$ in the tangle manifold of the `generic fiber' $(\rest q 1, \rest f 1 : \rest Q 1  \into \Totz {\rest q 1})$. Similarly, a \textbf{perturbation of tame tangles} is a tame tangle bundle over the stratified 1-simplex $\CStr[1]\op$ whose fundamental tangle truss bundle is a perturbation.
\end{defn}

\begin{notn}[Perturbations] A perturbation $(q,h)$ with special fiber $(T,f)$ and generic fiber $(S,h)$ will be denoted by $(q,h) : (T,f) \perturb (S,h)$
\end{notn}

\begin{eg}[Perturbations] The first and third example in \cref{fig:tame-tangle-bundles} are perturbations of tame tangles, while the second example fails to be a perturbation (as it is not `surjective' on the tangle manifold in the special fiber).
\end{eg}

\begin{rmk}[Fiber bundle perturbations] \label{rmk:fib-bun-pert} A `fiber bundle perturbation' is a perturbation that is a fiber bundle of tame tangles (in fact, in this case the surjectivity condition is vacuous, as it is implied by fiber transition condition in \cref{defn:tang-truss-bun-fiber} of fiber bundles of tangle trusses).
\end{rmk}

\begin{obs}[Perturbations compose] Given perturbations $(p,f) : (T,g) \to (T',g')$ and $(p',f') : (T',g') \to (T'',g'')$ there exists a unique tangle truss bundle $(r,h)$ over $[2]\op$ such that $\rest {(r,h)} {0 \ot 1} = (p,f)$ and $\rest {(r,h)} {1 \ot 2} = (p',f')$.\footnote{The bundle $(r,h)$ over $[2]\op$ can be constructed as follows. Write $f = (Q \into \Totz p)$ and $f' = (Q' \into \Totz {p'})$. We obtain $[1]\op$-labeled bundles $(p,\theta_Q)$ and $(p,\theta_{Q'})$ where $\theta_Q$ resp.\ $\theta_{Q'}$ are indicator maps of $Q$ resp.\ $Q'$ (see \cref{rmk:poset-emb-strat}). Using \cref{obs:labeled-n-truss-bord}, we can compose their classifying morphisms $[1]\op \to \kT^n([1]\op)$ to obtain a functor $[2]\op \to \kT^n([1]\op)$, and thus a $[1]\op$-labeled bundle $(q,\theta_{R})$ over $[2]\op$. From this one obtains the stratified truss bundle $(r,h : \theta_{R}\inv(0) \into \Totz r)$. One then shows that the conditions for $(r,h)$ to be a tangle truss bundle (see \cref{defn:tame-tang-bun}) are satisfied since they are satisfied by $(p,f)$ and $(p',f')$.} The restriction of $(r,h)$ to $(0 \ot 2)$ yields the \emph{composite perturbation} $(p,f) \ast (p',f') :  (T,g) \to (T'',g'')$.
\end{obs}

\begin{eg}[Perturbations and composites of perturbations] Two (composable) perturbations as well as their composite are shown in \cref{fig:two-perturbations-and-their-composite}.
\begin{figure}[ht]
    \centering
    \def\svgwidth{1\columnwidth}
    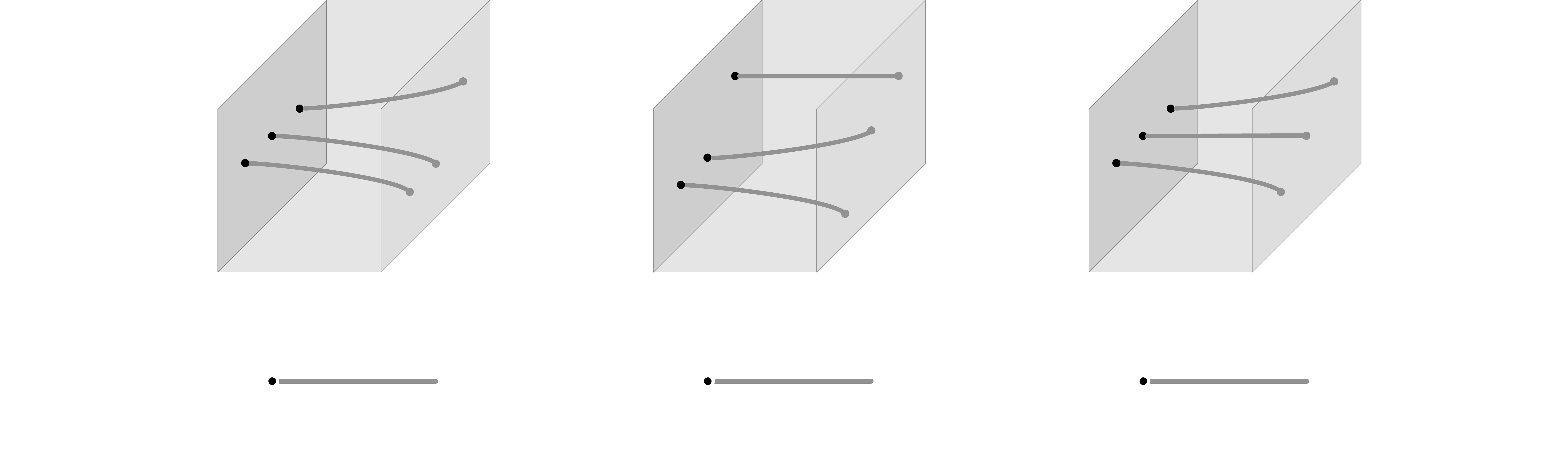

    \caption{Two perturbations and their composite}
    \label{fig:two-perturbations-and-their-composite}
\end{figure}
\end{eg}

We turn to stability. The idea is as follows: a tangle singularity is stable if and only if it cannot be perturbed into a tangle containing `strictly less complex' singularities. The comparison of `complexity' requires elaboration. While in the setting of singularities of topological maps it may not be clear how to obtain an appropriate measure complexity, a key innovation in the framed combinatorial-topological setting is that we can immediately draw from combinatorial representations for this purpose. Maybe the simple measure of complexity of a poset $P$ is its cardinality $\#P$.\footnote{The simplest is not always the best, and it seems premature to not expect revisions this definition of complexity in the future.}

\begin{defn}[Stability] An $m$-singularity $(T,f : Q \into T_n)$ is \textbf{stable} if there exists no perturbation $(T,f) \perturb (S,g : W \into S_n)$ such that for all $x \in W$ we have that $\#Q > \#W^{\leq x}$.
\end{defn}

\nid As an example, note that all singularities in \cref{fig:basic-tangles-singularities} are stable---we will meet further stable singularities in the next sections.

We define the following `path equivalence' on singularities.

\begin{defn}[$\sF$-equivalence] \label{defn:framed-isotopy} Stable singularities $(T,f)$ and $(S,g)$ are said to be \textbf{$\sF$-equivalent} (or `framed isotopic'), written $(T,f) \eqv_{\sF} (S,g)$, if there is a path $(T,f) \to (S,f)$. An equivalence class of $\sF$-equivalent stable singularities is also called an \textbf{$\sF$-orbit}.
\end{defn}

\nid Note that the path need not pass only through stable singularities, or even through singularities; for instance, the last example in \cref{fig:tame-tangle-bundles} is a path between $\iA_1^{\numovarone}$ singularities that passes through non-singularity tangles (note also that the path `loops', i.e.\ it has the same start and end point). A special case of $\sF$-equivalences are perturbations $(T,f) \perturb (S,f)$ between stable singularities (indeed, any such perturbation may be completed to a path over the open 1-truss $(\bullet \to \bullet \ot \bullet)$ by the identity perturbation $(S,f) = (S,f)$).

When working with stable singularities up to $\sF$-equivalence we will be interested in the `maximally perturbed' representatives of $\sF$-orbits.

\begin{defn}[Inductive stability] \label{term:full-extended-stab} A stable singularity $(T,f)$ is called \textbf{inductively stable} if all perturbations $(T,f) \perturb (S,g)$ to another stable singularity $(S,g)$ are of the form $[1]\op \times (T,f)$.
\end{defn}

\nid Roughly speaking, inductive stability asks not only for the singularity to be stable, but also its link to be written in the most stable form. As an example consider the perturbation between stable singularities in \cref{fig:a-perturbation-between-stable-singularities}: this perturbs a singularity of type $\iA_1^{\numovarone} \stack 1 \iA_1^{\numovarone} \to \id$ to a singularity of the form $\iA_1^{\numovarone} \stack 2 \iA_1^{\numovarone} \to \id$ (see \cref{fig:gluings-of-stratified-trusses} and \cref{fig:basic-tangles-singularities})---as we will see shortly, both are stable singularities, but only the latter singularity is inductively stable.

\begin{figure}[ht]
    \centering
    \def\svgwidth{1\columnwidth}
    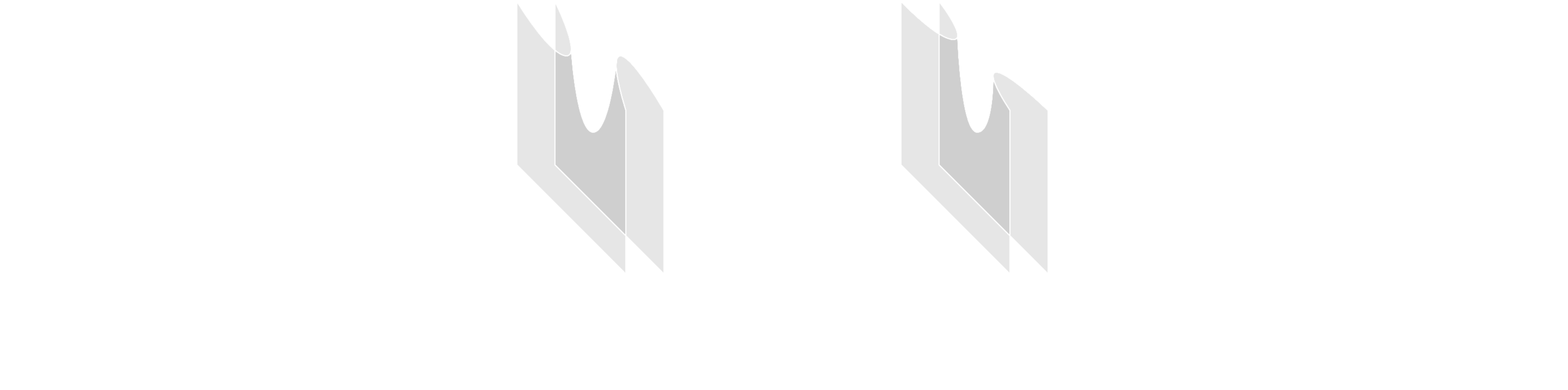

    \caption[A perturbation between stable singularities]{A perturbation from a stable, but not inductively stable singularity, to an inductively stable singularity}
    \label{fig:a-perturbation-between-stable-singularities}
\end{figure}

The example may suggest to the reader that we should always work with inductive stability to avoid non-generic behavior in links. However, stable, but not inductively stable singularities do occur naturally as well, and may provide `condensed' representatives of $\sF$-equivalence classes; we return to this in \cref{eg:D4} in our discussion of the classical $D_4$ singularity.

We end this section with a remark about the notion of coherences, which in some way is complementary to that of singularities.

\begin{rmk}[Stable coherences] \label{rmk:elementary-homotopies} Recall that a coherence is a path without singular points (see \cref{term:coherence}). There's now a parallel story to the above which addresses the question of `stable' (or `elementary') coherences. Namely, a stable coherence is a coherence that cannot be perturbed into strictly less complex coherences. We illustrate this with an example in \cref{fig:perturbing-the-triple-braid}: the triple braid is \emph{unstable} since it can be perturbed into three braids; in contrast, the braid itself is \emph{stable}. 
\begin{figure}[ht]
    \centering
    \def\svgwidth{1\columnwidth}
    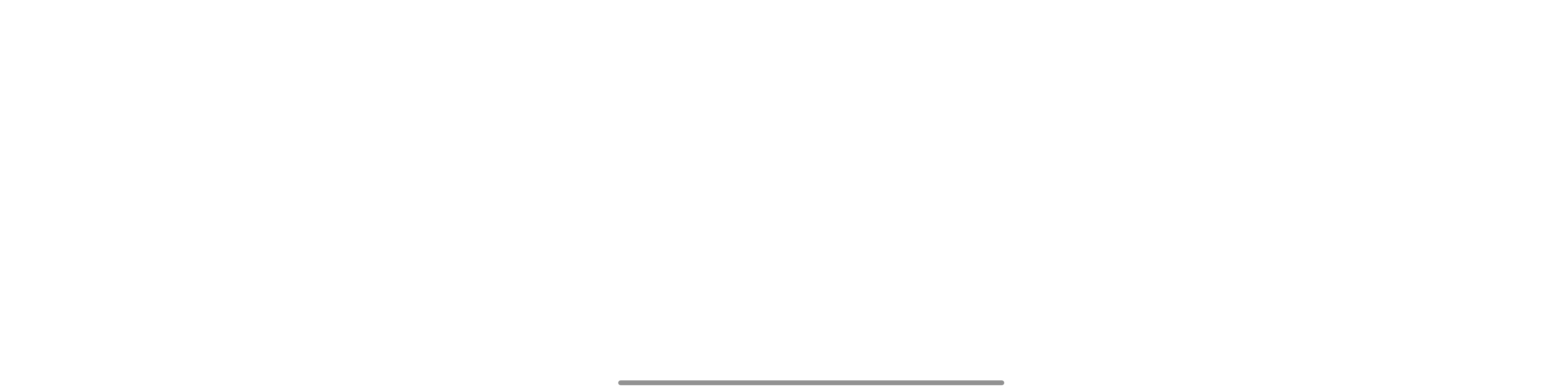

    \caption{Perturbing the triple braid into three (stable) braid coherences}
    \label{fig:perturbing-the-triple-braid}
\end{figure}
\end{rmk}

\section{Stability in low dimensions} \label{sec:stab-in-low-dim}

We discuss stable $m$-tangle singularities $(\II^n,f)$ for $n \leq 4$, giving definitive answers to the classification of such singularities for $m < 3$, and an indication for the case $m = 3$.

\subsection{Stable 2-tangle singularities in 3-space}

As a warm-up, let us determine the stable 2-singularities in codimension $1$ (up to $\sF$-equivalence). We shall attempt to present our discussion in intuitive geometric terms, and will therefore mainly focus on tame tangles rather than tangle trusses---however, there's a parallel, purely combinatorial argument based on the combinatorics of tangle trusses. We will work with all tame tangles up to framed stratified homeomorphisms (which, of course, amounts to working with their corresponding fundamental tangle trusses). Recall that we refer to the $(n-k)$th coordinate of $\II^n$ as its `$k$th categorical direction'.

\begin{rmk}[Reflections] The `reflection in $k$th categorical direction' (or simply, the `$k$-reflection') of the $n$-cube $\II^n = (-1,1)^n$ is the map $r : \II^n \to \II^n$ that takes $(x_1,...,x_{n-k-1},x_{n-k},x_{n-k+1}, ...,x_n)$ to $(x_1,...,x_{n-k-1},-x_{n-k},x_{n-k+1}, ...,x_n)$. Note that if $(\II^n,f)$ is a tame stratification then so is $(\II^n,f \circ r)$. On the combinatorial side, note that $\NFTrs (f \circ r)$ is the stratified truss obtained from $\NFTrs f$ by reversing fibers frame orders in the $k$th 1-truss bundle. Observe that reflections map tame tangles to tame tangles, and manifold diagrams to manifold diagrams.
\end{rmk}

\begin{prop}[Stable 2-tangle singularities in $\II^3$] \label{thm:stable-2-sing} Up to reflection and $\sF$-equivalence, there are three stable 2-singularities in $\II^3$, namely, those depicted in \cref{fig:stable-2-tangle-singularities-in-dimension-3}.
\end{prop}

\begin{rmk}[Naming and counting stable 2-tangle singularities] \label{notn:stable-2-tang-sing} The three singularities (and their respective reflections) in \cref{fig:stable-2-tangle-singularities-in-dimension-3} will be called `extrema', `saddles', and `cusps' respectively. Symbolically, we also refer to extrema and saddles together as the $\iA_1^{\numovar 2}$ singularities, and to the cusps as the $\iA_2$ singularities. Note that applying all available reflections, in total there are \emph{eight} $\sF$-equivalence classes of stable 2-tangle singularities in $\II^3$.
\end{rmk}

\nid Recall the notions of sides, gluings, domains, and codomains from \cref{constr:gluings} and \cref{term:src-and-tgt} (and keep in mind we may tacitly refine trusses before we glue them, see \cref{rmk:gluing-up-to-norm}). Translating between tame tangles and tangles trusses via $\FTrs$ and $\CMsh$ we will freely use these notions in the topological context.

\begin{rmk}[Codimension 1 tangles in $\II^1$ and $\II^2$] A $0$-tangle $(\II^1,f)$ is a disjoint union of $i \in \lN$ points in the interval $\II^1$ and we abbreviate $(\II^1,f)$ simply by $i$. Given $0$-tangles $i$ and $j$ we may compose them by gluing: note that $i \stack 1 j = i + j$ (see \cref{constr:gluings}).  One dimension up, a 1-tangle $(\II^2,f)$ is therefore of the form $f : i \to j$ (see \cref{term:src-and-tgt}). Such a tangle will contain an arrangement of $\iA_1^{\numovarone}$ singularities as illustrated in \cref{fig:source-and-target-of-a-1-tangle-in-dim-2}.
\begin{figure}[ht]
    \centering
    \def\svgwidth{1\columnwidth}
    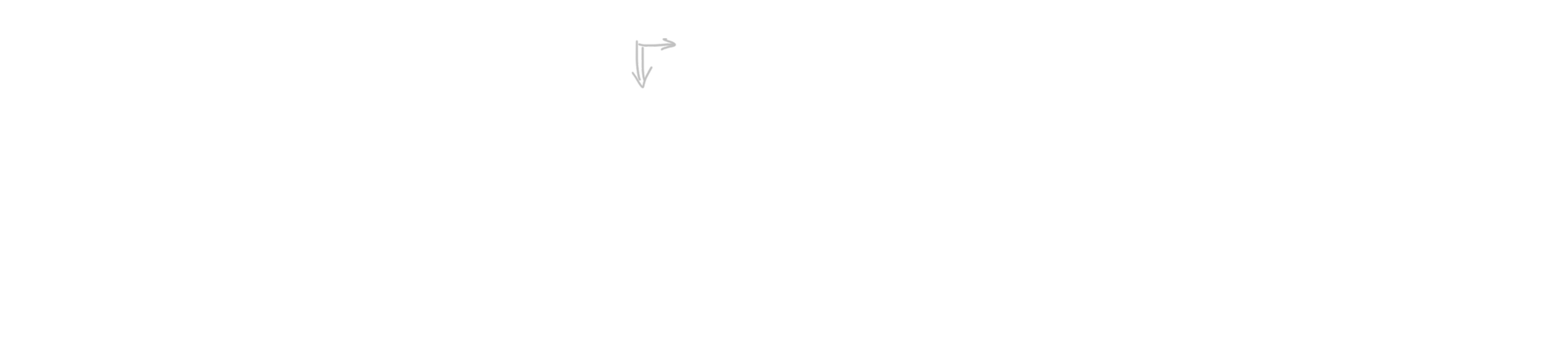

    \caption{A 1-tangle with six source and six target points}
    \label{fig:source-and-target-of-a-1-tangle-in-dim-2}
\end{figure}
\end{rmk}

\begin{notn}[Identities] Denote by $\id_i$ the 1-tangle $\II^1 \times i : i \to i$. This notation similarly applies to general tame tangles $f$, namely, we write $\id_f$ for $\II^1 \times f : f \to f$.
\end{notn}

\begin{proof}[Proof of \cref{thm:stable-2-sing}] We will show, (1), that the three given singularities are stable, and (2), that up to $\sF$-equivalence and reflections they are the only stable singularities. The first statement follows by enumerating \emph{all} 2-singularities $s : f \to g$ which contain at most two $\iA_1^{\numovarone}$ singularities in their link (i.e.\ in their source $f$ and target $g$ taken together). Up to reflections, the enumeration is given in \cref{fig:enumerating-up-to-two-a1}.
\begin{figure}[ht]
    \centering
    \def\svgwidth{1\columnwidth}
    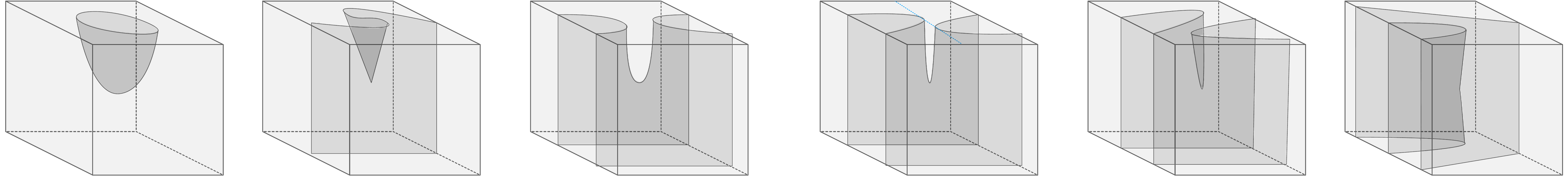

    \caption{Enumerating singularities whose link contains up to two $\iA_1^{\numovarone}$ singularities}
    \label{fig:enumerating-up-to-two-a1}
\end{figure}

\nid The first four singularities in this list are stable: indeed, one checks that none of them can be perturbed into strictly less complex singularities. This in particular includes the three singularities from \cref{fig:stable-2-tangle-singularities-in-dimension-3}, and thus shows statement (1).

To see statement (2), first note that the fourth singularity in \cref{fig:enumerating-up-to-two-a1} is stable but $\sF$-equivalent to the third, while the last two singularities in \cref{fig:enumerating-up-to-two-a1} are unstable: the respective perturbations that witness this are given in \cref{fig:perturbation-of-2-link-sings}. It remains to show that any singularity $s : f \to g$ not in the list of \cref{fig:enumerating-up-to-two-a1} is unstable (in particular, the link of $s$ will contain at least three $\iA_1^{\numovarone}$ singularities). We will do so by showing that $s$ can be perturbed into a tangle containing only $\iA_1^{\numovar 2}$ and $\iA_2$ singularities (all of which are strictly less complex than $s$).
\begin{figure}[ht]
    \centering
    \def\svgwidth{1\columnwidth}
    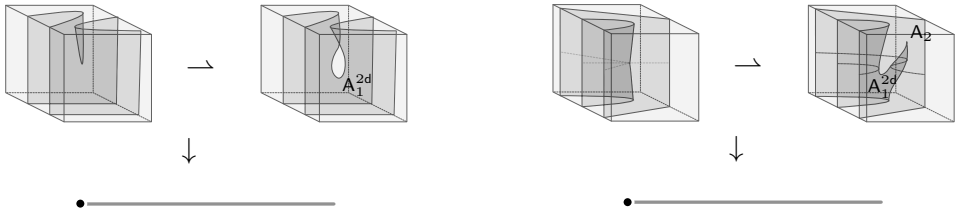

    \caption{The simplest unstable 2-singularities and their perturbations}
    \label{fig:perturbation-of-2-link-sings}
\end{figure}

We organize the argument into two steps. The first step will show that we may assume the target to be trivial. The second step will construct the required perturbation. As a running example we will consider the singularity $s$ shown in \cref{fig:an-unstable-singularity}.

\begin{figure}[ht]
    \centering
    \def\svgwidth{1\columnwidth}
    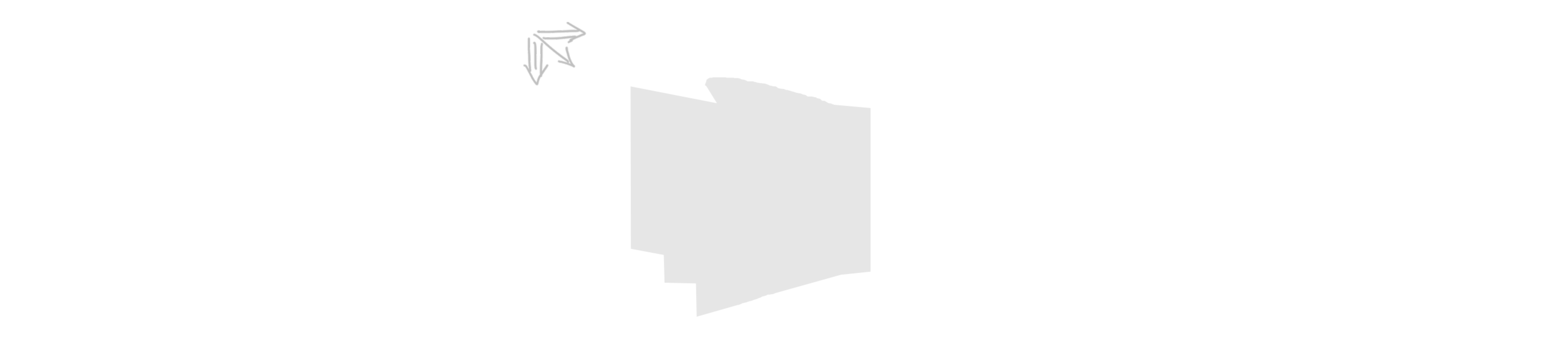

    \caption[An unstable singularity]{An unstable singularity $s : f \to g$}
    \label{fig:an-unstable-singularity}
\end{figure}

\emph{Step 1: Moving $g$ to the source.} Denote by $g\inv$ the 2-reflection of $g$. First, there is a tangle $\mu : \id_j \to g\inv \stack 2 g$ given by applying $\iA_1^{\numovar 2}$ singularities to generate the pairs of $\iA_1^{\numovarone}$ singularities in $g\inv \stack 2 g$; second, there is a 2-tangle singularity $s^{\cup} : h \to \id_i$ whose source $h := f \stack 2 g\inv$ is the 2-gluing of $f$ and $g\inv$, and whose target is the identity on $i$ strands; finally, these two tangles can be combined into a single tangle $\mu \ast s^{\cup} : f \to g$ given by the composite $(\id_f \stack 2 \mu) \stack 3 (s^{\cup} \stack 2 \id_g)$. For our choice of $s$ in \cref{fig:an-unstable-singularity}, this is illustrated in \cref{fig:bending-the-target-into-the-source}.
\begin{figure}[ht]
    \centering
    \def\svgwidth{1\columnwidth}
    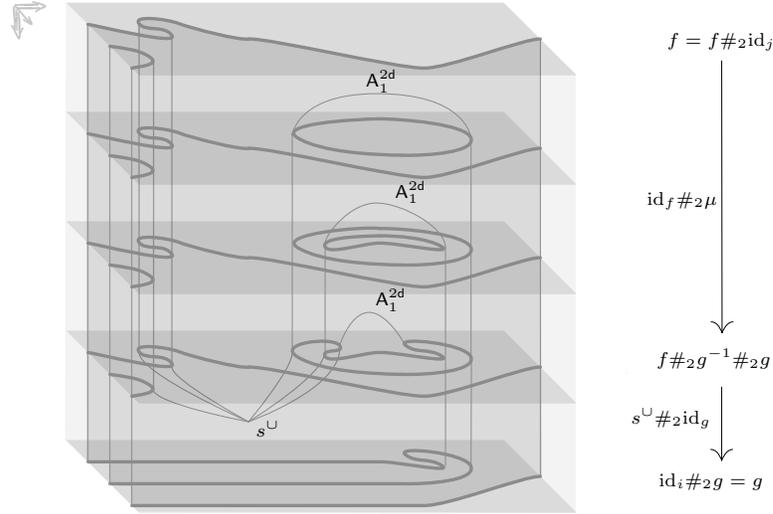

    \caption[Fusing source and target]{A perturbation of $s$ obtained by `bending' the target into the source}
    \label{fig:bending-the-target-into-the-source}
\end{figure}

Observe that there is a (unique) perturbation $s \perturb \mu \ast s^{\cup}$ which restricts to an identity on sources and targets. If we can construct a perturbation of $s^{\cup}$ into a tangle containing only $\iA_1^{\numovar 2}$ and $\iA_2$ singularities (without changing the source $f \stack 2 g\inv \stack 2 g$ or the target $\id$), we can compose the perturbations to obtain the required perturbation of $s$. Replacing $s$ by $s^{\cup}$, we have thus reduced the proof to the case where the target of $s$ is $\id_i$.

\emph{Step 2: Inductively simplifying the source.} To show $s : h \to \id$ is unstable, we first remove all wiggles from $h$ as follows. Since the link of $s$ is a 1-sphere, observe that the source $h$ is a union of line segments in the plane $\II^2$ (i.e.\ $h$ cannot contain 1-spheres itself). Each such line segment  contains a non-negative number of $\iA_1^{\numovarone}$ singularities; if this number is greater than 1, then we can always find a pair of consecutive $\iA_1^{\numovarone}$ singularities to which we can apply an $\iA_2$ singularity (at least after an appropriate isotopy, that ensures these two singularities are `close by one another'). The application of $\iA_2$ singularities to `remove wiggles' in this way is illustrated in \cref{fig:removing-wiggles-from-the-domain}. As shown in the figure, removing wiggles inductively, produces a path $\omega : h \to \tilde h$ such that each line segment in $\tilde h$ contains at most one $\iA_1^{\numovarone}$ singularity.
\begin{figure}[ht]
    \centering
    \def\svgwidth{1\columnwidth}
    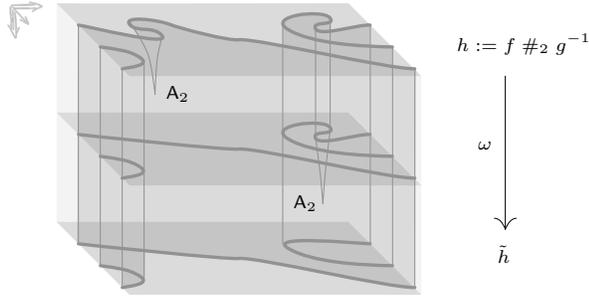

    \caption[Removing wiggles from the source]{Removing wiggles from the source of $s^{\cup}$}
    \label{fig:removing-wiggles-from-the-domain}
\end{figure}

If $\tilde h$ equals $\id_i$ (which always holds if $h$ has a single line segment), then we are done: in this case, observe that there is a (unique) perturbation $s \perturb \omega$ and this perturbs $s$ into a tangle containing $\iA_1^{\numovar 2}$ and $\iA_2$ singularities as claimed. Otherwise, we continue. Note that $\tilde h$ has at least two line segments. Consider the two line segments containing the \emph{first} point in the source resp.\ target of $\tilde h$: the two segments are marked in red resp.\ blue in the 1-tangle shown in the top-left of \cref{fig:inductively-removing-remaining-cups-and-caps} (note the two segments must be distinct since the link of $\tilde s$ is a 1-sphere $S^1$).
\begin{figure}[ht]
    \centering
    \def\svgwidth{1\columnwidth}
    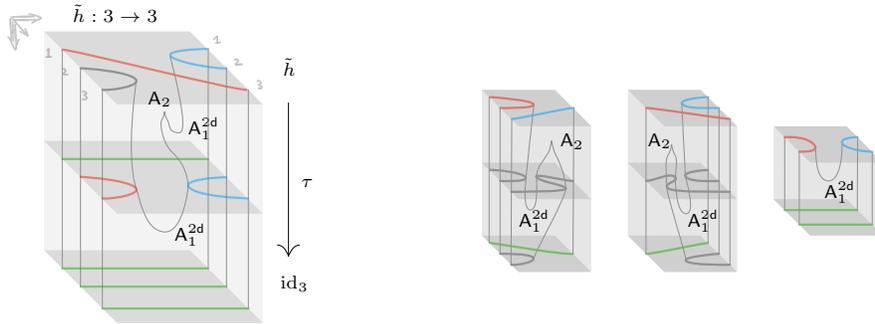

    \caption{Inductively removing remaining cups and caps}
    \label{fig:inductively-removing-remaining-cups-and-caps}
\end{figure}
We can operate on the two segments with one of the three operations shown on the right in \cref{fig:inductively-removing-remaining-cups-and-caps} (containing only $\iA_1^{\numovar 2}$ and $\iA_2$ singularities) to obtain a tangle $\tilde h \to  \id_1 \stack 1 \tilde h'$; in fact, applying such operations inductively (noting that $\tilde h'$ has fewer line segments than $\tilde h$), we obtain a tangle $\tau : \tilde h \to \id_1 \stack 1 ... \stack 1 \id_1 = \id_i$, as illustrated in  \cref{fig:inductively-removing-remaining-cups-and-caps}.
Observe that there is a (unique) perturbation $s \perturb \omega \stack 3 \tau$ which restricts to identities on sources and targets. Since $\omega \stack 3 \tau$ only contains $\iA_1^{\numovar 2}$ and $\iA_2$ singularities this proves the statement.
\end{proof}

\subsection{Stable 1-tangle singularities in 3-space}

Recall, for $m = 1$ and $n = 2$ there are exactly two $m$-tangle singularities $(\II^n,f)$, namely those discussed in \cref{eg:basic-tang-sing}. This situation `stabilizes' to the case $m = 1$, $n = 3$ as the following theorem records.

\begin{prop}[Stable 1-tangle singularities in dimension 3] \label{thm:stable-1-sing-dim-3} Up to $\sF$-equivalence and reflection, there is one stable 1-tangle singularity in $\II^3$: namely, the singularity depicted (together with its reflection) in \cref{fig:stable-1-tangle-singularities-in-dimension-3}.
\end{prop}

\begin{figure}[ht]
    \centering
    \def\svgwidth{1\columnwidth}
    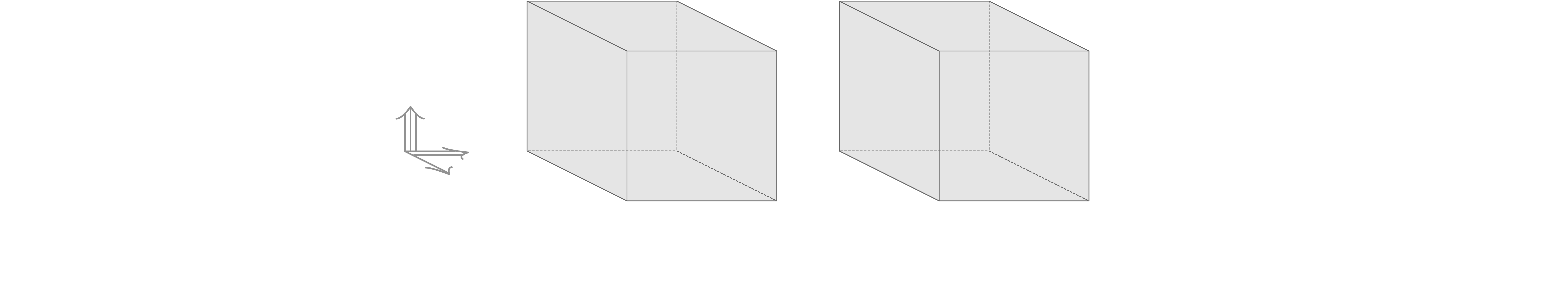

    \caption{Stable 1-tangle singularities in dimension 3}
    \label{fig:stable-1-tangle-singularities-in-dimension-3}
\end{figure}

\begin{proof} Links of 1-tangle singularities are $0$-spheres and thus consist of two points; when working in dimension 3, these two points can either both lie in the source 2-cube or in the target 2-cube. Observe that there are two choices for the relative positioning of these points: indeed, in addition to the singularities in \cref{fig:stable-1-tangle-singularities-in-dimension-3} we also find the singularities shown in \cref{fig:stable-1-tangle-singularities-in-dimension-3-non-generic}. However, the latter singularities are $\sF$-equivalent to the former two singularities which proves the statement.
\begin{figure}[ht]
    \centering
    \def\svgwidth{1\columnwidth}
    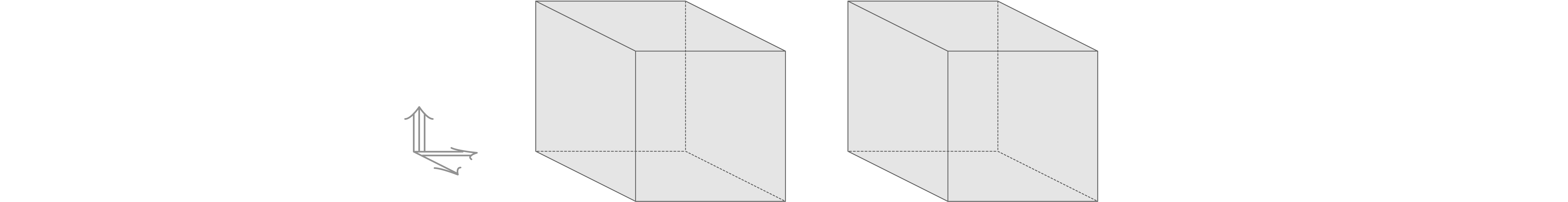

    \caption{Stable, but not inductively stable, 1-tangle singularities in dimension 3}
    \label{fig:stable-1-tangle-singularities-in-dimension-3-non-generic}
\end{figure}
\end{proof}

\nid Note that while the singularities in \cref{fig:stable-1-tangle-singularities-in-dimension-3-non-generic} are stable, they are not inductively stable (as their link may be perturbed to a more generic form).

\subsection{Stable 2-tangle singularities in 4-space}

The case of 2-tangle singularities $(\II^4,f)$ is, similarly to the case of 1-tangle singularities in $\II^3$, a stabilization of the case with one less codimension. That is, we obtain a first set of stable 2-tangle singularities $(\II^4,f)$ by stabilizing the singularities in \cref{fig:stable-2-tangle-singularities-in-dimension-3}: the resulting singularities are shown in \cref{fig:three-stable-2-tangle-singularities-in-dimension-4}.
\begin{figure}[ht]
    \centering
    \def\svgwidth{1\columnwidth}
    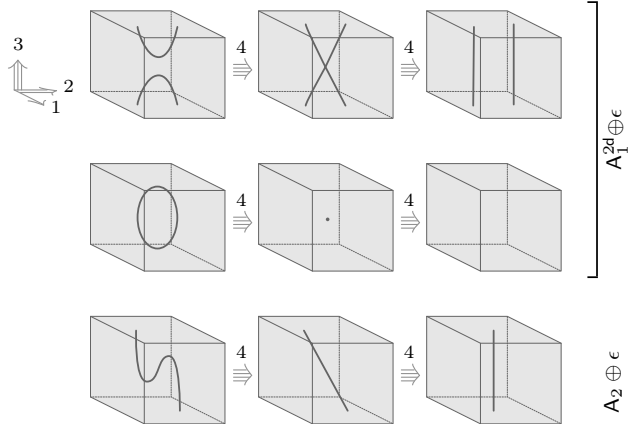

    \caption{Three stable 2-tangle singularities in dimension 4}
    \label{fig:three-stable-2-tangle-singularities-in-dimension-4}
\end{figure}

However, codimension 2 again also brings new phenomena not present in codimension 1. In this case, the ability to braid points in the plane provides the link of a new singularity which trivializes the braid; this singularity is depicted in \cref{fig:the-braid-eating-singularity}. 

\begin{figure}[ht]
    \centering
    \def\svgwidth{1\columnwidth}
    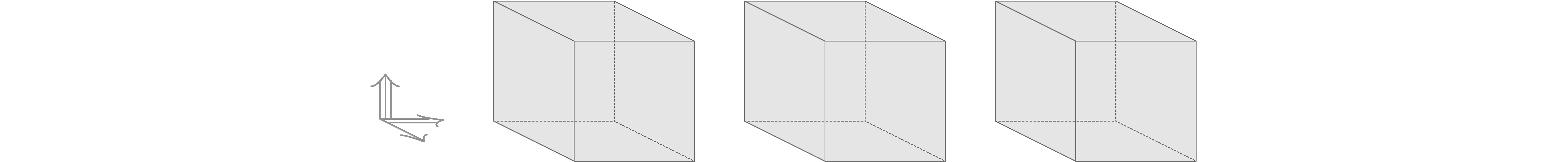

    \caption{The braid trivializing singularity}
    \label{fig:the-braid-eating-singularity}
\end{figure}

\begin{prop}[Stable 2-tangle singularities in dimension 4] \label{thm:stable-2-sing-dim-4} Up to reflection and $\sF$-equivalence, there are four stable 2-singularities in $\II^4$, namely, those depicted in \cref{fig:three-stable-2-tangle-singularities-in-dimension-4} and \cref{fig:the-braid-eating-singularity}.
\end{prop}

\begin{proof}[Proof Outline] We outline the steps of the proof. The argument is analogous to \cref{thm:stable-2-sing}. First, one shows by enumerating all singularities up to the size of the four singularities in \cref{fig:three-stable-2-tangle-singularities-in-dimension-4} and \cref{fig:the-braid-eating-singularity}, that these singularities are indeed stable and cannot be perturbed into simpler singularities. Next, one shows that all other singularities $s : f \to g$ are unstable. We may assume (analogous to the argument in \cref{thm:stable-2-sing}) that $g = \id$. As an example, consider the singularity shown in \cref{fig:an-unstable-2-singularity-in-dim-4}.
\begin{figure}[ht]
    \centering
    \def\svgwidth{1\columnwidth}
    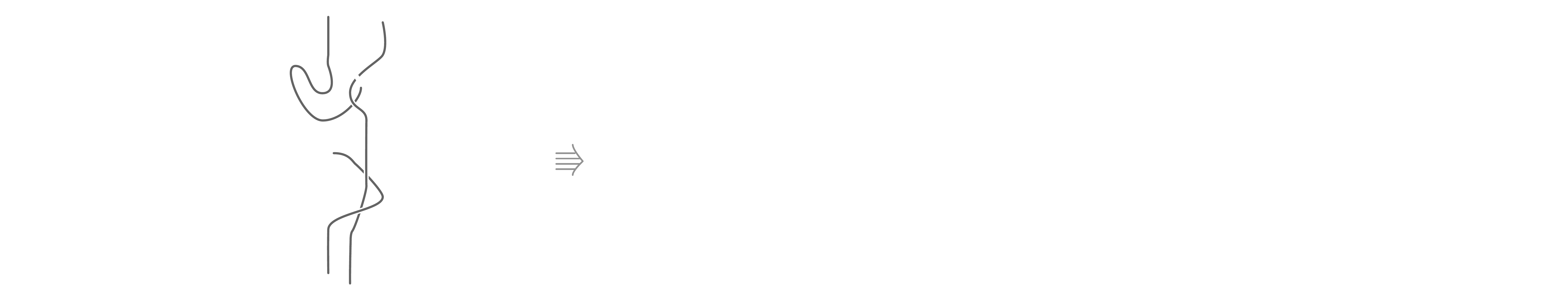

    \caption{An unstable 2-singularity in dimension 4}
    \label{fig:an-unstable-2-singularity-in-dim-4}
\end{figure}

We will produce a perturbation $s \perturb \rho$ such that $\rho$ only uses the four stable singularities determined above. The perturbation $\rho$ of $s$ can be built in three steps (for our example of $s$, the perturbation is shown in \cref{fig:perturbation-of-an-unstable-2-singularity-in-dim-4}). In the first step, we genericize the source $f$ of $s$ by a coherence $f \to \tilde f$ (see \cref{term:coherence}), such that $\tilde f = (W \into \II^3)$ has the following property: the projection $\pi : \II^3 \to \II^2$ restricts to projection $W \to \II^2$ which has empty or singleton preimages except for isolated points $x \in \II^2$, at which there is a framed neighborhood $U_x$ such that $W$ restricts on $\pi\inv(U_x)$ to a braid. (In other words, $\tilde f$ doesn't contain any other `wire crossings' but braids.) In the second step, we then delete all braids using the braid trivializing singularity. In the third and final step, we remove circles using the second singularity in \cref{fig:three-stable-2-tangle-singularities-in-dimension-4}, and then remove wiggles and simplify using the first and third singularity in \cref{fig:three-stable-2-tangle-singularities-in-dimension-4} (this last part is analogous to the argument in \cref{thm:stable-2-sing}). As a result we have shown that $s$ can be rewritten in terms of the four stable singularities, and is thus unstable as required.
\begin{figure}[ht]
    \centering
    \def\svgwidth{1\columnwidth}
    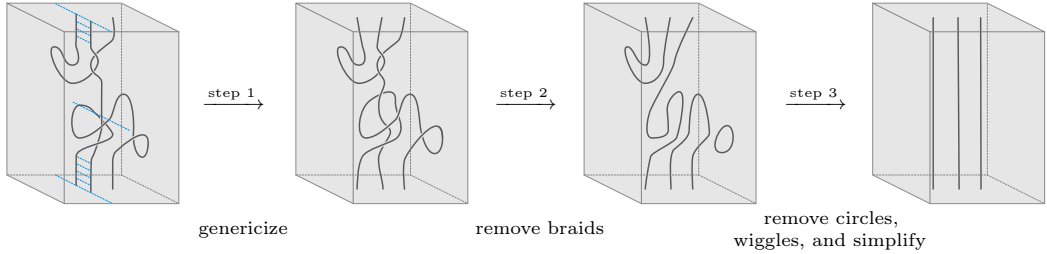

    \caption{3-step perturbation of an unstable 2-singularity in dimension 4}
    \label{fig:perturbation-of-an-unstable-2-singularity-in-dim-4}
\end{figure}
\end{proof}

\begin{rmk}[Fiber-bundle-perturbation stable singularities] A different classification of stable singularities than that of \cref{thm:stable-2-sing-dim-4} would have been obtained if we worked with `fiber bundle perturbations' in place of `perturbations' (see \cref{rmk:fib-bun-pert}): indeed, for any isotopy class of knots in $\II^3$ we can produce an $\sF$-equivalence class of fiber-bundle-perturbation stable singularities by taking the cone of a (sufficiently simple) representative of that knot isotopy class. In particular, there are infinitely many $\sF$-equivalence classes of fiber-bundle-perturbation stable 2-singularities in dimension 4.
\end{rmk}

\begin{rmk}[Stability for 3-tangle singularities in 4-space] \label{rmk:3-tang-in-dim-4} We briefly also remark on the case of 3-tangle singularities $(\II^4,f)$. We claim that there are nine classes of stable singularities in this case (up to reflections and $\sF$-equivalence) which can be reasonably organized into five `types' of singularities---these nine classes (and five types) of singularities are illustrated later in \cref{fig:stable-3-tangle-singularities}.
\end{rmk}

\section{Smooth singularities} \label{sec:classical-sings}

In this section we discuss the heuristic relation of tame tangle singularities and classical smooth singularities. At the end of the section we make several conjectures relating the framed topology of tame tangles to smooth structures on manifolds.

The relation between tame tangles singularities (or rather their corresponding conical manifold diagrams) and classical smooth singularities we are going to put forward isn't conceptually new: it has been known and used as a vague idea for a while \cite{carter1996diagrammatics} \cite{ncc2006}. One of our goals in this section is to make this idea less vague. As a first example of the idea, recall the stable 2-tangle singularities in dimension 3, namely the `extrema' and `saddle' singularities  $\iA_1^{\numovar 2}$ and the `cusp' singularities $\iA_2$. These are usually thought to have the following classical `counterparts':
\begin{enumerate}
    \item The counterpart of $\iA_1^{\numovar 2}$ are the `Morse singularities' $f(x_1,x_2) = \pm x_1^2 \pm x_2^2$,
\item The counterpart of $\iA_2$ are the `Morse-Cerf singularities' $f(x_1,u_1) = \pm x_1^3 \pm u_1 x_1$.
\end{enumerate}
As we will discuss in this section, the underlying principle of this relation is to translate classical smooth singularities into their (parametrized) graphs and consider these graphs as tame tangles. Observations in low dimensions suggest that the translation preserves stability---i.e.\ it translates stable classical smooth singularities to stable tame tangle singularities.

The resulting relation between the two approaches is somewhat mysterious, since their respective mathematical foundations are rather different. We will not give a rigorous explanation of this relation; instead, we would like to suggest up the problem of explaining it as an open challenge. Part of the importance of this question derives from the study of dualizability data in higher category theory as we will explain (see \cref{rmk:dualizability}).

\subsection{Recollections from smooth singularity theory}

We start with a brief recollection of classical singularity theory. Our main reference is the excellent modern treatment in \cite{mond2020singularities}. The subject has its roots in seminal work by Milnor, Thom, Mather, Arnold, Wall, and many others. In the following, all maps are assumed to be smooth unless we indicate otherwise.

The rough set-up of classical singularity theory is as follows. We study the space map germs $f : (\lR^m,0) \to (\lR^p,0)$ under the action of a group $\sG$: for $p = 1$, the common choice is to set $\sG$ to be the group $\sR$ of origin-preserving diffeomorphisms acting on the domain (which is usually called `right equivalence'); and, for $p > 1$, one common sets $\sG = \sA$, where $\sA$ acting by origin-preserving diffeomorphisms on both domain and codomain (which is usually called `left-right equivalence'). For our purposes, we will usually assume $p = 1$ and $\sG = \sR$ (unless otherwise noted) but many of the ideas below hold for more general choices of $\sG$.

Germs in the same $\sG$-orbit are said to be `$\sG$-equivalent'. One can define a notion of tangent space $T\sG f$ and normal space $T^1_\sG f$ at $f$ of the orbit $\sG f$ in the space map germs (see e.g.\ \cite[\S3.2]{mond2020singularities}, where this notational choice is explained as well).
We call $f$ `$\sG$-finite' if the dimension of the normal space $T^1 \sG f$ is finite (this dimension is also called the `$\sG$-codimension'). `$\sG$-determinacy' of a germ $f$ means that there is a $k < \infty$ such that any other germ $g$ with the same $k$-jet as $f$ is $\sG$-equivalent to $f$. Mather showed that $\sG$-finiteness is equivalent to $\sG$-determinacy of $f$ (see \cite[Thm.\ 6.1]{mond2020singularities}). (This is a core result in the classification of singularities since it allows us to only inspect finitely many derivatives to determine $\sG$-finite $\sG$-orbits.)
Given a $\sG$-finite $f$, one moreover says that $f$ is `$\sG$-simple' if in a small neighborhood of $f$ in the space of germs we only find finitely many other $\sG$-orbits. The classification of $\sG$-orbits for general values of $m$ and $p$ was of great interest and promise especially in the 1980s and 1990s---however, activity has dampened over time, and many of the remaining questions appear to be hard. A list of successful classifications is compiled in \cite[\S6.4]{mond2020singularities}.
Most prominently, one has the ADE classification of Arnold, which classifies \emph{all} $\sR$-simple $\sR$-orbits for $m > 0$ and $p = 1$. We list germs representing these $\sR$-orbits below (in each case, the germs we have $i$ `stem' variables $(x_1, ...,x_i)$, as well as a quadratic term $\pm x^2_{i+1} \pm x^2_{i+2} \pm ... \pm x^2_m$ in the other variables which we keep implicit).
\begin{enumerate}
\item $A_k$ germs of the form $x_1^{k+1}$, $k > 1$.
\item $D^\pm_k$ germs, form $x_1^{k-1} \pm x_1 x_2^2$, $k > 3$.
\item $E_6$, $E_7$, and $E_8$ germs, of the form $x_1^3 + x_2^4$, $x_1^3 + x_1 x_2^3$, and $x_1^3 + x_2^5$ respectively.
\end{enumerate}
Observe that the above germs have at most two stem variables, and that for those germs $f$ with two stem variables the 3-jets $j^3 f$ never vanish. Let us briefly sketch why non-simple germs are to be expected in other `higher' cases. Consider the space of $k$-jets of germs $(\lR^2,0) \to (\lR,0)$ in two variables at $0$; the action of $\sR$ on germs descends to an action on this space by $\mathrm{GL(\lR^2)}$.
Now, if $k \leq 3$, then there are finitely many $\sR$-orbits in the space of $k$-jets in two variables, but for $k > 3$, there are `moduli' of such orbits, i.e.\ all of the orbits  have positive codimension (indeed, the fiber $J^4(2,1) \to J^3(2,1)$ has dimension 5, with basis $x_1^4$, $x_1^3 x_2$, $x_1^2 x_2^2$, $x_1 x_2^3$, $x_2^4$, while the group $\mathrm{GL}(\lR^2)$ that is acting on it only has dimension $4$). $\sR$-orbits of $k$-jets in three variables we find moduli already when $k = 3$.
Thus we encounter non-simple $\sR$-orbits of germs when considering germs with vanishing 3-jets in two variables, or vanishing 2-jets in three variables. Conversely, for germs of two variables that are determined by their 3-jet, all $\sR$-orbits are $\sR$-simple and of codimension less than or equal to 5, see \cite[Thm.\ 6.12]{mond2020singularities}. (Thom's original seven `elementary singularities' exactly describe those orbits up to codimension 5, which are $A_2, A_3, A_4, A_5, D^\pm_4,D_5$, and the ADE classification extends this to two infinite families $A_k$, $D_k$, observing that they remain simple even in codimension higher than $5$.) This dashes the hope that finding a finite classification of singularities in codimensions $k$, for $k > 5$. The issue has led mathematicians to consider right equivalence by homeomorphism in place of diffeomorphisms, as discussed thoroughly in \cite{du1995geometry} (however, this leads to other difficulties with the classification problem of singularities; as the authors \emph{loc.cit.}\ note, ``while Arnold finds a zoo of singularities, we find a bestiary'').

In contrast, in the framed combinatorial-topological setup of singularity theory outlined in \cref{sec:fam-and-stability}, the worry of encountering a `moduli of singularities' is absent---since tangle singularities are combinatorially classifiable by tangle truss singularities there are only countably many such singularities (up to framed homeomorphism), and thus there can be, at most, countably many `$\sF$-orbits' (i.e. $\sF$-equivalence classes, see \cref{defn:framed-isotopy}).

\subsection{From smooth to tame singularities} \label{ssec:heur-rel-sing} A priori, it seems rather implausible that there should be any type of faithful relation between $\sF$-orbits of tame tangle singularities (see \cref{defn:framed-isotopy}) and $\sG$-orbits of smooth map germs (as discussed in the previous section). However, we will now illustrate that, while generally tame tangles are `more flexible' objects than smooth map germs, there seems to be a close relation between the two deriving from the translation of $\sG$-simple $\sG$-orbits into $\sF$-orbits given in \cref{constr:translating-smooth-to-tame-sing} below.

\begin{rmk}[Unfoldings] We recall one further notion from classical smooth singularity theory. Let $f : (\lR^m,0) \to (\lR^p,0)$ be a smooth $\sG$-finite germ. An `$l$-unfolding' $F$ of $f$ is a smooth map $F : (\lR^m \times \lR^l, 0) \to (\lR^p \times \lR^l,0)$ of the form $F(x,u) \equiv (f_u(x),u)$ such that $f_0 = f$. A `versal' unfolding $F$ is an unfolding of $f$ from which any other unfolding can be obtained by pullback (`versal' is the intersection of `universal' and `transversal', see \cite[\S5]{mond2020singularities}).
The versal unfolding $F$ is said to be `miniversal' if $l$ is minimal with respect to $F$ being versal. Equivalently, the versal unfolding $F$ is miniversal if its derivatives (in the direction of its $l$ parameters) form a basis for the normal space of the orbit $\sG f$ in the space of germs (see \cite[Thm.\ 5.1]{mond2020singularities}). We may thus think of miniversal unfoldings as `unfolding $f$ in all directions normal to the $\sG$-action'.
\end{rmk}

\begin{heur}[Translation of $\sG$-simple $\sG$-orbits to $\sF$-orbits] \label{constr:translating-smooth-to-tame-sing} Given a $\sG$-simple $\sG$-orbit, we will construct corresponding tame tangle singularity (representing an $\sF$-orbit). Assume the chosen orbit has codimension $l$. Pick a representative $f$ of the orbit, and a miniversal $l$-unfolding $F(x,u) \equiv (f_u(x),u)$ of $f$.
    In fact, for the singularities of interest to us, there are usually polynomial `normal forms' available for both $f$ and $F$.\footnote{Further discussion of `normal forms' can be found in \cite{arnol1972normal} \cite{arnold1976local} \cite{ushiki1984normal} \cite{wall1984notes}---ultimately, there doesn't seem to be a formal definition of what a normal form is. Intuitively, a normal form is a `simplest' representative of an orbit, but there might be several such choices (and if there are, then some might be more suitable than others for the heuristic).}
    Examples of normal forms include the earlier representations of the ADE singularities by simple polynomials. We assume to have chosen such normal forms for both $f$ and $F$ (importantly, this choice will ensure that $F$ looks sufficiently generic for the next step of the heuristic).
    Denote by $\tilde\Gamma_f$ the $u$-parametrized graph of the family of maps $f_u$: this is the subspace of $\lR^l \times \lR^p \times \lR^m$ consisting of points $(u, f_u(x), x)$ where $x \in \lR^m$ and $u \in \lR^l$. Here, the order of components matters: the first $l$ components describe the parameters $u$, the next $n$ components describe the values of $f_u$, and the last $m$ components the space of variables $x$. Set $\lR^n = \lR^l \times \lR^p \times \lR^m$ (in particular, $n := l + p + m$). Let $\II^n$ be the open unit $n$-cube as before.
    For an appropriate choice of \emph{framed} embedding $\phi : \II^n \into \lR^n$, the parametrized graph $\tilde\Gamma_f$ intersects the image of $\phi$ in a tame $m$-tangle singularity, which we denote by $\phi\inv\tilde\Gamma_f \into \II^n$---we illustrate this step in examples below. Keeping $\phi$ implicit, we simply write $\tilde\Gamma_f$ for an $(m+l)$-tangle singularity obtained from $f$ in this way.
\end{heur}

\begin{eg}[Translating extrema, saddles, cusps] \label{eg:translating-extrema} We return to the examples mentioned the beginning of \cref{sec:classical-sings}. Set $f(x_1,x_2) = x_1^2 + x_2^2$, $g(x_1,x_2) = x_1^2 - x_2^2$, and $h(x_1) = x_1^3$. While $f$ and $g$ are of codimension $0$ (i.e.\ their orbits are open in the space of germs), $h$ is of codimension $1$ and thus has a miniversal unfolding with one parameter: we choose this to be the family $h_{u_1} = x_1^3 - u_1 x_1$. The respective (parametrized) graphs $\tilde\Gamma_f$, $\tilde\Gamma_g$, and $\tilde\Gamma_h$ are shown in the top row of \cref{fig:translating-extrema-saddles-cusps}.
    In each case we implicitly pick a framed inclusion $\phi : \II^3 \into \lR^3$ (note that the order of coordinates is different in the last example compared to the first two); we mark the intersection of $\partial \II^3$ and $\tilde \Gamma$ in green. Below each example, we illustrate the resulting tangle singularity $\phi\inv \tilde \Gamma$. Note, that the heuristic outputs stable tangle singularities, and in this sense it preserves stability for the given examples.
\begin{figure}[ht]
    \centering
    \def\svgwidth{1\columnwidth}
    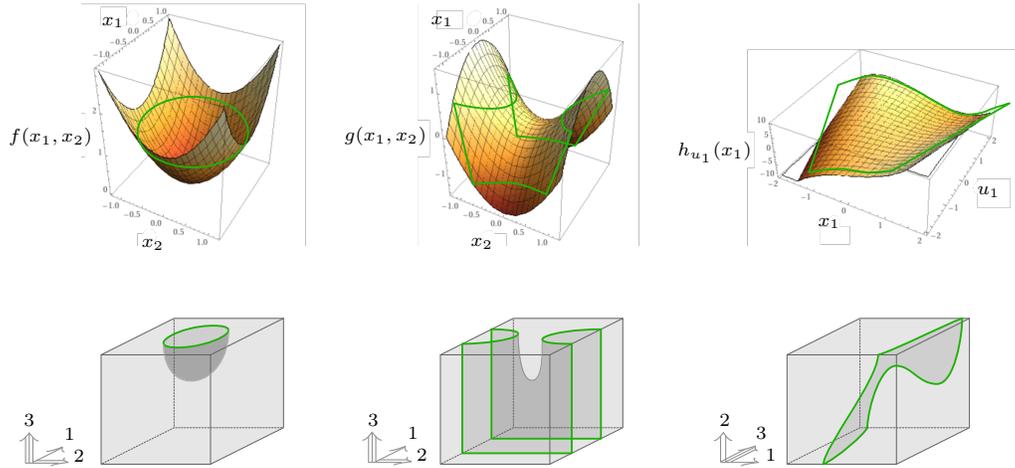

    \caption[Translating classical to tame tangle singularities]{Translating extrema, saddles and cusps into tangle singularities}
    \label{fig:translating-extrema-saddles-cusps}
\end{figure}
\end{eg}

Let us remark that we may use the heuristic to interpret classical smooth singularities as `higher algebraic relations'. Recall, tame $m$-tangles in codimension $p$ translate into cell $(m + p)$-diagrams (see \cref{obs:tang-to-pd}); this translation takes tame tangle \emph{singularities} to cell $n$-diagrams with a \emph{single} $n$-cell, which we can think of as a morphisms relating its source and target. We may thus interpret classical singularities as categorical `relators'.\footnote{Here, a terminology `relator' hints at the categorification of the term `relation', and follows the usual terminological convention of adding the Latin suffix `-or' when turning strict identities into higher categorical morphisms.}

From now on, we focus on the case $p = 1$, that is, on map germs $\lR^m \to \lR$. In terms of tame tangles, this means we will restrict our attention to tame tangle singularities in codimension 1. For this case, we now discuss a heuristic that describes how higher tangle singularities can be constructed by extrapolation from a pattern of `binary relators'.

\begin{term}[Binary relators] \label{term:binary-relator} Given tame tangles $f$ and $g$, a `relator' is a tangle singularity $s : f \to g$. We call $s$ a `binary relator' if $g = \id$ (resp.\ $f = \id$) and $f$ (resp.\ $g$) comprises exactly two tangle singularities. We will also use the terminology in a weaker, more conceptual sense, to mean that $f$ (resp.\ $g$) contains two `distinguished' singularities that govern the overall structure of $f$ (but $f$ may contain further singularities in addition to these two).
\end{term}

\begin{eg}[Binary relators] \label{eg:binary-relator} Recall the singularities $\iA_1^{\numovarone}, \iA_1^{\numovar 2}$ and $\iA_2$ from \cref{eg:basic-tang-sing} and \cref{notn:stable-2-tang-sing}: these stable $m$-tangle singularities are binary relators of stable $(m-1)$-tangle singularities, as follows.
\begin{enumerate}
\item[$-$] $\iA_1^{\numovarone}$ singularities relate composites of \emph{two} point singularities to identities.
\item[$-$] $\iA_1^{\numovar 2}$ singularities relate composites of \emph{two} $\iA_1^{\numovarone}$ singularities to identities.
\item[$-$] $\iA_2$ singularities relate composites of \emph{two} $\iA_1^{\numovarone}$ singularities to identities, but note that the way these two $\iA_1^{\numovarone}$ singularities are composed differs from the $\iA_1^{\numovar 2}$ case; we will refer to this composition as an `$\iA_2$-shape' composition. \qedhere
\end{enumerate}
\end{eg}


\begin{heur}[Tangle singularities from binary algebraic relators] \label{rmk:stable-tang-sing-from-bin-laws} We want to generalize the `binary relator' structure observed in the previous example to higher singularities. For this, we schematize the notation for singularities by the following heuristic. We write $\singa \iX {\vec k} \iY$ to mean the `kind' of $m$-tangle singularities that relate binary composites of two $\iY$ singularities to identities, where composites are in the `shape of' the $i$-singularity $\iX$. The subscript $\vec k = (k_i,k_{i-1},...k_1)$ is a vector of descending indices of so-called `primary directions': roughly speaking, by identifying the categorical directions $(n+1,k_i,k_{i-1},...k_1)$ of $\singa \iX {\vec k} \iY$ with the categorical directions $(i+1,i,i-1,...,1)$ of $\iX$ will enable us to `see' that $\singa \iX k \iY$ has the shape of $\iX$. Using this notation, we may for instance rewrite:
\begin{enumerate}
    \item[$-$] $\iA_1^{\numovarone} \equiv (\singa {\iA_1} 1 {\pts})$.
    \item[$-$] $\iA_1^{\numovar 2} \equiv (\singa {\iA_1} 2 {\iA_1^{\numovarone}})$.
    \item[$-$] $\iA_2 \equiv (\singa {\iA_2} {(2,1)} {\iA_1^{\numovarone}})$.
\end{enumerate}
To further illustrate the notation, let us proceed to the next dimension: we construct 3-tangle singularities as binary relators for 2-tangle singularities. We find five kinds of singularities (represented by a total of nine $\sF$-equivalence classes of singularities up to reflections), which we illustrate in \cref{fig:stable-3-tangle-singularities}, and which we further describe below.
\begin{figure}[ht]
    \centering
    \def\svgwidth{1\columnwidth}
    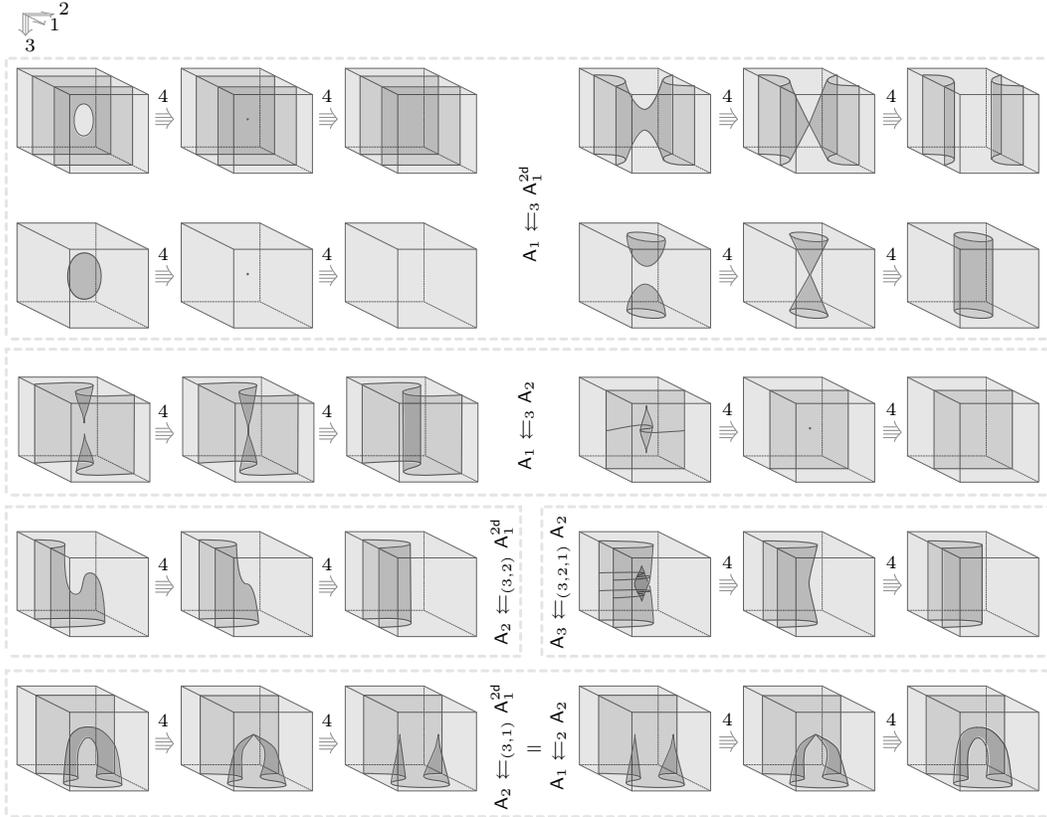

    \caption{Nine classes of 3-tangle singularities in dim 4, organized into five kinds}
    \label{fig:stable-3-tangle-singularities}
\end{figure}
\begin{enumerate}
    \item[$-$] $\singa {\iA_1} 3 {\iA_1^{\numovar 2}}$ (also written $\iA_1^{\numovar 3}$). This is a relator of two $\iA_1^{\numovar 2}$ singularities, one the 3-reflection of the other. Its shape is $\iA_1$: to see this more explicitly, intersect the source with a line in 3-direction, which recovers the source of $\iA_1$, see \cref{fig:seeing-the-a2-shaped-in-3-tangle-singularities} for an illustration.

    \item[$-$] $\singa {\iA_2} {(3,2)} {\iA_1^{\numovar 2}}$ (also written $\iA_2^{\numovar 2}$).\footnote{Here and elsewhere, the superscript `$\numovar k$' is meant to indicate that the singularity comes from a classical germ with $k$ variables; for example, we have $\iA_1^{\numovar k} = \tilde\Gamma_f$ with $f = x_1^2 + x_2^2 + ... + x_k^2$, and similarly $\iA_2^{\numovar k} = \tilde\Gamma_g$ with $g$ being the 1-parameter family $x_1^3 - u_1 x_1 + x_2^2 + ... + x_k^2$.} This a relator of two $\iA_1^{\numovar 2}$ singularities composed in shape of $\iA_2$ in primary directions $(3,2)$. To see the $\iA_2$-shape, intersect the source with the $(3,2)$-plane, to recover the source of the usual $\iA_2$ singularity; see \cref{fig:seeing-the-a2-shaped-in-3-tangle-singularities} for an illustration.

    \item[$-$] $\singa {\iA_2} {(3,1)} {\iA_1^{\numovar 2}}$. This is a relator of two $\iA_1^{\numovar 2}$ singularities composed in the shape of $\iA_2$ in primary direction $(3,1)$. To see the $\iA_2$-shape, intersect the source with the $(3,1)$-plane, to recover the source of the usual $\iA_2$ singularity; see \cref{fig:seeing-the-a2-shaped-in-3-tangle-singularities} for an illustration.

    \item[$-$] $\singa {\iA_1} 3 {\iA_2}$. This is a relator of two $\iA_2$ singularities, one the 3-reflection of the other. The underlying $\iA_1$ shape is exhibited in \cref{fig:seeing-the-a2-shaped-in-3-tangle-singularities}.

    \item[$-$] $\singa {\iA_3} {(3,2,1)} {\iA_2}$ (also written simply $\iA_3$. This relator takes two $\iA_2$ singularities, composed by in an $\iA_3$-shape, and relates them to an identity. Note this is a new shape appearing in this dimension, which (in this and higher dimensions) only acts on other $\iA_2$-shaped singularities.

    \item[$-$] $\singa {\iA_1} 2 {\iA_2}$. This is a relator of two $\iA_2$ singularities. The underlying $\iA_1$ shape is illustrated again in \cref{fig:seeing-the-a2-shaped-in-3-tangle-singularities}.
\end{enumerate}

\begin{figure}[ht]
    \centering
    \def\svgwidth{1\columnwidth}
    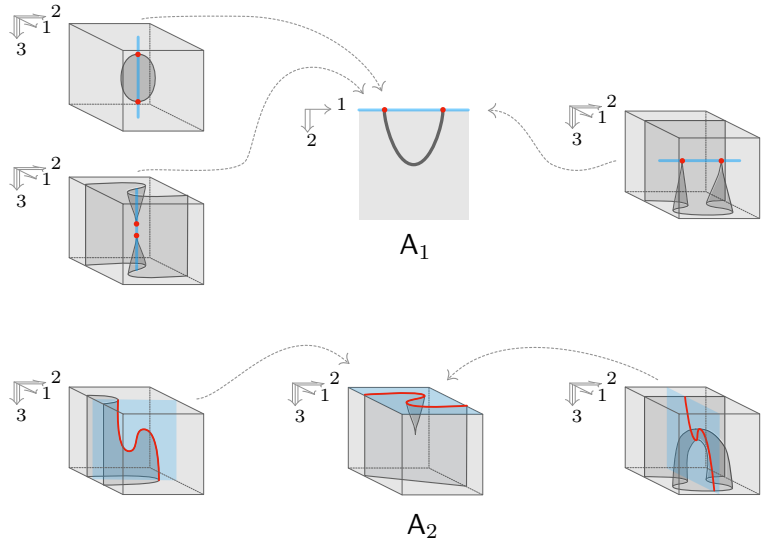

    \caption{Seeing $\iA_1$ and $\iA_2$ shapes in the sources of higher singularities}
    \label{fig:seeing-the-a2-shaped-in-3-tangle-singularities}
\end{figure}

\nid Several remarks are in order. Firstly, it is worth emphasizing again that the symbols introduced here denote \emph{kinds} of singularities, consisting of multiple $\sF$-equivalence classes of as given (up to reflections) in \cref{fig:stable-3-tangle-singularities}. Secondly, not all possible `combinations of symbols' (including choices of primary directions $\vec k$) yield valid singularities. One can describe reasonable conditions for composite symbols to only describe actual singularities (but this goes beyond the intended scope here). However, thirdly, even with such rules in place, there are non-trivial interactions between relators, as the next remark records.
\end{heur}

\begin{obs}[Relators interact] \cref{fig:stable-3-tangle-singularities} illustrates the following important point: the classes $\singa {\iA_2} {(3,1)} {\iA_1^{\numovar 2}}$ and $\singa {\iA_1} 2 {\iA_2}$ in fact coincide: indeed, up to a 4-reflection, their given respective representatives are identical (and both contain both a pair of $\iA_1^{\numovar 2}$ singularities and a pair of $\iA_2$ singularities, making them a `double binary relator' rather than a true binary relator in the sense of \cref{term:binary-relator}). Going into yet higher dimensions one finds a plethora of such interactions between relators: for instance, in \cref{eg:D3} we will meet the `triple binary relator' $\iD_3$.
\end{obs}

\begin{term}[The $\iD_2$ singularity class] Based on the previous observation, we jointly denote the classes of singularities $(\singa {\iA_2} {(3,1)} {\iA_1^{\numovar 2}})$ and $(\singa {\iA_1} 3 {\iA_2})$ by $\iD_2$. Instead of the representatives given in \cref{fig:stable-3-tangle-singularities}, we usually represent singularities in $\iD_2$ by moving all nontrivial content into the source (resp.\ target), making the target (resp.\ source) trivial: this is illustrated in \cref{fig:the-u-singularity}. (In fact, strictly speaking, only the representation in \cref{fig:the-u-singularity} is stable in the sense of our definition of stability; however, both representations can be expressed in terms of one another and the other classes of singularities---the idea for this translation is a higher dimensional version of step 1 in the proof of \cref{thm:stable-2-sing}.)
\begin{figure}[ht]
    \centering
    \def\svgwidth{1\columnwidth}
    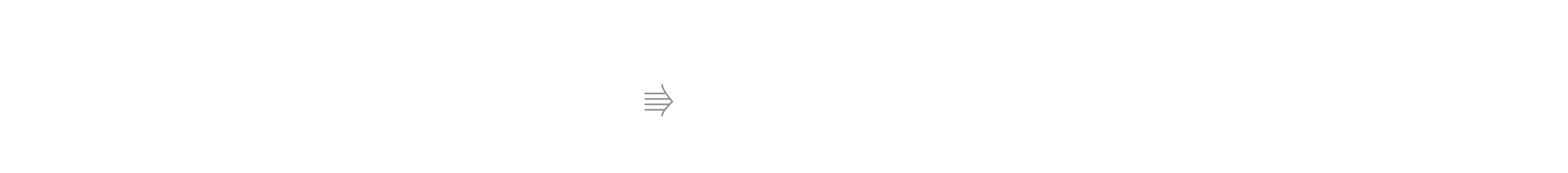

    \caption{The $\iD_2$ singularity class}
    \label{fig:the-u-singularity}
\end{figure}
\end{term}

The next remark serves to record in what way the set-up of tangles differs from that of smooth map germs.

\begin{rmk}[Non-classical tangle singularities] The singularity orbit $\iD_2$ has no classical smooth counterpart in the sense of \cref{constr:translating-smooth-to-tame-sing}: its source mixes 2-variable Morse singularities with 1-parameter 1-variable Morse-Cerf singularities and thus cannot be a graph of the form $\tilde\Gamma_f$. This provides a sense of how tame tangle singularities are `more flexible' than classical smooth maps, and the resulting theory of such singularities is finer.
\end{rmk}

\subsection{The elliptic umbilic singularity as a tame tangle} Having seen how $\iA_k$ singularities are binary relators of $\iA_{k-1}$ singularities it is now natural to ask about the second infinite series of singularities: do the ideas of \cref{constr:translating-smooth-to-tame-sing} and \cref{rmk:stable-tang-sing-from-bin-laws} still apply for $D_k$ singularities, $k \geq 4$? In the lowest dimension, namely, for the $D_4$ singularity (which in Thom's classification is also known as the `elliptic umbilic singularity'), the question is answered positively in this section. The insight into the nature of the $D_4$ singularity may be summarized as follows.
\begin{center}\itshape
    While the tame tangle singularity $\iD_4$ corresponding to the classical $D_4$ singularity is stable, it is not inductively stable. It admits a perturbation to an inductively stable singularity, which is a binary relator of $\iD_3$ singularities (see \cref{term:binary-relator}).
\end{center}
In order to visualize the $\iD_3$ and $\iD_4$ singularities, the following notation will be helpful.

\begin{notn}[Projecting 2-tangles] \label{notn:3-tangle-sings-as-proj} Going forward, we often depict 2-tangles in $\II^3$ by the projections of their `non-regular points' (i.e. points of transversal dimension $<2$) along  $\II^3 \to \II^2$. This simplifies the depiction of higher tangles as well: for instance, consider the first six 3-tangle examples from \cref{fig:stable-3-tangle-singularities} and the 3-tangle \cref{fig:the-u-singularity}. In \cref{fig:denoting-3-tangle-singularities-by-their-projections} below (ignoring colors for now), we re-illustrate these seven 3-tangles by only depicting projections of their non-regular points to $\II^2$. We usually also add, in color, the deformation that happens in between any of the sampled slices. Note that the notation is ambiguous: for instance, the first two examples now notationally coincide while they were originally different (see \cref{fig:stable-3-tangle-singularities}). To remedy this ambiguity we usually provide  an `initial condition' given by at least one full picture of a 2-tangle in the 3-cube $\II^3$, which then determines the rest of the higher tangle.
\begin{figure}[ht]
    \centering
    \def\svgwidth{1\columnwidth}
    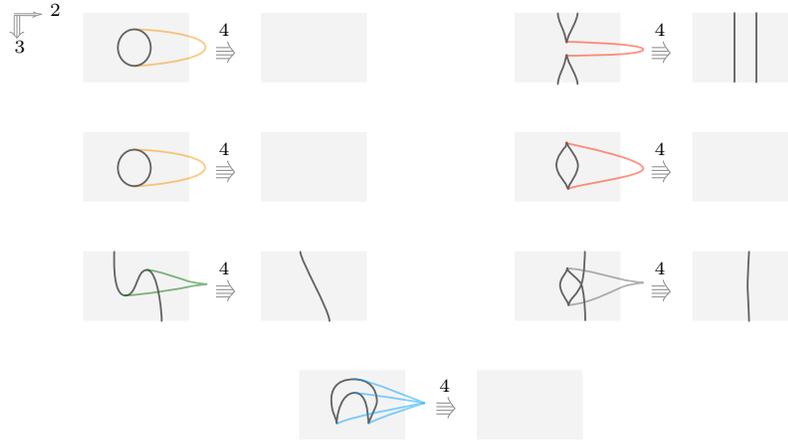

    \caption{Denoting 3-tangle singularities by their projections}
    \label{fig:denoting-3-tangle-singularities-by-their-projections}
\end{figure}
\end{notn}

\begin{eg}[The $\iD_3$ relator] \label{eg:D3} As a warm-up, let us illustrate the $\iD_3$ relator. In \cref{fig:the-source-of-the-d3-law} we depict its source. The is a relator of an $\iA_2$-shaped composite of $\iD_2$ singularities; we highlight the shape of (the source of) $\iA_2$ by a dashed red line in the blue plane, which spans the categorical directions 2 and 4. However, note that this relator cross-interacts with several relators (just as the $\iD_2$ singularity itself was an interacting relator of both $\iA_2$ singularities and of $\iA_1^{\numovar 2}$ singularities). Concretely, it is simultaneously also a $\singa {\iA_2} {(3,1)} {\iA_1^{\numovar 3}}$ relator and a $\singa {\iA_2} {(3,1)} {\iA_2^{\numovar 2}}$ relator---correspondingly, we observe binary occurrences of both $\iA_1^{\numovar 3}$ and of $\iA_2^{\numovar 2}$ singularities in the depicted source of $\iD_3$. ($\iD_3$ is the first example of such a `triple binary relator'.)
\begin{figure}[ht]
    \centering
    \def\svgwidth{1\columnwidth}
    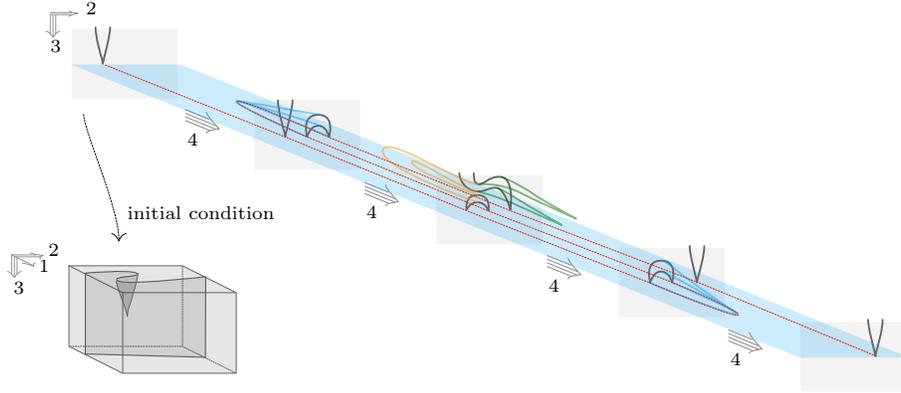

    \caption{The source of the $\iD_3$ relator}
    \label{fig:the-source-of-the-d3-law}
\end{figure}
\end{eg}

\begin{eg}[The $\iD_4$ relator] \label{eg:D4} The classical $D^-_4$ germ is represented by $x_1^3 - x_1 x_2^2$. The codimension of this germ is $3$, and a miniversal unfolding of $D^-_4$ can be given the normal form $x_1^3 - x_2 x_1^2 + u_1 (x_2^2 + x_1^2) - u_2 x_2 - u_3 x_1$. Following \cref{constr:translating-smooth-to-tame-sing} we consider the parametrized graph $\tilde\Gamma_f$ in $\lR^2 \times \lR^1 \times \lR^3 = \lR^6$  of this unfolding. Choosing an appropriate framed cube $\II^6 \into \lR^6$ around the origin, we obtain a tame tangle singularity. As for the earlier \cref{eg:translating-extrema} the process is straight-forward, but requires a good visualization of the 5-manifold $\tilde\Gamma_f$ in $\lR^6$. A more detailed description of this process can be found in \cite{dorn2022d4law}. We will denote the resulting tame 5-tangle singularity by $(\II^6, \iD_4)$. (Up to moving nontrivial content from its target to its source, we can assume that $\iD_4$ is of the form $\src \iD_4 \to \id$.)

    Preliminary work suggests that the source of $\iD_4$ cannot be simplified in a substantive way; in other words, $\iD_4$ appears to be stable. Importantly though, it is not true that $\iD_4$ is inductively stable (see \cref{term:full-extended-stab}); that is, the source $\src \iD_4$ of $\iD_4$ is not yet maximally generic.
    To obtain an inductively stable singularity we apply any maximal perturbation $\iD_4 \to \iD'_4$ between singularities---in fact, one can ensure in this way that the source $\src \iD'_4$ only contains stable 5-tangle singularities that follow the schematic form of \cref{rmk:stable-tang-sing-from-bin-laws}. The resulting source of the singularity $\iD'_4$, simplified slightly up to $\sF$-equivalence, is depicted in \cref{fig:the-source-of-the-d4-singularity}. (The $\iD'_4$ singularity is the cone of the depicted source, see also \cref{fig:qualitative-content-of-d4-singularity}.) In order to be able to interpret the picture, note the following.
    \begin{enumerate}
        \item[$-$] The source $\src \iD'_4$ of $\iD'_4$ is a 4-tangle in dimension 5, $(\II^5,\src \iD'_4)$.
        \item[$-$]  Each row in the picture is a 3-tangle slice, which restricts $(\II^5,\src \iD'_4)$ to `times' $\{t_i\} \times \II^4$. Each such slice is in turn depicted as explained in \cref{notn:3-tangle-sings-as-proj}: \emph{thin} colored lines representing evolution in the 4-direction (and we preserve the color-coding of singularity classes from \cref{fig:denoting-3-tangle-singularities-by-their-projections}) and an `initial condition' is given for the middle picture in the first row (dashed line).
        \item[$-$] The rows, read from top-to-bottom, together show the evolution of these 3-tangles in the 5-direction: \emph{thick} color-lines keep track of how 3-tangle singularities evolve. Note that 4-tangle singularities occur (the source is not an isotopy); a legend of these singularities is provided in the lower left.
        \item[$-$] The legend uses symbols as introduced in \cref{rmk:stable-tang-sing-from-bin-laws}. \qedhere
    \end{enumerate}
\begin{figure}[!ht]
    \centering
    \def\svgwidth{1\columnwidth}
    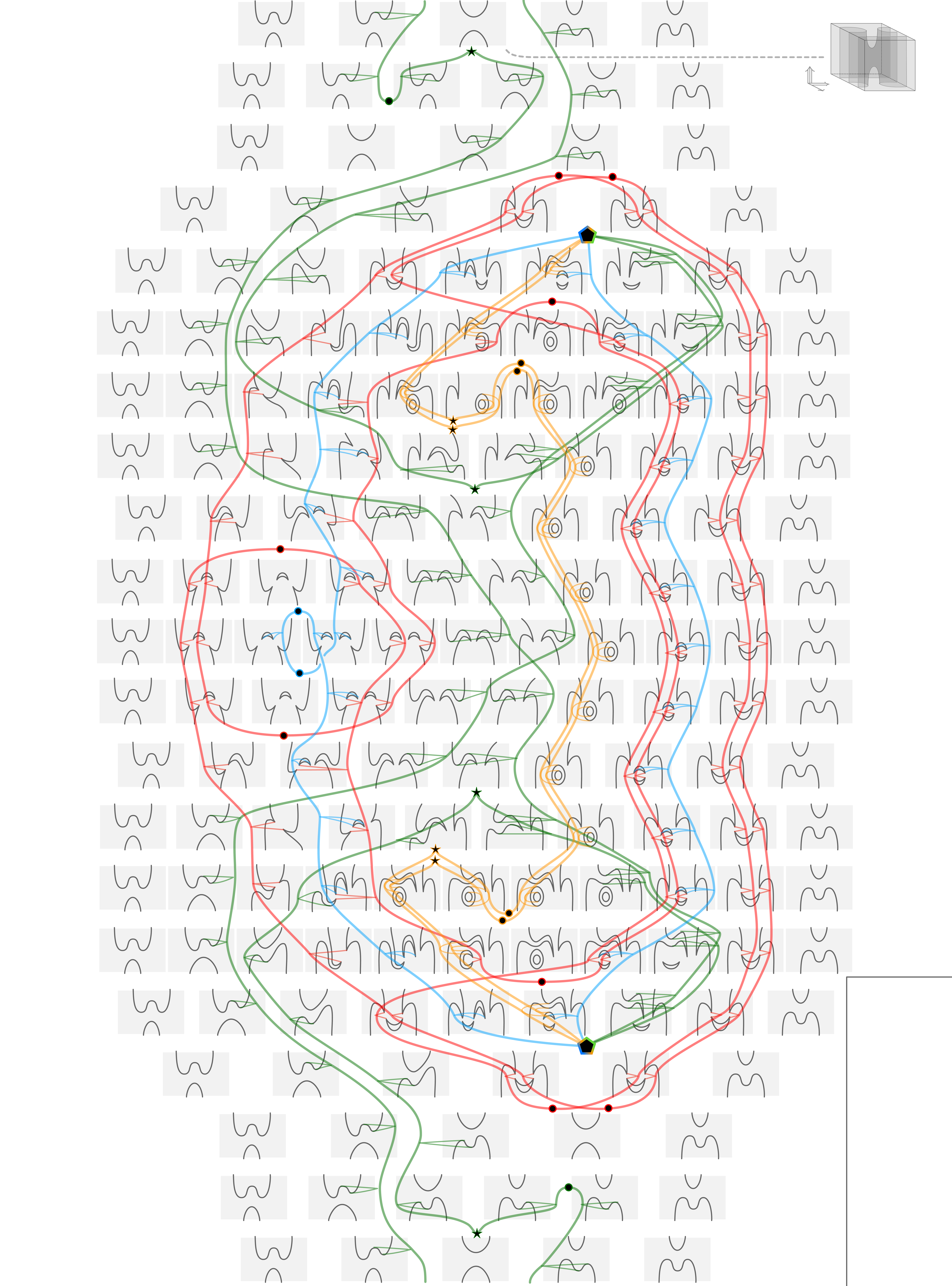

    \caption{The perturbed source of the $\iD_4$ singularity}
    \label{fig:the-source-of-the-d4-singularity}
\end{figure}
\end{eg}

\begin{rmk}[Comparison of $\iD_4$ source and an $\iA_1$-shaped composite of $\iD_3$'s] A schematic representation of the $\iD_4$ singularity is given on the left in \cref{fig:qualitative-content-of-d4-singularity}: this extracts from the previous \cref{fig:the-source-of-the-d4-singularity} only the locus traced out by $\iD_2$ and $\iD_3$ singularities, and depicts it in the (5,4)-plane, extended as by taking a cone into the 6th categorical direction. Contrast this with the $\iA_1$-shaped relator $\singa {\iA_1} 5 {\iD_3}$: unlike the latter relator, whose source simple glues together two 5-reflected $\iD_3$ singularities on their boundary, the source of $\iD_4$ composes two 5-reflected $\iD_3$ singularities up to a `Serre loop' of $\iD_2$ singularities (see \cite{douglas2020dualizable}).
\begin{figure}[ht]
    \centering
    \def\svgwidth{1\columnwidth}
    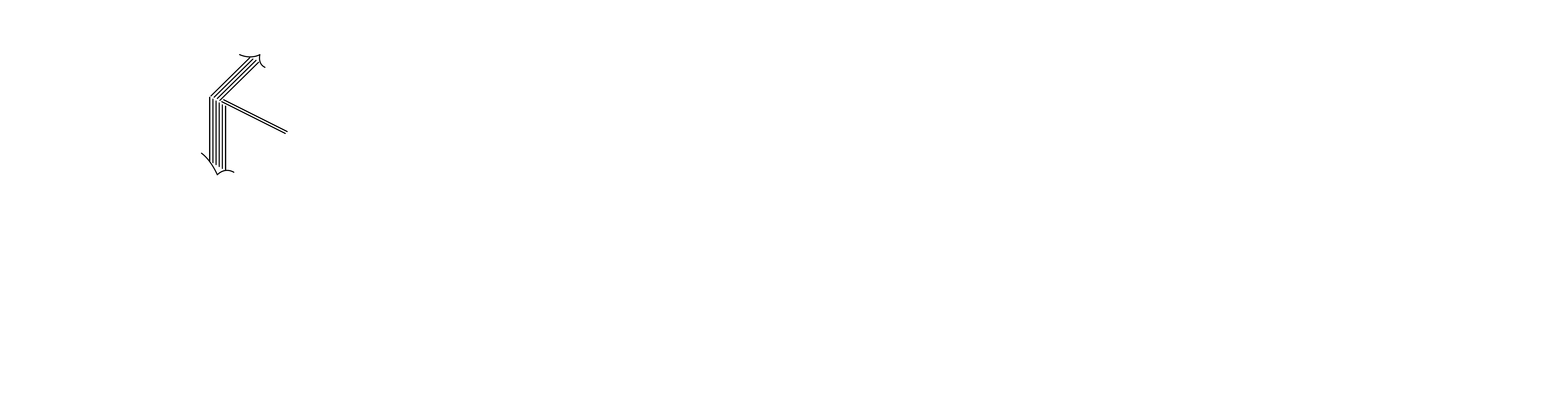

    \caption{Qualitative comparison of the $\iD_4$ singularity}
    \label{fig:qualitative-content-of-d4-singularity}
\end{figure}
\end{rmk}

\nid This completes our brief empiric discussion of the relation of classical and tame tangle singularities. We end this section with remarks about related ideas and further topics of research.

\begin{note}[Drawing exotic PL 4-spheres] \label{rmk:SPC4-pictorially} Recall our earlier discussion of PL manifold structures of tame tangles and its relation to SPC4 (see \cref{rmk:SPC4} and \cref{rmk:tame-4-tang-PL}). If there are exotic PL 4-spheres $\tilde S^4$ then these PL embed in $\II^5$ (where they bound exotic PL 5-disks); this is discussed in \cite{colding2012mean} in the smooth case, from which the PL case follows. By \cref{obs:pl-tang-gen-tame}, we can choose a sufficiently generic such embedding, obtaining a tame 4-tangle $\tilde S^4 \into \II^5$. Thus, if they exist, exotic PL 4-spheres can be represented by a picture like \cref{fig:the-source-of-the-d4-singularity} (built from suitably stable 4-tangle singularities in $\II^5$).
\end{note}

\begin{note}[Invertibility and dualizability] \label{rmk:dualizability} The binary relators touched upon earlier in \cref{rmk:stable-tang-sing-from-bin-laws} are, of course, closely related to the structure of coherently invertible (and coherently dualizable) objects in higher categories. This relation roots in the generalized tangle hypothesis (TH), which provides a link between normal framed tangles and dualizable objects (see \cite{baez1995higher}). In order to see how the TH applies to our discussion here, we briefly sketch how tame tangle singularities can be endowed with `normal framings' (in fact, in the form of a combinatorial structure).

    Let us first consider codimension-1 tame $m$-tangle singularities $f = (W \into \II^{m+1})$. The topological Schoenflies theorem \cite{brown1960proof} shows that the link $\partial \overline W \into \partial \bI^{m+1}$ of $f$ separates $S^m \iso \partial \bI^{m+1}$ into two $m$-disks. Thus, $\II^{m+1} \setminus W$ has two $(m+1)$-disk components. A \emph{signing} $\mathrm{sgn}$ of $f$ labels one of these disks with `$-$' and the other with `$+$'.
    For $m = 0$, signings thus distinguish two tame tangle singularities: the `positive' point $\ast : - \to +$ and `negative' point $\ast^\dagger : + \to -$. In both cases, the tangle manifold (which is merely a point) can be given a normal 1-frame that points into the adjacent 1-disk labeled by `$+$'. Via the TH, the singularity $\ast^\dagger$ can be thought of as the dual of $\ast^\dagger$ (both considered as objects in a monoidal higher category). For $m > 0$, we may similarly endow signed $m$-tangle singularities $(f,\mathrm{sgn})$ with a normal 1-frame which, at any point of the tangle manifold, points into the `$+$'-labeled disk. Again via the TH, we can think of $(f,\mathrm{sgn})$ as a part of the dualizability structure of the dualizable object $\ast$.

    The idea extends to higher codimensions by stabilization. Namely, we can stabilize any given signed codim-1 singularities $(f,\mathrm{sgn})$ and its normal 1-frame to a normal $k$-framed singularity $(f \oplus \eps^{k-1},\mathrm{sgn})$ in $\II^{m + k}$ (by post-composing $f$ with the embedding $\II^{m + 1} \iso \II^{m + 1} \times \{0\} \into \II^{m + k}$ and adding the last $(k-1)$ unit vectors to the normal frame). This procedure produces certain very specific normal framed $m$-tangle singularities in codim-$k$, and not all codim-$k$ singularities can be obtained in this way (for instance, the braid trivializing singularity cannot be obtained in this way). Again, via the TH, we can think of the resulting normal $k$-framed $m$-tangle singularities as expressing the dualizability structure of the dualizable object $\ast \oplus \eps^{k-1}$ (in a `$k$-tuply monoidal' higher category).
\end{note}

\nid We remark that the relation between tame tangle singularities and dualizability described in the preceding note, makes the classification of stable tame tangle singularities an interesting problem: indeed, such a classification (together with a classification of `stable coherences', such as braids and their higher analogues, see \cref{rmk:elementary-homotopies}) \emph{should} provide the generating higher morphisms from which to build homotopy groups and other homotopical behavior.

\begin{note}[The case $p > 1$] Understanding stable tangle singularities for codimension $p > 1$ goes beyond the ADE pattern. Such singularities appear naturally for instance in the attachment maps of cells of Thom spectra other than the sphere spectrum (see e.g.\ \cite{dorn2022TP}). The classification of singularities remains a `very open' problem in this general case (and classical work on the matter has been much sparser than in the case $p = 1$---nonetheless, interesting results exist when replacing `left-right equivalence' by certain coarser equivalence relations \cite{giusti1983}).
\end{note}

\subsection{Conjectures about the combinatorialization of smooth structures} \label{sec:conjectures}

The close connection illustrated in previous sections between classical smooth singularities and framed combinatorial topology motivated us in \cite[\S5]{fct} to make several conjectures regarding the `smooth behavior' of tame stratifications. We here recall and rephrase these conjectures in the context of manifold diagrams and tangles. For simplicity, we will work with closed manifolds, i.e.\ compact manifolds without boundary (though, with sufficient care, similar conjectures can be made for manifolds with boundary as well). Recall the notion of smooth tangles from \cref{term:transversal}. A `tame smooth tangle' $W \into \II^n$ is a smooth tangle that is a tame tangle (i.e.\ $W$ is a smooth manifold, and $W \into \II^n$ a smooth embedding that is tame and framed transversal). Our starting point is the following observation in low dimensions.

\begin{obs}[Combinatorialization of smooth structures in dim $\leq 4$] \label{obs:comb-of-dim4-struct} For $m \leq 4$, recall from \cref{rmk:tame-4-tang-PL} that any tame $m$-tangle $f : W \into \II^n$ carries a canonical smooth manifold structure on its tangle manifold $W$ (since this is true in the PL case and since PL and smooth manifold structures are equivalent in these dimensions). In fact, given any smooth manifold structure on a topological $m$-manifold $W$, then there exists a tame tangle $f : W \into \II^n$ whose canonical smooth structure coincides with the given smooth structure on $W$ (again, this follows since it is true in the PL case, see \cref{obs:realizing-PL-struct}). Moreover, if $f : W \into \II^n$ is a tame \emph{smooth} tangle and $f_{\mathrm{PL}} = (W_{\mathrm{PL}} \into \II^n)$ a PL tangle that is framed stratified homeomorphic to $f$, then the framed stratified homeomorphism $f \iso f_{\mathrm{PL}}$ can be chosen to be piecewise differentiable (by an argument analogous to the piecewise linear case, see \cref{obs:pl-tang-inherit-pl-struct}). In particular, there is a piecewise differentiable homeomorphism $W \iso W_{\mathrm{PL}}$, and, because the piecewise differentiable and the piecewise linear categories are equivalent, this implies that the given smooth structure of $W$ coincides with the canonical smooth structure of $f$ on $W$. The situation can be summarized as follows: in dimensions $m \leq 4$, framed stratified homeomorphism classes of tame $m$-tangles $f : W \into \II^n$ can be used to present any smooth structure on a given topological manifold $W$ (by endowing $W$ with the canonical smooth structure of $f$), and whenever $W \into \II^n$ is smooth itself, then the canonical smooth structure is simply the smooth structure of $W$. 

    In fact, general position arguments (analogous to the PL case in \cref{obs:pl-tang-gen-tame}) should also apply in the smooth case to prove that all smooth structures of $W$ can be realized as tame \emph{smooth} tangles $W \into \II^n$. While we will not attempt to prove this here, we formally record the statement (for all dimensions $m$) in \cref{conj:tame-dense} below.
\end{obs}

\nid Building on and generalizing the previous observation, our first conjecture is that the tameness and framed transversality is generic property of smooth embeddings $W \into \II^n$. (We endow mapping spaces with compact-open topology.)

\begin{conj}[Smooth embeddings are generically tame tangles] \label{conj:tame-dense} Given a closed smooth manifold $W$, the subspace of tame smooth tangles $W \into \II^n$ is dense in the space of all smooth embeddings $W \into \II^n$.
\end{conj}

\nid In the PL case this conjecture is true, and was discussed in \cref{obs:pl-tang-gen-tame}. Moreover, if we endow the mapping space with a topology in which neighborhoods keep differentials `close by' (such as the Whitney $C^\infty$-topology), then it should be possible to replace `dense' by `open and dense' in the conjecture.

While the preceding conjecture guarantees that for any (closed) smooth manifold $W$ there exists a tame smooth tangle $W \into \II^n$, we may conversely ask if the smooth structure of $W$ can be recovered from such a tangle. Let us assume for a moment that all smooth disks $D^n$ are standard (in reality, the question is open for $n = 4$ and $n = 5$, and hinges again on SPC4).
Then, given a tame smooth tangle $W \into \II^n$, each point of $W$ has a tame framed neighborhood which is a standard smooth disk. In classical Morse theory, exotic smooth structures can arise from attachments of standard disks if the attaching maps are themselves exotic.
However, in the context of tame tangles, such attachments are (by an inductive argument) standard gluings along lower dimensional tame tangles. Inductively, exotic attachments can therefore only happen on $0$-tangles; but $0$-tangles have no exotic structure. This motivates the following conjecture.

\begin{conj}[Framed homeomorphism of tame smooth tangles implies diffeomorphism] \label{conj:smooth-faithfull} If two tame smooth tangles $W \into \II^n$ and $W' \into \II^n$ are framed stratified homeomorphic, then $W$ and $W'$ are diffeomorphic.\footnote{Note that this conjecture isn't `obviously equivalent' to SPC4: indeed, the sketched argument would probably still work with the simpler assumption that smooth structures of exotic 5-disks are determined by their boundary exotic 4-spheres.}
\end{conj}




\nid Again, in the PL case this conjecture is true, and was discussed in \cref{obs:pl-tang-inherit-pl-struct}.

Note that a priori there are many tame smooth tangles $W \into \II^n$ for a given smooth manifold $W$; any such choice may be thought of as a choice of `higher Morse function' on $W$. Different choices should be related, as the next conjecture suggests. Recall the notion of framed isotopy from \cref{defn:framed-isotopy}. We say a framed isotopy is smooth if it is smooth as a tame tangle.

\begin{conj}[Smooth framed isotopy of tame smooth tangles implies diffeomorphism] \label{conj:smooth-full} Given smooth $m$-tame tangles $f : W \into \II^n$, $f' : W' \into \II^n$, if $f$ and $f'$ are smoothly framed isotopic then $W$ and $W'$ are diffeomorphic.
\end{conj}

\nid In the stable range $n > 2m + 1$, the converse should hold as well (i.e.\ diffeomorphism will imply the existence of a framed isotopy).

Finally, we turn to stable tame tangle singularities. First, note the previous conjectures suggest that there is a distinction between `smooth' and `non-smooth' tangle singularities, that is, we do not expect all tangle singularities to be smoothable---instead, non-smoothable examples of the following form should exist. Given an exotic smooth $k$-sphere $W$ (where $k > 4$) we can find a tame, framed transversal embedding $W \into \II^n$ by \cref{conj:tame-dense}. Then the cone on $W$ embeds in $\II^n \times \II$ and defines a tame tangle singularity. However, since exotic smooth spheres (above dimension 4) do not bound smooth disks, and assuming \cref{conj:smooth-faithfull} holds true, this tangle singularity itself cannot be smoothable. When considering stable singularities, one would similarly expect the class of smooth stable $m$-tangle singularities to be a proper subclass of the class of all stable $m$-tangle singularities for general $m$. 

Recall, for the low-dimensional cases discussed in \cref{sec:stab-in-low-dim}, we obtained classifications of stable tangle singularities up to $\sF$-equivalence that were `constructive' (i.e.\ we could explicitly construct and list all of the stable tangle singularity classes). In fact, the lists we obtained were finite as well. Our final conjecture states that similar classifications should be possible for singularities in all dimensions and codimensions.

\begin{conj}[Classifications of stable tangle singularities is tractable] $\sF$-equivalence classes of stable $m$-tangle singularities in codimension $p$ are constructively classifiable, and there are finitely many such classes for any pair $(m,p)$.
\end{conj}



\clearpage

\Urlmuskip=0mu plus 1mu\relax
\renewcommand{\baselinestretch}{.95}\small
\bibliographystyle{alpha}
\bibliography{library-mdiag}

\end{document}

%% file: preamble-mdiag.tex
\newcommand{\ignore}[1]{}
\newcommand{\formerlysmall}[1]{}

\settrims{0pt}{0pt}
\settypeblocksize{*}{35pc}{*}
\setlrmargins{*}{*}{1}
\setulmarginsandblock{1.1in}{1.4in}{*}
\setheadfoot{\onelineskip}{2\onelineskip}
\setheaderspaces{*}{1.5\onelineskip}{*}
\checkandfixthelayout
\cftsetindents{figure}{0em}{3.5em}
\cftsetindents{table}{0em}{3.5em}

\usepackage[margin=1.5in]{geometry}
 
\usepackage[centertags,sumlimits,intlimits,namelimits,reqno]{amsmath}
\usepackage{amssymb}
\usepackage{amsthm}
\usepackage{cmll}
\usepackage{mathtools}
\usepackage[cal=euler,scr=rsfso]{mathalfa}
\usepackage[boxslash]{stmaryrd}
\usepackage{enumitem}
\usepackage[T1]{fontenc}

\usepackage[pdfencoding=auto,psdextra]{hyperref}
\hypersetup{bookmarksdepth=4,
			bookmarksnumbered=true,
            colorlinks,breaklinks,
            urlcolor=[RGB]{7,7,96},
            linkcolor=[RGB]{7,7,96},
            citecolor=[RGB]{7,7,96}}

\usepackage{import}
\usepackage{pdfpages}
\usepackage{transparent}
\usepackage{xcolor}
\usepackage{setspace,wrapfig,verbatim,changepage,adjustbox, eso-pic}
\usepackage{braket,xspace,url,scalerel,microtype}
\usepackage{xstring}

\usepackage[capitalize,nameinlink]{cleveref}
\usepackage{float,bookmark,caption}
\usepackage[noindentafter,explicit]{titlesec}

\usepackage{tikz, tikz-cd}
\usetikzlibrary{arrows,positioning,calc,backgrounds,decorations.markings,decorations.pathmorphing,shapes, tikzmark}

\allowdisplaybreaks

\setlist{noitemsep, nolistsep}

\setsecheadstyle{\large\bfseries}

\titleformat{\subsection}[runin]
  {\normalfont\normalsize\bfseries}
  {\thesubsection}
  {1em}
  {#1.}

\titleformat{\subsubsection}[runin]
  {\normalfont\normalsize\itshape}
  {\thesubsubsection}
  {1em}
  {#1.}

\titlespacing\section{0pt}{12pt plus 4pt minus 2pt}{6pt plus 2pt minus 2pt}
\titlespacing\subsection{0pt}{6pt plus 4pt minus 2pt}{6pt plus 2pt minus 2pt}
\titlespacing\subsubsection{0pt}{6pt plus 4pt minus 2pt}{6pt plus 2pt minus 2pt}
  
\crefalias{chapter}{section}

\DeclareRobustCommand{\SkipTocEntry}[5]{}

\let\OLDthebibliography\thebibliography
\renewcommand\thebibliography[1]{
  \OLDthebibliography{#1}
  \setlength{\parskip}{0pt}
  \setlength{\itemsep}{0pt plus 0.3ex}
}

\makeatletter
\newcommand{\firstToC}{{%
  \@fileswfalse
  \renewcommand{\contentsname}{Contents}%
  \@starttoc{toc}{\contentsname}%
}}
\newcommand{\anotherToC}{{%
  \@fileswfalse
  \renewcommand{\contentsname}{Contents (TEMPORARY)}%
  \@starttoc{toc}{\contentsname}%
}}
\newcommand{\lastToC}{{%
  \renewcommand{\@tocwrite}[2]{}
  \renewcommand{\contentsname}{Contents}%
  \@starttoc{toc}{\contentsname}%
}}
\makeatother

\counterwithin{figure}{chapter}

 \tikzset{double line with arrow/.style args={#1,#2}{decorate,decoration={markings,%
mark=at position 0 with {\coordinate (ta-base-1) at (0,1pt);
\coordinate (ta-base-2) at (0,-1pt);},
mark=at position 1 with {\draw[#1] (ta-base-1) -- (0,1pt);
\draw[#2] (ta-base-2) -- (0,-1pt);
}}}}
 \tikzset{Equal/.style={-,double line with arrow={-,-}}}

\tikzset{
    labls/.style={anchor=south, rotate=90, inner sep=.5mm}
}
\tikzset{
    labln/.style={anchor=north, rotate=90, inner sep=1mm}
}
\tikzset{
    lablh/.style={anchor=north, rotate=45, inner sep=1mm}
}

\tikzset{mid vert/.style={/utils/exec=\tikzset{every node/.append style={outer sep=0.8ex}},
postaction=decorate,decoration={markings,
mark=at position 0.5 with {\draw[-] (0,#1) -- (0,-#1);}}},
mid vert/.default=0.75ex}

\definecolor{brightgreen}{rgb}{0.4, 1.0, 0.0}
\definecolor{brightturquoise}{rgb}{0.03, 0.91, 0.87}
\definecolor{brightpink}{rgb}{1.0, 0.0, 0.5}
\definecolor{carrotorange}{rgb}{0.93, 0.57, 0.13}

\definecolor{NBcolor}{rgb}{0.6,0.6,0.0}		

\newcounter{cdcnt}
\newcounter{cldcnt}
\newcounter{tdcnt}

\definecolor{CDcolor}{rgb}{0.0,0.5,0.75}	
\definecolor{XCDcolor}{rgb}{0.8,0.8,1}	
\definecolor{CLDcolor}{RGB}{220,100,10}
\definecolor{CLDEcolor}{rgb}{.6,0,.15}                    
\definecolor{XCLDcolor}{rgb}{1,.8,.8}		
\definecolor{TDcolor}{rgb}{0.4,0.4,0.0}		
\definecolor{RVcolor}{rgb}{0.4,0.0,0.8}		

\makeatletter%
\newbox\envmarksym@box
\setbox\envmarksym@box\hbox
  {%
    \begin{tikzpicture}
      \draw (0,0) -- (.5em,0) -- (.5em,.5em);
    \end{tikzpicture}%
  }
\newcommand*\envmarksym[1][.5em]
  {\resizebox{#1}{!}{\raisebox{.0ex}{\usebox\envmarksym@box}}}

\def\thmqedhere{\expandafter\csname\csname @currenvir\endcsname @qed\endcsname}

\def\defthm#1#2#3{
      \newtheorem{#1}{#2}[section]%
      \numberwithin{#1}{section}%
      \expandafter\def\csname #1autorefname\endcsname{#2}%
      \expandafter\let\csname c@#1\endcsname\c@thm}

\theoremstyle{plain} 

\newtheorem{introthm}{Theorem}
\expandafter\def\csname introthmautorefname\endcsname{Theorem}

\expandafter\def\csname introcorautorefname\endcsname{Corollary}
\expandafter\let\csname c@introcor\endcsname\c@introthm

\expandafter\def\csname intropropautorefname\endcsname{Proposition}
\expandafter\let\csname c@introprop\endcsname\c@introthm

\expandafter\def\csname introconjautorefname\endcsname{Conjecture}
\expandafter\let\csname c@introconj\endcsname\c@introthm

\newtheorem{introhyp}{Hypothesis}
\expandafter\def\csname introhypautorefname\endcsname{Hypothesis}
\expandafter\let\csname c@introhyp\endcsname\c@introthm

\newtheorem{introobs}{Observation}
\expandafter\def\csname introobsautorefname\endcsname{Observation}
\expandafter\let\csname c@introobs\endcsname\c@introthm

\newtheorem{thm}{Theorem}[section]
\numberwithin{thm}{section}

\defthm{cor}{Corollary}{Corollaries}
\defthm{prop}{Proposition}{Propositions}
\defthm{lem}{Lemma}{Lemmas}
\defthm{klem}{Key Lemma}{Key Lemmas}
\defthm{conj}{Conjecture}{Conjectures}
\defthm{disconj}{Disproven Conjecture}{Disproven Conjectures}
\defthm{hyp}{Hypothesis}{Hypotheses}
\defthm{indhyp}{Inductive Hypothesis}{Inductive Hypotheses}

\theoremstyle{definition}

\newtheorem{introdef}{Definition}
\expandafter\def\csname introdefautorefname\endcsname{Definition}
\expandafter\let\csname c@introdef\endcsname\c@introthm

\newtheorem{intrormk}{Remark}
\expandafter\def\csname intrormkautorefname\endcsname{Remark}
\expandafter\let\csname c@intrormk\endcsname\c@introthm

\defthm{defnx}{Definition}{Definitions}
\newenvironment{defn}
  {\pushQED{\qed}\defnx}
  {\popQED\enddefnx}

\defthm{altdefnx}{Alternative Definition}{Alternative Definitions}
\crefname{altdefnx}{Definition}{Definitions}
\newenvironment{altdefn}
  {\pushQED{\qed}\altdefnx}
  {\popQED\endaltdefnx}

\defthm{nonegx}{Non-Example}{Non-Examples}
\newenvironment{noneg}
  {\pushQED{\qed}\nonegx}
  {\popQED\endnonegx}

\defthm{egx}{Example}{Examples}
\newenvironment{eg}
  {\pushQED{\qed}\egx}
  {\popQED\endegx}

\defthm{constrx}{Construction}{Constructions}
\newenvironment{constr}
  {\pushQED{\qed}\constrx}
  {\popQED\endconstrx}

\defthm{heurx}{Heuristic}{Heuristics}
\newenvironment{heur}
  {\pushQED{\qed}\heurx}
  {\popQED\endheurx}

\theoremstyle{remark}

\defthm{rmkx}{Remark}{Remarks}
\newenvironment{rmk}
  {\pushQED{\qed}\rmkx}
  {\popQED\endrmkx}
  
\defthm{notex}{Note}{Notes}
\newenvironment{note}
  {\pushQED{\qed}\notex}
  {\popQED\endnotex}

\defthm{punchx}{Punchline}{Punchlines}

\defthm{previewx}{Preview}{Previews}

\defthm{obsx}{Observation}{Observations}
\newenvironment{obs}
  {\pushQED{\qed}\obsx}
  {\popQED\endobsx}

\defthm{probx}{Problem}{Problems}

\defthm{notnx}{Notation}{Notations}
\newenvironment{notn}
  {\pushQED{\qed}\notnx}
  {\popQED\endnotnx}

\defthm{termx}{Terminology}{Terminology}
\newenvironment{term}
  {\pushQED{\qed}\termx}
  {\popQED\endtermx}
\newtheorem*{term*}{Terminology}

\defthm{convx}{Convention}{Conventions}
\newenvironment{conv}
  {\pushQED{\qed}\convx}
  {\popQED\endconvx}
  
\defthm{recoll}{Recollection}{Recollections}
  
\newtheorem*{conv*}{Convention}
\newtheorem*{convrev*}{Convention Reversal}
\newtheorem*{convs*}{Conventions}
\newtheorem*{recoll*}{Recollection}

\let\c@equation\c@thm
\numberwithin{equation}{section}

\def\definecref{\newif\ifcref}
\ifx\creftrue\undefined
  \definecref
  \creffalse
\fi

\ifcref\usepackage{cleveref,aliascnt}\fi

\ifcref\else
  \@ifpackageloaded{mathtools}{\mathtoolsset{showonlyrefs,showmanualtags}}{}
\fi

\makeatother

\newcommand{\fleq}{\ensuremath{\preceq}}

\newcommand{\fles}{\ensuremath{\prec}}

\newcommand{\cone}{\mathsf{c}}
\newcommand{\link}{\mathsf{link}}
\newcommand{\fcone}{\mathsf{ce}}
\newcommand{\tdim}{\mathsf{tdim}}
\newcommand{\tstr}{\mathsf{tstr}}

\makeatletter
\newcommand{\oset}[3][0ex]{%
  \mathrel{\mathop{#3}\limits^{
    \vbox to#1{\kern-2\ex@
    \hbox{$\scriptstyle#2$}\vss}}}}
\makeatother
\newcommand{\TT}{\ensuremath{\mathbb{T}}}
\newcommand{\OTT}{\ensuremath{\mathring{\TT}}}
\newcommand{\CTT}{\ensuremath{\bar{\TT}}}

\newcommand{\truss}[1]{\ensuremath{{\mathsf{T}}\!\mathsf{rs}_{#1}}}

\newcommand{\otruss}[1]{\ensuremath{\mathring{\mathsf{T}}\!\mathsf{rs}_{#1}}}
\newcommand{\ctruss}[1]{\ensuremath{\bar{\mathsf{T}}\!\mathsf{rs}_{#1}}}
\newcommand{\strotruss}[1]{\ensuremath{\mathsf{Str}\!\mathring{\mathsf{T}}\!\mathsf{rs}_{#1}}}
\newcommand{\strctruss}[1]{\ensuremath{\mathsf{Str}\!\bar{\mathsf{T}}\!\mathsf{rs}_{#1}}}

\newcommand{\trussbun}[1]{\ensuremath{\mathsf{TrsBun}_{#1}}}

\newcommand{\ccs}[1]{\ensuremath{\mathsf{char}(#1)}}

\newcommand{\dcs}[1]{\ensuremath{\mathsf{discr}(#1)}}

\newcommand{\cellstr}{\ensuremath{\mathsf{cell}}}

\DeclareMathOperator\Totz{\mathsf{Tot}}
\DeclareMathOperator\Entr{\mathsf{Entr}}
\DeclareMathOperator\Exit{\mathsf{Exit}}
\newcommand{\TEntr}{\ensuremath{\mathcal{E}\mkern-2mu\mathit{ntr}}}
\DeclareMathOperator\stratify{\mathsf{CStr}}
\newcommand{\CStr}[1]{\ensuremath{\stratify {#1}}}

\DeclareMathOperator\fibdim{\ensuremath{\mathrm{fibdim}}}

\newcommand{\numovar}[1]{#1\mathsf{d}}
\newcommand{\numovarone}{}
\newcommand{\pts}{\mathsf{pt}}
\newcommand{\singa}[3]{{#1} \leftleftarrows_{#2} {#3}}
\newcommand{\singb}[5]{{#1} \leftleftarrows_{#2} {#3} \leftleftarrows_{#4} {#5}}

\DeclareMathOperator\src{\ensuremath{\mathsf{src}}}

\newcommand{\proto}[1]{\ensuremath{\xslashedrightarrow{#1}}}

\newcommand{\tmesh}[1]{\ensuremath{{\mathcal{M}\mkern-2mu\mathit{esh}_{#1}}}}
\newcommand{\tmeshbun}[1]{\ensuremath{{\mathcal{M}\mkern-2mu\mathit{eshBun}_{#1}}}}
\newcommand{\otmesh}[1]{\ensuremath{\mathring{\mathcal{M}}\mkern-2mu\mathit{esh}_{#1}}}
\newcommand{\ctmesh}[1]{\ensuremath{\bar{\mathcal{M}}\mkern-2mu\mathit{esh}_{#1}}}
\newcommand{\strotmesh}[1]{\ensuremath{\mathcal{S}\mkern-2mu\mathit{tr}\mathring{\mathcal{M}}\mkern-2mu\mathit{esh}_{#1}}}
\newcommand{\strctmesh}[1]{\mathcal{S}\mkern-2mu\mathit{tr}\ensuremath{\bar{\mathcal{M}}\mkern-2mu\mathit{esh}_{#1}}}

\DeclareMathOperator\FTrs{\mathsf{FTrs}}
\DeclareMathOperator\NFTrs{\mathsf{NFTrs}}

\DeclareMathOperator\CMsh{\mathsf{CMsh}}

\newcommand{\cret}{\ensuremath{\mathsf{cr}}}
\newcommand{\cint}{\ensuremath{\mathsf{ci}}}

\newcommand{\II}{\ensuremath{\mathbb{I}}}
\newcommand{\NF}[1]{\left\llbracket{#1}\right\rrbracket}

\usepackage{calc}

\DeclareFontFamily{U}{MnSymbolB}{}
\DeclareSymbolFont{MnSyB}{U}{MnSymbolB}{m}{n}
\SetSymbolFont{MnSyB}{bold}{U}{MnSymbolB}{b}{n}
\DeclareFontShape{U}{MnSymbolB}{m}{n}{
    <-6>  MnSymbolB5
   <6-7>  MnSymbolB6
   <7-8>  MnSymbolB7
   <8-9>  MnSymbolB8
   <9-10> MnSymbolB9
  <10-12> MnSymbolB10
  <12->   MnSymbolB12}{}
\DeclareFontShape{U}{MnSymbolB}{b}{n}{
    <-6>  MnSymbolB-Bold5
   <6-7>  MnSymbolB-Bold6
   <7-8>  MnSymbolB-Bold7
   <8-9>  MnSymbolB-Bold8
   <9-10> MnSymbolB-Bold9
  <10-12> MnSymbolB-Bold10
  <12->   MnSymbolB-Bold12}{}
  
\DeclareMathSymbol{\nperp}{\mathrel}{MnSyB}{217}

\makeatletter

\newcommand\mathrule[3][0pt]{%
  \ifdim#2>#3\math@hrule[#1]{#2}{#3}\else\math@vrule[#1]{#2}{#3}\fi}

\newcommand\math@hrule[3][0pt]{%
  \gdef\mystery@factor{0.07}%
 \@tempdima=#3%
  \rule[#1]{0pt}{#3}
  \raisebox{.5\@tempdima+#1}{%
    \makebox[#2][l]{\kern-.5\@tempdima\@@mathrule{#2}{#3}}}%
}

\newcommand\math@vrule[3][0pt]{%
  \gdef\mystery@factor{0.0}%
 \@tempdima=#2%
  \rule[#1]{0pt}{#3}
  \raisebox{-.0\@tempdima+#1}{%
    \kern0.5\@tempdima%
    \rotatebox{90}{\kern-0.5\@tempdima\makebox[#3][l]{\@@mathrule{#3}{#2}}}%
    \kern0.5\@tempdima}%
}

\def\@@mathrule#1#2{%
  \@tempdimb=#2%
  \@tempdima=\dimexpr#1-\mystery@factor\@tempdimb
  \pdfliteral{%
    q []0 d %
    1 J 
    \strip@pt\@tempdimb\space w \strip@pt\@tempdimb\space 0 m %
    \strip@pt\@tempdima\space 0 l S Q }}
\makeatother

\newif\iffootnoterule

\makeatletter
\AtBeginDocument{%
\let\latex@@footnoterule\footnoterule

\renewcommand\footnoterule{%
  \iffootnoterule
  \latex@@footnoterule%
  \else
  \advance\skip\footins 4\p@\@plus2\p@\relax%
  \fi
  }
}
\makeatother

\newcommand{\intrr}{^\circ}

\newcommand{\abs}[1]{\left|{#1}\right|}

\newcommand{\rest}[2]{{\left.\kern-\nulldelimiterspace #1 \right|_{#2}}}
\newcommand{\restup}[3]{\left.\kern-\nulldelimiterspace #1 \right|_{#2}^{#3}}

\newcommand{\eps}{\epsilon}

\newcommand{\xra}{\xrightarrow}
\newcommand{\nid}{\noindent}
\newcommand{\stack}[1]{\mathrel{\#_{#1}}}

\makeatletter
\let\ea\expandafter

\def\mdef#1#2{\ea\ea\ea\gdef\ea\ea\noexpand#1\ea{\ea\ensuremath\ea{#2}\xspace}}
\def\alwaysmath#1{\ea\ea\ea\global\ea\ea\ea\let\ea\ea\csname your@#1\endcsname\csname #1\endcsname
  \ea\def\csname #1\endcsname{\ensuremath{\csname your@#1\endcsname}\xspace}}

\DeclareRobustCommand\widecheck[1]{{\mathpalette\@widecheck{#1}}}
\def\@widecheck#1#2{%
    \setbox\z@\hbox{\m@th$#1#2$}%
    \setbox\tw@\hbox{\m@th$#1%
       \widehat{%
          \vrule\@width\z@\@height\ht\z@
          \vrule\@height\z@\@width\wd\z@}$}%
    \dp\tw@-\ht\z@
    \@tempdima\ht\z@ \advance\@tempdima2\ht\tw@ \divide\@tempdima\thr@@
    \setbox\tw@\hbox{%
       \raise\@tempdima\hbox{\scalebox{1}[-1]{\lower\@tempdima\box
\tw@}}}%
    {\ooalign{\box\tw@ \cr \box\z@}}}

\newcount\foreachcount

\def\foreachletter#1#2#3{\foreachcount=#1
  \ea\loop\ea\ea\ea#3\@alph\foreachcount
  \advance\foreachcount by 1
  \ifnum\foreachcount<#2\repeat}

\def\foreachLetter#1#2#3{\foreachcount=#1
  \ea\loop\ea\ea\ea#3\@Alph\foreachcount
  \advance\foreachcount by 1
  \ifnum\foreachcount<#2\repeat}

\let\oldit\it

\def\definescr#1{\ea\gdef\csname s#1\endcsname{\ensuremath{\mathscr{#1}}\xspace}}
\foreachLetter{1}{27}{\definescr}
\def\definecal#1{\ea\gdef\csname c#1\endcsname{\ensuremath{\mathcal{#1}}\xspace}}
\foreachLetter{1}{27}{\definecal}
\def\definebold#1{\ea\gdef\csname b#1\endcsname{\ensuremath{\mathbf{#1}}\xspace}}
\foreachLetter{1}{27}{\definebold}
\def\definebb#1{\ea\gdef\csname l#1\endcsname{\ensuremath{\mathbb{#1}}\xspace}}
\foreachLetter{1}{27}{\definebb}
\def\definefrak#1{\ea\gdef\csname k#1\endcsname{\ensuremath{\mathfrak{#1}}\xspace}}
\foreachletter{1}{27}{\definefrak}
\foreachLetter{1}{27}{\definefrak}
\def\definesf#1{\ea\gdef\csname i#1\endcsname{\ensuremath{\mathsf{#1}}\xspace}}
\foreachletter{1}{6}{\definesf}
\foreachletter{7}{14}{\definesf}
\foreachletter{15}{27}{\definesf}
\foreachLetter{1}{27}{\definesf}
\def\definebar#1{\ea\gdef\csname #1bar\endcsname{\ensuremath{\overline{#1}}\xspace}}
\foreachLetter{1}{27}{\definebar}
\foreachletter{1}{8}{\definebar} 
\foreachletter{9}{15}{\definebar} 
\foreachletter{16}{27}{\definebar}
\def\definetil#1{\ea\gdef\csname #1til\endcsname{\ensuremath{\widetilde{#1}}\xspace}}
\foreachLetter{1}{27}{\definetil}
\foreachletter{1}{27}{\definetil}
\def\definehat#1{\ea\gdef\csname #1hat\endcsname{\ensuremath{\widehat{#1}}\xspace}}
\foreachLetter{1}{27}{\definehat}
\foreachletter{1}{27}{\definehat}
\def\definechk#1{\ea\gdef\csname #1chk\endcsname{\ensuremath{\widecheck{#1}}\xspace}}
\foreachLetter{1}{27}{\definechk}
\foreachletter{1}{27}{\definechk}
\def\defineul#1{\ea\gdef\csname u#1\endcsname{\ensuremath{\underline{#1}}\xspace}}
\foreachLetter{1}{27}{\defineul}
\foreachletter{1}{27}{\defineul}

\let\it\oldit

\def\autofmt@b#1\autofmt@end{\mathbf{#1}}
\def\autofmt@l#1#2\autofmt@end{\mathbb{#1}\mathsf{#2}}
\def\autofmt@c#1#2\autofmt@end{\mathcal{#1}\mathit{#2}}
\def\autofmt@s#1#2\autofmt@end{\mathscr{#1}\mathit{#2}}
\def\autofmt@f#1\autofmt@end{\mathsf{#1}}
\def\autofmt@k#1\autofmt@end{\mathfrak{#1}}
\def\autofmt@u#1\autofmt@end{\underline{\smash{\mathsf{#1}}}}
\def\autofmt@U#1\autofmt@end{\underline{\underline{\smash{\mathsf{#1}}}}}
\def\autofmt@h#1\autofmt@end{\widehat{#1}}
\def\autofmt@r#1\autofmt@end{\overline{#1}}
\def\autofmt@t#1\autofmt@end{\widetilde{#1}}
\def\autofmt@k#1\autofmt@end{\check{#1}}

\def\auto@drop#1{}
\def\autodef#1{\ea\ea\ea\@autodef\ea\ea\ea#1\ea\auto@drop\string#1\autodef@end}
\def\@autodef#1#2#3\autodef@end{%
  \ea\def\ea#1\ea{\ea\ensuremath\ea{\csname autofmt@#2\endcsname#3\autofmt@end}\xspace}}
\def\autodefs@end{blarg!}
\def\autodefs#1{\@autodefs#1\autodefs@end}
\def\@autodefs#1{\ifx#1\autodefs@end%
  \def\autodefs@next{}%
  \else%
  \def\autodefs@next{\autodef#1\@autodefs}%
  \fi\autodefs@next}

\alwaysmath{alpha}
\alwaysmath{beta}
\alwaysmath{gamma}
\alwaysmath{Gamma}
\alwaysmath{delta}
\alwaysmath{Delta}
\alwaysmath{epsilon}
\mdef\ep{\varepsilon}
\alwaysmath{zeta}
\alwaysmath{eta}
\alwaysmath{theta}
\alwaysmath{Theta}
\alwaysmath{iota}
\alwaysmath{kappa}
\alwaysmath{lambda}
\alwaysmath{Lambda}
\alwaysmath{mu}
\alwaysmath{nu}
\alwaysmath{xi}
\alwaysmath{pi}
\alwaysmath{rho}
\alwaysmath{sigma}
\alwaysmath{Sigma}
\alwaysmath{tau}
\alwaysmath{upsilon}
\alwaysmath{Upsilon}
\alwaysmath{phi}
\alwaysmath{Pi}
\alwaysmath{Phi}
\mdef\ph{\varphi}
\alwaysmath{chi}
\alwaysmath{psi}
\alwaysmath{Psi}
\alwaysmath{omega}
\alwaysmath{Omega}

\renewcommand{\Set}[1]{\{#1\}}

\mdef\delbar{\overline{\partial}}

\newcommand{\inv}{^{-1}}

\mdef\hf{\textstyle\frac12 }
\mdef\thrd{\textstyle\frac13 }
\mdef\qtr{\textstyle\frac14 }

\newcommand{\op}{^{\mathrm{op}}}

\let\iso\cong
\let\eqv\simeq

\mdef\Id{\mathrm{Id}}
\mdef\id{\mathrm{id}}
\mdef\uni{\mathrm{is\_uni}}
\mdef\sym{\mathrm{is\_sym}}
\mdef\rel{\mathrm{is\_rel}}
\mdef\Sym{\mathrm{Sym}}
\mdef\SYM{\textsc{sym}}
\mdef\REL{\textsc{rel}}
\alwaysmath{ell}
\alwaysmath{infty}
\alwaysmath{odot}
\def\frc#1/#2.{\frac{#1}{#2}}   
\mdef\ten{\mathrel{\otimes}}

\mdef\sqten{\mathrel{\boxtimes}}

\DeclareMathOperator\ev{ev}

\mdef\Im{\mathrm{Im}}
\mdef\im{\mathrm{im}}
\let\lim\relax
\DeclareMathOperator\lim{lim}

\DeclareRobustCommand{\ocirc}{%
  \mathbin{\mathpalette\on@ntimes\relax}%
}
\newcommand{\on@ntimes}[2]{%
  \vcenter{\hbox{%
    \sbox0{\m@th$#1\otimes$}%
    \setlength\unitlength{\wd0}%
    \begin{picture}(1,1)
    \linethickness{0.35pt}
    \put(.5,.5){\circle{.8}}
    \end{picture}%
  }}%
}

\newcommand{\ot}{\ensuremath{\leftarrow}}

\let\mshar\rightsquigarrow
\let\perturb\rightharpoonup
\let\toot\rightleftarrows

\let\imp\Rightarrow

\let\into\hookrightarrow

\let\xinto\xhookrightarrow
\mdef\we{\overset{\sim}{\longrightarrow}}
\mdef\leftwe{\overset{\sim}{\longleftarrow}}

\let\epi\twoheadrightarrow
\let\leftepi\twoheadleftarrow

\providecommand{\leftsquigarrow}{%
  \mathrel{\mathpalette\reflect@squig\relax}%
}
\newcommand{\reflect@squig}[2]{%
  \reflectbox{$\m@th#1\rightsquigarrow$}%
}

\newcounter{sarrow}
\newcommand\xrsquigarrow[1]{%
\stepcounter{sarrow}%
\mathrel{\begin{tikzpicture}[baseline= {( $ (current bounding box.south) + (0,-0.5ex) $ )}]
\node[inner sep=0.5ex] (\thesarrow) {$\scriptstyle #1$};
\path[draw,<-,decorate,
  decoration={zigzag,amplitude=0.7pt,segment length=1.2mm,pre=lineto,pre length=4pt}]
    (\thesarrow.south east) -- (\thesarrow.south west);
\end{tikzpicture}}%
}

\let\xto\xrightarrow
\let\xot\xleftarrow
\def\rightarrowtailfill@{\arrowfill@{\Yright\joinrel\relbar}\relbar\rightarrow}
\newcommand\xrightarrowtail[2][]{\ext@arrow 0055{\rightarrowtailfill@}{#1}{#2}}

\def\twoheadrightarrowfill@{\arrowfill@{\relbar\joinrel\relbar}\relbar\twoheadrightarrow}
\newcommand\xtwoheadrightarrow[2][]{\ext@arrow 0055{\twoheadrightarrowfill@}{#1}{#2}}

\def\slashedarrowfill@#1#2#3#4#5{%
  $\m@th\thickmuskip0mu\medmuskip\thickmuskip\thinmuskip\thickmuskip
   \relax#5#1\mkern-7mu%
   \cleaders\hbox{$#5\mkern-2mu#2\mkern-2mu$}\hfill
   \mathclap{#3}\mathclap{#2}%
   \cleaders\hbox{$#5\mkern-2mu#2\mkern-2mu$}\hfill
   \mkern-7mu#4$%
}
\def\rightslashedarrowfill@{%
  \slashedarrowfill@\relbar\relbar\mapstochar\rightarrow}
\newcommand\xslashedrightarrow[2][]{%
  \ext@arrow 0055{\rightslashedarrowfill@}{#1}{#2}}
\def\leftslashedarrowfill@{%
  \slashedarrowfill@\relbar\relbar\mapstochar\leftarrow}
\newcommand\xslashedleftarrow[2][]{%
  \ext@arrow 0055{\leftslashedarrowfill@}{#1}{#2}}
\def\hookrightslashedarrowfill@{%
  \slashedarrowfill@\relbar\relbar\mapstochar\hookrightarrow}
\newcommand\xslashedhookrightarrow[2][]{%
  \ext@arrow 0055{\hookrightslashedarrowfill@}{#1}{#2}}
\mdef\hto{\xslashedrightarrow{}}
\mdef\htoo{\xslashedrightarrow{\quad}}

\newbox\dottedarrow@box
\setbox\dottedarrow@box\hbox
  {%
    \begin{tikzpicture}
      \draw[dotted, semithick ,right hook->] (0,0) -- (1.5em,0);
    \end{tikzpicture}%
  }
\newcommand*\dottedarrow
  {\relax\ifmmode\expandafter\dottedarrow@m\else\expandafter\dottedarrow@t\fi}
\newcommand*\dottedarrow@t[1][1.3em]
  {\resizebox{#1}{!}{\raisebox{-.04ex}{\usebox\dottedarrow@box}}}
\newcommand*\dottedarrow@m[1][]
  {%
    \if\relax\detokenize{#1}\relax
      \mathchoice
        {\dottedarrow@t}
        {\dottedarrow@t}
        {\dottedarrow@t[1.1em]}
        {\dottedarrow@t[0.9em]}%
    \else
      \dottedarrow@t[#1]%
    \fi
  }



%% file: figuresused/manifold-diagrams-inductive-idea.pdf_tex
\begingroup%
  \makeatletter%
  \providecommand\color[2][]{%
    \errmessage{(Inkscape) Color is used for the text in Inkscape, but the package 'color.sty' is not loaded}%
    \renewcommand\color[2][]{}%
  }%
  \providecommand\transparent[1]{%
    \errmessage{(Inkscape) Transparency is used (non-zero) for the text in Inkscape, but the package 'transparent.sty' is not loaded}%
    \renewcommand\transparent[1]{}%
  }%
  \providecommand\rotatebox[2]{#2}%
  \newcommand*\fsize{\dimexpr\f@size pt\relax}%
  \newcommand*\lineheight[1]{\fontsize{\fsize}{#1\fsize}\selectfont}%
  \ifx\svgwidth\undefined%
    \setlength{\unitlength}{2160bp}%
    \ifx\svgscale\undefined%
      \relax%
    \else%
      \setlength{\unitlength}{\unitlength * \real{\svgscale}}%
    \fi%
  \else%
    \setlength{\unitlength}{\svgwidth}%
  \fi%
  \global\let\svgwidth\undefined%
  \global\let\svgscale\undefined%
  \makeatother%
  \begin{picture}(1,0.415625)%
    \lineheight{1}%
    \setlength\tabcolsep{0pt}%
    \put(0,0){\includegraphics[width=\unitlength,page=1]{manifold-diagrams-inductive-idea.pdf}}%
    \put(0.10590278,0.39756944){\color[rgb]{0,0,0}\makebox(0,0)[lt]{\lineheight{1.25}\smash{\begin{tabular}[t]{l}dim 1\end{tabular}}}}%
    \put(0.35104167,0.39756944){\color[rgb]{0,0,0}\makebox(0,0)[lt]{\lineheight{1.25}\smash{\begin{tabular}[t]{l}dim 2\end{tabular}}}}%
    \put(0.66145833,0.39756944){\color[rgb]{0,0,0}\makebox(0,0)[lt]{\lineheight{1.25}\smash{\begin{tabular}[t]{l}dim 3\end{tabular}}}}%
    \put(0.45659722,0.28819444){\color[rgb]{0,0,0}\makebox(0,0)[lt]{\lineheight{1.25}\smash{\begin{tabular}[t]{l}$\to$\end{tabular}}}}%
    \put(0.45659722,0.0875){\color[rgb]{0,0,0}\makebox(0,0)[lt]{\lineheight{1.25}\smash{\begin{tabular}[t]{l}$\to$\end{tabular}}}}%
    \put(0.81423611,0.0875){\color[rgb]{0,0,0}\makebox(0,0)[lt]{\lineheight{1.25}\smash{\begin{tabular}[t]{l}$\to$\end{tabular}}}}%
    \put(0.81423611,0.28819444){\color[rgb]{0,0,0}\makebox(0,0)[lt]{\lineheight{1.25}\smash{\begin{tabular}[t]{l}$\to$\end{tabular}}}}%
    \put(0,0){\includegraphics[width=\unitlength,page=2]{manifold-diagrams-inductive-idea.pdf}}%
  \end{picture}%
\endgroup%

%% file: figuresused/meshes-cellulate-manifold-diagrams.pdf_tex
\begingroup%
  \makeatletter%
  \providecommand\color[2][]{%
    \errmessage{(Inkscape) Color is used for the text in Inkscape, but the package 'color.sty' is not loaded}%
    \renewcommand\color[2][]{}%
  }%
  \providecommand\transparent[1]{%
    \errmessage{(Inkscape) Transparency is used (non-zero) for the text in Inkscape, but the package 'transparent.sty' is not loaded}%
    \renewcommand\transparent[1]{}%
  }%
  \providecommand\rotatebox[2]{#2}%
  \newcommand*\fsize{\dimexpr\f@size pt\relax}%
  \newcommand*\lineheight[1]{\fontsize{\fsize}{#1\fsize}\selectfont}%
  \ifx\svgwidth\undefined%
    \setlength{\unitlength}{2160bp}%
    \ifx\svgscale\undefined%
      \relax%
    \else%
      \setlength{\unitlength}{\unitlength * \real{\svgscale}}%
    \fi%
  \else%
    \setlength{\unitlength}{\svgwidth}%
  \fi%
  \global\let\svgwidth\undefined%
  \global\let\svgscale\undefined%
  \makeatother%
  \begin{picture}(1,0.425)%
    \lineheight{1}%
    \setlength\tabcolsep{0pt}%
    \put(0,0){\includegraphics[width=\unitlength,page=1]{meshes-cellulate-manifold-diagrams.pdf}}%
    \put(0.34097222,0.33611111){\color[rgb]{0,0,0}\makebox(0,0)[lt]{\lineheight{1.25}\smash{\begin{tabular}[t]{l}$\leftsquigarrow$\end{tabular}}}}%
    \put(0.31493056,0.35625){\color[rgb]{0,0,0}\makebox(0,0)[lt]{\lineheight{1.25}\smash{\begin{tabular}[t]{l}{\tiny cellulates}\end{tabular}}}}%
    \put(0.34097222,0.11736111){\color[rgb]{0,0,0}\makebox(0,0)[lt]{\lineheight{1.25}\smash{\begin{tabular}[t]{l}$\leftsquigarrow$\end{tabular}}}}%
    \put(0.31493056,0.1375){\color[rgb]{0,0,0}\makebox(0,0)[lt]{\lineheight{1.25}\smash{\begin{tabular}[t]{l}{\tiny cellulates}\end{tabular}}}}%
    \put(0.42673611,0.02951389){\color[rgb]{0,0,0}\makebox(0,0)[lt]{\lineheight{1.25}\smash{\begin{tabular}[t]{l}$\underbracket[0.140ex]{\hspace{2.7in}}_{\text{\tiny 3-mesh}}$\end{tabular}}}}%
    \put(0.46145833,0.27118056){\color[rgb]{0,0,0}\makebox(0,0)[lt]{\lineheight{1.25}\smash{\begin{tabular}[t]{l}$\underbracket[0.140ex]{\hspace{1.65in}}_{\text{\tiny 2-mesh}}$\end{tabular}}}}%
    \put(0.61388889,0.33888889){\color[rgb]{0,0,0}\makebox(0,0)[lt]{\lineheight{1.25}\smash{\begin{tabular}[t]{l}$\to$\end{tabular}}}}%
    \put(0.69618056,0.33888889){\color[rgb]{0,0,0}\makebox(0,0)[lt]{\lineheight{1.25}\smash{\begin{tabular}[t]{l}$\to$\end{tabular}}}}%
    \put(0.85173611,0.12118056){\color[rgb]{0,0,0}\makebox(0,0)[lt]{\lineheight{1.25}\smash{\begin{tabular}[t]{l}$\to$\end{tabular}}}}%
    \put(0.76215278,0.12118056){\color[rgb]{0,0,0}\makebox(0,0)[lt]{\lineheight{1.25}\smash{\begin{tabular}[t]{l}$\to$\end{tabular}}}}%
    \put(0.63020833,0.12118056){\color[rgb]{0,0,0}\makebox(0,0)[lt]{\lineheight{1.25}\smash{\begin{tabular}[t]{l}$\to$\end{tabular}}}}%
    \put(0,0){\includegraphics[width=\unitlength,page=2]{meshes-cellulate-manifold-diagrams.pdf}}%
  \end{picture}%
\endgroup%

%% file: figuresused/framed-conicality-condition.pdf_tex
\begingroup%
  \makeatletter%
  \providecommand\color[2][]{%
    \errmessage{(Inkscape) Color is used for the text in Inkscape, but the package 'color.sty' is not loaded}%
    \renewcommand\color[2][]{}%
  }%
  \providecommand\transparent[1]{%
    \errmessage{(Inkscape) Transparency is used (non-zero) for the text in Inkscape, but the package 'transparent.sty' is not loaded}%
    \renewcommand\transparent[1]{}%
  }%
  \providecommand\rotatebox[2]{#2}%
  \newcommand*\fsize{\dimexpr\f@size pt\relax}%
  \newcommand*\lineheight[1]{\fontsize{\fsize}{#1\fsize}\selectfont}%
  \ifx\svgwidth\undefined%
    \setlength{\unitlength}{2160bp}%
    \ifx\svgscale\undefined%
      \relax%
    \else%
      \setlength{\unitlength}{\unitlength * \real{\svgscale}}%
    \fi%
  \else%
    \setlength{\unitlength}{\svgwidth}%
  \fi%
  \global\let\svgwidth\undefined%
  \global\let\svgscale\undefined%
  \makeatother%
  \begin{picture}(1,0.24305556)%
    \lineheight{1}%
    \setlength\tabcolsep{0pt}%
    \put(0,0){\includegraphics[width=\unitlength,page=1]{framed-conicality-condition.pdf}}%
    \put(0.64618056,0.12534722){\color[rgb]{0,0,0}\makebox(0,0)[lt]{\lineheight{1.25}\smash{\begin{tabular}[t]{l}$\iso$\end{tabular}}}}%
    \put(0.70694444,0.12534722){\color[rgb]{0,0,0}\makebox(0,0)[lt]{\lineheight{1.25}\smash{\begin{tabular}[t]{l}$\times$\end{tabular}}}}%
  \end{picture}%
\endgroup%

%% file: figuresused/a-manifold-diagram-and-its-dual-pasting-diagram.pdf_tex
\begingroup%
  \makeatletter%
  \providecommand\color[2][]{%
    \errmessage{(Inkscape) Color is used for the text in Inkscape, but the package 'color.sty' is not loaded}%
    \renewcommand\color[2][]{}%
  }%
  \providecommand\transparent[1]{%
    \errmessage{(Inkscape) Transparency is used (non-zero) for the text in Inkscape, but the package 'transparent.sty' is not loaded}%
    \renewcommand\transparent[1]{}%
  }%
  \providecommand\rotatebox[2]{#2}%
  \newcommand*\fsize{\dimexpr\f@size pt\relax}%
  \newcommand*\lineheight[1]{\fontsize{\fsize}{#1\fsize}\selectfont}%
  \ifx\svgwidth\undefined%
    \setlength{\unitlength}{2160bp}%
    \ifx\svgscale\undefined%
      \relax%
    \else%
      \setlength{\unitlength}{\unitlength * \real{\svgscale}}%
    \fi%
  \else%
    \setlength{\unitlength}{\svgwidth}%
  \fi%
  \global\let\svgwidth\undefined%
  \global\let\svgscale\undefined%
  \makeatother%
  \begin{picture}(1,0.14583333)%
    \lineheight{1}%
    \setlength\tabcolsep{0pt}%
    \put(0.26388889,0.05763889){\color[rgb]{0,0,0}\makebox(0,0)[lt]{\lineheight{1.25}\smash{\begin{tabular}[t]{l}$\leftsquigarrow$\end{tabular}}}}%
    \put(0.23784722,0.03611111){\color[rgb]{0,0,0}\makebox(0,0)[lt]{\lineheight{1.25}\smash{\begin{tabular}[t]{l}{\tiny cellulates}\end{tabular}}}}%
    \put(0.49166667,0.03611111){\color[rgb]{0,0,0}\makebox(0,0)[lt]{\lineheight{1.25}\smash{\begin{tabular}[t]{l}{\tiny pass to}\end{tabular}}}}%
    \put(0.48506944,0.01458333){\color[rgb]{0,0,0}\makebox(0,0)[lt]{\lineheight{1.25}\smash{\begin{tabular}[t]{l}{\tiny dual cells}\end{tabular}}}}%
    \put(0.70243056,0.03611111){\color[rgb]{0,0,0}\makebox(0,0)[lt]{\lineheight{1.25}\smash{\begin{tabular}[t]{l}{\tiny cellulates}\end{tabular}}}}%
    \put(0.72395833,0.05763889){\color[rgb]{0,0,0}\makebox(0,0)[lt]{\lineheight{1.25}\smash{\begin{tabular}[t]{l}$\rightsquigarrow$\end{tabular}}}}%
    \put(0.503125,0.05763889){\color[rgb]{0,0,0}\makebox(0,0)[lt]{\lineheight{1.25}\smash{\begin{tabular}[t]{l}$\leftrightsquigarrow$\end{tabular}}}}%
    \put(0,0){\includegraphics[width=\unitlength,page=1]{a-manifold-diagram-and-its-dual-pasting-diagram.pdf}}%
  \end{picture}%
\endgroup%

%% file: figuresused/stable-2-tangle-singularities-in-dimension-3.pdf_tex
\begingroup%
  \makeatletter%
  \providecommand\color[2][]{%
    \errmessage{(Inkscape) Color is used for the text in Inkscape, but the package 'color.sty' is not loaded}%
    \renewcommand\color[2][]{}%
  }%
  \providecommand\transparent[1]{%
    \errmessage{(Inkscape) Transparency is used (non-zero) for the text in Inkscape, but the package 'transparent.sty' is not loaded}%
    \renewcommand\transparent[1]{}%
  }%
  \providecommand\rotatebox[2]{#2}%
  \newcommand*\fsize{\dimexpr\f@size pt\relax}%
  \newcommand*\lineheight[1]{\fontsize{\fsize}{#1\fsize}\selectfont}%
  \ifx\svgwidth\undefined%
    \setlength{\unitlength}{2160bp}%
    \ifx\svgscale\undefined%
      \relax%
    \else%
      \setlength{\unitlength}{\unitlength * \real{\svgscale}}%
    \fi%
  \else%
    \setlength{\unitlength}{\svgwidth}%
  \fi%
  \global\let\svgwidth\undefined%
  \global\let\svgscale\undefined%
  \makeatother%
  \begin{picture}(1,0.19583333)%
    \lineheight{1}%
    \setlength\tabcolsep{0pt}%
    \put(0,0){\includegraphics[width=\unitlength,page=1]{stable-2-tangle-singularities-in-dimension-3.pdf}}%
  \end{picture}%
\endgroup%

%% file: figuresused/eckmann-hilton-vs-braid-simpler.pdf_tex
\begingroup%
  \makeatletter%
  \providecommand\color[2][]{%
    \errmessage{(Inkscape) Color is used for the text in Inkscape, but the package 'color.sty' is not loaded}%
    \renewcommand\color[2][]{}%
  }%
  \providecommand\transparent[1]{%
    \errmessage{(Inkscape) Transparency is used (non-zero) for the text in Inkscape, but the package 'transparent.sty' is not loaded}%
    \renewcommand\transparent[1]{}%
  }%
  \providecommand\rotatebox[2]{#2}%
  \newcommand*\fsize{\dimexpr\f@size pt\relax}%
  \newcommand*\lineheight[1]{\fontsize{\fsize}{#1\fsize}\selectfont}%
  \ifx\svgwidth\undefined%
    \setlength{\unitlength}{2160bp}%
    \ifx\svgscale\undefined%
      \relax%
    \else%
      \setlength{\unitlength}{\unitlength * \real{\svgscale}}%
    \fi%
  \else%
    \setlength{\unitlength}{\svgwidth}%
  \fi%
  \global\let\svgwidth\undefined%
  \global\let\svgscale\undefined%
  \makeatother%
  \begin{picture}(1,0.33784722)%
    \lineheight{1}%
    \setlength\tabcolsep{0pt}%
    \put(0,0){\includegraphics[width=\unitlength,page=1]{eckmann-hilton-vs-braid-simpler.pdf}}%
    \put(0.77604167,0.265625){\color[rgb]{0,0,0}\makebox(0,0)[lt]{\lineheight{1.25}\smash{\begin{tabular}[t]{l}$\alpha$\end{tabular}}}}%
    \put(0.70451389,0.23854167){\color[rgb]{0,0,0}\makebox(0,0)[lt]{\lineheight{1.25}\smash{\begin{tabular}[t]{l}$\beta$\end{tabular}}}}%
    \put(0.40277778,0.15034722){\color[rgb]{0,0,0}\makebox(0,0)[lt]{\lineheight{1.25}\smash{\begin{tabular}[t]{l}\rotatebox{-23}{\tiny $\alpha$}\end{tabular}}}}%
    \put(0.18506944,0.08506944){\color[rgb]{0,0,0}\makebox(0,0)[lt]{\lineheight{1.25}\smash{\begin{tabular}[t]{l}\rotatebox{-23}{\tiny $\alpha$}\end{tabular}}}}%
    \put(0.20034722,0.22986111){\color[rgb]{0,0,0}\makebox(0,0)[lt]{\lineheight{1.25}\smash{\begin{tabular}[t]{l}\rotatebox{-23}{\tiny $\alpha$}\end{tabular}}}}%
    \put(0.11736111,0.23263889){\color[rgb]{0,0,0}\makebox(0,0)[lt]{\lineheight{1.25}\smash{\begin{tabular}[t]{l}\rotatebox{-30}{\tiny $\beta$}\end{tabular}}}}%
    \put(0.3375,0.16979167){\color[rgb]{0,0,0}\makebox(0,0)[lt]{\lineheight{1.25}\smash{\begin{tabular}[t]{l}\rotatebox{-30}{\tiny $\beta$}\end{tabular}}}}%
    \put(0.13784722,0.12048611){\color[rgb]{0,0,0}\makebox(0,0)[lt]{\lineheight{1.25}\smash{\begin{tabular}[t]{l}\rotatebox{-30}{\tiny $\beta$}\end{tabular}}}}%
    \put(0.38541667,0.046875){\color[rgb]{0,0,0}\makebox(0,0)[lt]{\lineheight{1.25}\smash{\begin{tabular}[t]{l}\rotatebox{-23}{$\alpha$}\end{tabular}}}}%
    \put(0.35208333,0.01701389){\color[rgb]{0,0,0}\makebox(0,0)[lt]{\lineheight{1.25}\smash{\begin{tabular}[t]{l}\rotatebox{-30}{$\beta$}\end{tabular}}}}%
    \put(0.38263889,0.30729167){\color[rgb]{0,0,0}\makebox(0,0)[lt]{\lineheight{1.25}\smash{\begin{tabular}[t]{l}\rotatebox{-23}{$\alpha$}\end{tabular}}}}%
    \put(0.34930556,0.27743056){\color[rgb]{0,0,0}\makebox(0,0)[lt]{\lineheight{1.25}\smash{\begin{tabular}[t]{l}\rotatebox{-30}{$\beta$}\end{tabular}}}}%
    \put(0.25763889,0.275){\color[rgb]{0,0,0}\makebox(0,0)[lt]{\lineheight{1.25}\smash{\begin{tabular}[t]{l}{\tiny $\ev$}\end{tabular}}}}%
    \put(0.25868056,0.20833333){\color[rgb]{0,0,0}\makebox(0,0)[lt]{\lineheight{1.25}\smash{\begin{tabular}[t]{l}{\tiny $\ev$}\end{tabular}}}}%
    \put(0.25763889,0.13958333){\color[rgb]{0,0,0}\makebox(0,0)[lt]{\lineheight{1.25}\smash{\begin{tabular}[t]{l}{\tiny $\ev$}\end{tabular}}}}%
    \put(0.25763889,0.07708333){\color[rgb]{0,0,0}\makebox(0,0)[lt]{\lineheight{1.25}\smash{\begin{tabular}[t]{l}{\tiny $\ev$}\end{tabular}}}}%
    \put(0,0){\includegraphics[width=\unitlength,page=2]{eckmann-hilton-vs-braid-simpler.pdf}}%
  \end{picture}%
\endgroup%

%% file: figuresused/1-meshes-of-open-closed-and-mixed-boundary-type.pdf_tex
\begingroup%
  \makeatletter%
  \providecommand\color[2][]{%
    \errmessage{(Inkscape) Color is used for the text in Inkscape, but the package 'color.sty' is not loaded}%
    \renewcommand\color[2][]{}%
  }%
  \providecommand\transparent[1]{%
    \errmessage{(Inkscape) Transparency is used (non-zero) for the text in Inkscape, but the package 'transparent.sty' is not loaded}%
    \renewcommand\transparent[1]{}%
  }%
  \providecommand\rotatebox[2]{#2}%
  \newcommand*\fsize{\dimexpr\f@size pt\relax}%
  \newcommand*\lineheight[1]{\fontsize{\fsize}{#1\fsize}\selectfont}%
  \ifx\svgwidth\undefined%
    \setlength{\unitlength}{2160bp}%
    \ifx\svgscale\undefined%
      \relax%
    \else%
      \setlength{\unitlength}{\unitlength * \real{\svgscale}}%
    \fi%
  \else%
    \setlength{\unitlength}{\svgwidth}%
  \fi%
  \global\let\svgwidth\undefined%
  \global\let\svgscale\undefined%
  \makeatother%
  \begin{picture}(1,0.21354167)%
    \lineheight{1}%
    \setlength\tabcolsep{0pt}%
    \put(0.52256944,0.01493056){\color[rgb]{0,0,0}\makebox(0,0)[lt]{\lineheight{1.25}\smash{\begin{tabular}[t]{l}closed\end{tabular}}}}%
    \put(0,0){\includegraphics[width=\unitlength,page=1]{1-meshes-of-open-closed-and-mixed-boundary-type.pdf}}%
    \put(0.17534722,0.01493056){\color[rgb]{0,0,0}\makebox(0,0)[lt]{\lineheight{1.25}\smash{\begin{tabular}[t]{l}open\end{tabular}}}}%
    \put(0.82604167,0.01493056){\color[rgb]{0,0,0}\makebox(0,0)[lt]{\lineheight{1.25}\smash{\begin{tabular}[t]{l}mixed\end{tabular}}}}%
    \put(0,0){\includegraphics[width=\unitlength,page=2]{1-meshes-of-open-closed-and-mixed-boundary-type.pdf}}%
  \end{picture}%
\endgroup%

%% file: figuresused/1-mesh-bundles.pdf_tex
\begingroup%
  \makeatletter%
  \providecommand\color[2][]{%
    \errmessage{(Inkscape) Color is used for the text in Inkscape, but the package 'color.sty' is not loaded}%
    \renewcommand\color[2][]{}%
  }%
  \providecommand\transparent[1]{%
    \errmessage{(Inkscape) Transparency is used (non-zero) for the text in Inkscape, but the package 'transparent.sty' is not loaded}%
    \renewcommand\transparent[1]{}%
  }%
  \providecommand\rotatebox[2]{#2}%
  \newcommand*\fsize{\dimexpr\f@size pt\relax}%
  \newcommand*\lineheight[1]{\fontsize{\fsize}{#1\fsize}\selectfont}%
  \ifx\svgwidth\undefined%
    \setlength{\unitlength}{2160bp}%
    \ifx\svgscale\undefined%
      \relax%
    \else%
      \setlength{\unitlength}{\unitlength * \real{\svgscale}}%
    \fi%
  \else%
    \setlength{\unitlength}{\svgwidth}%
  \fi%
  \global\let\svgwidth\undefined%
  \global\let\svgscale\undefined%
  \makeatother%
  \begin{picture}(1,0.26145833)%
    \lineheight{1}%
    \setlength\tabcolsep{0pt}%
    \put(0,0){\includegraphics[width=\unitlength,page=1]{1-mesh-bundles.pdf}}%
    \put(0.39548611,0.18993056){\color[rgb]{0,0,0}\makebox(0,0)[lt]{\lineheight{1.25}\smash{\begin{tabular}[t]{l}{\tiny framing of}\end{tabular}}}}%
    \put(0.41701389,0.16909722){\color[rgb]{0,0,0}\makebox(0,0)[lt]{\lineheight{1.25}\smash{\begin{tabular}[t]{l}{\tiny fibers}\end{tabular}}}}%
    \put(0,0){\includegraphics[width=\unitlength,page=2]{1-mesh-bundles.pdf}}%
  \end{picture}%
\endgroup%

%% file: figuresused/fundamental-truss-of-our-earlier-1-mesh-examples.pdf_tex
\begingroup%
  \makeatletter%
  \providecommand\color[2][]{%
    \errmessage{(Inkscape) Color is used for the text in Inkscape, but the package 'color.sty' is not loaded}%
    \renewcommand\color[2][]{}%
  }%
  \providecommand\transparent[1]{%
    \errmessage{(Inkscape) Transparency is used (non-zero) for the text in Inkscape, but the package 'transparent.sty' is not loaded}%
    \renewcommand\transparent[1]{}%
  }%
  \providecommand\rotatebox[2]{#2}%
  \newcommand*\fsize{\dimexpr\f@size pt\relax}%
  \newcommand*\lineheight[1]{\fontsize{\fsize}{#1\fsize}\selectfont}%
  \ifx\svgwidth\undefined%
    \setlength{\unitlength}{2160bp}%
    \ifx\svgscale\undefined%
      \relax%
    \else%
      \setlength{\unitlength}{\unitlength * \real{\svgscale}}%
    \fi%
  \else%
    \setlength{\unitlength}{\svgwidth}%
  \fi%
  \global\let\svgwidth\undefined%
  \global\let\svgscale\undefined%
  \makeatother%
  \begin{picture}(1,0.19340278)%
    \lineheight{1}%
    \setlength\tabcolsep{0pt}%
    \put(0,0){\includegraphics[width=\unitlength,page=1]{fundamental-truss-of-our-earlier-1-mesh-examples.pdf}}%
  \end{picture}%
\endgroup%

%% file: figuresused/fundamental-1-truss-bundles-of-earlier-1-mesh-bundle.pdf_tex
\begingroup%
  \makeatletter%
  \providecommand\color[2][]{%
    \errmessage{(Inkscape) Color is used for the text in Inkscape, but the package 'color.sty' is not loaded}%
    \renewcommand\color[2][]{}%
  }%
  \providecommand\transparent[1]{%
    \errmessage{(Inkscape) Transparency is used (non-zero) for the text in Inkscape, but the package 'transparent.sty' is not loaded}%
    \renewcommand\transparent[1]{}%
  }%
  \providecommand\rotatebox[2]{#2}%
  \newcommand*\fsize{\dimexpr\f@size pt\relax}%
  \newcommand*\lineheight[1]{\fontsize{\fsize}{#1\fsize}\selectfont}%
  \ifx\svgwidth\undefined%
    \setlength{\unitlength}{2160bp}%
    \ifx\svgscale\undefined%
      \relax%
    \else%
      \setlength{\unitlength}{\unitlength * \real{\svgscale}}%
    \fi%
  \else%
    \setlength{\unitlength}{\svgwidth}%
  \fi%
  \global\let\svgwidth\undefined%
  \global\let\svgscale\undefined%
  \makeatother%
  \begin{picture}(1,0.29201389)%
    \lineheight{1}%
    \setlength\tabcolsep{0pt}%
    \put(0,0){\includegraphics[width=\unitlength,page=1]{fundamental-1-truss-bundles-of-earlier-1-mesh-bundle.pdf}}%
  \end{picture}%
\endgroup%

%% file: figuresused/dualization-of-mesh-maps.pdf_tex
\begingroup%
  \makeatletter%
  \providecommand\color[2][]{%
    \errmessage{(Inkscape) Color is used for the text in Inkscape, but the package 'color.sty' is not loaded}%
    \renewcommand\color[2][]{}%
  }%
  \providecommand\transparent[1]{%
    \errmessage{(Inkscape) Transparency is used (non-zero) for the text in Inkscape, but the package 'transparent.sty' is not loaded}%
    \renewcommand\transparent[1]{}%
  }%
  \providecommand\rotatebox[2]{#2}%
  \newcommand*\fsize{\dimexpr\f@size pt\relax}%
  \newcommand*\lineheight[1]{\fontsize{\fsize}{#1\fsize}\selectfont}%
  \ifx\svgwidth\undefined%
    \setlength{\unitlength}{2160bp}%
    \ifx\svgscale\undefined%
      \relax%
    \else%
      \setlength{\unitlength}{\unitlength * \real{\svgscale}}%
    \fi%
  \else%
    \setlength{\unitlength}{\svgwidth}%
  \fi%
  \global\let\svgwidth\undefined%
  \global\let\svgscale\undefined%
  \makeatother%
  \begin{picture}(1,0.38680556)%
    \lineheight{1}%
    \setlength\tabcolsep{0pt}%
    \put(0,0){\includegraphics[width=\unitlength,page=1]{dualization-of-mesh-maps.pdf}}%
    \put(0.32881944,0.34305556){\color[rgb]{0,0,0}\makebox(0,0)[lt]{\lineheight{1.25}\smash{\begin{tabular}[t]{l}cellular\end{tabular}}}}%
    \put(0.56979167,0.34305556){\color[rgb]{0,0,0}\makebox(0,0)[lt]{\lineheight{1.25}\smash{\begin{tabular}[t]{l}cocellular\end{tabular}}}}%
    \put(0.484375,0.36875){\color[rgb]{0,0,0}\makebox(0,0)[lt]{\lineheight{1.25}\smash{\begin{tabular}[t]{l}$\dagger$\end{tabular}}}}%
    \put(0.18402778,0.25277778){\color[rgb]{0,0,0}\makebox(0,0)[lt]{\lineheight{1.25}\smash{\begin{tabular}[t]{l}{\tiny ``face''}\end{tabular}}}}%
    \put(0.76041667,0.25277778){\color[rgb]{0,0,0}\makebox(0,0)[lt]{\lineheight{1.25}\smash{\begin{tabular}[t]{l}{\tiny ``embedding''}\end{tabular}}}}%
    \put(0.12152778,0.07916667){\color[rgb]{0,0,0}\makebox(0,0)[lt]{\lineheight{1.25}\smash{\begin{tabular}[t]{l}{\tiny ``degeneracy''}\end{tabular}}}}%
    \put(0.76041667,0.07916667){\color[rgb]{0,0,0}\makebox(0,0)[lt]{\lineheight{1.25}\smash{\begin{tabular}[t]{l}{\tiny ``coarsening''}\end{tabular}}}}%
    \put(0,0){\includegraphics[width=\unitlength,page=2]{dualization-of-mesh-maps.pdf}}%
  \end{picture}%
\endgroup%

%% file: figuresused/dualization-of-mesh-bundles.pdf_tex
\begingroup%
  \makeatletter%
  \providecommand\color[2][]{%
    \errmessage{(Inkscape) Color is used for the text in Inkscape, but the package 'color.sty' is not loaded}%
    \renewcommand\color[2][]{}%
  }%
  \providecommand\transparent[1]{%
    \errmessage{(Inkscape) Transparency is used (non-zero) for the text in Inkscape, but the package 'transparent.sty' is not loaded}%
    \renewcommand\transparent[1]{}%
  }%
  \providecommand\rotatebox[2]{#2}%
  \newcommand*\fsize{\dimexpr\f@size pt\relax}%
  \newcommand*\lineheight[1]{\fontsize{\fsize}{#1\fsize}\selectfont}%
  \ifx\svgwidth\undefined%
    \setlength{\unitlength}{2160bp}%
    \ifx\svgscale\undefined%
      \relax%
    \else%
      \setlength{\unitlength}{\unitlength * \real{\svgscale}}%
    \fi%
  \else%
    \setlength{\unitlength}{\svgwidth}%
  \fi%
  \global\let\svgwidth\undefined%
  \global\let\svgscale\undefined%
  \makeatother%
  \begin{picture}(1,0.26423611)%
    \lineheight{1}%
    \setlength\tabcolsep{0pt}%
    \put(0,0){\includegraphics[width=\unitlength,page=1]{dualization-of-mesh-bundles.pdf}}%
    \put(0.52256944,0.15868056){\color[rgb]{0,0,0}\makebox(0,0)[lt]{\lineheight{1.25}\smash{\begin{tabular}[t]{l}$\dagger$\end{tabular}}}}%
    \put(0,0){\includegraphics[width=\unitlength,page=2]{dualization-of-mesh-bundles.pdf}}%
  \end{picture}%
\endgroup%

%% file: figuresused/stratified-1-meshes-and-1-trusses.pdf_tex
\begingroup%
  \makeatletter%
  \providecommand\color[2][]{%
    \errmessage{(Inkscape) Color is used for the text in Inkscape, but the package 'color.sty' is not loaded}%
    \renewcommand\color[2][]{}%
  }%
  \providecommand\transparent[1]{%
    \errmessage{(Inkscape) Transparency is used (non-zero) for the text in Inkscape, but the package 'transparent.sty' is not loaded}%
    \renewcommand\transparent[1]{}%
  }%
  \providecommand\rotatebox[2]{#2}%
  \newcommand*\fsize{\dimexpr\f@size pt\relax}%
  \newcommand*\lineheight[1]{\fontsize{\fsize}{#1\fsize}\selectfont}%
  \ifx\svgwidth\undefined%
    \setlength{\unitlength}{2160bp}%
    \ifx\svgscale\undefined%
      \relax%
    \else%
      \setlength{\unitlength}{\unitlength * \real{\svgscale}}%
    \fi%
  \else%
    \setlength{\unitlength}{\svgwidth}%
  \fi%
  \global\let\svgwidth\undefined%
  \global\let\svgscale\undefined%
  \makeatother%
  \begin{picture}(1,0.13333333)%
    \lineheight{1}%
    \setlength\tabcolsep{0pt}%
    \put(0,0){\includegraphics[width=\unitlength,page=1]{stratified-1-meshes-and-1-trusses.pdf}}%
    \put(0.125,0.11041667){\color[rgb]{0,0,0}\makebox(0,0)[lt]{\lineheight{1.25}\smash{\begin{tabular}[t]{l}$M$\end{tabular}}}}%
    \put(0.30659722,0.06041667){\color[rgb]{0,0,0}\makebox(0,0)[lt]{\lineheight{1.25}\smash{\begin{tabular}[t]{l}{\tiny coarsen}\end{tabular}}}}%
    \put(0.12881944,0.00625){\color[rgb]{0,0,0}\makebox(0,0)[lt]{\lineheight{1.25}\smash{\begin{tabular}[t]{l}$f$\end{tabular}}}}%
    \put(0.78194444,0.00625){\color[rgb]{0,0,0}\makebox(0,0)[lt]{\lineheight{1.25}\smash{\begin{tabular}[t]{l}$\Entr(f)$\end{tabular}}}}%
    \put(0.80034722,0.11041667){\color[rgb]{0,0,0}\makebox(0,0)[lt]{\lineheight{1.25}\smash{\begin{tabular}[t]{l}$T$\end{tabular}}}}%
  \end{picture}%
\endgroup%

%% file: figuresused/open-and-closed-2-meshes-and-their-respective-2-trusses.pdf_tex
\begingroup%
  \makeatletter%
  \providecommand\color[2][]{%
    \errmessage{(Inkscape) Color is used for the text in Inkscape, but the package 'color.sty' is not loaded}%
    \renewcommand\color[2][]{}%
  }%
  \providecommand\transparent[1]{%
    \errmessage{(Inkscape) Transparency is used (non-zero) for the text in Inkscape, but the package 'transparent.sty' is not loaded}%
    \renewcommand\transparent[1]{}%
  }%
  \providecommand\rotatebox[2]{#2}%
  \newcommand*\fsize{\dimexpr\f@size pt\relax}%
  \newcommand*\lineheight[1]{\fontsize{\fsize}{#1\fsize}\selectfont}%
  \ifx\svgwidth\undefined%
    \setlength{\unitlength}{2160bp}%
    \ifx\svgscale\undefined%
      \relax%
    \else%
      \setlength{\unitlength}{\unitlength * \real{\svgscale}}%
    \fi%
  \else%
    \setlength{\unitlength}{\svgwidth}%
  \fi%
  \global\let\svgwidth\undefined%
  \global\let\svgscale\undefined%
  \makeatother%
  \begin{picture}(1,0.36006944)%
    \lineheight{1}%
    \setlength\tabcolsep{0pt}%
    \put(0,0){\includegraphics[width=\unitlength,page=1]{open-and-closed-2-meshes-and-their-respective-2-trusses.pdf}}%
    \put(0.11909722,0.02534722){\color[rgb]{0,0,0}\makebox(0,0)[lt]{\lineheight{1.25}\smash{\begin{tabular}[t]{l}$M$\end{tabular}}}}%
    \put(0.31180556,0.02534722){\color[rgb]{0,0,0}\makebox(0,0)[lt]{\lineheight{1.25}\smash{\begin{tabular}[t]{l}$\FTrs(M)$\end{tabular}}}}%
    \put(0.65173611,0.02534722){\color[rgb]{0,0,0}\makebox(0,0)[lt]{\lineheight{1.25}\smash{\begin{tabular}[t]{l}$N$\end{tabular}}}}%
    \put(0.82118056,0.02534722){\color[rgb]{0,0,0}\makebox(0,0)[lt]{\lineheight{1.25}\smash{\begin{tabular}[t]{l}$\FTrs(N)$\end{tabular}}}}%
  \end{picture}%
\endgroup%

%% file: figuresused/2-mesh-bundles-over-the-stratified-1-simplex.pdf_tex
\begingroup%
  \makeatletter%
  \providecommand\color[2][]{%
    \errmessage{(Inkscape) Color is used for the text in Inkscape, but the package 'color.sty' is not loaded}%
    \renewcommand\color[2][]{}%
  }%
  \providecommand\transparent[1]{%
    \errmessage{(Inkscape) Transparency is used (non-zero) for the text in Inkscape, but the package 'transparent.sty' is not loaded}%
    \renewcommand\transparent[1]{}%
  }%
  \providecommand\rotatebox[2]{#2}%
  \newcommand*\fsize{\dimexpr\f@size pt\relax}%
  \newcommand*\lineheight[1]{\fontsize{\fsize}{#1\fsize}\selectfont}%
  \ifx\svgwidth\undefined%
    \setlength{\unitlength}{2160bp}%
    \ifx\svgscale\undefined%
      \relax%
    \else%
      \setlength{\unitlength}{\unitlength * \real{\svgscale}}%
    \fi%
  \else%
    \setlength{\unitlength}{\svgwidth}%
  \fi%
  \global\let\svgwidth\undefined%
  \global\let\svgscale\undefined%
  \makeatother%
  \begin{picture}(1,0.36597222)%
    \lineheight{1}%
    \setlength\tabcolsep{0pt}%
    \put(0,0){\includegraphics[width=\unitlength,page=1]{2-mesh-bundles-over-the-stratified-1-simplex.pdf}}%
  \end{picture}%
\endgroup%

%% file: figuresused/mesh-coarsening.pdf_tex
\begingroup%
  \makeatletter%
  \providecommand\color[2][]{%
    \errmessage{(Inkscape) Color is used for the text in Inkscape, but the package 'color.sty' is not loaded}%
    \renewcommand\color[2][]{}%
  }%
  \providecommand\transparent[1]{%
    \errmessage{(Inkscape) Transparency is used (non-zero) for the text in Inkscape, but the package 'transparent.sty' is not loaded}%
    \renewcommand\transparent[1]{}%
  }%
  \providecommand\rotatebox[2]{#2}%
  \newcommand*\fsize{\dimexpr\f@size pt\relax}%
  \newcommand*\lineheight[1]{\fontsize{\fsize}{#1\fsize}\selectfont}%
  \ifx\svgwidth\undefined%
    \setlength{\unitlength}{2160bp}%
    \ifx\svgscale\undefined%
      \relax%
    \else%
      \setlength{\unitlength}{\unitlength * \real{\svgscale}}%
    \fi%
  \else%
    \setlength{\unitlength}{\svgwidth}%
  \fi%
  \global\let\svgwidth\undefined%
  \global\let\svgscale\undefined%
  \makeatother%
  \begin{picture}(1,0.28090278)%
    \lineheight{1}%
    \setlength\tabcolsep{0pt}%
    \put(0,0){\includegraphics[width=\unitlength,page=1]{mesh-coarsening.pdf}}%
    \put(0.29826389,0.17951389){\color[rgb]{0,0,0}\makebox(0,0)[lt]{\lineheight{1.25}\smash{\begin{tabular}[t]{l}$M$\end{tabular}}}}%
    \put(0.67951389,0.13784722){\color[rgb]{0,0,0}\makebox(0,0)[lt]{\lineheight{1.25}\smash{\begin{tabular}[t]{l}$f$\end{tabular}}}}%
    \put(0.29826389,0.00798611){\color[rgb]{0,0,0}\makebox(0,0)[lt]{\lineheight{1.25}\smash{\begin{tabular}[t]{l}$M'$\end{tabular}}}}%
    \put(0.4125,0.13784722){\color[rgb]{0,0,0}\makebox(0,0)[lt]{\lineheight{1.25}\smash{\begin{tabular}[t]{l}{\tiny coarsen}\end{tabular}}}}%
    \put(0,0){\includegraphics[width=\unitlength,page=2]{mesh-coarsening.pdf}}%
  \end{picture}%
\endgroup%

%% file: figuresused/cube-axis-labeling.pdf_tex
\begingroup%
  \makeatletter%
  \providecommand\color[2][]{%
    \errmessage{(Inkscape) Color is used for the text in Inkscape, but the package 'color.sty' is not loaded}%
    \renewcommand\color[2][]{}%
  }%
  \providecommand\transparent[1]{%
    \errmessage{(Inkscape) Transparency is used (non-zero) for the text in Inkscape, but the package 'transparent.sty' is not loaded}%
    \renewcommand\transparent[1]{}%
  }%
  \providecommand\rotatebox[2]{#2}%
  \newcommand*\fsize{\dimexpr\f@size pt\relax}%
  \newcommand*\lineheight[1]{\fontsize{\fsize}{#1\fsize}\selectfont}%
  \ifx\svgwidth\undefined%
    \setlength{\unitlength}{2160bp}%
    \ifx\svgscale\undefined%
      \relax%
    \else%
      \setlength{\unitlength}{\unitlength * \real{\svgscale}}%
    \fi%
  \else%
    \setlength{\unitlength}{\svgwidth}%
  \fi%
  \global\let\svgwidth\undefined%
  \global\let\svgscale\undefined%
  \makeatother%
  \begin{picture}(1,0.34722222)%
    \lineheight{1}%
    \setlength\tabcolsep{0pt}%
    \put(0,0){\includegraphics[width=\unitlength,page=1]{cube-axis-labeling.pdf}}%
    \put(0.55833333,0.328125){\color[rgb]{0,0,0}\makebox(0,0)[lt]{\lineheight{1.25}\smash{\begin{tabular}[t]{l}$\lR^2 = \lR \times \lR$\end{tabular}}}}%
    \put(0.56527778,0.27291667){\color[rgb]{0,0,0}\makebox(0,0)[lt]{\lineheight{1.25}\smash{\begin{tabular}[t]{l}$2$\end{tabular}}}}%
    \put(0.77465278,0.06875){\color[rgb]{0,0,0}\makebox(0,0)[lt]{\lineheight{1.25}\smash{\begin{tabular}[t]{l}$1$\end{tabular}}}}%
    \put(0,0){\includegraphics[width=\unitlength,page=2]{cube-axis-labeling.pdf}}%
    \put(0.45659722,0.19618056){\color[rgb]{0,0,0}\makebox(0,0)[lt]{\lineheight{1.25}\smash{\begin{tabular}[t]{l}$\xinto{\quad \gamma \quad}$\end{tabular}}}}%
    \put(0.27118056,0.290625){\color[rgb]{0,0,0}\makebox(0,0)[lt]{\lineheight{1.25}\smash{\begin{tabular}[t]{l}$X$\end{tabular}}}}%
  \end{picture}%
\endgroup%

%% file: figuresused/dualization-of-2-meshes.pdf_tex
\begingroup%
  \makeatletter%
  \providecommand\color[2][]{%
    \errmessage{(Inkscape) Color is used for the text in Inkscape, but the package 'color.sty' is not loaded}%
    \renewcommand\color[2][]{}%
  }%
  \providecommand\transparent[1]{%
    \errmessage{(Inkscape) Transparency is used (non-zero) for the text in Inkscape, but the package 'transparent.sty' is not loaded}%
    \renewcommand\transparent[1]{}%
  }%
  \providecommand\rotatebox[2]{#2}%
  \newcommand*\fsize{\dimexpr\f@size pt\relax}%
  \newcommand*\lineheight[1]{\fontsize{\fsize}{#1\fsize}\selectfont}%
  \ifx\svgwidth\undefined%
    \setlength{\unitlength}{2160bp}%
    \ifx\svgscale\undefined%
      \relax%
    \else%
      \setlength{\unitlength}{\unitlength * \real{\svgscale}}%
    \fi%
  \else%
    \setlength{\unitlength}{\svgwidth}%
  \fi%
  \global\let\svgwidth\undefined%
  \global\let\svgscale\undefined%
  \makeatother%
  \begin{picture}(1,0.26388889)%
    \lineheight{1}%
    \setlength\tabcolsep{0pt}%
    \put(0,0){\includegraphics[width=\unitlength,page=1]{dualization-of-2-meshes.pdf}}%
    \put(0.484375,0.22083333){\color[rgb]{0,0,0}\makebox(0,0)[lt]{\lineheight{1.25}\smash{\begin{tabular}[t]{l}$\dagger$\end{tabular}}}}%
    \put(0,0){\includegraphics[width=\unitlength,page=2]{dualization-of-2-meshes.pdf}}%
  \end{picture}%
\endgroup%

%% file: figuresused/combinatorializing-tame-stratifications.pdf_tex
\begingroup%
  \makeatletter%
  \providecommand\color[2][]{%
    \errmessage{(Inkscape) Color is used for the text in Inkscape, but the package 'color.sty' is not loaded}%
    \renewcommand\color[2][]{}%
  }%
  \providecommand\transparent[1]{%
    \errmessage{(Inkscape) Transparency is used (non-zero) for the text in Inkscape, but the package 'transparent.sty' is not loaded}%
    \renewcommand\transparent[1]{}%
  }%
  \providecommand\rotatebox[2]{#2}%
  \newcommand*\fsize{\dimexpr\f@size pt\relax}%
  \newcommand*\lineheight[1]{\fontsize{\fsize}{#1\fsize}\selectfont}%
  \ifx\svgwidth\undefined%
    \setlength{\unitlength}{2160bp}%
    \ifx\svgscale\undefined%
      \relax%
    \else%
      \setlength{\unitlength}{\unitlength * \real{\svgscale}}%
    \fi%
  \else%
    \setlength{\unitlength}{\svgwidth}%
  \fi%
  \global\let\svgwidth\undefined%
  \global\let\svgscale\undefined%
  \makeatother%
  \begin{picture}(1,0.26597222)%
    \lineheight{1}%
    \setlength\tabcolsep{0pt}%
    \put(0,0){\includegraphics[width=\unitlength,page=1]{combinatorializing-tame-stratifications.pdf}}%
    \put(0.17916667,0.10208333){\color[rgb]{0,0,0}\makebox(0,0)[lt]{\lineheight{1.25}\smash{\begin{tabular}[t]{l}$f$\end{tabular}}}}%
    \put(0.25798611,0.17881944){\color[rgb]{0,0,0}\makebox(0,0)[lt]{\lineheight{1.25}\smash{\begin{tabular}[t]{l}$\xot{\text{refines}}$\end{tabular}}}}%
    \put(0.45590278,0.10208333){\color[rgb]{0,0,0}\makebox(0,0)[lt]{\lineheight{1.25}\smash{\begin{tabular}[t]{l}$\iM^f$\end{tabular}}}}%
    \put(0.55972222,0.10208333){\color[rgb]{0,0,0}\makebox(0,0)[lt]{\lineheight{1.25}\smash{\begin{tabular}[t]{l}$T$\end{tabular}}}}%
    \put(0.75729167,0.17881944){\color[rgb]{0,0,0}\makebox(0,0)[lt]{\lineheight{1.25}\smash{\begin{tabular}[t]{l}$\xto{\quad g \quad}$\end{tabular}}}}%
    \put(0.39479167,0.09201389){\color[rgb]{0,0,0}\makebox(0,0)[lt]{\lineheight{1.25}\smash{\begin{tabular}[t]{l}$\downarrow$\end{tabular}}}}%
    \put(0.39479167,0.02847222){\color[rgb]{0,0,0}\makebox(0,0)[lt]{\lineheight{1.25}\smash{\begin{tabular}[t]{l}$\downarrow$\end{tabular}}}}%
    \put(0.65381944,0.09201389){\color[rgb]{0,0,0}\makebox(0,0)[lt]{\lineheight{1.25}\smash{\begin{tabular}[t]{l}$\downarrow$\end{tabular}}}}%
    \put(0.65381944,0.02847222){\color[rgb]{0,0,0}\makebox(0,0)[lt]{\lineheight{1.25}\smash{\begin{tabular}[t]{l}$\downarrow$\end{tabular}}}}%
    \put(0.83159722,0.10208333){\color[rgb]{0,0,0}\makebox(0,0)[lt]{\lineheight{1.25}\smash{\begin{tabular}[t]{l}$\Entr(f)$\end{tabular}}}}%
    \put(0,0){\includegraphics[width=\unitlength,page=2]{combinatorializing-tame-stratifications.pdf}}%
  \end{picture}%
\endgroup%

%% file: figuresused/framed-conicality-condition-and-failure.pdf_tex
\begingroup%
  \makeatletter%
  \providecommand\color[2][]{%
    \errmessage{(Inkscape) Color is used for the text in Inkscape, but the package 'color.sty' is not loaded}%
    \renewcommand\color[2][]{}%
  }%
  \providecommand\transparent[1]{%
    \errmessage{(Inkscape) Transparency is used (non-zero) for the text in Inkscape, but the package 'transparent.sty' is not loaded}%
    \renewcommand\transparent[1]{}%
  }%
  \providecommand\rotatebox[2]{#2}%
  \newcommand*\fsize{\dimexpr\f@size pt\relax}%
  \newcommand*\lineheight[1]{\fontsize{\fsize}{#1\fsize}\selectfont}%
  \ifx\svgwidth\undefined%
    \setlength{\unitlength}{2160bp}%
    \ifx\svgscale\undefined%
      \relax%
    \else%
      \setlength{\unitlength}{\unitlength * \real{\svgscale}}%
    \fi%
  \else%
    \setlength{\unitlength}{\svgwidth}%
  \fi%
  \global\let\svgwidth\undefined%
  \global\let\svgscale\undefined%
  \makeatother%
  \begin{picture}(1,0.22291667)%
    \lineheight{1}%
    \setlength\tabcolsep{0pt}%
    \put(0,0){\includegraphics[width=\unitlength,page=1]{framed-conicality-condition-and-failure.pdf}}%
    \put(0.16840278,0.01701389){\color[rgb]{0,0,0}\makebox(0,0)[lt]{\lineheight{1.25}\smash{\begin{tabular}[t]{l}$\iso$\end{tabular}}}}%
    \put(0,0){\includegraphics[width=\unitlength,page=2]{framed-conicality-condition-and-failure.pdf}}%
    \put(0.22569444,0.01701389){\color[rgb]{0,0,0}\makebox(0,0)[lt]{\lineheight{1.25}\smash{\begin{tabular}[t]{l}$\times$\end{tabular}}}}%
    \put(0,0){\includegraphics[width=\unitlength,page=3]{framed-conicality-condition-and-failure.pdf}}%
  \end{picture}%
\endgroup%

%% file: figuresused/manifold-diagrams-in-dim-4.pdf_tex
\begingroup%
  \makeatletter%
  \providecommand\color[2][]{%
    \errmessage{(Inkscape) Color is used for the text in Inkscape, but the package 'color.sty' is not loaded}%
    \renewcommand\color[2][]{}%
  }%
  \providecommand\transparent[1]{%
    \errmessage{(Inkscape) Transparency is used (non-zero) for the text in Inkscape, but the package 'transparent.sty' is not loaded}%
    \renewcommand\transparent[1]{}%
  }%
  \providecommand\rotatebox[2]{#2}%
  \newcommand*\fsize{\dimexpr\f@size pt\relax}%
  \newcommand*\lineheight[1]{\fontsize{\fsize}{#1\fsize}\selectfont}%
  \ifx\svgwidth\undefined%
    \setlength{\unitlength}{2161.5bp}%
    \ifx\svgscale\undefined%
      \relax%
    \else%
      \setlength{\unitlength}{\unitlength * \real{\svgscale}}%
    \fi%
  \else%
    \setlength{\unitlength}{\svgwidth}%
  \fi%
  \global\let\svgwidth\undefined%
  \global\let\svgscale\undefined%
  \makeatother%
  \begin{picture}(1,0.40804997)%
    \lineheight{1}%
    \setlength\tabcolsep{0pt}%
    \put(0,0){\includegraphics[width=\unitlength,page=1]{manifold-diagrams-in-dim-4.pdf}}%
    \put(0.04337266,0.39347675){\color[rgb]{0,0,0}\makebox(0,0)[lt]{\lineheight{1.25}\smash{\begin{tabular}[t]{l}$t_0$\end{tabular}}}}%
    \put(0.25156142,0.39347675){\color[rgb]{0,0,0}\makebox(0,0)[lt]{\lineheight{1.25}\smash{\begin{tabular}[t]{l}$t_1$\end{tabular}}}}%
    \put(0.45975017,0.39347675){\color[rgb]{0,0,0}\makebox(0,0)[lt]{\lineheight{1.25}\smash{\begin{tabular}[t]{l}$t_2$\end{tabular}}}}%
    \put(0.66759195,0.39347675){\color[rgb]{0,0,0}\makebox(0,0)[lt]{\lineheight{1.25}\smash{\begin{tabular}[t]{l}$t_3$\end{tabular}}}}%
    \put(0.87612769,0.39347675){\color[rgb]{0,0,0}\makebox(0,0)[lt]{\lineheight{1.25}\smash{\begin{tabular}[t]{l}$t_4$\end{tabular}}}}%
    \put(0,0){\includegraphics[width=\unitlength,page=2]{manifold-diagrams-in-dim-4.pdf}}%
  \end{picture}%
\endgroup%

%% file: figuresused/a-non-tame-diagram.pdf_tex
\begingroup%
  \makeatletter%
  \providecommand\color[2][]{%
    \errmessage{(Inkscape) Color is used for the text in Inkscape, but the package 'color.sty' is not loaded}%
    \renewcommand\color[2][]{}%
  }%
  \providecommand\transparent[1]{%
    \errmessage{(Inkscape) Transparency is used (non-zero) for the text in Inkscape, but the package 'transparent.sty' is not loaded}%
    \renewcommand\transparent[1]{}%
  }%
  \providecommand\rotatebox[2]{#2}%
  \newcommand*\fsize{\dimexpr\f@size pt\relax}%
  \newcommand*\lineheight[1]{\fontsize{\fsize}{#1\fsize}\selectfont}%
  \ifx\svgwidth\undefined%
    \setlength{\unitlength}{2160bp}%
    \ifx\svgscale\undefined%
      \relax%
    \else%
      \setlength{\unitlength}{\unitlength * \real{\svgscale}}%
    \fi%
  \else%
    \setlength{\unitlength}{\svgwidth}%
  \fi%
  \global\let\svgwidth\undefined%
  \global\let\svgscale\undefined%
  \makeatother%
  \begin{picture}(1,0.175)%
    \lineheight{1}%
    \setlength\tabcolsep{0pt}%
    \put(0,0){\includegraphics[width=\unitlength,page=1]{a-non-tame-diagram.pdf}}%
  \end{picture}%
\endgroup%

%% file: figuresused/stratified-truss-neighborhoods.pdf_tex
\begingroup%
  \makeatletter%
  \providecommand\color[2][]{%
    \errmessage{(Inkscape) Color is used for the text in Inkscape, but the package 'color.sty' is not loaded}%
    \renewcommand\color[2][]{}%
  }%
  \providecommand\transparent[1]{%
    \errmessage{(Inkscape) Transparency is used (non-zero) for the text in Inkscape, but the package 'transparent.sty' is not loaded}%
    \renewcommand\transparent[1]{}%
  }%
  \providecommand\rotatebox[2]{#2}%
  \newcommand*\fsize{\dimexpr\f@size pt\relax}%
  \newcommand*\lineheight[1]{\fontsize{\fsize}{#1\fsize}\selectfont}%
  \ifx\svgwidth\undefined%
    \setlength{\unitlength}{2160bp}%
    \ifx\svgscale\undefined%
      \relax%
    \else%
      \setlength{\unitlength}{\unitlength * \real{\svgscale}}%
    \fi%
  \else%
    \setlength{\unitlength}{\svgwidth}%
  \fi%
  \global\let\svgwidth\undefined%
  \global\let\svgscale\undefined%
  \makeatother%
  \begin{picture}(1,0.33263889)%
    \lineheight{1}%
    \setlength\tabcolsep{0pt}%
    \put(0,0){\includegraphics[width=\unitlength,page=1]{stratified-truss-neighborhoods.pdf}}%
    \put(0.26319444,0.08541667){\color[rgb]{0,0,0}\makebox(0,0)[lt]{\lineheight{1.25}\smash{\begin{tabular}[t]{l}$f^{\leq x}$\end{tabular}}}}%
    \put(0.27256944,0.26458333){\color[rgb]{0,0,0}\makebox(0,0)[lt]{\lineheight{1.25}\smash{\begin{tabular}[t]{l}$f$\end{tabular}}}}%
    \put(0.02569444,0.01805556){\color[rgb]{0,0,0}\makebox(0,0)[lt]{\lineheight{1.25}\smash{\begin{tabular}[t]{l}$T^{\leq x}$\end{tabular}}}}%
    \put(0.36284722,0.28888889){\color[rgb]{0,0,0}\makebox(0,0)[lt]{\lineheight{1.25}\smash{\begin{tabular}[t]{l}$\Entr(f)$\end{tabular}}}}%
    \put(0.37430556,0.01805556){\color[rgb]{0,0,0}\makebox(0,0)[lt]{\lineheight{1.25}\smash{\begin{tabular}[t]{l}$\Entr(f^{\leq x})$\end{tabular}}}}%
    \put(0.88819444,0.28402778){\color[rgb]{0,0,0}\makebox(0,0)[lt]{\lineheight{1.25}\smash{\begin{tabular}[t]{l}$\Entr(g)$\end{tabular}}}}%
    \put(0.88819444,0.01805556){\color[rgb]{0,0,0}\makebox(0,0)[lt]{\lineheight{1.25}\smash{\begin{tabular}[t]{l}$\Entr(g^{\leq y})$\end{tabular}}}}%
    \put(0.72604167,0.07847222){\color[rgb]{0,0,0}\makebox(0,0)[lt]{\lineheight{1.25}\smash{\begin{tabular}[t]{l}$g^{\leq x}$\end{tabular}}}}%
    \put(0.02951389,0.30347222){\color[rgb]{0,0,0}\makebox(0,0)[lt]{\lineheight{1.25}\smash{\begin{tabular}[t]{l}$T$\end{tabular}}}}%
    \put(0.58229167,0.19305556){\color[rgb]{0,0,0}\makebox(0,0)[lt]{\lineheight{1.25}\smash{\begin{tabular}[t]{l}$y$\end{tabular}}}}%
    \put(0.78298611,0.275){\color[rgb]{0,0,0}\makebox(0,0)[lt]{\lineheight{1.25}\smash{\begin{tabular}[t]{l}$g$\end{tabular}}}}%
    \put(0.57951389,0.28888889){\color[rgb]{0,0,0}\makebox(0,0)[lt]{\lineheight{1.25}\smash{\begin{tabular}[t]{l}$S$\end{tabular}}}}%
    \put(0.20034722,0.16527778){\color[rgb]{0,0,0}\makebox(0,0)[lt]{\lineheight{1.25}\smash{\begin{tabular}[t]{l}$x$\end{tabular}}}}%
    \put(0,0){\includegraphics[width=\unitlength,page=2]{stratified-truss-neighborhoods.pdf}}%
  \end{picture}%
\endgroup%

%% file: figuresused/truss-cones.pdf_tex
\begingroup%
  \makeatletter%
  \providecommand\color[2][]{%
    \errmessage{(Inkscape) Color is used for the text in Inkscape, but the package 'color.sty' is not loaded}%
    \renewcommand\color[2][]{}%
  }%
  \providecommand\transparent[1]{%
    \errmessage{(Inkscape) Transparency is used (non-zero) for the text in Inkscape, but the package 'transparent.sty' is not loaded}%
    \renewcommand\transparent[1]{}%
  }%
  \providecommand\rotatebox[2]{#2}%
  \newcommand*\fsize{\dimexpr\f@size pt\relax}%
  \newcommand*\lineheight[1]{\fontsize{\fsize}{#1\fsize}\selectfont}%
  \ifx\svgwidth\undefined%
    \setlength{\unitlength}{2160bp}%
    \ifx\svgscale\undefined%
      \relax%
    \else%
      \setlength{\unitlength}{\unitlength * \real{\svgscale}}%
    \fi%
  \else%
    \setlength{\unitlength}{\svgwidth}%
  \fi%
  \global\let\svgwidth\undefined%
  \global\let\svgscale\undefined%
  \makeatother%
  \begin{picture}(1,0.10173611)%
    \lineheight{1}%
    \setlength\tabcolsep{0pt}%
    \put(0,0){\includegraphics[width=\unitlength,page=1]{truss-cones.pdf}}%
  \end{picture}%
\endgroup%

%% file: figuresused/combinatorial-manifold-2-diagram.pdf_tex
\begingroup%
  \makeatletter%
  \providecommand\color[2][]{%
    \errmessage{(Inkscape) Color is used for the text in Inkscape, but the package 'color.sty' is not loaded}%
    \renewcommand\color[2][]{}%
  }%
  \providecommand\transparent[1]{%
    \errmessage{(Inkscape) Transparency is used (non-zero) for the text in Inkscape, but the package 'transparent.sty' is not loaded}%
    \renewcommand\transparent[1]{}%
  }%
  \providecommand\rotatebox[2]{#2}%
  \newcommand*\fsize{\dimexpr\f@size pt\relax}%
  \newcommand*\lineheight[1]{\fontsize{\fsize}{#1\fsize}\selectfont}%
  \ifx\svgwidth\undefined%
    \setlength{\unitlength}{2160bp}%
    \ifx\svgscale\undefined%
      \relax%
    \else%
      \setlength{\unitlength}{\unitlength * \real{\svgscale}}%
    \fi%
  \else%
    \setlength{\unitlength}{\svgwidth}%
  \fi%
  \global\let\svgwidth\undefined%
  \global\let\svgscale\undefined%
  \makeatother%
  \begin{picture}(1,0.28055556)%
    \lineheight{1}%
    \setlength\tabcolsep{0pt}%
    \put(0,0){\includegraphics[width=\unitlength,page=1]{combinatorial-manifold-2-diagram.pdf}}%
    \put(0.12881944,0.16180556){\color[rgb]{0,0,0}\makebox(0,0)[lt]{\lineheight{1.25}\smash{\begin{tabular}[t]{l}$f$\end{tabular}}}}%
    \put(0.12465278,0.25069444){\color[rgb]{0,0,0}\makebox(0,0)[lt]{\lineheight{1.25}\smash{\begin{tabular}[t]{l}$T$\end{tabular}}}}%
    \put(0.34340278,0.26597222){\color[rgb]{0,0,0}\makebox(0,0)[lt]{\lineheight{1.25}\smash{\begin{tabular}[t]{l}$x$\end{tabular}}}}%
    \put(0.46423611,0.25069444){\color[rgb]{0,0,0}\makebox(0,0)[lt]{\lineheight{1.25}\smash{\begin{tabular}[t]{l}$(T^{\leq x}, f^{\leq x})$\end{tabular}}}}%
    \put(0.75208333,0.25069444){\color[rgb]{0,0,0}\makebox(0,0)[lt]{\lineheight{1.25}\smash{\begin{tabular}[t]{l}$\NF{T^{\leq x}, f^{\leq x}}$\end{tabular}}}}%
    \put(0,0){\includegraphics[width=\unitlength,page=2]{combinatorial-manifold-2-diagram.pdf}}%
  \end{picture}%
\endgroup%

%% file: figuresused/combinatorial-manifold-3-diagram.pdf_tex
\begingroup%
  \makeatletter%
  \providecommand\color[2][]{%
    \errmessage{(Inkscape) Color is used for the text in Inkscape, but the package 'color.sty' is not loaded}%
    \renewcommand\color[2][]{}%
  }%
  \providecommand\transparent[1]{%
    \errmessage{(Inkscape) Transparency is used (non-zero) for the text in Inkscape, but the package 'transparent.sty' is not loaded}%
    \renewcommand\transparent[1]{}%
  }%
  \providecommand\rotatebox[2]{#2}%
  \newcommand*\fsize{\dimexpr\f@size pt\relax}%
  \newcommand*\lineheight[1]{\fontsize{\fsize}{#1\fsize}\selectfont}%
  \ifx\svgwidth\undefined%
    \setlength{\unitlength}{2160bp}%
    \ifx\svgscale\undefined%
      \relax%
    \else%
      \setlength{\unitlength}{\unitlength * \real{\svgscale}}%
    \fi%
  \else%
    \setlength{\unitlength}{\svgwidth}%
  \fi%
  \global\let\svgwidth\undefined%
  \global\let\svgscale\undefined%
  \makeatother%
  \begin{picture}(1,0.30868056)%
    \lineheight{1}%
    \setlength\tabcolsep{0pt}%
    \put(0,0){\includegraphics[width=\unitlength,page=1]{combinatorial-manifold-3-diagram.pdf}}%
    \put(0.07951389,0.10173611){\color[rgb]{0,0,0}\makebox(0,0)[lt]{\lineheight{1.25}\smash{\begin{tabular}[t]{l}$\Entr(f)$\end{tabular}}}}%
    \put(0.20694444,0.16770833){\color[rgb]{0,0,0}\makebox(0,0)[lt]{\lineheight{1.25}\smash{\begin{tabular}[t]{l}$f$\end{tabular}}}}%
    \put(0.228125,0.08576389){\color[rgb]{0,0,0}\makebox(0,0)[lt]{\lineheight{1.25}\smash{\begin{tabular}[t]{l}$x$\end{tabular}}}}%
    \put(0.20451389,0.27881944){\color[rgb]{0,0,0}\makebox(0,0)[lt]{\lineheight{1.25}\smash{\begin{tabular}[t]{l}$T$\end{tabular}}}}%
    \put(0,0){\includegraphics[width=\unitlength,page=2]{combinatorial-manifold-3-diagram.pdf}}%
  \end{picture}%
\endgroup%

%% file: figuresused/normalization-of-a-neighborhood.pdf_tex
\begingroup%
  \makeatletter%
  \providecommand\color[2][]{%
    \errmessage{(Inkscape) Color is used for the text in Inkscape, but the package 'color.sty' is not loaded}%
    \renewcommand\color[2][]{}%
  }%
  \providecommand\transparent[1]{%
    \errmessage{(Inkscape) Transparency is used (non-zero) for the text in Inkscape, but the package 'transparent.sty' is not loaded}%
    \renewcommand\transparent[1]{}%
  }%
  \providecommand\rotatebox[2]{#2}%
  \newcommand*\fsize{\dimexpr\f@size pt\relax}%
  \newcommand*\lineheight[1]{\fontsize{\fsize}{#1\fsize}\selectfont}%
  \ifx\svgwidth\undefined%
    \setlength{\unitlength}{2160bp}%
    \ifx\svgscale\undefined%
      \relax%
    \else%
      \setlength{\unitlength}{\unitlength * \real{\svgscale}}%
    \fi%
  \else%
    \setlength{\unitlength}{\svgwidth}%
  \fi%
  \global\let\svgwidth\undefined%
  \global\let\svgscale\undefined%
  \makeatother%
  \begin{picture}(1,0.50451389)%
    \lineheight{1}%
    \setlength\tabcolsep{0pt}%
    \put(0,0){\includegraphics[width=\unitlength,page=1]{normalization-of-a-neighborhood.pdf}}%
    \put(0.14236111,0.09270833){\color[rgb]{0,0,0}\makebox(0,0)[lt]{\lineheight{1.25}\smash{\begin{tabular}[t]{l}$\NF{T^{\leq x},f^{\leq x}}$\end{tabular}}}}%
    \put(0.20173611,0.17743056){\color[rgb]{0,0,0}\makebox(0,0)[lt]{\lineheight{1.25}\smash{\begin{tabular}[t]{l}{\tiny normalize}\end{tabular}}}}%
    \put(0.14236111,0.415625){\color[rgb]{0,0,0}\makebox(0,0)[lt]{\lineheight{1.25}\smash{\begin{tabular}[t]{l}$(T^{\leq x},f^{\leq x})$\end{tabular}}}}%
  \end{picture}%
\endgroup%

%% file: figuresused/a-compact-combinatorial-manifold-diagrams.pdf_tex
\begingroup%
  \makeatletter%
  \providecommand\color[2][]{%
    \errmessage{(Inkscape) Color is used for the text in Inkscape, but the package 'color.sty' is not loaded}%
    \renewcommand\color[2][]{}%
  }%
  \providecommand\transparent[1]{%
    \errmessage{(Inkscape) Transparency is used (non-zero) for the text in Inkscape, but the package 'transparent.sty' is not loaded}%
    \renewcommand\transparent[1]{}%
  }%
  \providecommand\rotatebox[2]{#2}%
  \newcommand*\fsize{\dimexpr\f@size pt\relax}%
  \newcommand*\lineheight[1]{\fontsize{\fsize}{#1\fsize}\selectfont}%
  \ifx\svgwidth\undefined%
    \setlength{\unitlength}{2160bp}%
    \ifx\svgscale\undefined%
      \relax%
    \else%
      \setlength{\unitlength}{\unitlength * \real{\svgscale}}%
    \fi%
  \else%
    \setlength{\unitlength}{\svgwidth}%
  \fi%
  \global\let\svgwidth\undefined%
  \global\let\svgscale\undefined%
  \makeatother%
  \begin{picture}(1,0.52951389)%
    \lineheight{1}%
    \setlength\tabcolsep{0pt}%
    \put(0,0){\includegraphics[width=\unitlength,page=1]{a-compact-combinatorial-manifold-diagrams.pdf}}%
    \put(0.00868056,0.16666667){\color[rgb]{0,0,0}\makebox(0,0)[lt]{\lineheight{1.25}\smash{\begin{tabular}[t]{l}$(T,f)$\end{tabular}}}}%
    \put(0.01631944,0.42708333){\color[rgb]{0,0,0}\makebox(0,0)[lt]{\lineheight{1.25}\smash{\begin{tabular}[t]{l}$x$\end{tabular}}}}%
    \put(0.14444444,0.42708333){\color[rgb]{0,0,0}\makebox(0,0)[lt]{\lineheight{1.25}\smash{\begin{tabular}[t]{l}$y$\end{tabular}}}}%
    \put(0.434375,0.33263889){\color[rgb]{0,0,0}\makebox(0,0)[lt]{\lineheight{1.25}\smash{\begin{tabular}[t]{l}$(T^{\leq x},f^{\leq x}) = \NF{T^{\leq x},f^{\leq x}}$\end{tabular}}}}%
    \put(0.13506944,0.00833333){\color[rgb]{0,0,0}\makebox(0,0)[lt]{\lineheight{1.25}\smash{\begin{tabular}[t]{l}$(T^{\leq y},f^{\leq y})$\end{tabular}}}}%
    \put(0.58645833,0.00833333){\color[rgb]{0,0,0}\makebox(0,0)[lt]{\lineheight{1.25}\smash{\begin{tabular}[t]{l}$\NF{T^{\leq y},f^{\leq y}}$\end{tabular}}}}%
    \put(0.21875,0.33888889){\color[rgb]{0,0,0}\makebox(0,0)[lt]{\lineheight{1.25}\smash{\begin{tabular}[t]{l}$\xto{f}$\end{tabular}}}}%
    \put(0.43576389,0.48993056){\color[rgb]{0,0,0}\makebox(0,0)[lt]{\lineheight{1.25}\smash{\begin{tabular}[t]{l}$\xto{f^{\leq x}}$\end{tabular}}}}%
    \put(0.34583333,0.15729167){\color[rgb]{0,0,0}\makebox(0,0)[lt]{\lineheight{1.25}\smash{\begin{tabular}[t]{l}$\xto{f^{\leq y}}$\end{tabular}}}}%
    \put(0,0){\includegraphics[width=\unitlength,page=2]{a-compact-combinatorial-manifold-diagrams.pdf}}%
    \put(0.40486111,0.06909722){\color[rgb]{0,0,0}\makebox(0,0)[lt]{\lineheight{1.25}\smash{\begin{tabular}[t]{l}{\tiny normalize}\end{tabular}}}}%
    \put(0,0){\includegraphics[width=\unitlength,page=3]{a-compact-combinatorial-manifold-diagrams.pdf}}%
    \put(0.78680556,0.49166667){\color[rgb]{0,0,0}\makebox(0,0)[lt]{\lineheight{1.25}\smash{\begin{tabular}[t]{l}$\xto{f^{\leq x}}$\end{tabular}}}}%
    \put(0,0){\includegraphics[width=\unitlength,page=4]{a-compact-combinatorial-manifold-diagrams.pdf}}%
    \put(0.64236111,0.09409722){\color[rgb]{0,0,0}\makebox(0,0)[lt]{\lineheight{1.25}\smash{\begin{tabular}[t]{l}$= \quad \TT^\emptyset \times \TT^+ \times$\end{tabular}}}}%
    \put(0.57881944,0.45833333){\color[rgb]{0,0,0}\makebox(0,0)[lt]{\lineheight{1.25}\smash{\begin{tabular}[t]{l}$= \quad \TT^- \times$\end{tabular}}}}%
    \put(0.89618056,0.45138889){\color[rgb]{0,0,0}\makebox(0,0)[lt]{\lineheight{1.25}\smash{\begin{tabular}[t]{l}{\HUGE $)$}\end{tabular}}}}%
    \put(0.98090278,0.08854167){\color[rgb]{0,0,0}\makebox(0,0)[lt]{\lineheight{1.25}\smash{\begin{tabular}[t]{l}{\HUGE $)$}\end{tabular}}}}%
    \put(0.83090278,0.08854167){\color[rgb]{0,0,0}\makebox(0,0)[lt]{\lineheight{1.25}\smash{\begin{tabular}[t]{l}{\HUGE $($}\end{tabular}}}}%
    \put(0.709375,0.45138889){\color[rgb]{0,0,0}\makebox(0,0)[lt]{\lineheight{1.25}\smash{\begin{tabular}[t]{l}{\HUGE $($}\end{tabular}}}}%
    \put(0,0){\includegraphics[width=\unitlength,page=5]{a-compact-combinatorial-manifold-diagrams.pdf}}%
    \put(0.57256944,0.16527778){\color[rgb]{0,0,0}\makebox(0,0)[lt]{\lineheight{1.25}\smash{\begin{tabular}[t]{l}$\xto{\NF{f^{\leq y}}}$\end{tabular}}}}%
    \put(0,0){\includegraphics[width=\unitlength,page=6]{a-compact-combinatorial-manifold-diagrams.pdf}}%
    \put(0.878125,0.09756944){\color[rgb]{0,0,0}\makebox(0,0)[lt]{\lineheight{1.25}\smash{\begin{tabular}[t]{l}$\xto{\NF{f^{\leq y}}}$\end{tabular}}}}%
    \put(0,0){\includegraphics[width=\unitlength,page=7]{a-compact-combinatorial-manifold-diagrams.pdf}}%
  \end{picture}%
\endgroup%

%% file: figuresused/compactification-of-a-stratified-truss.pdf_tex
\begingroup%
  \makeatletter%
  \providecommand\color[2][]{%
    \errmessage{(Inkscape) Color is used for the text in Inkscape, but the package 'color.sty' is not loaded}%
    \renewcommand\color[2][]{}%
  }%
  \providecommand\transparent[1]{%
    \errmessage{(Inkscape) Transparency is used (non-zero) for the text in Inkscape, but the package 'transparent.sty' is not loaded}%
    \renewcommand\transparent[1]{}%
  }%
  \providecommand\rotatebox[2]{#2}%
  \newcommand*\fsize{\dimexpr\f@size pt\relax}%
  \newcommand*\lineheight[1]{\fontsize{\fsize}{#1\fsize}\selectfont}%
  \ifx\svgwidth\undefined%
    \setlength{\unitlength}{2160bp}%
    \ifx\svgscale\undefined%
      \relax%
    \else%
      \setlength{\unitlength}{\unitlength * \real{\svgscale}}%
    \fi%
  \else%
    \setlength{\unitlength}{\svgwidth}%
  \fi%
  \global\let\svgwidth\undefined%
  \global\let\svgscale\undefined%
  \makeatother%
  \begin{picture}(1,0.24027778)%
    \lineheight{1}%
    \setlength\tabcolsep{0pt}%
    \put(0,0){\includegraphics[width=\unitlength,page=1]{compactification-of-a-stratified-truss.pdf}}%
    \put(0.42847222,0.1125){\color[rgb]{0,0,0}\makebox(0,0)[lt]{\lineheight{1.25}\smash{\begin{tabular}[t]{l}$T$\end{tabular}}}}%
    \put(0.76840278,0.10243056){\color[rgb]{0,0,0}\makebox(0,0)[lt]{\lineheight{1.25}\smash{\begin{tabular}[t]{l}$\tilde T$\end{tabular}}}}%
    \put(0.30381944,0.2){\color[rgb]{0,0,0}\makebox(0,0)[lt]{\lineheight{1.25}\smash{\begin{tabular}[t]{l}{\tiny compactify}\end{tabular}}}}%
    \put(0.31284722,0.21875){\color[rgb]{0,0,0}\makebox(0,0)[lt]{\lineheight{1.25}\smash{\begin{tabular}[t]{l}{\tiny cubically}\end{tabular}}}}%
    \put(0,0){\includegraphics[width=\unitlength,page=2]{compactification-of-a-stratified-truss.pdf}}%
    \put(0.30381944,0.2){\color[rgb]{0,0,0}\makebox(0,0)[lt]{\lineheight{1.25}\smash{\begin{tabular}[t]{l}{\tiny compactify}\end{tabular}}}}%
    \put(0.31284722,0.21875){\color[rgb]{0,0,0}\makebox(0,0)[lt]{\lineheight{1.25}\smash{\begin{tabular}[t]{l}{\tiny cubically}\end{tabular}}}}%
    \put(0.9375,0.19826389){\color[rgb]{0,0,0}\makebox(0,0)[lt]{\lineheight{1.25}\smash{\begin{tabular}[t]{l}$\tilde f$\end{tabular}}}}%
    \put(0.53993056,0.19965278){\color[rgb]{0,0,0}\makebox(0,0)[lt]{\lineheight{1.25}\smash{\begin{tabular}[t]{l}$f$\end{tabular}}}}%
    \put(0,0){\includegraphics[width=\unitlength,page=3]{compactification-of-a-stratified-truss.pdf}}%
    \put(0.63576389,0.2){\color[rgb]{0,0,0}\makebox(0,0)[lt]{\lineheight{1.25}\smash{\begin{tabular}[t]{l}{\tiny compactify}\end{tabular}}}}%
    \put(0.63576389,0.21770833){\color[rgb]{0,0,0}\makebox(0,0)[lt]{\lineheight{1.25}\smash{\begin{tabular}[t]{l}{\tiny retractably}\end{tabular}}}}%
    \put(0.025,0.10243056){\color[rgb]{0,0,0}\makebox(0,0)[lt]{\lineheight{1.25}\smash{\begin{tabular}[t]{l}$\overline T$\end{tabular}}}}%
    \put(0.21527778,0.19861111){\color[rgb]{0,0,0}\makebox(0,0)[lt]{\lineheight{1.25}\smash{\begin{tabular}[t]{l}$\overline f$\end{tabular}}}}%
    \put(0,0){\includegraphics[width=\unitlength,page=4]{compactification-of-a-stratified-truss.pdf}}%
  \end{picture}%
\endgroup%

%% file: figuresused/combinatorializations-of-manifold-2-diagrams.pdf_tex
\begingroup%
  \makeatletter%
  \providecommand\color[2][]{%
    \errmessage{(Inkscape) Color is used for the text in Inkscape, but the package 'color.sty' is not loaded}%
    \renewcommand\color[2][]{}%
  }%
  \providecommand\transparent[1]{%
    \errmessage{(Inkscape) Transparency is used (non-zero) for the text in Inkscape, but the package 'transparent.sty' is not loaded}%
    \renewcommand\transparent[1]{}%
  }%
  \providecommand\rotatebox[2]{#2}%
  \newcommand*\fsize{\dimexpr\f@size pt\relax}%
  \newcommand*\lineheight[1]{\fontsize{\fsize}{#1\fsize}\selectfont}%
  \ifx\svgwidth\undefined%
    \setlength{\unitlength}{2160bp}%
    \ifx\svgscale\undefined%
      \relax%
    \else%
      \setlength{\unitlength}{\unitlength * \real{\svgscale}}%
    \fi%
  \else%
    \setlength{\unitlength}{\svgwidth}%
  \fi%
  \global\let\svgwidth\undefined%
  \global\let\svgscale\undefined%
  \makeatother%
  \begin{picture}(1,0.30902778)%
    \lineheight{1}%
    \setlength\tabcolsep{0pt}%
    \put(0,0){\includegraphics[width=\unitlength,page=1]{combinatorializations-of-manifold-2-diagrams.pdf}}%
    \put(0.06284722,0.20972222){\color[rgb]{0,0,0}\makebox(0,0)[lt]{\lineheight{1.25}\smash{\begin{tabular}[t]{l}$f$\end{tabular}}}}%
    \put(0.06284722,0.03263889){\color[rgb]{0,0,0}\makebox(0,0)[lt]{\lineheight{1.25}\smash{\begin{tabular}[t]{l}$f$\end{tabular}}}}%
    \put(0.43923611,0.20972222){\color[rgb]{0,0,0}\makebox(0,0)[lt]{\lineheight{1.25}\smash{\begin{tabular}[t]{l}$M$\end{tabular}}}}%
    \put(0.43923611,0.03263889){\color[rgb]{0,0,0}\makebox(0,0)[lt]{\lineheight{1.25}\smash{\begin{tabular}[t]{l}$M$\end{tabular}}}}%
    \put(0.23159722,0.09479167){\color[rgb]{0,0,0}\makebox(0,0)[lt]{\lineheight{1.25}\smash{\begin{tabular}[t]{l}{\tiny refines}\end{tabular}}}}%
    \put(0.23055556,0.26805556){\color[rgb]{0,0,0}\makebox(0,0)[lt]{\lineheight{1.25}\smash{\begin{tabular}[t]{l}{\tiny refines}\end{tabular}}}}%
    \put(0.54236111,0.03263889){\color[rgb]{0,0,0}\makebox(0,0)[lt]{\lineheight{1.25}\smash{\begin{tabular}[t]{l}$T$\end{tabular}}}}%
    \put(0.57361111,0.20972222){\color[rgb]{0,0,0}\makebox(0,0)[lt]{\lineheight{1.25}\smash{\begin{tabular}[t]{l}$T$\end{tabular}}}}%
    \put(0.75520833,0.26875){\color[rgb]{0,0,0}\makebox(0,0)[lt]{\lineheight{1.25}\smash{\begin{tabular}[t]{l}$g$\end{tabular}}}}%
    \put(0.76527778,0.09930556){\color[rgb]{0,0,0}\makebox(0,0)[lt]{\lineheight{1.25}\smash{\begin{tabular}[t]{l}$g$\end{tabular}}}}%
    \put(0,0){\includegraphics[width=\unitlength,page=2]{combinatorializations-of-manifold-2-diagrams.pdf}}%
  \end{picture}%
\endgroup%

%% file: figuresused/combinatorialization-of-a-3-diagram.pdf_tex
\begingroup%
  \makeatletter%
  \providecommand\color[2][]{%
    \errmessage{(Inkscape) Color is used for the text in Inkscape, but the package 'color.sty' is not loaded}%
    \renewcommand\color[2][]{}%
  }%
  \providecommand\transparent[1]{%
    \errmessage{(Inkscape) Transparency is used (non-zero) for the text in Inkscape, but the package 'transparent.sty' is not loaded}%
    \renewcommand\transparent[1]{}%
  }%
  \providecommand\rotatebox[2]{#2}%
  \newcommand*\fsize{\dimexpr\f@size pt\relax}%
  \newcommand*\lineheight[1]{\fontsize{\fsize}{#1\fsize}\selectfont}%
  \ifx\svgwidth\undefined%
    \setlength{\unitlength}{2160bp}%
    \ifx\svgscale\undefined%
      \relax%
    \else%
      \setlength{\unitlength}{\unitlength * \real{\svgscale}}%
    \fi%
  \else%
    \setlength{\unitlength}{\svgwidth}%
  \fi%
  \global\let\svgwidth\undefined%
  \global\let\svgscale\undefined%
  \makeatother%
  \begin{picture}(1,0.30868056)%
    \lineheight{1}%
    \setlength\tabcolsep{0pt}%
    \put(0,0){\includegraphics[width=\unitlength,page=1]{combinatorialization-of-a-3-diagram.pdf}}%
    \put(0.17534722,0.09131944){\color[rgb]{0,0,0}\makebox(0,0)[lt]{\lineheight{1.25}\smash{\begin{tabular}[t]{l}$f$\end{tabular}}}}%
    \put(0.52361111,0.03020833){\color[rgb]{0,0,0}\makebox(0,0)[lt]{\lineheight{1.25}\smash{\begin{tabular}[t]{l}$T$\end{tabular}}}}%
    \put(0.73055556,0.17256944){\color[rgb]{0,0,0}\makebox(0,0)[lt]{\lineheight{1.25}\smash{\begin{tabular}[t]{l}$g$\end{tabular}}}}%
  \end{picture}%
\endgroup%

%% file: figuresused/combinatorialization-of-a-compact-manifold-diagrams.pdf_tex
\begingroup%
  \makeatletter%
  \providecommand\color[2][]{%
    \errmessage{(Inkscape) Color is used for the text in Inkscape, but the package 'color.sty' is not loaded}%
    \renewcommand\color[2][]{}%
  }%
  \providecommand\transparent[1]{%
    \errmessage{(Inkscape) Transparency is used (non-zero) for the text in Inkscape, but the package 'transparent.sty' is not loaded}%
    \renewcommand\transparent[1]{}%
  }%
  \providecommand\rotatebox[2]{#2}%
  \newcommand*\fsize{\dimexpr\f@size pt\relax}%
  \newcommand*\lineheight[1]{\fontsize{\fsize}{#1\fsize}\selectfont}%
  \ifx\svgwidth\undefined%
    \setlength{\unitlength}{2160bp}%
    \ifx\svgscale\undefined%
      \relax%
    \else%
      \setlength{\unitlength}{\unitlength * \real{\svgscale}}%
    \fi%
  \else%
    \setlength{\unitlength}{\svgwidth}%
  \fi%
  \global\let\svgwidth\undefined%
  \global\let\svgscale\undefined%
  \makeatother%
  \begin{picture}(1,0.12638889)%
    \lineheight{1}%
    \setlength\tabcolsep{0pt}%
    \put(0,0){\includegraphics[width=\unitlength,page=1]{combinatorialization-of-a-compact-manifold-diagrams.pdf}}%
    \put(0.09895833,0.01145833){\color[rgb]{0,0,0}\makebox(0,0)[lt]{\lineheight{1.25}\smash{\begin{tabular}[t]{l}$f$\end{tabular}}}}%
    \put(0.29479167,0.01145833){\color[rgb]{0,0,0}\makebox(0,0)[lt]{\lineheight{1.25}\smash{\begin{tabular}[t]{l}$M$\end{tabular}}}}%
    \put(0.57465278,0.01145833){\color[rgb]{0,0,0}\makebox(0,0)[lt]{\lineheight{1.25}\smash{\begin{tabular}[t]{l}$T$\end{tabular}}}}%
    \put(0.80034722,0.07708333){\color[rgb]{0,0,0}\makebox(0,0)[lt]{\lineheight{1.25}\smash{\begin{tabular}[t]{l}$g$\end{tabular}}}}%
    \put(0,0){\includegraphics[width=\unitlength,page=2]{combinatorialization-of-a-compact-manifold-diagrams.pdf}}%
  \end{picture}%
\endgroup%

%% file: figuresused/interior-compactifying-diagram.pdf_tex
\begingroup%
  \makeatletter%
  \providecommand\color[2][]{%
    \errmessage{(Inkscape) Color is used for the text in Inkscape, but the package 'color.sty' is not loaded}%
    \renewcommand\color[2][]{}%
  }%
  \providecommand\transparent[1]{%
    \errmessage{(Inkscape) Transparency is used (non-zero) for the text in Inkscape, but the package 'transparent.sty' is not loaded}%
    \renewcommand\transparent[1]{}%
  }%
  \providecommand\rotatebox[2]{#2}%
  \newcommand*\fsize{\dimexpr\f@size pt\relax}%
  \newcommand*\lineheight[1]{\fontsize{\fsize}{#1\fsize}\selectfont}%
  \ifx\svgwidth\undefined%
    \setlength{\unitlength}{2160bp}%
    \ifx\svgscale\undefined%
      \relax%
    \else%
      \setlength{\unitlength}{\unitlength * \real{\svgscale}}%
    \fi%
  \else%
    \setlength{\unitlength}{\svgwidth}%
  \fi%
  \global\let\svgwidth\undefined%
  \global\let\svgscale\undefined%
  \makeatother%
  \begin{picture}(1,0.18958333)%
    \lineheight{1}%
    \setlength\tabcolsep{0pt}%
    \put(0,0){\includegraphics[width=\unitlength,page=1]{interior-compactifying-diagram.pdf}}%
  \end{picture}%
\endgroup%

%% file: figuresused/pasting-diagrams-and-their-dual-manifold-diagrams.pdf_tex
\begingroup%
  \makeatletter%
  \providecommand\color[2][]{%
    \errmessage{(Inkscape) Color is used for the text in Inkscape, but the package 'color.sty' is not loaded}%
    \renewcommand\color[2][]{}%
  }%
  \providecommand\transparent[1]{%
    \errmessage{(Inkscape) Transparency is used (non-zero) for the text in Inkscape, but the package 'transparent.sty' is not loaded}%
    \renewcommand\transparent[1]{}%
  }%
  \providecommand\rotatebox[2]{#2}%
  \newcommand*\fsize{\dimexpr\f@size pt\relax}%
  \newcommand*\lineheight[1]{\fontsize{\fsize}{#1\fsize}\selectfont}%
  \ifx\svgwidth\undefined%
    \setlength{\unitlength}{2160bp}%
    \ifx\svgscale\undefined%
      \relax%
    \else%
      \setlength{\unitlength}{\unitlength * \real{\svgscale}}%
    \fi%
  \else%
    \setlength{\unitlength}{\svgwidth}%
  \fi%
  \global\let\svgwidth\undefined%
  \global\let\svgscale\undefined%
  \makeatother%
  \begin{picture}(1,0.40590278)%
    \lineheight{1}%
    \setlength\tabcolsep{0pt}%
    \put(0,0){\includegraphics[width=\unitlength,page=1]{pasting-diagrams-and-their-dual-manifold-diagrams.pdf}}%
  \end{picture}%
\endgroup%

%% file: figuresused/interpreting-cell-diagrams-as-pasting-diagrams.pdf_tex
\begingroup%
  \makeatletter%
  \providecommand\color[2][]{%
    \errmessage{(Inkscape) Color is used for the text in Inkscape, but the package 'color.sty' is not loaded}%
    \renewcommand\color[2][]{}%
  }%
  \providecommand\transparent[1]{%
    \errmessage{(Inkscape) Transparency is used (non-zero) for the text in Inkscape, but the package 'transparent.sty' is not loaded}%
    \renewcommand\transparent[1]{}%
  }%
  \providecommand\rotatebox[2]{#2}%
  \newcommand*\fsize{\dimexpr\f@size pt\relax}%
  \newcommand*\lineheight[1]{\fontsize{\fsize}{#1\fsize}\selectfont}%
  \ifx\svgwidth\undefined%
    \setlength{\unitlength}{2160bp}%
    \ifx\svgscale\undefined%
      \relax%
    \else%
      \setlength{\unitlength}{\unitlength * \real{\svgscale}}%
    \fi%
  \else%
    \setlength{\unitlength}{\svgwidth}%
  \fi%
  \global\let\svgwidth\undefined%
  \global\let\svgscale\undefined%
  \makeatother%
  \begin{picture}(1,0.20625)%
    \lineheight{1}%
    \setlength\tabcolsep{0pt}%
    \put(0,0){\includegraphics[width=\unitlength,page=1]{interpreting-cell-diagrams-as-pasting-diagrams.pdf}}%
  \end{picture}%
\endgroup%

%% file: figuresused/blow-up-of-classical-cells-to-computadic-cells.pdf_tex
\begingroup%
  \makeatletter%
  \providecommand\color[2][]{%
    \errmessage{(Inkscape) Color is used for the text in Inkscape, but the package 'color.sty' is not loaded}%
    \renewcommand\color[2][]{}%
  }%
  \providecommand\transparent[1]{%
    \errmessage{(Inkscape) Transparency is used (non-zero) for the text in Inkscape, but the package 'transparent.sty' is not loaded}%
    \renewcommand\transparent[1]{}%
  }%
  \providecommand\rotatebox[2]{#2}%
  \newcommand*\fsize{\dimexpr\f@size pt\relax}%
  \newcommand*\lineheight[1]{\fontsize{\fsize}{#1\fsize}\selectfont}%
  \ifx\svgwidth\undefined%
    \setlength{\unitlength}{2160bp}%
    \ifx\svgscale\undefined%
      \relax%
    \else%
      \setlength{\unitlength}{\unitlength * \real{\svgscale}}%
    \fi%
  \else%
    \setlength{\unitlength}{\svgwidth}%
  \fi%
  \global\let\svgwidth\undefined%
  \global\let\svgscale\undefined%
  \makeatother%
  \begin{picture}(1,0.27569444)%
    \lineheight{1}%
    \setlength\tabcolsep{0pt}%
    \put(0,0){\includegraphics[width=\unitlength,page=1]{blow-up-of-classical-cells-to-computadic-cells.pdf}}%
  \end{picture}%
\endgroup%

%% file: figuresused/coned-embeddings-vs-stratified-cones.pdf_tex
\begingroup%
  \makeatletter%
  \providecommand\color[2][]{%
    \errmessage{(Inkscape) Color is used for the text in Inkscape, but the package 'color.sty' is not loaded}%
    \renewcommand\color[2][]{}%
  }%
  \providecommand\transparent[1]{%
    \errmessage{(Inkscape) Transparency is used (non-zero) for the text in Inkscape, but the package 'transparent.sty' is not loaded}%
    \renewcommand\transparent[1]{}%
  }%
  \providecommand\rotatebox[2]{#2}%
  \newcommand*\fsize{\dimexpr\f@size pt\relax}%
  \newcommand*\lineheight[1]{\fontsize{\fsize}{#1\fsize}\selectfont}%
  \ifx\svgwidth\undefined%
    \setlength{\unitlength}{2160bp}%
    \ifx\svgscale\undefined%
      \relax%
    \else%
      \setlength{\unitlength}{\unitlength * \real{\svgscale}}%
    \fi%
  \else%
    \setlength{\unitlength}{\svgwidth}%
  \fi%
  \global\let\svgwidth\undefined%
  \global\let\svgscale\undefined%
  \makeatother%
  \begin{picture}(1,0.19756944)%
    \lineheight{1}%
    \setlength\tabcolsep{0pt}%
    \put(0,0){\includegraphics[width=\unitlength,page=1]{coned-embeddings-vs-stratified-cones.pdf}}%
    \put(0.09965278,0.003125){\color[rgb]{0,0,0}\makebox(0,0)[lt]{\lineheight{1.25}\smash{\begin{tabular}[t]{l}$f = (M \into \partial\bI^2)$\end{tabular}}}}%
    \put(0.47048611,0.003125){\color[rgb]{0,0,0}\makebox(0,0)[lt]{\lineheight{1.25}\smash{\begin{tabular}[t]{l}$\fcone(f)$\end{tabular}}}}%
    \put(0.79166667,0.003125){\color[rgb]{0,0,0}\makebox(0,0)[lt]{\lineheight{1.25}\smash{\begin{tabular}[t]{l}$\cone(f)$\end{tabular}}}}%
    \put(0,0){\includegraphics[width=\unitlength,page=2]{coned-embeddings-vs-stratified-cones.pdf}}%
  \end{picture}%
\endgroup%

%% file: figuresused/framed-transversality-condition-and-failure.pdf_tex
\begingroup%
  \makeatletter%
  \providecommand\color[2][]{%
    \errmessage{(Inkscape) Color is used for the text in Inkscape, but the package 'color.sty' is not loaded}%
    \renewcommand\color[2][]{}%
  }%
  \providecommand\transparent[1]{%
    \errmessage{(Inkscape) Transparency is used (non-zero) for the text in Inkscape, but the package 'transparent.sty' is not loaded}%
    \renewcommand\transparent[1]{}%
  }%
  \providecommand\rotatebox[2]{#2}%
  \newcommand*\fsize{\dimexpr\f@size pt\relax}%
  \newcommand*\lineheight[1]{\fontsize{\fsize}{#1\fsize}\selectfont}%
  \ifx\svgwidth\undefined%
    \setlength{\unitlength}{2160bp}%
    \ifx\svgscale\undefined%
      \relax%
    \else%
      \setlength{\unitlength}{\unitlength * \real{\svgscale}}%
    \fi%
  \else%
    \setlength{\unitlength}{\svgwidth}%
  \fi%
  \global\let\svgwidth\undefined%
  \global\let\svgscale\undefined%
  \makeatother%
  \begin{picture}(1,0.33333333)%
    \lineheight{1}%
    \setlength\tabcolsep{0pt}%
    \put(0,0){\includegraphics[width=\unitlength,page=1]{framed-transversality-condition-and-failure.pdf}}%
    \put(0.03784722,0.27083333){\color[rgb]{0,0,0}\makebox(0,0)[lt]{\lineheight{1.25}\smash{\begin{tabular}[t]{l}{\tiny 1-transversal}\end{tabular}}}}%
    \put(0.03784722,0.05138889){\color[rgb]{0,0,0}\makebox(0,0)[lt]{\lineheight{1.25}\smash{\begin{tabular}[t]{l}{\tiny 0-transversal}\end{tabular}}}}%
    \put(0.85138889,0.190625){\color[rgb]{0,0,0}\makebox(0,0)[lt]{\lineheight{1.25}\smash{\begin{tabular}[t]{l}{\tiny non-transversal}\end{tabular}}}}%
    \put(0,0){\includegraphics[width=\unitlength,page=2]{framed-transversality-condition-and-failure.pdf}}%
    \put(0.15104167,0.17569444){\color[rgb]{0,0,0}\makebox(0,0)[lt]{\lineheight{1.25}\smash{\begin{tabular}[t]{l}$\iso$\end{tabular}}}}%
    \put(0,0){\includegraphics[width=\unitlength,page=3]{framed-transversality-condition-and-failure.pdf}}%
    \put(0.20833333,0.17569444){\color[rgb]{0,0,0}\makebox(0,0)[lt]{\lineheight{1.25}\smash{\begin{tabular}[t]{l}$\times$\end{tabular}}}}%
    \put(0,0){\includegraphics[width=\unitlength,page=4]{framed-transversality-condition-and-failure.pdf}}%
  \end{picture}%
\endgroup%

%% file: figuresused/tame-tangle-in-dim-2-and-3.pdf_tex
\begingroup%
  \makeatletter%
  \providecommand\color[2][]{%
    \errmessage{(Inkscape) Color is used for the text in Inkscape, but the package 'color.sty' is not loaded}%
    \renewcommand\color[2][]{}%
  }%
  \providecommand\transparent[1]{%
    \errmessage{(Inkscape) Transparency is used (non-zero) for the text in Inkscape, but the package 'transparent.sty' is not loaded}%
    \renewcommand\transparent[1]{}%
  }%
  \providecommand\rotatebox[2]{#2}%
  \newcommand*\fsize{\dimexpr\f@size pt\relax}%
  \newcommand*\lineheight[1]{\fontsize{\fsize}{#1\fsize}\selectfont}%
  \ifx\svgwidth\undefined%
    \setlength{\unitlength}{2160bp}%
    \ifx\svgscale\undefined%
      \relax%
    \else%
      \setlength{\unitlength}{\unitlength * \real{\svgscale}}%
    \fi%
  \else%
    \setlength{\unitlength}{\svgwidth}%
  \fi%
  \global\let\svgwidth\undefined%
  \global\let\svgscale\undefined%
  \makeatother%
  \begin{picture}(1,0.19236111)%
    \lineheight{1}%
    \setlength\tabcolsep{0pt}%
    \put(0,0){\includegraphics[width=\unitlength,page=1]{tame-tangle-in-dim-2-and-3.pdf}}%
  \end{picture}%
\endgroup%

%% file: figuresused/a-stratified-truss-with-combinatorially-transversal.pdf_tex
\begingroup%
  \makeatletter%
  \providecommand\color[2][]{%
    \errmessage{(Inkscape) Color is used for the text in Inkscape, but the package 'color.sty' is not loaded}%
    \renewcommand\color[2][]{}%
  }%
  \providecommand\transparent[1]{%
    \errmessage{(Inkscape) Transparency is used (non-zero) for the text in Inkscape, but the package 'transparent.sty' is not loaded}%
    \renewcommand\transparent[1]{}%
  }%
  \providecommand\rotatebox[2]{#2}%
  \newcommand*\fsize{\dimexpr\f@size pt\relax}%
  \newcommand*\lineheight[1]{\fontsize{\fsize}{#1\fsize}\selectfont}%
  \ifx\svgwidth\undefined%
    \setlength{\unitlength}{2160bp}%
    \ifx\svgscale\undefined%
      \relax%
    \else%
      \setlength{\unitlength}{\unitlength * \real{\svgscale}}%
    \fi%
  \else%
    \setlength{\unitlength}{\svgwidth}%
  \fi%
  \global\let\svgwidth\undefined%
  \global\let\svgscale\undefined%
  \makeatother%
  \begin{picture}(1,0.52256944)%
    \lineheight{1}%
    \setlength\tabcolsep{0pt}%
    \put(0,0){\includegraphics[width=\unitlength,page=1]{a-stratified-truss-with-combinatorially-transversal.pdf}}%
    \put(0.13090278,0.40451389){\color[rgb]{0,0,0}\makebox(0,0)[lt]{\lineheight{1.25}\smash{\begin{tabular}[t]{l}\rotatebox{-90}{\tiny $\xrsquigarrow{\hspace{12pt}}$}\end{tabular}}}}%
    \put(0.15520833,0.38194444){\color[rgb]{0,0,0}\makebox(0,0)[lt]{\lineheight{1.25}\smash{\begin{tabular}[t]{l}{\tiny normalize}\end{tabular}}}}%
    \put(0.73229167,0.38194444){\color[rgb]{0,0,0}\makebox(0,0)[lt]{\lineheight{1.25}\smash{\begin{tabular}[t]{l}{\tiny normalized but not of the}\end{tabular}}}}%
    \put(0,0){\includegraphics[width=\unitlength,page=2]{a-stratified-truss-with-combinatorially-transversal.pdf}}%
    \put(0.43611111,0.48125){\color[rgb]{0,0,0}\makebox(0,0)[lt]{\lineheight{1.25}\smash{\begin{tabular}[t]{l}{\tiny $Q \into T_2$}\end{tabular}}}}%
    \put(0.73090278,0.35590278){\color[rgb]{0,0,0}\makebox(0,0)[lt]{\lineheight{1.25}\smash{\begin{tabular}[t]{l}{\tiny form $\OTT^k \times (C,D \into C_{n-k}) $}\end{tabular}}}}%
    \put(0.41840278,0.03402778){\color[rgb]{0,0,0}\makebox(0,0)[lt]{\lineheight{1.25}\smash{\begin{tabular}[t]{l}{\tiny $\iso \OTT^{0} \times (C,D \into C_{2})$}\end{tabular}}}}%
    \put(0.04131944,0.29270833){\color[rgb]{0,0,0}\makebox(0,0)[lt]{\lineheight{1.25}\smash{\begin{tabular}[t]{l}{\tiny $\iso \OTT^{1} \times (C,D \into C_{1})$}\end{tabular}}}}%
    \put(0.04131944,0.26701389){\color[rgb]{0,0,0}\makebox(0,0)[lt]{\lineheight{1.25}\smash{\begin{tabular}[t]{l}{\tiny $\imp$ 1-dim 1-transversal}\end{tabular}}}}%
    \put(0.41840278,0.00833333){\color[rgb]{0,0,0}\makebox(0,0)[lt]{\lineheight{1.25}\smash{\begin{tabular}[t]{l}{\tiny $\imp$ 1-dim 0-transversal}\end{tabular}}}}%
    \put(0.73090278,0.32986111){\color[rgb]{0,0,0}\makebox(0,0)[lt]{\lineheight{1.25}\smash{\begin{tabular}[t]{l}{\tiny $\imp$ not transversal}\end{tabular}}}}%
    \put(0,0){\includegraphics[width=\unitlength,page=3]{a-stratified-truss-with-combinatorially-transversal.pdf}}%
  \end{picture}%
\endgroup%

%% file: figuresused/manifold-diagram-refinements.pdf_tex
\begingroup%
  \makeatletter%
  \providecommand\color[2][]{%
    \errmessage{(Inkscape) Color is used for the text in Inkscape, but the package 'color.sty' is not loaded}%
    \renewcommand\color[2][]{}%
  }%
  \providecommand\transparent[1]{%
    \errmessage{(Inkscape) Transparency is used (non-zero) for the text in Inkscape, but the package 'transparent.sty' is not loaded}%
    \renewcommand\transparent[1]{}%
  }%
  \providecommand\rotatebox[2]{#2}%
  \newcommand*\fsize{\dimexpr\f@size pt\relax}%
  \newcommand*\lineheight[1]{\fontsize{\fsize}{#1\fsize}\selectfont}%
  \ifx\svgwidth\undefined%
    \setlength{\unitlength}{2160bp}%
    \ifx\svgscale\undefined%
      \relax%
    \else%
      \setlength{\unitlength}{\unitlength * \real{\svgscale}}%
    \fi%
  \else%
    \setlength{\unitlength}{\svgwidth}%
  \fi%
  \global\let\svgwidth\undefined%
  \global\let\svgscale\undefined%
  \makeatother%
  \begin{picture}(1,0.20833333)%
    \lineheight{1}%
    \setlength\tabcolsep{0pt}%
    \put(0,0){\includegraphics[width=\unitlength,page=1]{manifold-diagram-refinements.pdf}}%
    \put(0.15381944,0.11736111){\color[rgb]{0,0,0}\makebox(0,0)[lt]{\lineheight{1.25}\smash{\begin{tabular}[t]{l}{\tiny refines}\end{tabular}}}}%
    \put(0,0){\includegraphics[width=\unitlength,page=2]{manifold-diagram-refinements.pdf}}%
    \put(0.7125,0.11736111){\color[rgb]{0,0,0}\makebox(0,0)[lt]{\lineheight{1.25}\smash{\begin{tabular}[t]{l}{\tiny refines}\end{tabular}}}}%
    \put(0,0){\includegraphics[width=\unitlength,page=3]{manifold-diagram-refinements.pdf}}%
  \end{picture}%
\endgroup%

%% file: figuresused/cell-and-dual-cell-structures-of-tame-tangles.pdf_tex
\begingroup%
  \makeatletter%
  \providecommand\color[2][]{%
    \errmessage{(Inkscape) Color is used for the text in Inkscape, but the package 'color.sty' is not loaded}%
    \renewcommand\color[2][]{}%
  }%
  \providecommand\transparent[1]{%
    \errmessage{(Inkscape) Transparency is used (non-zero) for the text in Inkscape, but the package 'transparent.sty' is not loaded}%
    \renewcommand\transparent[1]{}%
  }%
  \providecommand\rotatebox[2]{#2}%
  \newcommand*\fsize{\dimexpr\f@size pt\relax}%
  \newcommand*\lineheight[1]{\fontsize{\fsize}{#1\fsize}\selectfont}%
  \ifx\svgwidth\undefined%
    \setlength{\unitlength}{2160bp}%
    \ifx\svgscale\undefined%
      \relax%
    \else%
      \setlength{\unitlength}{\unitlength * \real{\svgscale}}%
    \fi%
  \else%
    \setlength{\unitlength}{\svgwidth}%
  \fi%
  \global\let\svgwidth\undefined%
  \global\let\svgscale\undefined%
  \makeatother%
  \begin{picture}(1,0.33055556)%
    \lineheight{1}%
    \setlength\tabcolsep{0pt}%
    \put(0,0){\includegraphics[width=\unitlength,page=1]{cell-and-dual-cell-structures-of-tame-tangles.pdf}}%
    \put(0.26215278,0.08298611){\color[rgb]{0,0,0}\makebox(0,0)[lt]{\lineheight{1.25}\smash{\begin{tabular}[t]{l}$\into$\end{tabular}}}}%
    \put(0.265625,0.27361111){\color[rgb]{0,0,0}\makebox(0,0)[lt]{\lineheight{1.25}\smash{\begin{tabular}[t]{l}$\into$\end{tabular}}}}%
    \put(0,0){\includegraphics[width=\unitlength,page=2]{cell-and-dual-cell-structures-of-tame-tangles.pdf}}%
  \end{picture}%
\endgroup%

%% file: figuresused/tame-tangle-bundles.pdf_tex
\begingroup%
  \makeatletter%
  \providecommand\color[2][]{%
    \errmessage{(Inkscape) Color is used for the text in Inkscape, but the package 'color.sty' is not loaded}%
    \renewcommand\color[2][]{}%
  }%
  \providecommand\transparent[1]{%
    \errmessage{(Inkscape) Transparency is used (non-zero) for the text in Inkscape, but the package 'transparent.sty' is not loaded}%
    \renewcommand\transparent[1]{}%
  }%
  \providecommand\rotatebox[2]{#2}%
  \newcommand*\fsize{\dimexpr\f@size pt\relax}%
  \newcommand*\lineheight[1]{\fontsize{\fsize}{#1\fsize}\selectfont}%
  \ifx\svgwidth\undefined%
    \setlength{\unitlength}{2160bp}%
    \ifx\svgscale\undefined%
      \relax%
    \else%
      \setlength{\unitlength}{\unitlength * \real{\svgscale}}%
    \fi%
  \else%
    \setlength{\unitlength}{\svgwidth}%
  \fi%
  \global\let\svgwidth\undefined%
  \global\let\svgscale\undefined%
  \makeatother%
  \begin{picture}(1,0.51493056)%
    \lineheight{1}%
    \setlength\tabcolsep{0pt}%
    \put(0,0){\includegraphics[width=\unitlength,page=1]{tame-tangle-bundles.pdf}}%
    \put(0.70486111,0.03993056){\color[rgb]{0,0,0}\makebox(0,0)[lt]{\lineheight{1.25}\smash{\begin{tabular}[t]{l}$\downarrow$\end{tabular}}}}%
    \put(0.76736111,0.31319444){\color[rgb]{0,0,0}\makebox(0,0)[lt]{\lineheight{1.25}\smash{\begin{tabular}[t]{l}$\downarrow$\end{tabular}}}}%
    \put(0.14965278,0.31666667){\color[rgb]{0,0,0}\makebox(0,0)[lt]{\lineheight{1.25}\smash{\begin{tabular}[t]{l}$\downarrow$\end{tabular}}}}%
    \put(0.18402778,0.053125){\color[rgb]{0,0,0}\makebox(0,0)[lt]{\lineheight{1.25}\smash{\begin{tabular}[t]{l}$\downarrow$\end{tabular}}}}%
    \put(0.01909722,0.48506944){\color[rgb]{0,0,0}\makebox(0,0)[lt]{\lineheight{1.25}\smash{\begin{tabular}[t]{l}(1)\end{tabular}}}}%
    \put(0.33506944,0.48506944){\color[rgb]{0,0,0}\makebox(0,0)[lt]{\lineheight{1.25}\smash{\begin{tabular}[t]{l}(2)\end{tabular}}}}%
    \put(0.64409722,0.48506944){\color[rgb]{0,0,0}\makebox(0,0)[lt]{\lineheight{1.25}\smash{\begin{tabular}[t]{l}(3)\end{tabular}}}}%
    \put(0,0){\includegraphics[width=\unitlength,page=2]{tame-tangle-bundles.pdf}}%
    \put(0.43611111,0.31666667){\color[rgb]{0,0,0}\makebox(0,0)[lt]{\lineheight{1.25}\smash{\begin{tabular}[t]{l}$\downarrow$\end{tabular}}}}%
    \put(0.43229167,0.22951389){\color[rgb]{0,0,0}\makebox(0,0)[lt]{\lineheight{1.25}\smash{\begin{tabular}[t]{l}(5)\end{tabular}}}}%
    \put(0.01909722,0.22951389){\color[rgb]{0,0,0}\makebox(0,0)[lt]{\lineheight{1.25}\smash{\begin{tabular}[t]{l}(4)\end{tabular}}}}%
  \end{picture}%
\endgroup%

%% file: figuresused/tang-bun-non-eg.pdf_tex
\begingroup%
  \makeatletter%
  \providecommand\color[2][]{%
    \errmessage{(Inkscape) Color is used for the text in Inkscape, but the package 'color.sty' is not loaded}%
    \renewcommand\color[2][]{}%
  }%
  \providecommand\transparent[1]{%
    \errmessage{(Inkscape) Transparency is used (non-zero) for the text in Inkscape, but the package 'transparent.sty' is not loaded}%
    \renewcommand\transparent[1]{}%
  }%
  \providecommand\rotatebox[2]{#2}%
  \newcommand*\fsize{\dimexpr\f@size pt\relax}%
  \newcommand*\lineheight[1]{\fontsize{\fsize}{#1\fsize}\selectfont}%
  \ifx\svgwidth\undefined%
    \setlength{\unitlength}{2160bp}%
    \ifx\svgscale\undefined%
      \relax%
    \else%
      \setlength{\unitlength}{\unitlength * \real{\svgscale}}%
    \fi%
  \else%
    \setlength{\unitlength}{\svgwidth}%
  \fi%
  \global\let\svgwidth\undefined%
  \global\let\svgscale\undefined%
  \makeatother%
  \begin{picture}(1,0.23125)%
    \lineheight{1}%
    \setlength\tabcolsep{0pt}%
    \put(0,0){\includegraphics[width=\unitlength,page=1]{tang-bun-non-eg.pdf}}%
    \put(0.71979167,0.02951389){\color[rgb]{0,0,0}\makebox(0,0)[lt]{\lineheight{1.25}\smash{\begin{tabular}[t]{l}$\downarrow$\end{tabular}}}}%
    \put(0,0){\includegraphics[width=\unitlength,page=2]{tang-bun-non-eg.pdf}}%
    \put(0.23923611,0.03298611){\color[rgb]{0,0,0}\makebox(0,0)[lt]{\lineheight{1.25}\smash{\begin{tabular}[t]{l}$\downarrow$\end{tabular}}}}%
  \end{picture}%
\endgroup%

%% file: figuresused/gluings-of-stratified-trusses-comb.pdf_tex
\begingroup%
  \makeatletter%
  \providecommand\color[2][]{%
    \errmessage{(Inkscape) Color is used for the text in Inkscape, but the package 'color.sty' is not loaded}%
    \renewcommand\color[2][]{}%
  }%
  \providecommand\transparent[1]{%
    \errmessage{(Inkscape) Transparency is used (non-zero) for the text in Inkscape, but the package 'transparent.sty' is not loaded}%
    \renewcommand\transparent[1]{}%
  }%
  \providecommand\rotatebox[2]{#2}%
  \newcommand*\fsize{\dimexpr\f@size pt\relax}%
  \newcommand*\lineheight[1]{\fontsize{\fsize}{#1\fsize}\selectfont}%
  \ifx\svgwidth\undefined%
    \setlength{\unitlength}{2160bp}%
    \ifx\svgscale\undefined%
      \relax%
    \else%
      \setlength{\unitlength}{\unitlength * \real{\svgscale}}%
    \fi%
  \else%
    \setlength{\unitlength}{\svgwidth}%
  \fi%
  \global\let\svgwidth\undefined%
  \global\let\svgscale\undefined%
  \makeatother%
  \begin{picture}(1,0.24131944)%
    \lineheight{1}%
    \setlength\tabcolsep{0pt}%
    \put(0,0){\includegraphics[width=\unitlength,page=1]{gluings-of-stratified-trusses-comb.pdf}}%
    \put(0.17326389,0.17465278){\color[rgb]{0,0,0}\makebox(0,0)[lt]{\lineheight{1.25}\smash{\begin{tabular}[t]{l}$T$\end{tabular}}}}%
    \put(0.17361111,0.07048611){\color[rgb]{0,0,0}\makebox(0,0)[lt]{\lineheight{1.25}\smash{\begin{tabular}[t]{l}$S$\end{tabular}}}}%
    \put(0.29895833,0.003125){\color[rgb]{0,0,0}\makebox(0,0)[lt]{\lineheight{1.25}\smash{\begin{tabular}[t]{l}$T \stack 2 S$\end{tabular}}}}%
    \put(0.46944444,0.003125){\color[rgb]{0,0,0}\makebox(0,0)[lt]{\lineheight{1.25}\smash{\begin{tabular}[t]{l}$S \stack 2 T$\end{tabular}}}}%
    \put(0.64479167,0.003125){\color[rgb]{0,0,0}\makebox(0,0)[lt]{\lineheight{1.25}\smash{\begin{tabular}[t]{l}$T \stack 1 S$\end{tabular}}}}%
    \put(0.83194444,0.003125){\color[rgb]{0,0,0}\makebox(0,0)[lt]{\lineheight{1.25}\smash{\begin{tabular}[t]{l}$S \stack 1 T$\end{tabular}}}}%
    \put(0,0){\includegraphics[width=\unitlength,page=2]{gluings-of-stratified-trusses-comb.pdf}}%
  \end{picture}%
\endgroup%

%% file: figuresused/gluings-of-stratified-trusses.pdf_tex
\begingroup%
  \makeatletter%
  \providecommand\color[2][]{%
    \errmessage{(Inkscape) Color is used for the text in Inkscape, but the package 'color.sty' is not loaded}%
    \renewcommand\color[2][]{}%
  }%
  \providecommand\transparent[1]{%
    \errmessage{(Inkscape) Transparency is used (non-zero) for the text in Inkscape, but the package 'transparent.sty' is not loaded}%
    \renewcommand\transparent[1]{}%
  }%
  \providecommand\rotatebox[2]{#2}%
  \newcommand*\fsize{\dimexpr\f@size pt\relax}%
  \newcommand*\lineheight[1]{\fontsize{\fsize}{#1\fsize}\selectfont}%
  \ifx\svgwidth\undefined%
    \setlength{\unitlength}{2160bp}%
    \ifx\svgscale\undefined%
      \relax%
    \else%
      \setlength{\unitlength}{\unitlength * \real{\svgscale}}%
    \fi%
  \else%
    \setlength{\unitlength}{\svgwidth}%
  \fi%
  \global\let\svgwidth\undefined%
  \global\let\svgscale\undefined%
  \makeatother%
  \begin{picture}(1,0.19270833)%
    \lineheight{1}%
    \setlength\tabcolsep{0pt}%
    \put(0,0){\includegraphics[width=\unitlength,page=1]{gluings-of-stratified-trusses.pdf}}%
    \put(0.15694444,0.13506944){\color[rgb]{0,0,0}\makebox(0,0)[lt]{\lineheight{1.25}\smash{\begin{tabular}[t]{l}$f$\end{tabular}}}}%
    \put(0.15729167,0.03090278){\color[rgb]{0,0,0}\makebox(0,0)[lt]{\lineheight{1.25}\smash{\begin{tabular}[t]{l}$g$\end{tabular}}}}%
    \put(0.28611111,0.00520833){\color[rgb]{0,0,0}\makebox(0,0)[lt]{\lineheight{1.25}\smash{\begin{tabular}[t]{l}$f \stack 2 g$\end{tabular}}}}%
    \put(0.45659722,0.00520833){\color[rgb]{0,0,0}\makebox(0,0)[lt]{\lineheight{1.25}\smash{\begin{tabular}[t]{l}$g \stack 2 f$\end{tabular}}}}%
    \put(0.63194444,0.00520833){\color[rgb]{0,0,0}\makebox(0,0)[lt]{\lineheight{1.25}\smash{\begin{tabular}[t]{l}$f \stack 1 g$\end{tabular}}}}%
    \put(0.80520833,0.00520833){\color[rgb]{0,0,0}\makebox(0,0)[lt]{\lineheight{1.25}\smash{\begin{tabular}[t]{l}$g \stack 1 f$\end{tabular}}}}%
    \put(0,0){\includegraphics[width=\unitlength,page=2]{gluings-of-stratified-trusses.pdf}}%
  \end{picture}%
\endgroup%

%% file: figuresused/basic-tangles-singularities.pdf_tex
\begingroup%
  \makeatletter%
  \providecommand\color[2][]{%
    \errmessage{(Inkscape) Color is used for the text in Inkscape, but the package 'color.sty' is not loaded}%
    \renewcommand\color[2][]{}%
  }%
  \providecommand\transparent[1]{%
    \errmessage{(Inkscape) Transparency is used (non-zero) for the text in Inkscape, but the package 'transparent.sty' is not loaded}%
    \renewcommand\transparent[1]{}%
  }%
  \providecommand\rotatebox[2]{#2}%
  \newcommand*\fsize{\dimexpr\f@size pt\relax}%
  \newcommand*\lineheight[1]{\fontsize{\fsize}{#1\fsize}\selectfont}%
  \ifx\svgwidth\undefined%
    \setlength{\unitlength}{2160bp}%
    \ifx\svgscale\undefined%
      \relax%
    \else%
      \setlength{\unitlength}{\unitlength * \real{\svgscale}}%
    \fi%
  \else%
    \setlength{\unitlength}{\svgwidth}%
  \fi%
  \global\let\svgwidth\undefined%
  \global\let\svgscale\undefined%
  \makeatother%
  \begin{picture}(1,0.16354167)%
    \lineheight{1}%
    \setlength\tabcolsep{0pt}%
    \put(0,0){\includegraphics[width=\unitlength,page=1]{basic-tangles-singularities.pdf}}%
    \put(0.13333333,0.00416667){\color[rgb]{0,0,0}\makebox(0,0)[lt]{\lineheight{1.25}\smash{\begin{tabular}[t]{l}{\scriptsize $\iA_1^{\numovar 0} \equiv \pts$}\end{tabular}}}}%
    \put(0.37465278,0.03576389){\color[rgb]{0,0,0}\makebox(0,0)[lt]{\lineheight{1.25}\smash{\begin{tabular}[t]{l}$\underbracket[0.140ex]{\hspace{1.65in}}_{\iA_1^{\numovar 1} ~\equiv~ \iA_1}$\end{tabular}}}}%
    \put(0.821875,0.00416667){\color[rgb]{0,0,0}\makebox(0,0)[lt]{\lineheight{1.25}\smash{\begin{tabular}[t]{l}{\scriptsize $\iA_1^{\numovar 0} \oplus \eps$}\end{tabular}}}}%
    \put(0,0){\includegraphics[width=\unitlength,page=2]{basic-tangles-singularities.pdf}}%
  \end{picture}%
\endgroup%

%% file: figuresused/normal-singularity.pdf_tex
\begingroup%
  \makeatletter%
  \providecommand\color[2][]{%
    \errmessage{(Inkscape) Color is used for the text in Inkscape, but the package 'color.sty' is not loaded}%
    \renewcommand\color[2][]{}%
  }%
  \providecommand\transparent[1]{%
    \errmessage{(Inkscape) Transparency is used (non-zero) for the text in Inkscape, but the package 'transparent.sty' is not loaded}%
    \renewcommand\transparent[1]{}%
  }%
  \providecommand\rotatebox[2]{#2}%
  \newcommand*\fsize{\dimexpr\f@size pt\relax}%
  \newcommand*\lineheight[1]{\fontsize{\fsize}{#1\fsize}\selectfont}%
  \ifx\svgwidth\undefined%
    \setlength{\unitlength}{2160bp}%
    \ifx\svgscale\undefined%
      \relax%
    \else%
      \setlength{\unitlength}{\unitlength * \real{\svgscale}}%
    \fi%
  \else%
    \setlength{\unitlength}{\svgwidth}%
  \fi%
  \global\let\svgwidth\undefined%
  \global\let\svgscale\undefined%
  \makeatother%
  \begin{picture}(1,0.18784722)%
    \lineheight{1}%
    \setlength\tabcolsep{0pt}%
    \put(0,0){\includegraphics[width=\unitlength,page=1]{normal-singularity.pdf}}%
    \put(0.38541667,0.08090278){\color[rgb]{0,0,0}\makebox(0,0)[lt]{\lineheight{1.25}\smash{\begin{tabular}[t]{l}{\tiny $x$}\end{tabular}}}}%
    \put(0,0){\includegraphics[width=\unitlength,page=2]{normal-singularity.pdf}}%
    \put(0.43854167,0.18090278){\color[rgb]{0,0,0}\makebox(0,0)[lt]{\lineheight{1.25}\smash{\begin{tabular}[t]{l}{\tiny normal singularity at $x$}\end{tabular}}}}%
    \put(0,0){\includegraphics[width=\unitlength,page=3]{normal-singularity.pdf}}%
    \put(0.29513889,0.16076389){\color[rgb]{0,0,0}\makebox(0,0)[lt]{\lineheight{1.25}\smash{\begin{tabular}[t]{l}{\tiny 1}\end{tabular}}}}%
    \put(0.30381944,0.17569444){\color[rgb]{0,0,0}\makebox(0,0)[lt]{\lineheight{1.25}\smash{\begin{tabular}[t]{l}{\tiny 2}\end{tabular}}}}%
    \put(0.27152778,0.140625){\color[rgb]{0,0,0}\makebox(0,0)[lt]{\lineheight{1.25}\smash{\begin{tabular}[t]{l}{\tiny 3}\end{tabular}}}}%
    \put(0,0){\includegraphics[width=\unitlength,page=4]{normal-singularity.pdf}}%
  \end{picture}%
\endgroup%

%% file: figuresused/two-perturbations-and-their-composite.pdf_tex
\begingroup%
  \makeatletter%
  \providecommand\color[2][]{%
    \errmessage{(Inkscape) Color is used for the text in Inkscape, but the package 'color.sty' is not loaded}%
    \renewcommand\color[2][]{}%
  }%
  \providecommand\transparent[1]{%
    \errmessage{(Inkscape) Transparency is used (non-zero) for the text in Inkscape, but the package 'transparent.sty' is not loaded}%
    \renewcommand\transparent[1]{}%
  }%
  \providecommand\rotatebox[2]{#2}%
  \newcommand*\fsize{\dimexpr\f@size pt\relax}%
  \newcommand*\lineheight[1]{\fontsize{\fsize}{#1\fsize}\selectfont}%
  \ifx\svgwidth\undefined%
    \setlength{\unitlength}{2160bp}%
    \ifx\svgscale\undefined%
      \relax%
    \else%
      \setlength{\unitlength}{\unitlength * \real{\svgscale}}%
    \fi%
  \else%
    \setlength{\unitlength}{\svgwidth}%
  \fi%
  \global\let\svgwidth\undefined%
  \global\let\svgscale\undefined%
  \makeatother%
  \begin{picture}(1,0.28680556)%
    \lineheight{1}%
    \setlength\tabcolsep{0pt}%
    \put(0.21736111,0.0125){\color[rgb]{0,0,0}\makebox(0,0)[lt]{\lineheight{1.25}\smash{\begin{tabular}[t]{l}$p$\end{tabular}}}}%
    \put(0.21666667,0.06944444){\color[rgb]{0,0,0}\makebox(0,0)[lt]{\lineheight{1.25}\smash{\begin{tabular}[t]{l}$\downarrow$\end{tabular}}}}%
    \put(0.49340278,0.06944444){\color[rgb]{0,0,0}\makebox(0,0)[lt]{\lineheight{1.25}\smash{\begin{tabular}[t]{l}$\downarrow$\end{tabular}}}}%
    \put(0.771875,0.06944444){\color[rgb]{0,0,0}\makebox(0,0)[lt]{\lineheight{1.25}\smash{\begin{tabular}[t]{l}$\downarrow$\end{tabular}}}}%
    \put(0.49236111,0.0125){\color[rgb]{0,0,0}\makebox(0,0)[lt]{\lineheight{1.25}\smash{\begin{tabular}[t]{l}$p'$\end{tabular}}}}%
    \put(0.75520833,0.0125){\color[rgb]{0,0,0}\makebox(0,0)[lt]{\lineheight{1.25}\smash{\begin{tabular}[t]{l}$p \ast p'$\end{tabular}}}}%
    \put(0,0){\includegraphics[width=\unitlength,page=1]{two-perturbations-and-their-composite.pdf}}%
  \end{picture}%
\endgroup%

%% file: figuresused/a-perturbation-between-stable-singularities.pdf_tex
\begingroup%
  \makeatletter%
  \providecommand\color[2][]{%
    \errmessage{(Inkscape) Color is used for the text in Inkscape, but the package 'color.sty' is not loaded}%
    \renewcommand\color[2][]{}%
  }%
  \providecommand\transparent[1]{%
    \errmessage{(Inkscape) Transparency is used (non-zero) for the text in Inkscape, but the package 'transparent.sty' is not loaded}%
    \renewcommand\transparent[1]{}%
  }%
  \providecommand\rotatebox[2]{#2}%
  \newcommand*\fsize{\dimexpr\f@size pt\relax}%
  \newcommand*\lineheight[1]{\fontsize{\fsize}{#1\fsize}\selectfont}%
  \ifx\svgwidth\undefined%
    \setlength{\unitlength}{2160bp}%
    \ifx\svgscale\undefined%
      \relax%
    \else%
      \setlength{\unitlength}{\unitlength * \real{\svgscale}}%
    \fi%
  \else%
    \setlength{\unitlength}{\svgwidth}%
  \fi%
  \global\let\svgwidth\undefined%
  \global\let\svgscale\undefined%
  \makeatother%
  \begin{picture}(1,0.24861111)%
    \lineheight{1}%
    \setlength\tabcolsep{0pt}%
    \put(0,0){\includegraphics[width=\unitlength,page=1]{a-perturbation-between-stable-singularities.pdf}}%
    \put(0.484375,0.15555556){\color[rgb]{0,0,0}\makebox(0,0)[lt]{\lineheight{1.25}\smash{\begin{tabular}[t]{l}$\perturb$\end{tabular}}}}%
    \put(0.49409722,0.03472222){\color[rgb]{0,0,0}\makebox(0,0)[lt]{\lineheight{1.25}\smash{\begin{tabular}[t]{l}$\downarrow$\end{tabular}}}}%
    \put(0,0){\includegraphics[width=\unitlength,page=2]{a-perturbation-between-stable-singularities.pdf}}%
  \end{picture}%
\endgroup%

%% file: figuresused/perturbing-the-triple-braid.pdf_tex
\begingroup%
  \makeatletter%
  \providecommand\color[2][]{%
    \errmessage{(Inkscape) Color is used for the text in Inkscape, but the package 'color.sty' is not loaded}%
    \renewcommand\color[2][]{}%
  }%
  \providecommand\transparent[1]{%
    \errmessage{(Inkscape) Transparency is used (non-zero) for the text in Inkscape, but the package 'transparent.sty' is not loaded}%
    \renewcommand\transparent[1]{}%
  }%
  \providecommand\rotatebox[2]{#2}%
  \newcommand*\fsize{\dimexpr\f@size pt\relax}%
  \newcommand*\lineheight[1]{\fontsize{\fsize}{#1\fsize}\selectfont}%
  \ifx\svgwidth\undefined%
    \setlength{\unitlength}{2160bp}%
    \ifx\svgscale\undefined%
      \relax%
    \else%
      \setlength{\unitlength}{\unitlength * \real{\svgscale}}%
    \fi%
  \else%
    \setlength{\unitlength}{\svgwidth}%
  \fi%
  \global\let\svgwidth\undefined%
  \global\let\svgscale\undefined%
  \makeatother%
  \begin{picture}(1,0.24861111)%
    \lineheight{1}%
    \setlength\tabcolsep{0pt}%
    \put(0,0){\includegraphics[width=\unitlength,page=1]{perturbing-the-triple-braid.pdf}}%
    \put(0.48402778,0.15555556){\color[rgb]{0,0,0}\makebox(0,0)[lt]{\lineheight{1.25}\smash{\begin{tabular}[t]{l}$\perturb$\end{tabular}}}}%
    \put(0.49375,0.03472222){\color[rgb]{0,0,0}\makebox(0,0)[lt]{\lineheight{1.25}\smash{\begin{tabular}[t]{l}$\downarrow$\end{tabular}}}}%
    \put(0,0){\includegraphics[width=\unitlength,page=2]{perturbing-the-triple-braid.pdf}}%
  \end{picture}%
\endgroup%

%% file: figuresused/source-and-target-of-a-1-tangle-in-dim-2.pdf_tex
\begingroup%
  \makeatletter%
  \providecommand\color[2][]{%
    \errmessage{(Inkscape) Color is used for the text in Inkscape, but the package 'color.sty' is not loaded}%
    \renewcommand\color[2][]{}%
  }%
  \providecommand\transparent[1]{%
    \errmessage{(Inkscape) Transparency is used (non-zero) for the text in Inkscape, but the package 'transparent.sty' is not loaded}%
    \renewcommand\transparent[1]{}%
  }%
  \providecommand\rotatebox[2]{#2}%
  \newcommand*\fsize{\dimexpr\f@size pt\relax}%
  \newcommand*\lineheight[1]{\fontsize{\fsize}{#1\fsize}\selectfont}%
  \ifx\svgwidth\undefined%
    \setlength{\unitlength}{2160bp}%
    \ifx\svgscale\undefined%
      \relax%
    \else%
      \setlength{\unitlength}{\unitlength * \real{\svgscale}}%
    \fi%
  \else%
    \setlength{\unitlength}{\svgwidth}%
  \fi%
  \global\let\svgwidth\undefined%
  \global\let\svgscale\undefined%
  \makeatother%
  \begin{picture}(1,0.21944444)%
    \lineheight{1}%
    \setlength\tabcolsep{0pt}%
    \put(0,0){\includegraphics[width=\unitlength,page=1]{source-and-target-of-a-1-tangle-in-dim-2.pdf}}%
    \put(0.609375,0.20625){\color[rgb]{0,0,0}\makebox(0,0)[lt]{\lineheight{1.25}\smash{\begin{tabular}[t]{l}{\tiny $6$}\end{tabular}}}}%
    \put(0,0){\includegraphics[width=\unitlength,page=2]{source-and-target-of-a-1-tangle-in-dim-2.pdf}}%
    \put(0.609375,0.00347222){\color[rgb]{0,0,0}\makebox(0,0)[lt]{\lineheight{1.25}\smash{\begin{tabular}[t]{l}{\tiny $6$}\end{tabular}}}}%
    \put(0,0){\includegraphics[width=\unitlength,page=3]{source-and-target-of-a-1-tangle-in-dim-2.pdf}}%
  \end{picture}%
\endgroup%

%% file: figuresused/enumerating-all-singularities-whose-link-contains-up-to-two-a1.pdf_tex
\begingroup%
  \makeatletter%
  \providecommand\color[2][]{%
    \errmessage{(Inkscape) Color is used for the text in Inkscape, but the package 'color.sty' is not loaded}%
    \renewcommand\color[2][]{}%
  }%
  \providecommand\transparent[1]{%
    \errmessage{(Inkscape) Transparency is used (non-zero) for the text in Inkscape, but the package 'transparent.sty' is not loaded}%
    \renewcommand\transparent[1]{}%
  }%
  \providecommand\rotatebox[2]{#2}%
  \newcommand*\fsize{\dimexpr\f@size pt\relax}%
  \newcommand*\lineheight[1]{\fontsize{\fsize}{#1\fsize}\selectfont}%
  \ifx\svgwidth\undefined%
    \setlength{\unitlength}{2160bp}%
    \ifx\svgscale\undefined%
      \relax%
    \else%
      \setlength{\unitlength}{\unitlength * \real{\svgscale}}%
    \fi%
  \else%
    \setlength{\unitlength}{\svgwidth}%
  \fi%
  \global\let\svgwidth\undefined%
  \global\let\svgscale\undefined%
  \makeatother%
  \begin{picture}(1,0.11215278)%
    \lineheight{1}%
    \setlength\tabcolsep{0pt}%
    \put(0,0){\includegraphics[width=\unitlength,page=1]{enumerating-all-singularities-whose-link-contains-up-to-two-a1.pdf}}%
    \put(0.48784722,0.04826389){\color[rgb]{0,0,0}\makebox(0,0)[lt]{\lineheight{1.25}\smash{\begin{tabular}[t]{l}{\tiny $\eqv_\sF$}\end{tabular}}}}%
    \put(0,0){\includegraphics[width=\unitlength,page=2]{enumerating-all-singularities-whose-link-contains-up-to-two-a1.pdf}}%
  \end{picture}%
\endgroup%

%% file: figuresused/perturbation-of-2-link-sings.pdf_tex
\begingroup%
  \makeatletter%
  \providecommand\color[2][]{%
    \errmessage{(Inkscape) Color is used for the text in Inkscape, but the package 'color.sty' is not loaded}%
    \renewcommand\color[2][]{}%
  }%
  \providecommand\transparent[1]{%
    \errmessage{(Inkscape) Transparency is used (non-zero) for the text in Inkscape, but the package 'transparent.sty' is not loaded}%
    \renewcommand\transparent[1]{}%
  }%
  \providecommand\rotatebox[2]{#2}%
  \newcommand*\fsize{\dimexpr\f@size pt\relax}%
  \newcommand*\lineheight[1]{\fontsize{\fsize}{#1\fsize}\selectfont}%
  \ifx\svgwidth\undefined%
    \setlength{\unitlength}{2160bp}%
    \ifx\svgscale\undefined%
      \relax%
    \else%
      \setlength{\unitlength}{\unitlength * \real{\svgscale}}%
    \fi%
  \else%
    \setlength{\unitlength}{\svgwidth}%
  \fi%
  \global\let\svgwidth\undefined%
  \global\let\svgscale\undefined%
  \makeatother%
  \begin{picture}(1,0.19479167)%
    \lineheight{1}%
    \setlength\tabcolsep{0pt}%
    \put(0,0){\includegraphics[width=\unitlength,page=1]{perturbation-of-2-link-sings.pdf}}%
    \put(0.24131944,0.128125){\color[rgb]{0,0,0}\makebox(0,0)[lt]{\lineheight{1.25}\smash{\begin{tabular}[t]{l}$\perturb$\end{tabular}}}}%
    \put(0.23715278,0.04895833){\color[rgb]{0,0,0}\makebox(0,0)[lt]{\lineheight{1.25}\smash{\begin{tabular}[t]{l}$\downarrow$\end{tabular}}}}%
    \put(0,0){\includegraphics[width=\unitlength,page=2]{perturbation-of-2-link-sings.pdf}}%
    \put(0.38993056,0.1125){\color[rgb]{0,0,0}\makebox(0,0)[lt]{\lineheight{1.25}\smash{\begin{tabular}[t]{l}{\tiny $\iA_1^{\numovar 2}$}\end{tabular}}}}%
    \put(0.88993056,0.11111111){\color[rgb]{0,0,0}\makebox(0,0)[lt]{\lineheight{1.25}\smash{\begin{tabular}[t]{l}{\tiny $\iA_1^{\numovar 2}$}\end{tabular}}}}%
    \put(0.93020833,0.1625){\color[rgb]{0,0,0}\makebox(0,0)[lt]{\lineheight{1.25}\smash{\begin{tabular}[t]{l}{\tiny $\iA_2$}\end{tabular}}}}%
    \put(0.76215278,0.12986111){\color[rgb]{0,0,0}\makebox(0,0)[lt]{\lineheight{1.25}\smash{\begin{tabular}[t]{l}$\perturb$\end{tabular}}}}%
    \put(0.75798611,0.05069444){\color[rgb]{0,0,0}\makebox(0,0)[lt]{\lineheight{1.25}\smash{\begin{tabular}[t]{l}$\downarrow$\end{tabular}}}}%
    \put(0,0){\includegraphics[width=\unitlength,page=3]{perturbation-of-2-link-sings.pdf}}%
  \end{picture}%
\endgroup%

%% file: figuresused/an-unstable-singularity.pdf_tex
\begingroup%
  \makeatletter%
  \providecommand\color[2][]{%
    \errmessage{(Inkscape) Color is used for the text in Inkscape, but the package 'color.sty' is not loaded}%
    \renewcommand\color[2][]{}%
  }%
  \providecommand\transparent[1]{%
    \errmessage{(Inkscape) Transparency is used (non-zero) for the text in Inkscape, but the package 'transparent.sty' is not loaded}%
    \renewcommand\transparent[1]{}%
  }%
  \providecommand\rotatebox[2]{#2}%
  \newcommand*\fsize{\dimexpr\f@size pt\relax}%
  \newcommand*\lineheight[1]{\fontsize{\fsize}{#1\fsize}\selectfont}%
  \ifx\svgwidth\undefined%
    \setlength{\unitlength}{2160bp}%
    \ifx\svgscale\undefined%
      \relax%
    \else%
      \setlength{\unitlength}{\unitlength * \real{\svgscale}}%
    \fi%
  \else%
    \setlength{\unitlength}{\svgwidth}%
  \fi%
  \global\let\svgwidth\undefined%
  \global\let\svgscale\undefined%
  \makeatother%
  \begin{picture}(1,0.22083333)%
    \lineheight{1}%
    \setlength\tabcolsep{0pt}%
    \put(0,0){\includegraphics[width=\unitlength,page=1]{an-unstable-singularity.pdf}}%
    \put(0.63854167,0.15069444){\color[rgb]{0,0,0}\makebox(0,0)[lt]{\lineheight{1.25}\smash{\begin{tabular}[t]{l}{\tiny $f$}\end{tabular}}}}%
    \put(0.61631944,0.09305556){\color[rgb]{0,0,0}\makebox(0,0)[lt]{\lineheight{1.25}\smash{\begin{tabular}[t]{l}{\tiny $s$}\end{tabular}}}}%
    \put(0,0){\includegraphics[width=\unitlength,page=2]{an-unstable-singularity.pdf}}%
    \put(0.63506944,0.03888889){\color[rgb]{0,0,0}\makebox(0,0)[lt]{\lineheight{1.25}\smash{\begin{tabular}[t]{l}{\tiny $g$}\end{tabular}}}}%
    \put(0,0){\includegraphics[width=\unitlength,page=3]{an-unstable-singularity.pdf}}%
  \end{picture}%
\endgroup%

%% file: figuresused/bending-the-target-into-the-source.pdf_tex
\begingroup%
  \makeatletter%
  \providecommand\color[2][]{%
    \errmessage{(Inkscape) Color is used for the text in Inkscape, but the package 'color.sty' is not loaded}%
    \renewcommand\color[2][]{}%
  }%
  \providecommand\transparent[1]{%
    \errmessage{(Inkscape) Transparency is used (non-zero) for the text in Inkscape, but the package 'transparent.sty' is not loaded}%
    \renewcommand\transparent[1]{}%
  }%
  \providecommand\rotatebox[2]{#2}%
  \newcommand*\fsize{\dimexpr\f@size pt\relax}%
  \newcommand*\lineheight[1]{\fontsize{\fsize}{#1\fsize}\selectfont}%
  \ifx\svgwidth\undefined%
    \setlength{\unitlength}{2160bp}%
    \ifx\svgscale\undefined%
      \relax%
    \else%
      \setlength{\unitlength}{\unitlength * \real{\svgscale}}%
    \fi%
  \else%
    \setlength{\unitlength}{\svgwidth}%
  \fi%
  \global\let\svgwidth\undefined%
  \global\let\svgscale\undefined%
  \makeatother%
  \begin{picture}(1,0.51666667)%
    \lineheight{1}%
    \setlength\tabcolsep{0pt}%
    \put(0,0){\includegraphics[width=\unitlength,page=1]{bending-the-target-into-the-source.pdf}}%
    \put(0.81597222,0.46041667){\color[rgb]{0,0,0}\makebox(0,0)[lt]{\lineheight{1.25}\smash{\begin{tabular}[t]{l}{\tiny $f = f \#_2 \id_j$}\end{tabular}}}}%
    \put(0.796875,0.30902778){\color[rgb]{0,0,0}\makebox(0,0)[lt]{\lineheight{1.25}\smash{\begin{tabular}[t]{l}{\tiny $\id_f \#_2 \mu$}\end{tabular}}}}%
    \put(0,0){\includegraphics[width=\unitlength,page=2]{bending-the-target-into-the-source.pdf}}%
    \put(0.80590278,0.15625){\color[rgb]{0,0,0}\makebox(0,0)[lt]{\lineheight{1.25}\smash{\begin{tabular}[t]{l}{\tiny $f \#_2 g\inv \#_2 g$}\end{tabular}}}}%
    \put(0.78159722,0.10069444){\color[rgb]{0,0,0}\makebox(0,0)[lt]{\lineheight{1.25}\smash{\begin{tabular}[t]{l}{\tiny $s^{\cup} \#_2 \id_g$}\end{tabular}}}}%
    \put(0,0){\includegraphics[width=\unitlength,page=3]{bending-the-target-into-the-source.pdf}}%
    \put(0.80833333,0.04097222){\color[rgb]{0,0,0}\makebox(0,0)[lt]{\lineheight{1.25}\smash{\begin{tabular}[t]{l}{\tiny ${\id_i \#_2 g} = g$}\end{tabular}}}}%
    \put(0.53854167,0.2125){\color[rgb]{0,0,0}\makebox(0,0)[lt]{\lineheight{1.25}\smash{\begin{tabular}[t]{l}{\tiny $\iA_1^{\numovar 2}$}\end{tabular}}}}%
    \put(0.42465278,0.08819444){\color[rgb]{0,0,0}\makebox(0,0)[lt]{\lineheight{1.25}\smash{\begin{tabular}[t]{l}{\tiny $s^{\cup}$}\end{tabular}}}}%
    \put(0.55729167,0.31527778){\color[rgb]{0,0,0}\makebox(0,0)[lt]{\lineheight{1.25}\smash{\begin{tabular}[t]{l}{\tiny $\iA_1^{\numovar 2}$}\end{tabular}}}}%
    \put(0.52951389,0.42083333){\color[rgb]{0,0,0}\makebox(0,0)[lt]{\lineheight{1.25}\smash{\begin{tabular}[t]{l}{\tiny $\iA_1^{\numovar 2}$}\end{tabular}}}}%
    \put(0,0){\includegraphics[width=\unitlength,page=4]{bending-the-target-into-the-source.pdf}}%
  \end{picture}%
\endgroup%

%% file: figuresused/removing-wiggles-from-the-domain.pdf_tex
\begingroup%
  \makeatletter%
  \providecommand\color[2][]{%
    \errmessage{(Inkscape) Color is used for the text in Inkscape, but the package 'color.sty' is not loaded}%
    \renewcommand\color[2][]{}%
  }%
  \providecommand\transparent[1]{%
    \errmessage{(Inkscape) Transparency is used (non-zero) for the text in Inkscape, but the package 'transparent.sty' is not loaded}%
    \renewcommand\transparent[1]{}%
  }%
  \providecommand\rotatebox[2]{#2}%
  \newcommand*\fsize{\dimexpr\f@size pt\relax}%
  \newcommand*\lineheight[1]{\fontsize{\fsize}{#1\fsize}\selectfont}%
  \ifx\svgwidth\undefined%
    \setlength{\unitlength}{2160bp}%
    \ifx\svgscale\undefined%
      \relax%
    \else%
      \setlength{\unitlength}{\unitlength * \real{\svgscale}}%
    \fi%
  \else%
    \setlength{\unitlength}{\svgwidth}%
  \fi%
  \global\let\svgwidth\undefined%
  \global\let\svgscale\undefined%
  \makeatother%
  \begin{picture}(1,0.27951389)%
    \lineheight{1}%
    \setlength\tabcolsep{0pt}%
    \put(0,0){\includegraphics[width=\unitlength,page=1]{removing-wiggles-from-the-domain.pdf}}%
    \put(0.70729167,0.23263889){\color[rgb]{0,0,0}\makebox(0,0)[lt]{\lineheight{1.25}\smash{\begin{tabular}[t]{l}{\tiny $h := f \stack 2 g\inv$}\end{tabular}}}}%
    \put(0,0){\includegraphics[width=\unitlength,page=2]{removing-wiggles-from-the-domain.pdf}}%
    \put(0.72743056,0.13819444){\color[rgb]{0,0,0}\makebox(0,0)[lt]{\lineheight{1.25}\smash{\begin{tabular}[t]{l}{\tiny $\omega$}\end{tabular}}}}%
    \put(0,0){\includegraphics[width=\unitlength,page=3]{removing-wiggles-from-the-domain.pdf}}%
    \put(0.74618056,0.03055556){\color[rgb]{0,0,0}\makebox(0,0)[lt]{\lineheight{1.25}\smash{\begin{tabular}[t]{l}{\tiny $\tilde h$}\end{tabular}}}}%
    \put(0.43055556,0.18680556){\color[rgb]{0,0,0}\makebox(0,0)[lt]{\lineheight{1.25}\smash{\begin{tabular}[t]{l}{\tiny $\iA_2$}\end{tabular}}}}%
    \put(0.55208333,0.08263889){\color[rgb]{0,0,0}\makebox(0,0)[lt]{\lineheight{1.25}\smash{\begin{tabular}[t]{l}{\tiny $\iA_2$}\end{tabular}}}}%
    \put(0,0){\includegraphics[width=\unitlength,page=4]{removing-wiggles-from-the-domain.pdf}}%
  \end{picture}%
\endgroup%

%% file: figuresused/inductively-removing-remaining-cups-and-caps.pdf_tex
\begingroup%
  \makeatletter%
  \providecommand\color[2][]{%
    \errmessage{(Inkscape) Color is used for the text in Inkscape, but the package 'color.sty' is not loaded}%
    \renewcommand\color[2][]{}%
  }%
  \providecommand\transparent[1]{%
    \errmessage{(Inkscape) Transparency is used (non-zero) for the text in Inkscape, but the package 'transparent.sty' is not loaded}%
    \renewcommand\transparent[1]{}%
  }%
  \providecommand\rotatebox[2]{#2}%
  \newcommand*\fsize{\dimexpr\f@size pt\relax}%
  \newcommand*\lineheight[1]{\fontsize{\fsize}{#1\fsize}\selectfont}%
  \ifx\svgwidth\undefined%
    \setlength{\unitlength}{2160bp}%
    \ifx\svgscale\undefined%
      \relax%
    \else%
      \setlength{\unitlength}{\unitlength * \real{\svgscale}}%
    \fi%
  \else%
    \setlength{\unitlength}{\svgwidth}%
  \fi%
  \global\let\svgwidth\undefined%
  \global\let\svgscale\undefined%
  \makeatother%
  \begin{picture}(1,0.30972222)%
    \lineheight{1}%
    \setlength\tabcolsep{0pt}%
    \put(0,0){\includegraphics[width=\unitlength,page=1]{inductively-removing-remaining-cups-and-caps.pdf}}%
    \put(0.67881944,0.16875){\color[rgb]{0,0,0}\makebox(0,0)[lt]{\lineheight{1.25}\smash{\begin{tabular}[t]{l}{\tiny $\iA_2$}\end{tabular}}}}%
    \put(0.59548611,0.16805556){\color[rgb]{0,0,0}\makebox(0,0)[lt]{\lineheight{1.25}\smash{\begin{tabular}[t]{l}{\tiny $\iA_2$}\end{tabular}}}}%
    \put(0.846875,0.11944444){\color[rgb]{0,0,0}\makebox(0,0)[lt]{\lineheight{1.25}\smash{\begin{tabular}[t]{l}{\tiny $\iA_1^{\numovar 2}$}\end{tabular}}}}%
    \put(0.71666667,0.09166667){\color[rgb]{0,0,0}\makebox(0,0)[lt]{\lineheight{1.25}\smash{\begin{tabular}[t]{l}{\tiny $\iA_1^{\numovar 2}$}\end{tabular}}}}%
    \put(0.55590278,0.09166667){\color[rgb]{0,0,0}\makebox(0,0)[lt]{\lineheight{1.25}\smash{\begin{tabular}[t]{l}{\tiny $\iA_1^{\numovar 2}$}\end{tabular}}}}%
    \put(0.33159722,0.23541667){\color[rgb]{0,0,0}\makebox(0,0)[lt]{\lineheight{1.25}\smash{\begin{tabular}[t]{l}{\tiny $\tilde h$}\end{tabular}}}}%
    \put(0,0){\includegraphics[width=\unitlength,page=2]{inductively-removing-remaining-cups-and-caps.pdf}}%
    \put(0.32916667,0.03263889){\color[rgb]{0,0,0}\makebox(0,0)[lt]{\lineheight{1.25}\smash{\begin{tabular}[t]{l}{\tiny $\id_3$}\end{tabular}}}}%
    \put(0.34895833,0.12986111){\color[rgb]{0,0,0}\makebox(0,0)[lt]{\lineheight{1.25}\smash{\begin{tabular}[t]{l}{\tiny $\tau$}\end{tabular}}}}%
    \put(0.24131944,0.17986111){\color[rgb]{0,0,0}\makebox(0,0)[lt]{\lineheight{1.25}\smash{\begin{tabular}[t]{l}{\tiny $\iA_1^{\numovar 2}$}\end{tabular}}}}%
    \put(0.22847222,0.07847222){\color[rgb]{0,0,0}\makebox(0,0)[lt]{\lineheight{1.25}\smash{\begin{tabular}[t]{l}{\tiny $\iA_1^{\numovar 2}$}\end{tabular}}}}%
    \put(0.203125,0.20833333){\color[rgb]{0,0,0}\makebox(0,0)[lt]{\lineheight{1.25}\smash{\begin{tabular}[t]{l}{\tiny $\iA_2$}\end{tabular}}}}%
    \put(0.13159722,0.28680556){\color[rgb]{0,0,0}\makebox(0,0)[lt]{\lineheight{1.25}\smash{\begin{tabular}[t]{l}{\tiny $\tilde h : 3 \to 3$}\end{tabular}}}}%
    \put(0,0){\includegraphics[width=\unitlength,page=3]{inductively-removing-remaining-cups-and-caps.pdf}}%
  \end{picture}%
\endgroup%

%% file: figuresused/stable-1-tangle-singularities-in-dimension-3.pdf_tex
\begingroup%
  \makeatletter%
  \providecommand\color[2][]{%
    \errmessage{(Inkscape) Color is used for the text in Inkscape, but the package 'color.sty' is not loaded}%
    \renewcommand\color[2][]{}%
  }%
  \providecommand\transparent[1]{%
    \errmessage{(Inkscape) Transparency is used (non-zero) for the text in Inkscape, but the package 'transparent.sty' is not loaded}%
    \renewcommand\transparent[1]{}%
  }%
  \providecommand\rotatebox[2]{#2}%
  \newcommand*\fsize{\dimexpr\f@size pt\relax}%
  \newcommand*\lineheight[1]{\fontsize{\fsize}{#1\fsize}\selectfont}%
  \ifx\svgwidth\undefined%
    \setlength{\unitlength}{2160bp}%
    \ifx\svgscale\undefined%
      \relax%
    \else%
      \setlength{\unitlength}{\unitlength * \real{\svgscale}}%
    \fi%
  \else%
    \setlength{\unitlength}{\svgwidth}%
  \fi%
  \global\let\svgwidth\undefined%
  \global\let\svgscale\undefined%
  \makeatother%
  \begin{picture}(1,0.19444444)%
    \lineheight{1}%
    \setlength\tabcolsep{0pt}%
    \put(0,0){\includegraphics[width=\unitlength,page=1]{stable-1-tangle-singularities-in-dimension-3.pdf}}%
    \put(0.29201389,0.07777778){\color[rgb]{0,0,0}\makebox(0,0)[lt]{\lineheight{1.25}\smash{\begin{tabular}[t]{l}{\tiny 1}\end{tabular}}}}%
    \put(0.30520833,0.09618056){\color[rgb]{0,0,0}\makebox(0,0)[lt]{\lineheight{1.25}\smash{\begin{tabular}[t]{l}{\tiny 2}\end{tabular}}}}%
    \put(0.25729167,0.13576389){\color[rgb]{0,0,0}\makebox(0,0)[lt]{\lineheight{1.25}\smash{\begin{tabular}[t]{l}{\tiny 3}\end{tabular}}}}%
    \put(0,0){\includegraphics[width=\unitlength,page=2]{stable-1-tangle-singularities-in-dimension-3.pdf}}%
    \put(0.32708333,0.03784722){\color[rgb]{0,0,0}\makebox(0,0)[lt]{\lineheight{1.25}\smash{\begin{tabular}[t]{l}$\underbracket[0.140ex]{\hspace{2.05in}}_{\iA_1 \oplus \eps}$\end{tabular}}}}%
  \end{picture}%
\endgroup%

%% file: figuresused/stable-1-tangle-singularities-in-dimension-3-non-generic.pdf_tex
\begingroup%
  \makeatletter%
  \providecommand\color[2][]{%
    \errmessage{(Inkscape) Color is used for the text in Inkscape, but the package 'color.sty' is not loaded}%
    \renewcommand\color[2][]{}%
  }%
  \providecommand\transparent[1]{%
    \errmessage{(Inkscape) Transparency is used (non-zero) for the text in Inkscape, but the package 'transparent.sty' is not loaded}%
    \renewcommand\transparent[1]{}%
  }%
  \providecommand\rotatebox[2]{#2}%
  \newcommand*\fsize{\dimexpr\f@size pt\relax}%
  \newcommand*\lineheight[1]{\fontsize{\fsize}{#1\fsize}\selectfont}%
  \ifx\svgwidth\undefined%
    \setlength{\unitlength}{2160bp}%
    \ifx\svgscale\undefined%
      \relax%
    \else%
      \setlength{\unitlength}{\unitlength * \real{\svgscale}}%
    \fi%
  \else%
    \setlength{\unitlength}{\svgwidth}%
  \fi%
  \global\let\svgwidth\undefined%
  \global\let\svgscale\undefined%
  \makeatother%
  \begin{picture}(1,0.12847222)%
    \lineheight{1}%
    \setlength\tabcolsep{0pt}%
    \put(0,0){\includegraphics[width=\unitlength,page=1]{stable-1-tangle-singularities-in-dimension-3-non-generic.pdf}}%
    \put(0.29826389,0.01041667){\color[rgb]{0,0,0}\makebox(0,0)[lt]{\lineheight{1.25}\smash{\begin{tabular}[t]{l}{\tiny 1}\end{tabular}}}}%
    \put(0.31076389,0.03020833){\color[rgb]{0,0,0}\makebox(0,0)[lt]{\lineheight{1.25}\smash{\begin{tabular}[t]{l}{\tiny 2}\end{tabular}}}}%
    \put(0.26319444,0.06979167){\color[rgb]{0,0,0}\makebox(0,0)[lt]{\lineheight{1.25}\smash{\begin{tabular}[t]{l}{\tiny 3}\end{tabular}}}}%
    \put(0,0){\includegraphics[width=\unitlength,page=2]{stable-1-tangle-singularities-in-dimension-3-non-generic.pdf}}%
  \end{picture}%
\endgroup%

%% file: figuresused/three-stable-2-tangle-singularities-in-dimension-4.pdf_tex
\begingroup%
  \makeatletter%
  \providecommand\color[2][]{%
    \errmessage{(Inkscape) Color is used for the text in Inkscape, but the package 'color.sty' is not loaded}%
    \renewcommand\color[2][]{}%
  }%
  \providecommand\transparent[1]{%
    \errmessage{(Inkscape) Transparency is used (non-zero) for the text in Inkscape, but the package 'transparent.sty' is not loaded}%
    \renewcommand\transparent[1]{}%
  }%
  \providecommand\rotatebox[2]{#2}%
  \newcommand*\fsize{\dimexpr\f@size pt\relax}%
  \newcommand*\lineheight[1]{\fontsize{\fsize}{#1\fsize}\selectfont}%
  \ifx\svgwidth\undefined%
    \setlength{\unitlength}{2160bp}%
    \ifx\svgscale\undefined%
      \relax%
    \else%
      \setlength{\unitlength}{\unitlength * \real{\svgscale}}%
    \fi%
  \else%
    \setlength{\unitlength}{\svgwidth}%
  \fi%
  \global\let\svgwidth\undefined%
  \global\let\svgscale\undefined%
  \makeatother%
  \begin{picture}(1,0.39479167)%
    \lineheight{1}%
    \setlength\tabcolsep{0pt}%
    \put(0,0){\includegraphics[width=\unitlength,page=1]{three-stable-2-tangle-singularities-in-dimension-4.pdf}}%
    \put(0.27708333,0.29583333){\color[rgb]{0,0,0}\makebox(0,0)[lt]{\lineheight{1.25}\smash{\begin{tabular}[t]{l}{\tiny 1}\end{tabular}}}}%
    \put(0.28923611,0.31631944){\color[rgb]{0,0,0}\makebox(0,0)[lt]{\lineheight{1.25}\smash{\begin{tabular}[t]{l}{\tiny 2}\end{tabular}}}}%
    \put(0.24131944,0.35590278){\color[rgb]{0,0,0}\makebox(0,0)[lt]{\lineheight{1.25}\smash{\begin{tabular}[t]{l}{\tiny 3}\end{tabular}}}}%
    \put(0,0){\includegraphics[width=\unitlength,page=2]{three-stable-2-tangle-singularities-in-dimension-4.pdf}}%
    \put(0.45347222,0.35347222){\color[rgb]{0,0,0}\makebox(0,0)[lt]{\lineheight{1.25}\smash{\begin{tabular}[t]{l}{\tiny 4}\end{tabular}}}}%
    \put(0,0){\includegraphics[width=\unitlength,page=3]{three-stable-2-tangle-singularities-in-dimension-4.pdf}}%
    \put(0.61319444,0.35347222){\color[rgb]{0,0,0}\makebox(0,0)[lt]{\lineheight{1.25}\smash{\begin{tabular}[t]{l}{\tiny 4}\end{tabular}}}}%
    \put(0,0){\includegraphics[width=\unitlength,page=4]{three-stable-2-tangle-singularities-in-dimension-4.pdf}}%
    \put(0.45347222,0.20798611){\color[rgb]{0,0,0}\makebox(0,0)[lt]{\lineheight{1.25}\smash{\begin{tabular}[t]{l}{\tiny 4}\end{tabular}}}}%
    \put(0,0){\includegraphics[width=\unitlength,page=5]{three-stable-2-tangle-singularities-in-dimension-4.pdf}}%
    \put(0.61319444,0.20798611){\color[rgb]{0,0,0}\makebox(0,0)[lt]{\lineheight{1.25}\smash{\begin{tabular}[t]{l}{\tiny 4}\end{tabular}}}}%
    \put(0,0){\includegraphics[width=\unitlength,page=6]{three-stable-2-tangle-singularities-in-dimension-4.pdf}}%
    \put(0.45347222,0.0625){\color[rgb]{0,0,0}\makebox(0,0)[lt]{\lineheight{1.25}\smash{\begin{tabular}[t]{l}{\tiny 4}\end{tabular}}}}%
    \put(0,0){\includegraphics[width=\unitlength,page=7]{three-stable-2-tangle-singularities-in-dimension-4.pdf}}%
    \put(0.61319444,0.0625){\color[rgb]{0,0,0}\makebox(0,0)[lt]{\lineheight{1.25}\smash{\begin{tabular}[t]{l}{\tiny 4}\end{tabular}}}}%
    \put(0,0){\includegraphics[width=\unitlength,page=8]{three-stable-2-tangle-singularities-in-dimension-4.pdf}}%
    \put(0.78541667,0.13888889){\color[rgb]{0,0,0}\makebox(0,0)[lt]{\lineheight{1.25}\smash{\begin{tabular}[t]{l}\rotatebox{90}{$\underbracket[0.140ex]{\hspace{1.45in}}_{\iA_1^{\numovar 2} \oplus \eps}$}\end{tabular}}}}%
    \put(0.80381944,0.01006944){\color[rgb]{0,0,0}\makebox(0,0)[lt]{\lineheight{1.25}\smash{\begin{tabular}[t]{l}\rotatebox{90}{\scriptsize $\iA_2 \oplus \eps$}\end{tabular}}}}%
    \put(0,0){\includegraphics[width=\unitlength,page=9]{three-stable-2-tangle-singularities-in-dimension-4.pdf}}%
  \end{picture}%
\endgroup%

%% file: figuresused/the-braid-eating-singularity.pdf_tex
\begingroup%
  \makeatletter%
  \providecommand\color[2][]{%
    \errmessage{(Inkscape) Color is used for the text in Inkscape, but the package 'color.sty' is not loaded}%
    \renewcommand\color[2][]{}%
  }%
  \providecommand\transparent[1]{%
    \errmessage{(Inkscape) Transparency is used (non-zero) for the text in Inkscape, but the package 'transparent.sty' is not loaded}%
    \renewcommand\transparent[1]{}%
  }%
  \providecommand\rotatebox[2]{#2}%
  \newcommand*\fsize{\dimexpr\f@size pt\relax}%
  \newcommand*\lineheight[1]{\fontsize{\fsize}{#1\fsize}\selectfont}%
  \ifx\svgwidth\undefined%
    \setlength{\unitlength}{2160bp}%
    \ifx\svgscale\undefined%
      \relax%
    \else%
      \setlength{\unitlength}{\unitlength * \real{\svgscale}}%
    \fi%
  \else%
    \setlength{\unitlength}{\svgwidth}%
  \fi%
  \global\let\svgwidth\undefined%
  \global\let\svgscale\undefined%
  \makeatother%
  \begin{picture}(1,0.10347222)%
    \lineheight{1}%
    \setlength\tabcolsep{0pt}%
    \put(0,0){\includegraphics[width=\unitlength,page=1]{the-braid-eating-singularity.pdf}}%
    \put(0.27708333,0.00520833){\color[rgb]{0,0,0}\makebox(0,0)[lt]{\lineheight{1.25}\smash{\begin{tabular}[t]{l}{\tiny 1}\end{tabular}}}}%
    \put(0.28923611,0.02569444){\color[rgb]{0,0,0}\makebox(0,0)[lt]{\lineheight{1.25}\smash{\begin{tabular}[t]{l}{\tiny 2}\end{tabular}}}}%
    \put(0.24131944,0.06527778){\color[rgb]{0,0,0}\makebox(0,0)[lt]{\lineheight{1.25}\smash{\begin{tabular}[t]{l}{\tiny 3}\end{tabular}}}}%
    \put(0,0){\includegraphics[width=\unitlength,page=2]{the-braid-eating-singularity.pdf}}%
    \put(0.45347222,0.06284722){\color[rgb]{0,0,0}\makebox(0,0)[lt]{\lineheight{1.25}\smash{\begin{tabular}[t]{l}{\tiny 4}\end{tabular}}}}%
    \put(0,0){\includegraphics[width=\unitlength,page=3]{the-braid-eating-singularity.pdf}}%
    \put(0.61319444,0.06284722){\color[rgb]{0,0,0}\makebox(0,0)[lt]{\lineheight{1.25}\smash{\begin{tabular}[t]{l}{\tiny 4}\end{tabular}}}}%
    \put(0,0){\includegraphics[width=\unitlength,page=4]{the-braid-eating-singularity.pdf}}%
  \end{picture}%
\endgroup%

%% file: figuresused/an-unstable-2-singularity-in-dim-4.pdf_tex
\begingroup%
  \makeatletter%
  \providecommand\color[2][]{%
    \errmessage{(Inkscape) Color is used for the text in Inkscape, but the package 'color.sty' is not loaded}%
    \renewcommand\color[2][]{}%
  }%
  \providecommand\transparent[1]{%
    \errmessage{(Inkscape) Transparency is used (non-zero) for the text in Inkscape, but the package 'transparent.sty' is not loaded}%
    \renewcommand\transparent[1]{}%
  }%
  \providecommand\rotatebox[2]{#2}%
  \newcommand*\fsize{\dimexpr\f@size pt\relax}%
  \newcommand*\lineheight[1]{\fontsize{\fsize}{#1\fsize}\selectfont}%
  \ifx\svgwidth\undefined%
    \setlength{\unitlength}{2160bp}%
    \ifx\svgscale\undefined%
      \relax%
    \else%
      \setlength{\unitlength}{\unitlength * \real{\svgscale}}%
    \fi%
  \else%
    \setlength{\unitlength}{\svgwidth}%
  \fi%
  \global\let\svgwidth\undefined%
  \global\let\svgscale\undefined%
  \makeatother%
  \begin{picture}(1,0.19166667)%
    \lineheight{1}%
    \setlength\tabcolsep{0pt}%
    \put(0,0){\includegraphics[width=\unitlength,page=1]{an-unstable-2-singularity-in-dim-4.pdf}}%
    \put(0.35694444,0.10486111){\color[rgb]{0,0,0}\makebox(0,0)[lt]{\lineheight{1.25}\smash{\begin{tabular}[t]{l}{\tiny 4}\end{tabular}}}}%
    \put(0,0){\includegraphics[width=\unitlength,page=2]{an-unstable-2-singularity-in-dim-4.pdf}}%
    \put(0.63263889,0.10486111){\color[rgb]{0,0,0}\makebox(0,0)[lt]{\lineheight{1.25}\smash{\begin{tabular}[t]{l}{\tiny 4}\end{tabular}}}}%
    \put(0,0){\includegraphics[width=\unitlength,page=3]{an-unstable-2-singularity-in-dim-4.pdf}}%
  \end{picture}%
\endgroup%

%% file: figuresused/perturbation-of-an-unstable-2-singularity-in-dim-4.pdf_tex
\begingroup%
  \makeatletter%
  \providecommand\color[2][]{%
    \errmessage{(Inkscape) Color is used for the text in Inkscape, but the package 'color.sty' is not loaded}%
    \renewcommand\color[2][]{}%
  }%
  \providecommand\transparent[1]{%
    \errmessage{(Inkscape) Transparency is used (non-zero) for the text in Inkscape, but the package 'transparent.sty' is not loaded}%
    \renewcommand\transparent[1]{}%
  }%
  \providecommand\rotatebox[2]{#2}%
  \newcommand*\fsize{\dimexpr\f@size pt\relax}%
  \newcommand*\lineheight[1]{\fontsize{\fsize}{#1\fsize}\selectfont}%
  \ifx\svgwidth\undefined%
    \setlength{\unitlength}{2160bp}%
    \ifx\svgscale\undefined%
      \relax%
    \else%
      \setlength{\unitlength}{\unitlength * \real{\svgscale}}%
    \fi%
  \else%
    \setlength{\unitlength}{\svgwidth}%
  \fi%
  \global\let\svgwidth\undefined%
  \global\let\svgscale\undefined%
  \makeatother%
  \begin{picture}(1,0.22916667)%
    \lineheight{1}%
    \setlength\tabcolsep{0pt}%
    \put(0,0){\includegraphics[width=\unitlength,page=1]{perturbation-of-an-unstable-2-singularity-in-dim-4.pdf}}%
    \put(0.19270833,0.128125){\color[rgb]{0,0,0}\makebox(0,0)[lt]{\lineheight{1.25}\smash{\begin{tabular}[t]{l}{\tiny $\xra{\text{step 1}}$}\end{tabular}}}}%
    \put(0.47083333,0.128125){\color[rgb]{0,0,0}\makebox(0,0)[lt]{\lineheight{1.25}\smash{\begin{tabular}[t]{l}{\tiny $\xra{\text{step 2}}$}\end{tabular}}}}%
    \put(0.74652778,0.128125){\color[rgb]{0,0,0}\makebox(0,0)[lt]{\lineheight{1.25}\smash{\begin{tabular}[t]{l}{\tiny $\xra{\text{step 3}}$}\end{tabular}}}}%
    \put(0.18923611,0.009375){\color[rgb]{0,0,0}\makebox(0,0)[lt]{\lineheight{1.25}\smash{\begin{tabular}[t]{l}{\tiny genericize}\end{tabular}}}}%
    \put(0.453125,0.009375){\color[rgb]{0,0,0}\makebox(0,0)[lt]{\lineheight{1.25}\smash{\begin{tabular}[t]{l}{\tiny remove braids}\end{tabular}}}}%
    \put(0.72743056,0.01840278){\color[rgb]{0,0,0}\makebox(0,0)[lt]{\lineheight{1.25}\smash{\begin{tabular}[t]{l}{\tiny remove circles,}\end{tabular}}}}%
    \put(0.69826389,-0.00243056){\color[rgb]{0,0,0}\makebox(0,0)[lt]{\lineheight{1.25}\smash{\begin{tabular}[t]{l}{\tiny wiggles, and simplify}\end{tabular}}}}%
    \put(0,0){\includegraphics[width=\unitlength,page=2]{perturbation-of-an-unstable-2-singularity-in-dim-4.pdf}}%
  \end{picture}%
\endgroup%

%% file: figuresused/translating-extrema-saddles-cusps.pdf_tex
\begingroup%
  \makeatletter%
  \providecommand\color[2][]{%
    \errmessage{(Inkscape) Color is used for the text in Inkscape, but the package 'color.sty' is not loaded}%
    \renewcommand\color[2][]{}%
  }%
  \providecommand\transparent[1]{%
    \errmessage{(Inkscape) Transparency is used (non-zero) for the text in Inkscape, but the package 'transparent.sty' is not loaded}%
    \renewcommand\transparent[1]{}%
  }%
  \providecommand\rotatebox[2]{#2}%
  \newcommand*\fsize{\dimexpr\f@size pt\relax}%
  \newcommand*\lineheight[1]{\fontsize{\fsize}{#1\fsize}\selectfont}%
  \ifx\svgwidth\undefined%
    \setlength{\unitlength}{2160bp}%
    \ifx\svgscale\undefined%
      \relax%
    \else%
      \setlength{\unitlength}{\unitlength * \real{\svgscale}}%
    \fi%
  \else%
    \setlength{\unitlength}{\svgwidth}%
  \fi%
  \global\let\svgwidth\undefined%
  \global\let\svgscale\undefined%
  \makeatother%
  \begin{picture}(1,0.44583333)%
    \lineheight{1}%
    \setlength\tabcolsep{0pt}%
    \put(0,0){\includegraphics[width=\unitlength,page=1]{translating-extrema-saddles-cusps.pdf}}%
    \put(0.92743056,0.26388889){\color[rgb]{0,0,0}\makebox(0,0)[lt]{\lineheight{1.25}\smash{\begin{tabular}[t]{l}{\tiny $u_1$}\end{tabular}}}}%
    \put(0.77604167,0.23402778){\color[rgb]{0,0,0}\makebox(0,0)[lt]{\lineheight{1.25}\smash{\begin{tabular}[t]{l}{\tiny $x_1$}\end{tabular}}}}%
    \put(0.64236111,0.30069444){\color[rgb]{0,0,0}\makebox(0,0)[lt]{\lineheight{1.25}\smash{\begin{tabular}[t]{l}{\tiny $h_{u_1}(x_1)$}\end{tabular}}}}%
    \put(0.44548611,0.21458333){\color[rgb]{0,0,0}\makebox(0,0)[lt]{\lineheight{1.25}\smash{\begin{tabular}[t]{l}{\tiny $x_2$}\end{tabular}}}}%
    \put(0.40763889,0.42638889){\color[rgb]{0,0,0}\makebox(0,0)[lt]{\lineheight{1.25}\smash{\begin{tabular}[t]{l}{\tiny $x_1$}\end{tabular}}}}%
    \put(0.32534722,0.31527778){\color[rgb]{0,0,0}\makebox(0,0)[lt]{\lineheight{1.25}\smash{\begin{tabular}[t]{l}{\tiny $g(x_1,x_2)$}\end{tabular}}}}%
    \put(0.13229167,0.21354167){\color[rgb]{0,0,0}\makebox(0,0)[lt]{\lineheight{1.25}\smash{\begin{tabular}[t]{l}{\tiny $x_2$}\end{tabular}}}}%
    \put(0.09444444,0.42534722){\color[rgb]{0,0,0}\makebox(0,0)[lt]{\lineheight{1.25}\smash{\begin{tabular}[t]{l}{\tiny $x_1$}\end{tabular}}}}%
    \put(0.00590278,0.31423611){\color[rgb]{0,0,0}\makebox(0,0)[lt]{\lineheight{1.25}\smash{\begin{tabular}[t]{l}{\tiny $f(x_1,x_2)$}\end{tabular}}}}%
    \put(0,0){\includegraphics[width=\unitlength,page=2]{translating-extrema-saddles-cusps.pdf}}%
    \put(0.05902778,0.03020833){\color[rgb]{0,0,0}\makebox(0,0)[lt]{\lineheight{1.25}\smash{\begin{tabular}[t]{l}{\tiny 1}\end{tabular}}}}%
    \put(0.06770833,0.00972222){\color[rgb]{0,0,0}\makebox(0,0)[lt]{\lineheight{1.25}\smash{\begin{tabular}[t]{l}{\tiny 2}\end{tabular}}}}%
    \put(0.02048611,0.04305556){\color[rgb]{0,0,0}\makebox(0,0)[lt]{\lineheight{1.25}\smash{\begin{tabular}[t]{l}{\tiny 3}\end{tabular}}}}%
    \put(0,0){\includegraphics[width=\unitlength,page=3]{translating-extrema-saddles-cusps.pdf}}%
    \put(0.38541667,0.03090278){\color[rgb]{0,0,0}\makebox(0,0)[lt]{\lineheight{1.25}\smash{\begin{tabular}[t]{l}{\tiny 1}\end{tabular}}}}%
    \put(0.39409722,0.01041667){\color[rgb]{0,0,0}\makebox(0,0)[lt]{\lineheight{1.25}\smash{\begin{tabular}[t]{l}{\tiny 2}\end{tabular}}}}%
    \put(0.34756944,0.04375){\color[rgb]{0,0,0}\makebox(0,0)[lt]{\lineheight{1.25}\smash{\begin{tabular}[t]{l}{\tiny 3}\end{tabular}}}}%
    \put(0,0){\includegraphics[width=\unitlength,page=4]{translating-extrema-saddles-cusps.pdf}}%
    \put(0.72604167,0.01006944){\color[rgb]{0,0,0}\makebox(0,0)[lt]{\lineheight{1.25}\smash{\begin{tabular}[t]{l}{\tiny 1}\end{tabular}}}}%
    \put(0.68229167,0.04340278){\color[rgb]{0,0,0}\makebox(0,0)[lt]{\lineheight{1.25}\smash{\begin{tabular}[t]{l}{\tiny 2}\end{tabular}}}}%
    \put(0.71701389,0.03090278){\color[rgb]{0,0,0}\makebox(0,0)[lt]{\lineheight{1.25}\smash{\begin{tabular}[t]{l}{\tiny 3}\end{tabular}}}}%
    \put(0,0){\includegraphics[width=\unitlength,page=5]{translating-extrema-saddles-cusps.pdf}}%
  \end{picture}%
\endgroup%

%% file: figuresused/stable-3-tangle-singularities-2.pdf_tex
\begingroup%
  \makeatletter%
  \providecommand\color[2][]{%
    \errmessage{(Inkscape) Color is used for the text in Inkscape, but the package 'color.sty' is not loaded}%
    \renewcommand\color[2][]{}%
  }%
  \providecommand\transparent[1]{%
    \errmessage{(Inkscape) Transparency is used (non-zero) for the text in Inkscape, but the package 'transparent.sty' is not loaded}%
    \renewcommand\transparent[1]{}%
  }%
  \providecommand\rotatebox[2]{#2}%
  \newcommand*\fsize{\dimexpr\f@size pt\relax}%
  \newcommand*\lineheight[1]{\fontsize{\fsize}{#1\fsize}\selectfont}%
  \ifx\svgwidth\undefined%
    \setlength{\unitlength}{2160bp}%
    \ifx\svgscale\undefined%
      \relax%
    \else%
      \setlength{\unitlength}{\unitlength * \real{\svgscale}}%
    \fi%
  \else%
    \setlength{\unitlength}{\svgwidth}%
  \fi%
  \global\let\svgwidth\undefined%
  \global\let\svgscale\undefined%
  \makeatother%
  \begin{picture}(1,0.76875)%
    \lineheight{1}%
    \setlength\tabcolsep{0pt}%
    \put(0.49826389,0.06215278){\color[rgb]{0,0,0}\makebox(0,0)[lt]{\lineheight{1.25}\smash{\begin{tabular}[t]{l}{\tiny $=$}\end{tabular}}}}%
    \put(0.490625,0.33159722){\color[rgb]{0,0,0}\makebox(0,0)[lt]{\lineheight{1.25}\smash{\begin{tabular}[t]{l}\rotatebox{90}{\tiny $\singa {\iA_1} 3 {\iA_2}$}\end{tabular}}}}%
    \put(0.52152778,0.16076389){\color[rgb]{0,0,0}\makebox(0,0)[lt]{\lineheight{1.25}\smash{\begin{tabular}[t]{l}\rotatebox{90}{\tiny $\singa {\iA_3} {(3,2,1)} {\iA_2}$}\end{tabular}}}}%
    \put(0.52152778,0.02638889){\color[rgb]{0,0,0}\makebox(0,0)[lt]{\lineheight{1.25}\smash{\begin{tabular}[t]{l}\rotatebox{90}{\tiny $\singa {\iA_1} 2 {\iA_2}$}\end{tabular}}}}%
    \put(0.49131944,0.52951389){\color[rgb]{0,0,0}\makebox(0,0)[lt]{\lineheight{1.25}\smash{\begin{tabular}[t]{l}\rotatebox{90}{\tiny $\singa {\iA_1} 3 {\iA_1^{\numovar 2}}$}\end{tabular}}}}%
    \put(0.46423611,0.16493056){\color[rgb]{0,0,0}\makebox(0,0)[lt]{\lineheight{1.25}\smash{\begin{tabular}[t]{l}\rotatebox{90}{\tiny $\singa {\iA_2} {(3,2)} {\iA_1^{\numovar 2}}$}\end{tabular}}}}%
    \put(0.46423611,0.00798611){\color[rgb]{0,0,0}\makebox(0,0)[lt]{\lineheight{1.25}\smash{\begin{tabular}[t]{l}\rotatebox{90}{\tiny $\singa {\iA_2} {(3,1)} {\iA_1^{\numovar 2}}$}\end{tabular}}}}%
    \put(0,0){\includegraphics[width=\unitlength,page=1]{stable-3-tangle-singularities-2.pdf}}%
    \put(0.14791667,0.68159722){\color[rgb]{0,0,0}\makebox(0,0)[lt]{\lineheight{1.25}\smash{\begin{tabular}[t]{l}{\tiny 4}\end{tabular}}}}%
    \put(0,0){\includegraphics[width=\unitlength,page=2]{stable-3-tangle-singularities-2.pdf}}%
    \put(0.30381944,0.68159722){\color[rgb]{0,0,0}\makebox(0,0)[lt]{\lineheight{1.25}\smash{\begin{tabular}[t]{l}{\tiny 4}\end{tabular}}}}%
    \put(0,0){\includegraphics[width=\unitlength,page=3]{stable-3-tangle-singularities-2.pdf}}%
    \put(0.14791667,0.53159722){\color[rgb]{0,0,0}\makebox(0,0)[lt]{\lineheight{1.25}\smash{\begin{tabular}[t]{l}{\tiny 4}\end{tabular}}}}%
    \put(0,0){\includegraphics[width=\unitlength,page=4]{stable-3-tangle-singularities-2.pdf}}%
    \put(0.30381944,0.53159722){\color[rgb]{0,0,0}\makebox(0,0)[lt]{\lineheight{1.25}\smash{\begin{tabular}[t]{l}{\tiny 4}\end{tabular}}}}%
    \put(0,0){\includegraphics[width=\unitlength,page=5]{stable-3-tangle-singularities-2.pdf}}%
    \put(0.14791667,0.08680556){\color[rgb]{0,0,0}\makebox(0,0)[lt]{\lineheight{1.25}\smash{\begin{tabular}[t]{l}{\tiny 4}\end{tabular}}}}%
    \put(0,0){\includegraphics[width=\unitlength,page=6]{stable-3-tangle-singularities-2.pdf}}%
    \put(0.30381944,0.08680556){\color[rgb]{0,0,0}\makebox(0,0)[lt]{\lineheight{1.25}\smash{\begin{tabular}[t]{l}{\tiny 4}\end{tabular}}}}%
    \put(0,0){\includegraphics[width=\unitlength,page=7]{stable-3-tangle-singularities-2.pdf}}%
    \put(0.14791667,0.23611111){\color[rgb]{0,0,0}\makebox(0,0)[lt]{\lineheight{1.25}\smash{\begin{tabular}[t]{l}{\tiny 4}\end{tabular}}}}%
    \put(0,0){\includegraphics[width=\unitlength,page=8]{stable-3-tangle-singularities-2.pdf}}%
    \put(0.30381944,0.23611111){\color[rgb]{0,0,0}\makebox(0,0)[lt]{\lineheight{1.25}\smash{\begin{tabular}[t]{l}{\tiny 4}\end{tabular}}}}%
    \put(0,0){\includegraphics[width=\unitlength,page=9]{stable-3-tangle-singularities-2.pdf}}%
    \put(0.68472222,0.68159722){\color[rgb]{0,0,0}\makebox(0,0)[lt]{\lineheight{1.25}\smash{\begin{tabular}[t]{l}{\tiny 4}\end{tabular}}}}%
    \put(0,0){\includegraphics[width=\unitlength,page=10]{stable-3-tangle-singularities-2.pdf}}%
    \put(0.840625,0.68159722){\color[rgb]{0,0,0}\makebox(0,0)[lt]{\lineheight{1.25}\smash{\begin{tabular}[t]{l}{\tiny 4}\end{tabular}}}}%
    \put(0,0){\includegraphics[width=\unitlength,page=11]{stable-3-tangle-singularities-2.pdf}}%
    \put(0.68472222,0.53159722){\color[rgb]{0,0,0}\makebox(0,0)[lt]{\lineheight{1.25}\smash{\begin{tabular}[t]{l}{\tiny 4}\end{tabular}}}}%
    \put(0,0){\includegraphics[width=\unitlength,page=12]{stable-3-tangle-singularities-2.pdf}}%
    \put(0.840625,0.53159722){\color[rgb]{0,0,0}\makebox(0,0)[lt]{\lineheight{1.25}\smash{\begin{tabular}[t]{l}{\tiny 4}\end{tabular}}}}%
    \put(0,0){\includegraphics[width=\unitlength,page=13]{stable-3-tangle-singularities-2.pdf}}%
    \put(0.68472222,0.08819444){\color[rgb]{0,0,0}\makebox(0,0)[lt]{\lineheight{1.25}\smash{\begin{tabular}[t]{l}{\tiny 4}\end{tabular}}}}%
    \put(0,0){\includegraphics[width=\unitlength,page=14]{stable-3-tangle-singularities-2.pdf}}%
    \put(0.840625,0.08819444){\color[rgb]{0,0,0}\makebox(0,0)[lt]{\lineheight{1.25}\smash{\begin{tabular}[t]{l}{\tiny 4}\end{tabular}}}}%
    \put(0,0){\includegraphics[width=\unitlength,page=15]{stable-3-tangle-singularities-2.pdf}}%
    \put(0.68472222,0.2375){\color[rgb]{0,0,0}\makebox(0,0)[lt]{\lineheight{1.25}\smash{\begin{tabular}[t]{l}{\tiny 4}\end{tabular}}}}%
    \put(0,0){\includegraphics[width=\unitlength,page=16]{stable-3-tangle-singularities-2.pdf}}%
    \put(0.840625,0.2375){\color[rgb]{0,0,0}\makebox(0,0)[lt]{\lineheight{1.25}\smash{\begin{tabular}[t]{l}{\tiny 4}\end{tabular}}}}%
    \put(0,0){\includegraphics[width=\unitlength,page=17]{stable-3-tangle-singularities-2.pdf}}%
    \put(0.14791667,0.38194444){\color[rgb]{0,0,0}\makebox(0,0)[lt]{\lineheight{1.25}\smash{\begin{tabular}[t]{l}{\tiny 4}\end{tabular}}}}%
    \put(0,0){\includegraphics[width=\unitlength,page=18]{stable-3-tangle-singularities-2.pdf}}%
    \put(0.30381944,0.38194444){\color[rgb]{0,0,0}\makebox(0,0)[lt]{\lineheight{1.25}\smash{\begin{tabular}[t]{l}{\tiny 4}\end{tabular}}}}%
    \put(0,0){\includegraphics[width=\unitlength,page=19]{stable-3-tangle-singularities-2.pdf}}%
    \put(0.68472222,0.38333333){\color[rgb]{0,0,0}\makebox(0,0)[lt]{\lineheight{1.25}\smash{\begin{tabular}[t]{l}{\tiny 4}\end{tabular}}}}%
    \put(0,0){\includegraphics[width=\unitlength,page=20]{stable-3-tangle-singularities-2.pdf}}%
    \put(0.840625,0.38333333){\color[rgb]{0,0,0}\makebox(0,0)[lt]{\lineheight{1.25}\smash{\begin{tabular}[t]{l}{\tiny 4}\end{tabular}}}}%
    \put(0,0){\includegraphics[width=\unitlength,page=21]{stable-3-tangle-singularities-2.pdf}}%
    \put(0.04479167,0.75){\color[rgb]{0,0,0}\makebox(0,0)[lt]{\lineheight{1.25}\smash{\begin{tabular}[t]{l}{\tiny 1}\end{tabular}}}}%
    \put(0.05347222,0.76493056){\color[rgb]{0,0,0}\makebox(0,0)[lt]{\lineheight{1.25}\smash{\begin{tabular}[t]{l}{\tiny 2}\end{tabular}}}}%
    \put(0.02118056,0.72986111){\color[rgb]{0,0,0}\makebox(0,0)[lt]{\lineheight{1.25}\smash{\begin{tabular}[t]{l}{\tiny 3}\end{tabular}}}}%
    \put(0,0){\includegraphics[width=\unitlength,page=22]{stable-3-tangle-singularities-2.pdf}}%
  \end{picture}%
\endgroup%

%% file: figuresused/seeing-the-a2-shaped-in-3-tangle-singularities.pdf_tex
\begingroup%
  \makeatletter%
  \providecommand\color[2][]{%
    \errmessage{(Inkscape) Color is used for the text in Inkscape, but the package 'color.sty' is not loaded}%
    \renewcommand\color[2][]{}%
  }%
  \providecommand\transparent[1]{%
    \errmessage{(Inkscape) Transparency is used (non-zero) for the text in Inkscape, but the package 'transparent.sty' is not loaded}%
    \renewcommand\transparent[1]{}%
  }%
  \providecommand\rotatebox[2]{#2}%
  \newcommand*\fsize{\dimexpr\f@size pt\relax}%
  \newcommand*\lineheight[1]{\fontsize{\fsize}{#1\fsize}\selectfont}%
  \ifx\svgwidth\undefined%
    \setlength{\unitlength}{2160bp}%
    \ifx\svgscale\undefined%
      \relax%
    \else%
      \setlength{\unitlength}{\unitlength * \real{\svgscale}}%
    \fi%
  \else%
    \setlength{\unitlength}{\svgwidth}%
  \fi%
  \global\let\svgwidth\undefined%
  \global\let\svgscale\undefined%
  \makeatother%
  \begin{picture}(1,0.49826389)%
    \lineheight{1}%
    \setlength\tabcolsep{0pt}%
    \put(0,0){\includegraphics[width=\unitlength,page=1]{seeing-the-a2-shaped-in-3-tangle-singularities.pdf}}%
    \put(0.48993056,0.26875){\color[rgb]{0,0,0}\makebox(0,0)[lt]{\lineheight{1.25}\smash{\begin{tabular}[t]{l}$\iA_1$\end{tabular}}}}%
    \put(0.496875,0.00243056){\color[rgb]{0,0,0}\makebox(0,0)[lt]{\lineheight{1.25}\smash{\begin{tabular}[t]{l}$\iA_2$\end{tabular}}}}%
    \put(0,0){\includegraphics[width=\unitlength,page=2]{seeing-the-a2-shaped-in-3-tangle-singularities.pdf}}%
    \put(0.14930556,0.13055556){\color[rgb]{0,0,0}\makebox(0,0)[lt]{\lineheight{1.25}\smash{\begin{tabular}[t]{l}{\tiny 1}\end{tabular}}}}%
    \put(0.15833333,0.14583333){\color[rgb]{0,0,0}\makebox(0,0)[lt]{\lineheight{1.25}\smash{\begin{tabular}[t]{l}{\tiny 2}\end{tabular}}}}%
    \put(0.12534722,0.11006944){\color[rgb]{0,0,0}\makebox(0,0)[lt]{\lineheight{1.25}\smash{\begin{tabular}[t]{l}{\tiny 3}\end{tabular}}}}%
    \put(0,0){\includegraphics[width=\unitlength,page=3]{seeing-the-a2-shaped-in-3-tangle-singularities.pdf}}%
    \put(0.14930556,0.33472222){\color[rgb]{0,0,0}\makebox(0,0)[lt]{\lineheight{1.25}\smash{\begin{tabular}[t]{l}{\tiny 1}\end{tabular}}}}%
    \put(0.15833333,0.35){\color[rgb]{0,0,0}\makebox(0,0)[lt]{\lineheight{1.25}\smash{\begin{tabular}[t]{l}{\tiny 2}\end{tabular}}}}%
    \put(0.12534722,0.31423611){\color[rgb]{0,0,0}\makebox(0,0)[lt]{\lineheight{1.25}\smash{\begin{tabular}[t]{l}{\tiny 3}\end{tabular}}}}%
    \put(0,0){\includegraphics[width=\unitlength,page=4]{seeing-the-a2-shaped-in-3-tangle-singularities.pdf}}%
    \put(0.43020833,0.40625){\color[rgb]{0,0,0}\makebox(0,0)[lt]{\lineheight{1.25}\smash{\begin{tabular}[t]{l}{\tiny 1}\end{tabular}}}}%
    \put(0.40138889,0.37222222){\color[rgb]{0,0,0}\makebox(0,0)[lt]{\lineheight{1.25}\smash{\begin{tabular}[t]{l}{\tiny 2}\end{tabular}}}}%
    \put(0,0){\includegraphics[width=\unitlength,page=5]{seeing-the-a2-shaped-in-3-tangle-singularities.pdf}}%
    \put(0.14930556,0.47916667){\color[rgb]{0,0,0}\makebox(0,0)[lt]{\lineheight{1.25}\smash{\begin{tabular}[t]{l}{\tiny 1}\end{tabular}}}}%
    \put(0.15833333,0.49444444){\color[rgb]{0,0,0}\makebox(0,0)[lt]{\lineheight{1.25}\smash{\begin{tabular}[t]{l}{\tiny 2}\end{tabular}}}}%
    \put(0.12534722,0.45868056){\color[rgb]{0,0,0}\makebox(0,0)[lt]{\lineheight{1.25}\smash{\begin{tabular}[t]{l}{\tiny 3}\end{tabular}}}}%
    \put(0,0){\includegraphics[width=\unitlength,page=6]{seeing-the-a2-shaped-in-3-tangle-singularities.pdf}}%
    \put(0.41423611,0.12951389){\color[rgb]{0,0,0}\makebox(0,0)[lt]{\lineheight{1.25}\smash{\begin{tabular}[t]{l}{\tiny 1}\end{tabular}}}}%
    \put(0.42326389,0.14479167){\color[rgb]{0,0,0}\makebox(0,0)[lt]{\lineheight{1.25}\smash{\begin{tabular}[t]{l}{\tiny 2}\end{tabular}}}}%
    \put(0.39027778,0.10902778){\color[rgb]{0,0,0}\makebox(0,0)[lt]{\lineheight{1.25}\smash{\begin{tabular}[t]{l}{\tiny 3}\end{tabular}}}}%
    \put(0,0){\includegraphics[width=\unitlength,page=7]{seeing-the-a2-shaped-in-3-tangle-singularities.pdf}}%
    \put(0.678125,0.12951389){\color[rgb]{0,0,0}\makebox(0,0)[lt]{\lineheight{1.25}\smash{\begin{tabular}[t]{l}{\tiny 1}\end{tabular}}}}%
    \put(0.68715278,0.14479167){\color[rgb]{0,0,0}\makebox(0,0)[lt]{\lineheight{1.25}\smash{\begin{tabular}[t]{l}{\tiny 2}\end{tabular}}}}%
    \put(0.65416667,0.10902778){\color[rgb]{0,0,0}\makebox(0,0)[lt]{\lineheight{1.25}\smash{\begin{tabular}[t]{l}{\tiny 3}\end{tabular}}}}%
    \put(0,0){\includegraphics[width=\unitlength,page=8]{seeing-the-a2-shaped-in-3-tangle-singularities.pdf}}%
    \put(0.678125,0.39340278){\color[rgb]{0,0,0}\makebox(0,0)[lt]{\lineheight{1.25}\smash{\begin{tabular}[t]{l}{\tiny 1}\end{tabular}}}}%
    \put(0.68715278,0.40868056){\color[rgb]{0,0,0}\makebox(0,0)[lt]{\lineheight{1.25}\smash{\begin{tabular}[t]{l}{\tiny 2}\end{tabular}}}}%
    \put(0.65416667,0.37291667){\color[rgb]{0,0,0}\makebox(0,0)[lt]{\lineheight{1.25}\smash{\begin{tabular}[t]{l}{\tiny 3}\end{tabular}}}}%
    \put(0,0){\includegraphics[width=\unitlength,page=9]{seeing-the-a2-shaped-in-3-tangle-singularities.pdf}}%
  \end{picture}%
\endgroup%

%% file: figuresused/the-u-singularity.pdf_tex
\begingroup%
  \makeatletter%
  \providecommand\color[2][]{%
    \errmessage{(Inkscape) Color is used for the text in Inkscape, but the package 'color.sty' is not loaded}%
    \renewcommand\color[2][]{}%
  }%
  \providecommand\transparent[1]{%
    \errmessage{(Inkscape) Transparency is used (non-zero) for the text in Inkscape, but the package 'transparent.sty' is not loaded}%
    \renewcommand\transparent[1]{}%
  }%
  \providecommand\rotatebox[2]{#2}%
  \newcommand*\fsize{\dimexpr\f@size pt\relax}%
  \newcommand*\lineheight[1]{\fontsize{\fsize}{#1\fsize}\selectfont}%
  \ifx\svgwidth\undefined%
    \setlength{\unitlength}{2160bp}%
    \ifx\svgscale\undefined%
      \relax%
    \else%
      \setlength{\unitlength}{\unitlength * \real{\svgscale}}%
    \fi%
  \else%
    \setlength{\unitlength}{\svgwidth}%
  \fi%
  \global\let\svgwidth\undefined%
  \global\let\svgscale\undefined%
  \makeatother%
  \begin{picture}(1,0.11041667)%
    \lineheight{1}%
    \setlength\tabcolsep{0pt}%
    \put(0,0){\includegraphics[width=\unitlength,page=1]{the-u-singularity.pdf}}%
    \put(0.41388889,0.06180556){\color[rgb]{0,0,0}\makebox(0,0)[lt]{\lineheight{1.25}\smash{\begin{tabular}[t]{l}{\tiny 4}\end{tabular}}}}%
    \put(0,0){\includegraphics[width=\unitlength,page=2]{the-u-singularity.pdf}}%
    \put(0.57361111,0.06180556){\color[rgb]{0,0,0}\makebox(0,0)[lt]{\lineheight{1.25}\smash{\begin{tabular}[t]{l}{\tiny 4}\end{tabular}}}}%
    \put(0,0){\includegraphics[width=\unitlength,page=3]{the-u-singularity.pdf}}%
    \put(0.24548611,0.08993056){\color[rgb]{0,0,0}\makebox(0,0)[lt]{\lineheight{1.25}\smash{\begin{tabular}[t]{l}{\tiny 1}\end{tabular}}}}%
    \put(0.25451389,0.10520833){\color[rgb]{0,0,0}\makebox(0,0)[lt]{\lineheight{1.25}\smash{\begin{tabular}[t]{l}{\tiny 2}\end{tabular}}}}%
    \put(0.22152778,0.06944444){\color[rgb]{0,0,0}\makebox(0,0)[lt]{\lineheight{1.25}\smash{\begin{tabular}[t]{l}{\tiny 3}\end{tabular}}}}%
  \end{picture}%
\endgroup%

%% file: figuresused/denoting-3-tangle-singularities-by-their-projections.pdf_tex
\begingroup%
  \makeatletter%
  \providecommand\color[2][]{%
    \errmessage{(Inkscape) Color is used for the text in Inkscape, but the package 'color.sty' is not loaded}%
    \renewcommand\color[2][]{}%
  }%
  \providecommand\transparent[1]{%
    \errmessage{(Inkscape) Transparency is used (non-zero) for the text in Inkscape, but the package 'transparent.sty' is not loaded}%
    \renewcommand\transparent[1]{}%
  }%
  \providecommand\rotatebox[2]{#2}%
  \newcommand*\fsize{\dimexpr\f@size pt\relax}%
  \newcommand*\lineheight[1]{\fontsize{\fsize}{#1\fsize}\selectfont}%
  \ifx\svgwidth\undefined%
    \setlength{\unitlength}{2160bp}%
    \ifx\svgscale\undefined%
      \relax%
    \else%
      \setlength{\unitlength}{\unitlength * \real{\svgscale}}%
    \fi%
  \else%
    \setlength{\unitlength}{\svgwidth}%
  \fi%
  \global\let\svgwidth\undefined%
  \global\let\svgscale\undefined%
  \makeatother%
  \begin{picture}(1,0.41319444)%
    \lineheight{1}%
    \setlength\tabcolsep{0pt}%
    \put(0,0){\includegraphics[width=\unitlength,page=1]{denoting-3-tangle-singularities-by-their-projections.pdf}}%
    \put(0.290625,0.38784722){\color[rgb]{0,0,0}\makebox(0,0)[lt]{\lineheight{1.25}\smash{\begin{tabular}[t]{l}{\tiny 4}\end{tabular}}}}%
    \put(0,0){\includegraphics[width=\unitlength,page=2]{denoting-3-tangle-singularities-by-their-projections.pdf}}%
    \put(0.290625,0.27430556){\color[rgb]{0,0,0}\makebox(0,0)[lt]{\lineheight{1.25}\smash{\begin{tabular}[t]{l}{\tiny 4}\end{tabular}}}}%
    \put(0,0){\includegraphics[width=\unitlength,page=3]{denoting-3-tangle-singularities-by-their-projections.pdf}}%
    \put(0.290625,0.16076389){\color[rgb]{0,0,0}\makebox(0,0)[lt]{\lineheight{1.25}\smash{\begin{tabular}[t]{l}{\tiny 4}\end{tabular}}}}%
    \put(0,0){\includegraphics[width=\unitlength,page=4]{denoting-3-tangle-singularities-by-their-projections.pdf}}%
    \put(0.70520833,0.38784722){\color[rgb]{0,0,0}\makebox(0,0)[lt]{\lineheight{1.25}\smash{\begin{tabular}[t]{l}{\tiny 4}\end{tabular}}}}%
    \put(0,0){\includegraphics[width=\unitlength,page=5]{denoting-3-tangle-singularities-by-their-projections.pdf}}%
    \put(0.70520833,0.27430556){\color[rgb]{0,0,0}\makebox(0,0)[lt]{\lineheight{1.25}\smash{\begin{tabular}[t]{l}{\tiny 4}\end{tabular}}}}%
    \put(0,0){\includegraphics[width=\unitlength,page=6]{denoting-3-tangle-singularities-by-their-projections.pdf}}%
    \put(0.70520833,0.16076389){\color[rgb]{0,0,0}\makebox(0,0)[lt]{\lineheight{1.25}\smash{\begin{tabular}[t]{l}{\tiny 4}\end{tabular}}}}%
    \put(0,0){\includegraphics[width=\unitlength,page=7]{denoting-3-tangle-singularities-by-their-projections.pdf}}%
    \put(0.50034722,0.04930556){\color[rgb]{0,0,0}\makebox(0,0)[lt]{\lineheight{1.25}\smash{\begin{tabular}[t]{l}{\tiny 4}\end{tabular}}}}%
    \put(0,0){\includegraphics[width=\unitlength,page=8]{denoting-3-tangle-singularities-by-their-projections.pdf}}%
    \put(0.12986111,0.40729167){\color[rgb]{0,0,0}\makebox(0,0)[lt]{\lineheight{1.25}\smash{\begin{tabular}[t]{l}{\tiny 2}\end{tabular}}}}%
    \put(0.09618056,0.371875){\color[rgb]{0,0,0}\makebox(0,0)[lt]{\lineheight{1.25}\smash{\begin{tabular}[t]{l}{\tiny 3}\end{tabular}}}}%
    \put(0,0){\includegraphics[width=\unitlength,page=9]{denoting-3-tangle-singularities-by-their-projections.pdf}}%
  \end{picture}%
\endgroup%

%% file: figuresused/the-source-of-the-d3-law.pdf_tex
\begingroup%
  \makeatletter%
  \providecommand\color[2][]{%
    \errmessage{(Inkscape) Color is used for the text in Inkscape, but the package 'color.sty' is not loaded}%
    \renewcommand\color[2][]{}%
  }%
  \providecommand\transparent[1]{%
    \errmessage{(Inkscape) Transparency is used (non-zero) for the text in Inkscape, but the package 'transparent.sty' is not loaded}%
    \renewcommand\transparent[1]{}%
  }%
  \providecommand\rotatebox[2]{#2}%
  \newcommand*\fsize{\dimexpr\f@size pt\relax}%
  \newcommand*\lineheight[1]{\fontsize{\fsize}{#1\fsize}\selectfont}%
  \ifx\svgwidth\undefined%
    \setlength{\unitlength}{2160bp}%
    \ifx\svgscale\undefined%
      \relax%
    \else%
      \setlength{\unitlength}{\unitlength * \real{\svgscale}}%
    \fi%
  \else%
    \setlength{\unitlength}{\svgwidth}%
  \fi%
  \global\let\svgwidth\undefined%
  \global\let\svgscale\undefined%
  \makeatother%
  \begin{picture}(1,0.36666667)%
    \lineheight{1}%
    \setlength\tabcolsep{0pt}%
    \put(0,0){\includegraphics[width=\unitlength,page=1]{the-source-of-the-d3-law.pdf}}%
    \put(0.14131944,0.3625){\color[rgb]{0,0,0}\makebox(0,0)[lt]{\lineheight{1.25}\smash{\begin{tabular}[t]{l}{\tiny 2}\end{tabular}}}}%
    \put(0.10763889,0.32708333){\color[rgb]{0,0,0}\makebox(0,0)[lt]{\lineheight{1.25}\smash{\begin{tabular}[t]{l}{\tiny 3}\end{tabular}}}}%
    \put(0,0){\includegraphics[width=\unitlength,page=2]{the-source-of-the-d3-law.pdf}}%
    \put(0.23784722,0.2375){\color[rgb]{0,0,0}\makebox(0,0)[lt]{\lineheight{1.25}\smash{\begin{tabular}[t]{l}{\tiny 4}\end{tabular}}}}%
    \put(0.40833333,0.16875){\color[rgb]{0,0,0}\makebox(0,0)[lt]{\lineheight{1.25}\smash{\begin{tabular}[t]{l}{\tiny 4}\end{tabular}}}}%
    \put(0.58159722,0.09895833){\color[rgb]{0,0,0}\makebox(0,0)[lt]{\lineheight{1.25}\smash{\begin{tabular}[t]{l}{\tiny 4}\end{tabular}}}}%
    \put(0.75451389,0.02847222){\color[rgb]{0,0,0}\makebox(0,0)[lt]{\lineheight{1.25}\smash{\begin{tabular}[t]{l}{\tiny 4}\end{tabular}}}}%
    \put(0,0){\includegraphics[width=\unitlength,page=3]{the-source-of-the-d3-law.pdf}}%
    \put(0.18055556,0.16770833){\color[rgb]{0,0,0}\makebox(0,0)[lt]{\lineheight{1.25}\smash{\begin{tabular}[t]{l}{\tiny initial condition}\end{tabular}}}}%
    \put(0,0){\includegraphics[width=\unitlength,page=4]{the-source-of-the-d3-law.pdf}}%
    \put(0.09652778,0.11701389){\color[rgb]{0,0,0}\makebox(0,0)[lt]{\lineheight{1.25}\smash{\begin{tabular}[t]{l}{\tiny 1}\end{tabular}}}}%
    \put(0.10555556,0.13229167){\color[rgb]{0,0,0}\makebox(0,0)[lt]{\lineheight{1.25}\smash{\begin{tabular}[t]{l}{\tiny 2}\end{tabular}}}}%
    \put(0.07256944,0.09652778){\color[rgb]{0,0,0}\makebox(0,0)[lt]{\lineheight{1.25}\smash{\begin{tabular}[t]{l}{\tiny 3}\end{tabular}}}}%
  \end{picture}%
\endgroup%

%% file: figuresused/the-source-of-the-d4-singularity.pdf_tex
\begingroup%
  \makeatletter%
  \providecommand\color[2][]{%
    \errmessage{(Inkscape) Color is used for the text in Inkscape, but the package 'color.sty' is not loaded}%
    \renewcommand\color[2][]{}%
  }%
  \providecommand\transparent[1]{%
    \errmessage{(Inkscape) Transparency is used (non-zero) for the text in Inkscape, but the package 'transparent.sty' is not loaded}%
    \renewcommand\transparent[1]{}%
  }%
  \providecommand\rotatebox[2]{#2}%
  \newcommand*\fsize{\dimexpr\f@size pt\relax}%
  \newcommand*\lineheight[1]{\fontsize{\fsize}{#1\fsize}\selectfont}%
  \ifx\svgwidth\undefined%
    \setlength{\unitlength}{3120bp}%
    \ifx\svgscale\undefined%
      \relax%
    \else%
      \setlength{\unitlength}{\unitlength * \real{\svgscale}}%
    \fi%
  \else%
    \setlength{\unitlength}{\svgwidth}%
  \fi%
  \global\let\svgwidth\undefined%
  \global\let\svgscale\undefined%
  \makeatother%
  \begin{picture}(1,1.35096154)%
    \lineheight{1}%
    \setlength\tabcolsep{0pt}%
    \put(0,0){\includegraphics[width=\unitlength,page=1]{the-source-of-the-d4-singularity.pdf}}%
    \put(0.94278846,0.30721154){\makebox(0,0)[lt]{\lineheight{1.25}\smash{\begin{tabular}[t]{l}{\tiny$\singa {\iA_1} {3}{\iA_2}$}\end{tabular}}}}%
    \put(0,0){\includegraphics[width=\unitlength,page=2]{the-source-of-the-d4-singularity.pdf}}%
    \put(0.94302885,0.27379808){\color[rgb]{0,0,0}\makebox(0,0)[lt]{\lineheight{1.25}\smash{\begin{tabular}[t]{l}{\tiny$\iA_2^{\numovar 2}$}\end{tabular}}}}%
    \put(0,0){\includegraphics[width=\unitlength,page=3]{the-source-of-the-d4-singularity.pdf}}%
    \put(0.94302885,0.24350962){\color[rgb]{0,0,0}\makebox(0,0)[lt]{\lineheight{1.25}\smash{\begin{tabular}[t]{l}{\tiny$\iA_1^{\numovar 3}$}\end{tabular}}}}%
    \put(0,0){\includegraphics[width=\unitlength,page=4]{the-source-of-the-d4-singularity.pdf}}%
    \put(0.94302885,0.21514423){\color[rgb]{0,0,0}\makebox(0,0)[lt]{\lineheight{1.25}\smash{\begin{tabular}[t]{l}{\tiny$\iD_2$}\end{tabular}}}}%
    \put(0,0){\includegraphics[width=\unitlength,page=5]{the-source-of-the-d4-singularity.pdf}}%
    \put(0.94278846,0.18966346){\makebox(0,0)[lt]{\lineheight{1.25}\smash{\begin{tabular}[t]{l}{\tiny$\singa {\iA_3} {(4,3,2)}{\iA_2^{\numovar 2}}$}\end{tabular}}}}%
    \put(0,0){\includegraphics[width=\unitlength,page=6]{the-source-of-the-d4-singularity.pdf}}%
    \put(0.94278846,0.15961538){\makebox(0,0)[lt]{\lineheight{1.25}\smash{\begin{tabular}[t]{l}{\tiny$\singa {\iA_2}{(4,3)} {\iA_1^{\numovar 3}}$}\end{tabular}}}}%
    \put(0,0){\includegraphics[width=\unitlength,page=7]{the-source-of-the-d4-singularity.pdf}}%
    \put(0.94278846,0.12956731){\makebox(0,0)[lt]{\lineheight{1.25}\smash{\begin{tabular}[t]{l}{\tiny$\singb {\iA_1}{4} {\iA_1} {3} {\iA_2}$}\end{tabular}}}}%
    \put(0,0){\includegraphics[width=\unitlength,page=8]{the-source-of-the-d4-singularity.pdf}}%
    \put(0.94278846,0.09927885){\makebox(0,0)[lt]{\lineheight{1.25}\smash{\begin{tabular}[t]{l}{\tiny$\singa {\iA_1} {4}{\iA_2^{\numovar 2}}$}\end{tabular}}}}%
    \put(0,0){\includegraphics[width=\unitlength,page=9]{the-source-of-the-d4-singularity.pdf}}%
    \put(0.94278846,0.06923077){\makebox(0,0)[lt]{\lineheight{1.25}\smash{\begin{tabular}[t]{l}{\tiny$\singa {\iA_1} {4}{\iA_1^{\numovar 3}}$}\end{tabular}}}}%
    \put(0,0){\includegraphics[width=\unitlength,page=10]{the-source-of-the-d4-singularity.pdf}}%
    \put(0.94278846,0.03413462){\makebox(0,0)[lt]{\lineheight{1.25}\smash{\begin{tabular}[t]{l}{\tiny$\singa {\iA_1} {4}{\iD_2}$}\end{tabular}}}}%
    \put(0,0){\includegraphics[width=\unitlength,page=11]{the-source-of-the-d4-singularity.pdf}}%
    \put(0.94302885,0.00456731){\color[rgb]{0,0,0}\makebox(0,0)[lt]{\lineheight{1.25}\smash{\begin{tabular}[t]{l}{\tiny$\iD_3$}\end{tabular}}}}%
    \put(0,0){\includegraphics[width=\unitlength,page=12]{the-source-of-the-d4-singularity.pdf}}%
    \put(0.03605769,1.15096154){\color[rgb]{0,0,0}\makebox(0,0)[lt]{\lineheight{1.25}\smash{\begin{tabular}[t]{l}{\Large 5}\end{tabular}}}}%
    \put(0.14278846,1.29807692){\color[rgb]{0,0,0}\makebox(0,0)[lt]{\lineheight{1.25}\smash{\begin{tabular}[t]{l}{\Large 4}\end{tabular}}}}%
  \end{picture}%
\endgroup%

%% file: figuresused/qualitative-content-of-d4-singularity.pdf_tex
\begingroup%
  \makeatletter%
  \providecommand\color[2][]{%
    \errmessage{(Inkscape) Color is used for the text in Inkscape, but the package 'color.sty' is not loaded}%
    \renewcommand\color[2][]{}%
  }%
  \providecommand\transparent[1]{%
    \errmessage{(Inkscape) Transparency is used (non-zero) for the text in Inkscape, but the package 'transparent.sty' is not loaded}%
    \renewcommand\transparent[1]{}%
  }%
  \providecommand\rotatebox[2]{#2}%
  \newcommand*\fsize{\dimexpr\f@size pt\relax}%
  \newcommand*\lineheight[1]{\fontsize{\fsize}{#1\fsize}\selectfont}%
  \ifx\svgwidth\undefined%
    \setlength{\unitlength}{2160bp}%
    \ifx\svgscale\undefined%
      \relax%
    \else%
      \setlength{\unitlength}{\unitlength * \real{\svgscale}}%
    \fi%
  \else%
    \setlength{\unitlength}{\svgwidth}%
  \fi%
  \global\let\svgwidth\undefined%
  \global\let\svgscale\undefined%
  \makeatother%
  \begin{picture}(1,0.26041667)%
    \lineheight{1}%
    \setlength\tabcolsep{0pt}%
    \put(0,0){\includegraphics[width=\unitlength,page=1]{qualitative-content-of-d4-singularity.pdf}}%
    \put(0.16805556,0.22916667){\color[rgb]{0,0,0}\makebox(0,0)[lt]{\lineheight{1.25}\smash{\begin{tabular}[t]{l}{\tiny 4}\end{tabular}}}}%
    \put(0.19375,0.16458333){\color[rgb]{0,0,0}\makebox(0,0)[lt]{\lineheight{1.25}\smash{\begin{tabular}[t]{l}{\tiny 5}\end{tabular}}}}%
    \put(0.134375,0.12673611){\color[rgb]{0,0,0}\makebox(0,0)[lt]{\lineheight{1.25}\smash{\begin{tabular}[t]{l}{\tiny 6}\end{tabular}}}}%
    \put(0,0){\includegraphics[width=\unitlength,page=2]{qualitative-content-of-d4-singularity.pdf}}%
    \put(0.65798611,0.22916667){\color[rgb]{0,0,0}\makebox(0,0)[lt]{\lineheight{1.25}\smash{\begin{tabular}[t]{l}{\tiny 4}\end{tabular}}}}%
    \put(0.68368056,0.16458333){\color[rgb]{0,0,0}\makebox(0,0)[lt]{\lineheight{1.25}\smash{\begin{tabular}[t]{l}{\tiny 5}\end{tabular}}}}%
    \put(0.62430556,0.12673611){\color[rgb]{0,0,0}\makebox(0,0)[lt]{\lineheight{1.25}\smash{\begin{tabular}[t]{l}{\tiny 6}\end{tabular}}}}%
    \put(0,0){\includegraphics[width=\unitlength,page=3]{qualitative-content-of-d4-singularity.pdf}}%
    \put(0.32913751,0.03645833){\color[rgb]{0,0,0}\makebox(0,0)[lt]{\lineheight{1.25}\smash{\begin{tabular}[t]{l}$\iD_4$\end{tabular}}}}%
    \put(0.81423611,0.03645833){\color[rgb]{0,0,0}\makebox(0,0)[lt]{\lineheight{1.25}\smash{\begin{tabular}[t]{l}$\singa {\iA_1} 5 {\iD_3}$\end{tabular}}}}%
  \end{picture}%
\endgroup%